\documentclass[a4paper,12pt]{amsart}
\usepackage{amsmath, amscd, amssymb, mathrsfs, mathtools, fullpage, enumerate, svninfo,color,enumitem}
\usepackage[all]{xy}
\usepackage{yfonts}
\usepackage{stmaryrd}
\usepackage[pageanchor]{hyperref}
\usepackage{nomencl}

\title{Eisenstein-Kronecker classes, integrality of critical values of Hecke  $L$-functions and $p$-adic interpolation}

\author{Guido Kings}
\address{Fakult\"at f\"ur Mathematik \\
Universit\"at Regensburg\\
93040 Regensburg\\
Germany}
\author{Johannes Sprang}
\address{
Fakultät für Mathematik\\
Universität Duisburg-Essen\\
45127 Essen\\
Germany}

\thanks{This research was supported by the DFG grant: SFB 1085 Higher invariants}

\theoremstyle{plain}
    \newtheorem{theorem}{Theorem}[section]
\newtheorem*{theorem*}{Theorem}
    \newtheorem{lemma}[theorem]{Lemma}
    \newtheorem{proposition}[theorem]{Proposition}
    \newtheorem{corollary}[theorem]{Corollary}
    \newtheorem*{corollary*}{Corollary}
    \newtheorem{notation}[theorem]{Notation}

\theoremstyle{remark}
    \newtheorem{remark}[theorem]{Remark}

\theoremstyle{definition}
    \newtheorem{definition}[theorem]{Definition}

\DeclareMathOperator{\TSym}{TSym}
\DeclareMathOperator{\Sym}{Sym}
\DeclareMathOperator{\Hom}{Hom}
\DeclareMathOperator{\sHom}{\underline{\Hom}}
\DeclareMathOperator{\End}{End}
\DeclareMathOperator{\Ext}{Ext}
\DeclareMathOperator{\sExt}{\underline{\Ext}}

\DeclareMathOperator{\Spec}{Spec}

\DeclareMathOperator{\GL}{GL}

\DeclareMathOperator{\pr}{pr}

\DeclareMathOperator{\mom}{mom}

\newcommand{\norm}{N}

\DeclareMathOperator{\Gal}{Gal}

\DeclareMathOperator{\Map}{Map}

\newcommand{\one}[1]{\mathbf{1}^{(#1)}}
\newcommand{\bfone}{\mathbf{1}}
\DeclareMathOperator{\Lie}{Lie}

\DeclareMathOperator{\Res}{Res}
\DeclareMathOperator{\Aut}{Aut}

\DeclareMathOperator{\Inf}{Inf}

\DeclareMathOperator{\vol}{vol}

\newcommand{\et}{\text{\textup{\'et}}}

\newcommand{\ord}{\mathrm{ord}}

\newcommand{\res}{\mathrm{res}}
\newcommand{\cor}{\mathrm{cor}}

\newcommand{\sD}{\mathscr{D}}

\newcommand{\sF}{\mathscr{F}}
\newcommand{\sG}{\mathscr{G}}
\newcommand{\sH}{\mathscr{H}}
\newcommand{\sI}{\mathscr{I}}
\newcommand{\sJ}{\mathscr{J}}

\newcommand{\sL}{\mathscr{L}}

\newcommand{\sO}{\mathscr{O}}

\newcommand{\sP}{\mathscr{P}}

\newcommand{\cA}{\mathcal{A}}
\newcommand{\cB}{\mathcal{B}}
\newcommand{\cC}{\mathcal{C}}
\newcommand{\cD}{\mathcal{D}}
\newcommand{\cE}{\mathcal{E}}
\newcommand{\cF}{\mathcal{F}}
\newcommand{\cG}{\mathcal{G}}
\newcommand{\cH}{\mathcal{H}}
\newcommand{\cI}{\mathcal{I}}

\newcommand{\cL}{\mathcal{L}}

\newcommand{\cO}{\mathcal{O}}
\newcommand{\cP}{\mathcal{P}}

\newcommand{\cR}{\mathcal{R}}
\newcommand{\cS}{\mathcal{S}}
\newcommand{\cT}{\mathcal{T}}
\newcommand{\cU}{\mathcal{U}}

\newcommand{\cX}{\mathcal{X}}

\newcommand{\fra}{\mathfrak{a}}
\newcommand{\frb}{\mathfrak{b}}
\newcommand{\frc}{\mathfrak{c}}
\newcommand{\frd}{\mathfrak{d}}
\newcommand{\frf}{\mathfrak{f}}

\newcommand{\frmm}{\mathfrak{m}}

\newcommand{\frp}{\mathfrak{p}}

\newcommand{\CC}{\mathbb{C}}

\newcommand{\QQ}{\mathbb{Q}}
\newcommand{\RR}{\mathbb{R}}
\newcommand{\ZZ}{\mathbb{Z}}

\newcommand{\NN}{\mathbb{N}}

\newcommand{\Q}{\mathbb{Q}}
\newcommand{\N}{\mathbb{N}}
\newcommand{\Z}{\mathbb{Z}}
\newcommand{\R}{\mathbb{R}}
\newcommand{\C}{\mathbb{C}}
\newcommand{\Qp}{{\QQ_p}}
\newcommand{\Zp}{{\ZZ_p}}
\newcommand{\Qbar}{\overline{\Q}}
\newcommand{\Gm}{\mathbb{G}_m}
\newcommand{\T}{\mathbb{T}}

\newcommand{\HCI}{\operatorname{HChar}}
\newcommand{\CI}{\operatorname{Crit}}

\newcommand{\Eis}{\mathrm{EK}}

\newcommand{\dR}{\mathrm{dR}}

\newcommand{\prolim}{\varprojlim}
\newcommand{\indlim}{\varinjlim}

\newcommand{\isom}{\cong}

\newcommand{\id}{\mathrm{id}}

\renewcommand{\Im}{\mathrm{Im }}
\newcommand{\Log}[1]{{\sL}og^{(#1)}}

\newcommand{\sLog}{{\sL}og}

\newcommand{\fin}{\operatorname{fin}}

\newcommand{\ul}[1]{\underline{#1}}
\newcommand{\ol}[1]{\overline{#1}}
\DeclareMathOperator{\wP}{\widehat{\sP}}

\newcommand{\zbar}{\overline{z}}

\newcommand{\lambdabar}{\overline{\lambda}}
\newcommand{\dashsum}{\sideset{}{'}{\sum}}

\newcommand{\EK}{EK}
\DeclareMathOperator{\tors}{{tors}}
\DeclareMathOperator{\free}{{free}}

\def\endpiece{xxx}
\def\makeAlphabet[#1]{\expandafter\makeA#1,xxx,}
\def\makealphabet[#1]{\expandafter\makea#1,xxx,}
\def\makeA#1,{\def\temp{#1}\ifx\temp\endpiece\else%
\mkbb{#1}\mkfrak{#1}\mkbf{#1}\mkcal{#1}\mkscr{#1}\expandafter\makeA\fi}%
\def\makea#1,{\def\temp{#1}\ifx\temp\endpiece\else\mkfrak{#1}\mkbf{#1}\expandafter\makea\fi}%
\def\mkbb#1{\expandafter\def\csname bb#1\endcsname{\mathbb{#1}}}%
\def\mkfrak#1{\expandafter\def\csname fr#1\endcsname{\mathfrak{#1}}}
\def\mkbf#1{\expandafter\def\csname b#1\endcsname{\mathbf{#1}}}
\def\mkcal#1{\expandafter\def\csname c#1\endcsname{\mathcal{#1}}}
\def\mkscr#1{\expandafter\def\csname s#1\endcsname{\mathscr{#1}}}
\def\makeop[#1]{\xmakeop#1,xxx,}
\def\mkop#1{\expandafter\def\csname #1\endcsname{{\mathrm{#1}}}} %
\def\xmakeop#1,{\def\temp{#1}\ifx\temp\endpiece\else\mkop{#1}\expandafter\xmakeop\fi}%

\makeop[Mic,Mod,im]
\newcommand{\Hnorm}[1]{\left\lVert #1  \right\rVert_H^2}
\newcommand{\Gmf}{\widehat{\mathbb{G}}_m}
\newcommand{\OCp}{\sO_{\CC_p}}
\newcommand{\Cp}{{\CC_p}}
\newcommand{\Meas}{\mathrm{Meas}}

\newcommand{\sws}{\mathfrak{s}}
\newcommand{\rms}{\mathrm{s}}
\newcommand{\rmt}{\mathrm{t}}
\newcommand{\swt}{\mathfrak{t}}

\newcommand{\Local}{\mathrm{Local}}
\newcommand{\Chr}{\widehat{\mathrm{Char}}}

\AtBeginDocument{%
   \def\MR#1{}
}

\let\Im\relax
\DeclareMathOperator{\Im}{Im} 
\let\Re\relax
\DeclareMathOperator{\Re}{Re}

\makenomenclature



\numberwithin{equation}{subsection}
\begin{document}

\begin{abstract}
We show that for an arbitrary totally complex number field $L$ the (regularized) critical $L$-values of algebraic Hecke characters of $L$ divided by certain periods are algebraic integers. This relies on a new construction of an equivariant coherent cohomology class with values in the completion of the Poincar\'e bundle on an abelian scheme $\cA$. From this we obtain a cohomology class for the automorphism group of a CM abelian scheme $\cA$ with values in some canonical bundles, which can be explicitly calculated in terms of Eisenstein-Kronecker series. As a further consequence, using an infinitesimal trivialization of the Poincar\'e bundle, we construct a $p$-adic measure interpolating the critical $L$-values in the ordinary case. This generalizes previous results for CM fields by Damerell, Shimura and Katz and settles the algebraicity and $p$-adic interpolation in the remaining open cases of critical values of Hecke $L$-functions.
\end{abstract}

\setcounter{tocdepth}{1}
\maketitle

\tableofcontents
\section*{Introduction} 

\subsection*{Critical values of Hecke $L$-functions} Let $\chi$ be an algebraic Hecke character of a number field $L$  and $L(\chi,s)=\sum_{ \fra}\frac{\chi(\fra)}{\norm(\fra)^{s}}$ its $L$-function as defined by Hecke, where the sum runs over all integral ideals of $L$ coprime to the conductor of $\chi$. 
It is a  classical problem to investigate the special values of $L(\chi,s)$ at integers $s=k$.
As $L(\chi,k)=L(\chi\cdot\norm_{L/\Q}^{-k},0)$ one can concentrate on the case $s=0$. 

Assume now that $s=0$ is critical for $\chi$, which means essentially that the $\Gamma$-factors occurring in the functional equation of $L(\chi,s)$ have a finite value at $s=0$. It is known that this can happen only if $L$ is totally real or if $L$ contains a CM field $K$ (recall that a CM field is a totally imaginary quadratic extension of a totally real field). 

In the totally real case one knows, thanks to work of Euler, Klingen and Siegel \cite{Siegel-Fourier}, that  the $L(\chi,0)$ are algebraic numbers. This result was later refined by Barsky, Cassou-Nogues, Deligne and Ribet who showed that certain regularized values $L(\chi,0)$ are algebraic integers and they could construct a $p$-adic $L$-function. 

In the case where $L=K$ is a CM field, work of  Damerell \cite{Damerell} (for $K/\Q$ imaginary quadratic) and of Shimura \cite{Shimura} in general, showed that 
$L(\chi,0)$ divided by certain periods
are algebraic numbers. Later Katz \cite{Katz-CM} showed that a certain regularization of these values are algebraic integers and he constructed a $p$-adic $L$-function in the ordinary case. Blasius \cite{Blasius} finally proved  Deligne's conjecture about critical $L$-values for CM fields $K$. For extensions of CM fields, i.e. $L\neq K$, Colmez \cite{Colmez-algebraic} could show algebraicity of special $L$-values for certain extensions $L$ of imaginary quadratic fields $K$. Harder and Schappacher announced in 
\cite{Harder-Schappacher} (see also \cite{Harder}) 
an approach to the Deligne conjecture for Hecke $L$-series relying on Harder's Eisenstein cohomology, but so far full results were only published for quadratic extensions of CM fields (\cite{Harder-GL2}). 

The general case where $L$ is an arbitrary extension of degree $n:=[L:K]$ of an arbitrary CM field $K$ was still open. Here an algebraic Hecke character has the form $\chi=\varrho(\chi_0\circ\norm_{L/K})$ where $\chi_0$ is an algebraic Hecke character of $K$ and $\varrho$ is a Hecke character of $L$ of finite order. 

%

\subsection*{The main results} In this paper we show for critical algebraic Hecke characters $\chi$ of an arbitrary finite extension $L$ of an arbitrary CM field $K$ the 
integrality of the (regularized) $L$-values $L(\chi,0)$ divided by some periods and construct a $p$-adic $L$-function interpolating these values in the ordinary case.  This
generalizes the results of Shimura \cite{Shimura} and Katz \cite{Katz-CM} from CM fields to arbitrary totally complex fields. 

For the purpose of this introduction we give a slightly weaker statement of our first main result as follows:
%

\begin{theorem*}[Integrality of critical values, see Theorem \ref{thm:special-values}]
Let $\chi$ be a critical algebraic Hecke character of $L$ (see Section \ref{section:cm-types}). Let $L(\chi,0)$ be the value of its $L$-function at $0$ and $\frc$ be an integral ideal in $L$ coprime to $\frf$.    Then there is a number field $k$ with ring of integers $\sO_k$ such that 
\begin{equation*}\label{eq:intro-L-value-1}
\frac{(\chi(\frc)\norm\frc-1)L(\chi,0)}{\Omega^{\chi}}\in \sO_k[\frac{1}{\frf N(\frc)d_L}] 
\end{equation*}
where $\Omega^{\chi}$ is an explicit product of powers of $2\pi i$ and certain periods of abelian varieties with CM by $L$ over $\sO_k[\frac{1}{\frf N(\frc)d_L}]$ attached to  $\chi$.
\end{theorem*}
Our main theorem is more precise about the periods, the ring $\sO_k$ and we have a similar result for the trivial conductor $\frf=\sO_L$ for which we refer to 
Theorem \ref{thm:special-values}.

We would like to point out that our approach to this theorem is different from that of Katz even in the case of a CM field $L=K$. Furthermore, we get all critical values of the $L$-functions at the same time whereas Katz has to use the functional equation. 

The algebraicity  statement of the above theorem is the generalization of a formula of Shimura which was suggested by Katz in \cite{Katz-CM}. We prove this generalized Shimura formula in Corollary \ref{cor:Katz-formula}.  For Hecke characters of certain extensions of imaginary quadratic fields, already Colmez has shown rationality statements for $L(\chi,0)$ in \cite{Colmez-algebraic}. For arbitrary extensions of imaginary quadratic fields Bergeron-Charollois-Garcia (see \cite{BCG} and \cite{BCG2}) proved that the critical values of $L$-functions attached to Hecke characters can be expressed as polynomials in Eisenstein--Kronecker series evaluated at torsion points on elliptic curves. Thus they obtain similar integrality results in this special case.


Kufner \cite{Kufner} has recently been able to deduce the full Deligne conjecture from our main result, so that the Deligne conjecture for CM motives of rank one arising in the cohomology of abelian varieties is now completely settled. Conjecturally, these are all CM motives of rank one. 
 
At the time of writing of this paper, Kufner's result  was not yet available. But we can prove with our methods easily (we thank Blasius for pointing this out) a weak form of Deligne's conjecture \cite{Deligne-conj}:
\begin{corollary*}[Weak Deligne conjecture, see Corollary \ref{cor:deligne-conjecture}]
One has
\begin{equation*}
\frac{L(\chi,0)}{c^{+} R_{L/\Q}M(\chi)}\in\Qbar^{\times},
\end{equation*}
where ${c^{+} R_{L/\Q}M(\chi)}$ is the period of the motive 
$R_{L/\Q}M(\chi)$ defined by Deligne in \cite{Deligne-conj}.
\end{corollary*}

We will use our integrality results for special values of $L$-functions to study  the $p$-adic interpolation of these $L$-values. Our method gives a a direct geometric construction of $p$-adic measures for  ordinary primes using a trivialization of the Poincar\'e bundle. We  obtain the following generalization of the $p$-adic $L$-function constructed by Katz \cite{Katz-CM}, which we give here for the sake of the introduction in a slightly imprecise version:
\begin{theorem*}[$p$-adic interpolation, see Theorem \ref{thm_p-adic-interpolation}] Let $\Sigma$ be a CM type of $L$, which is is ordinary for the prime number $p$ (see Section \ref{sec:p-adic-geometric-setup}).	For every fractional ideal $\frf$ there exists a $p$-adic measure $\mu_{\frf}$ on $\Gal(L(p^\infty\frf)/L)$ with the following interpolation property: For every critical algebraic Hecke character $\chi$ of attached CM type $\Sigma$ and  conductor dividing $p^\infty\frf$, we have
	\begin{equation*}
		\frac{1}{\Omega^{\chi}_p}\int_{\Gal(L(p^\infty\frf)/L)} \chi(g)d\mu_{\frf}(g)=\{\mbox{explicit local factors}\}\frac{L(\chi,0)}{\Omega^{\chi}}
	\end{equation*}
for an explicit $p$-adic period $\Omega_p^{\chi}$ which depends only on the infinity-type of $\chi$.
\end{theorem*}

Note that for certain extensions of imaginary quadratic fields, such an interpolation was proved by Colmez-Schneps \cite{Colmez-Schneps} for the $L$-values of Hecke characters treated in \cite{Colmez-algebraic}.
 
The main ingredient in the proofs of these theorems above is an equivariant coherent cohomology class on an abelian scheme with values in $\wP$, the completion of the Poincar\'e bundle $\sP$ along $\cA\times_{\cS}e^{\vee}(\cS)\subset \cA\times_{\cS}\cA^{\vee}$. 
Let $d=\dim_\cS \cA$ be the relative dimension and 
$\Gamma\subset \Aut_{\cS}(\cA)$  a subgroup of automorphisms of $\cA$ over $\cS$. Then for a finite subscheme $\cD\subset \cA$ consisting of torsion sections and a $\Gamma$-invariant function $f$ on $\cD$ we construct a class
\begin{equation*}
\EK_{\Gamma}(f)\in H^{d-1}(\cA\smallsetminus\cD,\Gamma;\wP\otimes \Omega^{d}_{\cA/\cS} ),
\end{equation*}
which we call the \emph{equivariant coherent Eisenstein-Kronecker class}. 
It is absolutely essential for our applications to special values of Hecke $L$-functions to  have classes in equivariant cohomology. We use the case of an abelian scheme with CM by the ring of integers $\sO_L$ in the number field $L$ and where  $\Gamma\subset\sO_L^{\times}$ is a subgroup of finite index, but there are many other interesting cases. For example for an abelian scheme $\cA/\cS$ one can consider the $n$-fold product $\cA^{n}$ of $\cA$ over $\cS$, which has an action of $\GL_n(\Z)$, or if $\cA$ has already CM by $\sO_L$, by $\GL_n(\sO_L)$.  We point out that for these groups no arithmetic moduli space exists, but the following theorem contains a construction of group  cohomology classes with values in sections of certain algebraic bundles associated to $\cA$.

\begin{theorem*}[Eisenstein-Kronecker class, see Section \ref{subsection:Eisenstein-Kronecker-class}]
Let $\cA$ be an abelian scheme over $\cR:=\Spec R$ of relative dimension $d$ and $\Gamma\subset \Aut_{\cR}(\cA)$. Then for any integers $a,b\ge 0$ there is an \emph{Eisenstein-Kronecker class}
\begin{equation*}
\Eis_{\Gamma}^{b,a}(f,x)\in H^{d-1}(\Gamma,H^{0}(\cR,\TSym^{a}(\omega_{\cA/\cR})\otimes\TSym^{b}(\sH)\otimes \omega^{d}_{\cA/\cR} )),
\end{equation*}
depending on a $\Gamma$-invariant function $f$ on $\cD$ and a torsion section $x\in (\cA\smallsetminus \cD)(\cR)$. Here $\omega_{\cA/\cR}:=e^{*}\Omega_{\cA/\cR}$, $\sH\isom H^{1}_{dR}(\cA^{\vee}/\cR)$ and $\TSym$ denotes the tensor symmetric power algebra.
\end{theorem*}

Equivariant cohomology classes in connection with Eisenstein series have received a lot of attention lately, not only in the work of Bergeron-Charollois-Garcia \cite{BCG} already mentioned, but also for example in the work of Graf \cite{Graf}, of Sharifi-Venkatesh  \cite{Sharifi-Venkatesh}, in the work of Bannai-Hagihara-Yamada-Yamamoto \cite{Bannai-et-al}, in the work of Fl\'orez-Karabulut-Wong \cite{FKW}, and in the work \cite{BKL-topological-polylog}.
The related Shintani-cocycles were studied in the work of Charollois-Dasgupta \cite{Charollois-Dasgupta}.

Besides constructing such a group cohomology class, it is also important to be able to compute an explicit representative. This is accomplished in Section \ref{section:explicit-computation} of the paper. Using computations of Levin \cite{levin} adapted to our equivariant case one can show that $\Eis_{\Gamma}^{b,a}(f,x)$ can be represented in terms of  generalized Eisenstein-Kronecker series, which was the reason for its name.  Integrating these series over the classifying space of $\Gamma$ then leads to partial $L$-values of $L(\chi,0)$ by a standard computation.

\subsection*{Overview of the approach to the theorems}
We would like now to discuss our approach to the critical values of Hecke $L$-functions and why the strategy of Shimura or Katz has no straightforward generalization. Both, Shimura and Katz rely on real analytic Eisenstein series for the Hilbert modular group associated to the totally real subfield $F$ inside the CM field $K$ and the fact that their values at CM points are related to an $L$-value of a Hecke character. To see that these are algebraic numbers, they start by showing that the holomorphic Eisenstein series are algebraic using the $q$-expansion principle. The real analytic Eisenstein series are then derivatives by the Maa\ss-Shimura operators of the holomorphic ones. Katz had the important insight, that  the Maa\ss-Shimura operators arise from the Gau\ss-Manin connection and the Hodge-splitting.  The algebraicity of these operators is then guaranteed by the algebraicity of the Hodge-splitting at CM points. 

This approach of Shimura and Katz breaks down for arbitrary extensions $L$ of $K$ of degree $n$ as the natural habitat of the Eisenstein series in question is the locally symmetric space associated to the arithmetic group $\GL_n(\sO_K)$, which is not hermitian and hence can not be the $\C$-valued points of an algebraic variety. 

To handle this difficulty, we introduce several new viewpoints and concepts. First, it turns out that one should not work  on the base space of an abelian scheme, but on the abelian scheme with CM itself. That such an approach is possible was the pioneering insight by Bannai and Kobayashi \cite{Bannai-Kobayashi} who showed that the results of Damerell and the $p$-adic interpolation of $L$-values by Katz for elliptic curves $\cE$ with CM can be already achieved without any use of modular forms. In fact, our work started by trying to understand their approach conceptually and to generalize it to higher dimensional abelian varieties. Working on a single abelian variety means that $q$-expansions to check algebraicity of Eisenstein series are no longer available. It is an important feature of our work that the algebraicity is built into our construction of the Eisenstein-Kronecker classes. 
More precisely, we construct an equivariant coherent, hence algebraic, cohomology class, whose base extension to $\C$ can be described by currents in terms of generalized Eisenstein-Kronecker series. This turns the approach of Shimura and Katz upside down.  A second problem is that abelian schemes with CM by $L$, where $L$ is a proper extension of a CM field $K$, is only linear with respect to the maximal totally real field inside $K$.   We solve this by working systematically with the abelian scheme and its dual, which gives new insights into the periods occurring in the formulas for the special values of Hecke $L$-functions. For example this inspired the new construction of Blasius' reflex motive by Kufner in \cite{Kufner}.

Another main insight of this paper is that one should no longer work with sections of algebraic bundles as in the work of Shimura and Katz or of  Bannai-Kobayashi, but with equivariant coherent cohomology classes. The equivariant approach in our paper has its source in the papers \cite{BKL-topological-polylog} and \cite{Graf}. An equivariant motivic version of the polylog was advertised earlier by the first named author. Graf showed in fact, that one can recover the (algebraic) Eisenstein series used by Shimura and Katz with this method. This was for us the first indication that such an equivariant cohomological approach to the $L$-values should work.

For the $p$-adic interpolation we rely on the ideas of the second named author in \cite{Sprang_EisensteinKronecker}, which demonstrate that the differential operators used by Shimura and Katz can be realized using the canonical connection on the Poincar\'e bundle on the abelian variety times the universal vector extension of its dual. This gives a conceptual explanation of the ideas of Bannai-Kobayashi \cite{Bannai-Kobayashi}. Underlying the approach in \cite{Sprang_EisensteinKronecker} is the fact that the completion of the Poincar\'e bundle with connection is an incarnation of the de Rham logarithm sheaf, which was first shown in the  thesis of Scheider \cite{Scheider}. With this insight the paper \cite{Sprang_EisensteinKronecker} gives a new geometric construction of the $p$-adic Eisenstein measures of Katz in the elliptic case.

\subsection*{Outline of the paper} In the first section we study the action of the units $\sO_L^{\times}$ of $L$ on the algebraic de Rham cohomology of an abelian scheme with CM by $\sO_L$. 

In the second section we construct the equivariant coherent Eisenstein-Kronecker class, which has values in the completion of the Poincar\'e bundle. We first develop the necessary properties of the completion of the Poincar\'e bundle. This Eisenstein-Kronecker class is still insufficient for our purposes as it would lead only to holomorphic Eisenstein series. To remedy this and to get real-analytic Eisenstein series, we extend the class to the completion of the Poincar\'e bundle on the universal vector extension, where it acquires an integrable connection. This connection allows us to construct derivatives of the Eisenstein-Kronecker class. 

The third section is the technical heart of the paper. Here the equivariant Eisenstein-Kronecker class is explicitly computed in terms of generalized Eisenstein-Kronecker series. The computation of this class is facilitated by the relation of our class with the polylogarithm on abelian schemes, which allows to follow the strategy developed by Levin \cite{levin} to compute this class. 

The fourth section contains a proof of the integrality statement of the Hecke $L$-function cited above. Here we also fix our complex periods. It turns out that one should use the periods of $\cA$ and of $\cA^{\vee}$ at the same time. The relation with the work of Katz and others in the CM case is also explained there.

The final section is devoted to the $p$-adic interpolation of the critical $L$-values by a $p$-adic measure. The method used here relies on \cite{Sprang_EisensteinKronecker}, with some streamlining from \cite{Katz-another-look} and is also inspired by and follows partly \cite{Katz-CM}.

\subsection*{Acknowledgements} The attentive reader easily realizes our  debt to the work of Katz \cite{Katz-CM}, Bannai-Kobayashi \cite{Bannai-Kobayashi} and Beilinson-Levin \cite{BeilinsonLevin}. We thank Don Blasius for pointing out an inaccuracy in an earlier version of this paper and for many very useful remarks, which led to several improvements. We thank Yuhao Zhang  for a list of misprints and Han-Ung Kufner for many comments and discussions. We are especially grateful to the referees of this paper, whose extremely careful reading allowed to catch  inaccuracies and several misprints and helped a lot to improve the readability of this paper.  Further, we gratefully acknowledge the support through the SFB Higher Invariants in Regensburg during the work on this paper.

After the results of this work were completed, we were kindly informed by Bergeron that he together with Charollois and Garcia  obtained also integrality results for critical $L$-values of Hecke $L$-functions for extensions of an imaginary quadratic field, but with a completely different method (see \cite{BCG} and \cite{BCG2}). The first named author thanks Bergeron for an invitation to Paris and interesting discussions with him and Charollois about the different approaches used. 
\section{Abelian schemes with CM}

In this section we fix our notation concerning number fields and CM types, discuss abelian schemes with CM, prove some decomposition results and describe our set up. 
\subsection{CM types}\label{section:cm-types} We fix some notations concerning number fields and CM types. Let $L$ be an algebraic number field. We denote by $\sO_L$ its ring of integers and by $d_L$ its discriminant. We fix the algebraic closure $\Qbar\subset \C$ of $\Q$ in $\C$ and let $J_L:=\Hom_\Q(L,\Qbar)$ the set of embeddings of $L$ into $\Qbar$.

For each subset $\Xi\subset J_L$ we denote by $I_\Xi$ the free abelian group $\Z[\Xi]$ generated by $\Xi$. For $\mu=\sum_{ \xi\in \Xi}\mu(\xi)\xi\in I_\Xi$ we let 
\begin{equation*}
|\mu|:=\sum_{\xi\in \Xi}\mu(\xi)\in\Z
\end{equation*}
and we denote by $I_\Xi^{+}$ the $\mu=\sum_{ \xi\in \Xi}\mu(\xi)\xi\in I_\Xi$ with all $\mu(\xi)\ge 0$. 
In the case $\Xi=J_L$, we write
\begin{equation*}
I_L:=I_{J_L}=\Z[J_L]
\end{equation*}
for the free abelian group generated by $J_L$. 
We denote by
\begin{equation*}
\T_L:=\Res_{\sO_L/\Z}\Gm
\end{equation*}
the restriction of $\Gm$ to $\Spec \Z$. Let $L^{\Gal}$ be a Galois closure of $L$ with ring of integers $\sO_{L^{\Gal}}$. In the following, let $R$ be a $\sO_{L^{\Gal}}[1/d_L]$-algebra, where $d_L$ is the discriminant of $L$. The base change of $\T_L$ to $\cR:=\Spec R$
\begin{equation}\label{eq:split-torus}
\T_{L,R}:=\T_L\times_{\Spec\Z}\cR
\end{equation} 
is a split torus with $\T_{L,R}(R)=(\sO_{L}\otimes_{\Z}R)^{\times}$ and  character group $I_L$. More precisely, for such $R$, $\sO_{L}\otimes_\Z R$ is a semi-simple algebra and one has an isomorphism
\begin{align*}
\sO_L\otimes_\Z R\isom \bigoplus_{\sigma\in J_L}R&&  \ell\otimes 1\mapsto(\sigma(\ell))_{\sigma\in J_L}.
\end{align*}
Let $L/K$ be an extension of number fields. Then the norm defines a homomorphism
\begin{equation*}
\norm_{L/K}:\T_{L,R}\to \T_{K,R}
\end{equation*}
which induces a map
$\norm_{L/K}^{*}:I_K\to I_L$ by 
$\norm_{L/K}^{*}(\mu)(\sigma)=\mu(\sigma{\mid_K})$.
We say that a character $\mu\in I_L$ is \emph{lifted} from $I_K$, if it is of the form $\mu=\norm_{L/K}^{*}(\mu_0)$ with $\mu_0\in I_K$.

An algebraic action of $\sO_L^{\times}$ on an $R$-module $M$ is the same as an action of $\T_{L,R}$ on the quasi-coherent sheaf $\widetilde{M}$ on $\cR=\Spec R$, which amounts to a comodule structure under the group ring $R[I_L]$. The following is well-known and easy to check:
\begin{theorem}[Decomposition {\cite[Expos\'e I, Prop. 4.7.3]{SGA3}}]\label{thm:decomposition} Let $M$ be an algebraic $\T_{L,R}$-module as above, then there is an isomorphism
\begin{equation*}
M\isom \bigoplus_{\mu\in I_L}M(\mu)
\end{equation*}
where $\T_{L,R}$ acts on $M(\mu)$ via the character $\mu:\T_{L,R}\to \Gm$.
The set of $\mu\in I_L$ such that $M(\mu)\neq 0$ will be called the \emph{type of $M$}. 
\end{theorem}

Suppose that $R$ is contained in a number field $k$ (torsion free suffices). Then one has 
$\sO_L^{\times}\subset (\sO_L\otimes_\ZZ R)^{\times}=\T_{L,R}(R)$.
We are interested in the characters of $\T_{L,R}$ which are trivial
on a subgroup of finite index $\Gamma\subset \sO_L^{\times}$.
These are the possible infinity types of algebraic Hecke characters and were determined by Weil (see \cite{Serre-l-adic} for an elaboration). Recall that a CM field $K$ is a totally imaginary quadratic extension of a totally real field $F$. Depending on whether $L$ contains a CM field, there are two cases: 

1) $L$ does not contain a CM field. Then the only characters $\mu:\T_{L,R}\to \Gm$ trivial on some $\Gamma$ are the powers of the $\norm_{L/\Q}:\T_{L,R}\to \mathbb{G}_{m,R}$, i.e. $\mu$ is of the form
\begin{equation*}
\mu=w\sum_{ \sigma\in J_L}\sigma,\quad (\text{for some }w\in \ZZ).
\end{equation*}

2) $L$ does contain a CM field. Let $K$ be the maximal one. Then a character $\mu$ which is trivial on $\Gamma$ has to be lifted,  $\mu=\norm_{L/K}^{*}(\mu_0)$, from a character  $\mu_0\in I_K$, which satisfies the following condition: There exists an integer $w\in \Z$, such that
\begin{equation}\label{eq:alpha-beta-condition}
\mu_0(\overline{\sigma}_0)+\mu_0(\sigma_0)=w\mbox{ for all }\sigma_0\in \Sigma_K.
\end{equation}

Suppose now that $L$ contains a CM field $K$, hence $L$ is totally complex.
In this case we let 
\begin{align*}
n:=[L:K]&& 2g:=[K:\Q]&& \mbox{ and }&& 2d:=2gn=[L:\Q].
\end{align*}
Complex conjugation on $\Qbar\subset \C$ induces an involution 
$\sigma\mapsto \ol{\sigma}$ on $J_L$ and hence on $I_L$.
Recall that a CM type $\Sigma_K\subset J_K$ for a CM field $K$ is a subset of the embeddings of $K$ into $\Qbar$ such that for $\ol{\Sigma}_K:=\{\ol{\sigma}\mid\sigma\in\Sigma_K\}$ one has
\begin{align*}
\Sigma_K\cup \ol{\Sigma}_K=J_K&&\mbox{and}&& \Sigma_K\cap \ol{\Sigma}_K=\emptyset.
\end{align*} 
Given a CM type $\Sigma_K$ of $K$ one defines the CM type $\Sigma$ of $L$ lifted from $\Sigma_K$ to be
\begin{align}\label{eq:CM-type}
\Sigma:=\{\sigma\in J_L\mid \sigma\mid_K\in \Sigma_K\}&&
\mbox{and}&&
\ol{\Sigma}:=\{\sigma\in J_L\mid \sigma\mid_K\in \ol{\Sigma}_K\}.
\end{align}
\nomenclature{$\Sigma$}{\nomrefpage}
Complex conjugation defines bijections $\Sigma_K\isom \ol{\Sigma}_K$ and $\Sigma\isom \ol{\Sigma}$.
If we write ``$\Sigma$ is a CM type of $L$'' we always mean that it is lifted from some CM type of a sub CM field $K\subset L$. 

The following definition is crucial for the whole paper. 
\begin{definition}\label{def:critical} Let $L$\nomenclature{$L$}{\nomrefpage} be a number field, which contains the (maximal) CM field $K$ \nomenclature{$K$}{\nomrefpage}.
\begin{enumerate}
\item A character $\mu\in I_L$ is said to be of \emph{(algebraic) Hecke character type}, if there exists a character $\mu_0\in I_K$ and an integer $w\in \Z$ such that
\begin{align*}
\mu=\norm_{L/K}^{*}(\mu_0)
\end{align*}
and such that $\mu_0(\overline{\sigma}_0)+\mu_0(\sigma_0)=w$ for all $\sigma_0\in \Sigma_K$.
Let $\Gamma\subset \sO_L^{\times}$ be a subgroup of finite index. We write 
\begin{equation*}
\HCI_L(\Gamma):=\{\mu \in I_L\mid \mu\mbox{ of Hecke character type such that }\mu(\Gamma)=1\}
\end{equation*}
\nomenclature{$\HCI_L(\Gamma)$}{\nomrefpage}%
for the characters of (algebraic) Hecke character type, which are trivial on $\Gamma$. Then all characters of (algebraic) Hecke character type are $\bigcup_{\Gamma}\HCI_L(\Gamma)$, where $\Gamma$ runs over all subgroups of finite index in $\sO_L^{\times}$.
\item We call a $\mu\in \HCI_L(\Gamma)$  \emph{critical}, if 
\begin{align*}
	\Sigma:=\{\sigma \in J_L \mid \mu(\sigma)<0\}
\end{align*}
is a CM type of $L$ (i.e. it is lifted from a CM type $\Sigma_0$ of $K$). In this case, we call $\Sigma$ the \emph{attached CM type} of $\mu$. We write
\begin{equation*}
\CI_L(\Gamma):=\{\mu \in \HCI_L(\Gamma)\mid \mu\mbox{ of critical type}\}
\end{equation*}
for the set of critical Hecke character types.
\nomenclature{$\CI_L(\Gamma)$}{\nomrefpage}%
\end{enumerate}
\end{definition}

\begin{notation}\label{not:I-Sigma}
Let $L$ be a number field, which contains the (maximal) CM field $K$ and let $\mu \in \CI_L(\Gamma)$ with attached CM type $\Sigma$. By the definition of the CM type, we have $\mu(\sigma)\leq -1$ for $\sigma\in \Sigma$ and $\mu(\ol{\sigma})\geq 0$ for $\ol{\sigma}\in \ol{\Sigma}$. This suggests to introduce the following notation: The CM type $\Sigma$ of $\mu$ gives a decomposition
\begin{align*}
I_L=I_{\ol{\Sigma}}\oplus I_{{\Sigma}}
\end{align*}
(note the order of $\ol{\Sigma}$ and $\Sigma$), and we let $\ul{1}\in I_\Sigma$ be the element defined by $\ul{1}(\sigma)=1$ for all $\sigma\in \Sigma$.\nomenclature{$\alpha$}{\nomrefpage}\nomenclature{$\beta$}{\nomrefpage} With this notation, we can write $\mu$ uniquely in either of the following forms:
\begin{align*}
	\mu =\beta-\alpha, \quad \text{ with } \beta \in I^+_{\ol{\Sigma}}, \alpha-\ul{1} \in I^+_{\Sigma},
\end{align*}
or
\begin{align*}
	\mu =\beta-\alpha-\ul{1}, \quad \text{ with } \beta \in I^+_{\ol{\Sigma}}, \alpha \in I^+_{\Sigma}.
\end{align*}
\end{notation}

\begin{remark}
The $\mu$ of (algebraic) Hecke character type are the possible infinity types of algebraic Hecke characters and the critical ones have a critical $L$-value at $0$ (see Section \ref{section-special-values} below). Of course, the set $\HCI_L(\Gamma)$ can also be described by the Serre group $S_\frmm$, but for our purposes the above description is sufficient.
\end{remark}
We can now determine the $\Gamma\subset \sO_L^{\times}$ invariants of an algebraic $\T_{L,R}$-module $M$.
\begin{proposition}\label{prop:decomposition}
Let $M$ be an algebraic $\T_{L,R}$-module and suppose that $R$ is contained in an algebraic number field $k$. Let $\Gamma\subset \sO_L^{\times}$ be a subgroup of finite index. Then the $\Gamma$-invariants of $M$ are
\begin{equation*}
M^{\Gamma}=\bigoplus_{\mu\in \HCI_L(\Gamma)}M(\mu).
\end{equation*}
\end{proposition}
\begin{proof}
On $M(\mu)$ the group $\Gamma\subset \sO_L^{\times}$ acts via the character $\mu:\sO_L^{\times}\to R^{\times}$. This action is trivial if and only if $\mu\in \HCI_L(\Gamma)$ is of Hecke character type.
\end{proof}
We end this section introducing a convenient notation, which will be used throughout the text.  
\begin{definition}\label{def:TSym-convention}Let $M$ be an algebraic $\T_{L,R}$-module and $ \Xi\subset J_L$. Then we write 
\begin{equation*}
M(\Xi):=\bigoplus_{\sigma\in \Xi}M(\sigma)
\end{equation*}
and for $\alpha=\sum_{\sigma\in \Xi}\alpha(\sigma)\sigma\in I_\Xi^{+}$ a positive element  we define
\begin{equation*}
\TSym_R^{\alpha}(M(\Xi)):=\bigotimes_{\sigma\in \Xi}\TSym^{\alpha(\sigma)}_R(M(\sigma)).
\end{equation*}
We will, by slight abuse of notation, use this definition  also for a module of the form 
$\bigoplus_{\sigma\in \Xi}M(-\sigma)$.
Further, for an element $m\in M(\Xi)$, we denote by $m^{[\alpha]}\in \TSym_R^{\alpha}(M(\Xi))$ the element 
\begin{equation*}
m^{[\alpha]}=\bigotimes_{\sigma\in \Xi}m(\sigma)^{[\alpha(\sigma)]}
\end{equation*}
where $m(\sigma)^{[\alpha(\sigma)]}=m(\sigma)\otimes\cdots\otimes m(\sigma)$ $\alpha(\sigma)$-times. 
\end{definition}
Here $m(\sigma)\mapsto m(\sigma)^{[k]}$ is the divided power structure in $\TSym_R(M(\sigma))$ and one has $m(\sigma)^{k}=k!m(\sigma)^{[k]}$, where the left hand side is the $k$-th power calculated in $\TSym_RM(\sigma)$.
With the above convention one has the formula
\begin{equation*}
\TSym_R^{k}(M(\Xi))\isom \bigoplus_{\stackrel{\alpha\in I_\Xi^{+}}{|\alpha|=k}}\TSym_R^{\alpha}(M(\Xi))
\end{equation*}
\subsection{Abelian schemes with CM} We fix some notations concerning abelian schemes. 

Let $\cS$ be a (noetherian) scheme and $\pi:\cA\to \cS$\nomenclature{$\pi:\cA\to \cS$}{\nomrefpage} an abelian scheme of relative dimension $d$. We write $e:\cS\to \cA$ for its unit section. 
We denote by $\Lie(\cA/\cS)$ the relative Lie algebra and let 
\begin{equation*}
\omega_{\cA/\cS}:=e^{*}\Omega^{1}_{\cA/\cS}\isom \pi_*\Omega^{1}_{\cA/\cS}
\end{equation*}
be the sheaf of translation invariant differential forms. 
We
define for $i\geq 0$
\begin{align}\label{eq:omega-i-def}
\omega_{\cA/\cS}^{i}:=\Lambda^{i}\omega_{\cA/\cS}&&\omega_{\cA/\cS}^{-i}:=\sHom_{\sO_\cS}(\omega_{\cA/\cS}^{i},\sO_\cS)\isom\Lambda^{i}\Lie(\cA/\cS).
\end{align}
\nomenclature{$\omega_{\cA/\cS}^{i}$}{\nomrefpage}
Let $\pi^{\vee}:\cA^{\vee}\to \cS$\nomenclature{$\pi^{\vee}:\cA^{\vee}\to \cS$}{\nomrefpage} be the dual abelian scheme and  write  $e^{\vee}:\cS\to \cA^{\vee}$ for its unit section.
Then we have also the objects $\Lie(\cA^{\vee}/\cS)$ and $\omega_{\cA^{\vee}/\cS}^{i}$.
Denote by $H^{1}_{dR}(\cA/\cS)$ the first relative de Rham cohomology of $\cA/\cS$.  We notice that there is a canonical pairing
\begin{equation*}
H^{1}_{dR}(\cA/\cS)\times H^{1}_{dR}(\cA^{\vee}/\cS)\to \sO_{\cS}.
\end{equation*}
\begin{definition}\label{def:sH}
We let 
\begin{equation*}
\sH:=\sH_\cA:=\sHom_{\sO_\cS}(H^{1}_{dR}(\cA/\cS),\sO_{\cS})\isom H^{1}_{dR}(\cA^{\vee}/\cS),
\end{equation*}
\nomenclature{$\sH$}{\nomrefpage}%
which sits in the short exact sequence
\begin{equation*}
0\to \omega_{\cA^{\vee}/\cS}\to \sH\to \Lie(\cA/\cS)\to 0.
\end{equation*}
\end{definition}
The isomorphism $R^{1}\pi_*\sO_{\cA}\isom \Lie(\cA^{\vee}/\cS)$ induces  $R^{d}\pi_*\sO_{\cA}\isom \Lambda^{d}\Lie(\cA^{\vee}/\cS)$, which by Grothendieck duality gives rise to an isomorphism
\begin{equation}\label{eq:omega-equality}
\omega_{\cA^{\vee}/\cS}^{d}\isom \pi_*\Omega_{\cA/\cS}^{d}\isom \omega_{\cA/\cS}^{d}.
\end{equation}
\begin{definition}
Let $L$ be a totally complex number field with $[L:\Q]=2d$.  An abelian scheme $\cA/\cS$ of relative dimension $d$ has complex multiplication (CM) by $\sO_L$, if there exists an injection
\begin{equation*}
\iota:\sO_L\to \End_\cS(\cA)
\end{equation*}
such that $\Lie(\cA/\cS)$ is a projective $\sO_L\otimes_{\Z}\sO_\cS$-module.
In particular, $\cA/\cS$ has an action of $\sO_L^{\times}$. If $\cS=\Spec R$ with $R$ a $\sO_{L^{\Gal}}[1/d_L]$-algebra as in Section \ref{section:cm-types},   $\Lie(\cA/\cS)$, $\omega_{\cA^{\vee}/\cS}$ and $\sH$ are all  algebraic $\T_{L,R}$-modules.
\end{definition}

We recall the Serre construction. Let $T$ be a ring (not necessarily commutative) finite free over $\Z$ and $\cA/\cS$ an abelian scheme with an injection $T\to \End_{\cS}(\cA)$ and $\fra$ be a projective finitely presented right $T$-module. Then we define as usual the abelian scheme  $\cA(\fra):= \fra\otimes_{T}\cA$ (Serre construction), which for every $\cS$-scheme $\cS'$ satisfies
\begin{equation}
\cA(\fra)(\cS')=(\fra\otimes_{T}\cA)(\cS')=\fra\otimes_{T}\cA(\cS').
\end{equation}
\nomenclature{$\cA(\fra)$}{\nomrefpage}
Notice that $\Lie(\cA(\fra)/\cS)\isom \fra\otimes_{T}\Lie(\cA/\cS)$.  
If $\cA$ has CM by $\sO_L$ and $\fra\subset \frb$ are fractional ideals, then 
the homomorphism $\cA(\fra)\to \cA(\frb)$ is an isogeny of degree $[\frb:\fra]$. 
\begin{definition}\label{def:c-isogeny}
Let $\frc\subset L$ be an integral ideal, then the isogeny induced by $\sO_L\subset \frc^{-1}$ is denoted by
\begin{equation*}
[\frc]:\cA\to \cA(\frc^{-1})=\frc^{-1}\otimes_{\sO_L}\cA
\end{equation*}
and we write $\cA[\frc]$ for the kernel of $[\frc]$.
In particular, $\deg[\frc]=\norm\frc$.
\end{definition}
\begin{remark}
In this paper we consider abelian schemes $\cA$ with CM by $\sO_L$ together with the action of a subgroup $\Gamma\subset \sO_L^{\times}$ of finite index. Let $K\subset L$ be a CM subfield and $\cA_0$ be an abelian scheme with CM by $\sO_K$. Then one can form 
\begin{equation*}
\sO_L\otimes_{\sO_K}\cA_0
\end{equation*}
which has an action $\GL_{\sO_K}(\sO_L)$. The Eisenstein-Kronecker class which we construct in this paper actually is equivariant for this bigger group, but we will not use this fact for our main results. 
\end{remark}
%
%
%

\subsection{Decomposition results and set up} Let us write $L^{\Gal}$ for the Galois closure of $L$.  Let $R$ be an $\sO_{L^{\Gal}}[1/d_L]$-algebra and $\cR:=\Spec R$.
We study for  an abelian scheme $\cA/\cR$ with CM by $\sO_L$
the decomposition of the sheaves $\Lie(\cA/\cR)$, $\omega_{\cA^{\vee}/\cR}$ and $\sH$ under 
the $\T_{L,R}$-action. 

We start with a general remark. For  an $\sO_L\otimes_\Z R$-module $M$ one has a canonical isomorphism
\begin{equation*}
\Hom_R(M,R)\isom \Hom_{\sO_L\otimes_\Z R}(M,\Hom_{\Z}(\sO_L,\Z)\otimes R),
\end{equation*}
which induces an $\sO_L\otimes_\Z R$-linear pairing 
\begin{equation}\label{eq:R-O-pairing}
M\times \Hom_R(M,R)\to \Hom_{\Z}(\sO_L,\Z)\otimes R\isom \frd^{-1}_L\otimes R
\end{equation}
where $\frd_L^{-1}$ is the inverse different.
As the discriminant $d_L$ of $L$ is invertible in $R$, the trace map induces an isomorphism $\Hom_{\Z}(\sO_L,\Z)\otimes R\isom\sO_L\otimes R$, which identifies $\Hom_R(M,R)$ with the $\sO_L\otimes R$-dual of $M$. We will use this identification without further comment.
\begin{proposition}\label{prop:CM-decompostion} Assume that $R\subset\C$ and that $\cA/\cR$ has CM by $\sO_L$. Then there exists a (lifted) CM type $\Sigma$ of $L$, called the \emph{CM type of $\cA$}, such that
\begin{align*}
	\Lie(\cA/\cR)\cong \bigoplus_{\sigma\in\Sigma} \Lie(\cA/\cR)(\sigma)&&\mbox{and}&&
	\omega_{\cA^\vee/\cR}\cong \bigoplus_{\ol{\sigma}\in\ol{\Sigma}} \omega_{\cA^\vee/\cR}(\overline{\sigma}).
\end{align*}
One has a decomposition of $\T_{L,R}$-modules 
\begin{equation*}
\sH\isom \sH(\ol{\Sigma})\oplus \sH(\Sigma)
\end{equation*} 
using the convention in Definition \ref{def:TSym-convention}.
In particular, one has a splitting of the Hodge filtration
\begin{equation*}
0\to \omega_{\cA^{\vee}/\cR}\to\sH\to \Lie(\cA/\cR)\to 0.
\end{equation*}
The $L\otimes\C$-module $H_1(\cA(\C),\C)$ splits into 
\begin{equation*}
H_1(\cA(\C),\C)\isom \omega_{\cA^{\vee}/\C}\oplus \ol{\omega_{\cA^{\vee}/\C}}.
\end{equation*}
Further, one has $\sH(\ol{\Sigma})\otimes_R\C\isom \omega_{\cA^{\vee}/\CC}$ and 
$\sH(\Sigma)\otimes_R\C\isom \Lie(\cA/\C)$ so that $\ol{\omega_{\cA^{\vee}/\CC}}\isom \Lie(\cA/\C)$.
\end{proposition}
\begin{proof}
By the decomposition Theorem \ref{thm:decomposition} it remains to determine the types of $\Lie(\cA/\cR)$, $\omega_{\cA^{\vee}/\cR}$ and $\sH$.  This can be checked after base change to $\CC$. For $\sH\otimes_R\CC$ we get
	  \[
	  	\sH\otimes_R\CC\cong H_1(\cA(\C),\C)\isom L\otimes_\QQ \CC\cong \bigoplus_{\sigma\in\Hom_\QQ(L,\CC)}\CC,
	  \]
	  so every embedding appears exactly once in the rational representation. This determines the type of $\sH$. By the Hodge decomposition one has $H_1(\cA(\C),\C)\isom \omega_{\cA^{\vee}/\CC}\oplus \ol{\omega_{\cA^{\vee}/\CC}}$. This decomposition determines a CM type $\Sigma$ on $L$, defined such that ${\omega_{\cA^{\vee}/\CC}}$ has type $\ol{\Sigma}$. The sequence of $\T_{L,R}$-modules
\begin{equation*}
0\to \omega_{\cA^{\vee}/\cR}\to\sH_\cA\to \Lie(\cA/\cR)\to 0,
\end{equation*}
determines the type of $\Lie(\cA/\cR)$.
\end{proof}
We deduce from this proposition the decomposition of $\omega_{\cA/\cR}$ which is the $R$-dual of $\Lie(\cA/\cR)$ and where $\T_{L,R}$ acts via the inverse. Thus one obtains:
\begin{corollary}\label{cor:decomposition-sH}
	The $\T_{L,R}$-module $\omega_{\cA/\cR}$ splits into
	\[
		\omega_{\cA/\cR}\cong \bigoplus_{\sigma\in\Sigma} \omega_{\cA/\cR}(-\sigma),
	\]
	where $\gamma\in \Gamma=\sO_L^{\times}$ acts on $\omega_{\cA/\cR}(-\sigma)$ by multiplication with $\sigma(\gamma)^{-1}$.	Moreover, the Hodge filtration gives rise to a canonical isomorphism of $\T_{L,R}$-modules
\begin{equation*}
\omega_{\cA^{\vee}/\cR}^{d}\isom \omega_{\cA/\cR}^{d}\otimes \Lambda^{2d}\sH
\end{equation*}
where $\T_{L,R}$ acts via the norm on $\Lambda^{2d}\sH$. With the canonical isomorphism $\Lambda^{2d}\sH\isom \sO_{\cR}$ this induces an isomorphism $\omega_{\cA^{\vee}/\cR}^{d}\isom \omega_{\cA/\cR}^{d}$ of $\Gamma$-modules.
\end{corollary}
\begin{proof}
By the Proposition \ref{prop:CM-decompostion} the Hodge filtration is an exact sequence of $\T_{L,R}$-modules and the $\Gamma$-action on $\Lambda^{2d}\sH$ is via the norm $\norm:\sO_L^{\times}\to \Z^{\times}$, which is trivial as $L$ is totally complex.
\end{proof}
\begin{corollary}\label{cor:pairing} Define $\C^{\Sigma}:=\prod_{\sigma\in \Sigma}\C\isom L\otimes_\Q\C(\Sigma)$ and similarly $\C^{\overline{\Sigma}}$.
The period pairing of $L\otimes\C$-modules
\begin{equation*}
\langle \cdot,\cdot\rangle_{\cA}:H^{1}_{dR}(\cA/\C)\times H_1(\cA(\C),\C)\to L\otimes \C
\end{equation*}
induces 
\begin{equation*}
\langle \cdot,\cdot\rangle_\cA:\omega_{\cA/\C}\times H_1(\cA(\C),\Z)\to \C^{\Sigma}.
\end{equation*}
Using the canonical $\sO_L$-linear pairing 
\begin{equation*}
    \langle\cdot,\cdot\rangle_{dR,\cA}:H^{1}_{\dR}(\cA/\C)\times H^{1}_{\dR}(\cA^{\vee}/\C)\to L\otimes\C
\end{equation*}
and using the Hodge decomposition one gets a non-degenerate paring 
\begin{equation*}
\langle\cdot,\cdot\rangle_{\dR,\cA}:\ol{\omega_{\cA/\C}}\times \omega_{\cA^{\vee}/\C}\to \C^{\ol{\Sigma}}.
\end{equation*}
\end{corollary}
The next decomposition will come up in our construction of coherent Eisenstein-Kronecker classes. 
\begin{corollary}\label{cor:Gamma-invariants} 
Let $\Gamma\subseteq \sO_L^\times$ be a subgroup of finite index and $\ul{1}\in I_\Sigma$ the element which is $1$ at each $\sigma\in\Sigma$ defined in Notation \ref{not:I-Sigma}. For $\alpha\in I_{\Sigma}^+$ and $\beta\in I_{\ol{\Sigma}}^+$, we have 
\begin{multline*}
	\left(\TSym^{\alpha}(\omega_{\cA/\cR})\otimes_R \TSym^{\beta}(\omega_{\cA^{\vee}/\cR})\otimes_R\omega_{\cA/\cR}^{d}\right)^{\Gamma}=\\\begin{cases} \TSym^{\alpha}(\omega_{\cA/\cR})\otimes_R \TSym^{\beta}(\omega_{\cA^{\vee}/\cR})\otimes_R\omega_{\cA/\cR}^{d}& \beta-\alpha-\ul{1}\in \CI_L(\Gamma),\\
	0 & \text{otherwise}.
	  \end{cases}
\end{multline*}
\end{corollary}
\begin{proof} 
The action of $\gamma$ on $\TSym^{\alpha}(\omega_{\cA/\cR})\otimes_R \TSym^{\beta}(\omega_{\cA^{\vee}/\cR)})\otimes_R\omega_{\cA/\cR}^{d}$ is by multiplication with
\[
	\prod_{\sigma\in\Sigma}\sigma(\gamma)^{-\alpha(\sigma)-1}\prod_{\overline{\sigma}\in\overline{\Sigma}} \overline{\sigma}(\gamma)^{\beta(\overline{\sigma})}
\]
which is equal to $1$ if and only if $\mu=\beta-\alpha-\ul{1}\in \HCI_L(\Gamma)$ is of Hecke character type with respect to $\Gamma$. As $\Sigma=\{\sigma\in J_L\mid \mu(\sigma)<0\}$ is a CM type, one sees that every $\mu\in \HCI_L(\Gamma)$ of the form $\mu=\beta-\alpha-\ul{1}$ has to be critical. 
\end{proof}

We want to consider abelian schemes $\cA/\cR$ with CM by $\sO_L$ and with fixed $(R\otimes_{\Z}\sO_L)(\Sigma)$-bases 
of $\omega_{\cA/\cR}$ and of $\omega_{\cA^{\vee}/\cR}$. 
\begin{definition}
Let $\cA/\cR$ be an abelian scheme with CM by $\sO_L$ of type $\Sigma$. A basis $\omega(\cA)$  of 
$\omega_{\cA/\cR}$ (resp. $\omega(\cA^{\vee})$ of $\omega_{\cA^{\vee}/\cR}$) is a collection of global sections
\begin{align*}
\omega(\cA)(-\sigma)\in H^{0}(\cR, \omega_{\cA/\cR}(-\sigma))&&(\mbox{resp. }
\omega(\cA^{\vee})(\ol{\sigma})\in H^{0}(\cR, \omega_{\cA^{\vee}/\cR}(\ol{\sigma})))
\end{align*}
\nomenclature{$\omega(\cA)$}{\nomrefpage}
for $\sigma\in \Sigma$, which generate $\omega_{\cA/\cR}$ (resp. $\omega_{\cA^{\vee}/\cR})$
as $R\otimes_\Z \sO_L$-module.
\end{definition}
We note that a basis $\omega(\cA)$ of $\omega_{\cA/\cR}$ is the same as an isomorphism
\begin{equation}\label{eq:omega-as-map}
\omega(\cA):\Lie(\cA/\cR)\isom (\sO_L\otimes_{\ZZ} R)(\Sigma)
\end{equation}
using the convention in Definition \ref{def:TSym-convention}. The same remark applies, of course, to $\omega(\cA^{\vee})$.
\begin{remark}
The $\omega_{\cA/\cR}(-\sigma)$ and the $\omega_{\cA^{\vee}/\cR}(\ol{\sigma})$ are locally free $R$-modules of rank one. Hence in general there exists a basis only if one either shrinks $\cR=\Spec R$ or one performs a base change to a bigger ring. Both possibilities influence the main  integrality result of this paper.
\end{remark}
As $\sH\isom \omega_{\cA^{\vee}/\cR}\oplus \Lie(\cA/\cR)$, to specify a pair of bases $(\omega(\cA),\omega(\cA^{\vee}))$ is equivalent to specify a basis of $\sH$. Consider the exact sequence
\begin{equation*}
0\to H_1(\cA(\C),\Z)\to \Lie(\cA/\C)\to \cA(\C)\to 0.
\end{equation*}
Then  $\omega(\cA)$ induces an isomorphism 
\begin{equation}\label{eq:C-Sigma}
\omega(\cA):\Lie(\cA/\C)\isom (L\otimes \C)(\Sigma)=\C^{\Sigma}
\end{equation}
and the image of $H_1(\cA(\C),\Z)$ in $\C^{\Sigma}$ is the \emph{period lattice} $\Lambda$ of $\cA$ and one gets
\begin{equation}\label{eq:uniformization}
\cA(\C)\isom \C^{\Sigma}/\Lambda.
\end{equation}

\begin{definition}\label{def:psi-torsion}
Let $\cA/\cS$ be an abelian scheme and  $\varphi\colon\cA\to \cB$ be an isogeny. A section $x\colon\cS\to \ker\varphi$ is called a $\varphi$-torsion section.
\end{definition}
\begin{notation}\label{notation-for-CM-ab}Let $R$ be a $\sO_{L^{Gal}}[1/d_L]$-algebra with a fixed embedding $\iota_{\infty}:R\subset \C$. Let $\cR:=\Spec R$ and $\cA/\cR$ be an abelian scheme with CM by $\sO_L$. We will write  
\begin{equation*}
(\cA/\cR,\omega(\cA),\omega(\cA^{\vee}),x)
\end{equation*} 
to denote that we have fixed bases of $\omega_{\cA/\cR}$ and $\omega_{\cA^{\vee}/\cR}$ and that $x:\cR\to \cA[\frf]$ is an $[\frf]$-torsion section for an integral ideal $\frf\subset \sO_L$ and the isogeny
$[\frf]:\cA\to \cA(\frf^{-1})$. We write
\begin{equation*}
R[1/\frf]:=\bigcap_{\{\frmm\mid\frmm\cap \sO_L\nmid \frf \}}R_\frmm
\end{equation*}
where the intersection is taken over all localizations at  maximal ideals of $R$, such that $\frmm\cap \sO_L$ does not divide $\frf$.  
\end{notation}
Note that $\sO_L[1/\frf]$ are the $\frf$-integers.
If we write any subset of $(\cA/\cR,\omega(\cA),\omega(\cA^{\vee}),x)$ it means that only the conditions to define the mentioned objects are in force, e.g. $(\cA/\cR,\omega(\cA))$ is an abelian scheme with CM by $\sO_L$, the ring $R$ satisfies the above conditions and $\omega(\cA)$ is a basis of $\omega_{\cA/\cR}$ etc.

Note that $\cA/\cR$ has an action by 
$\sO_L^{\times}$
and that the subgroup $\sO_{\frf}^{\times}$ of units congruent to $1$ modulo $\frf$ fixes $x$. 
\begin{remark}
The theory of complex multiplication shows that one can choose the ring $R$ above to be a subring of an algebraic number field.
\end{remark}

\section{Equivariant coherent Eisenstein-Kronecker classes} In this section we present our construction of an equivariant coherent Eisenstein-Kronecker class. This construction uses the completion of the Poincar\'e bundle as main input and is modelled after the construction of the (de Rham) polylogarithm on abelian schemes.

\subsection{Preliminaries on the completed Poincar\'e bundle} \label{sec_Prelim_Poincare}
Let us fix some notations for this section.
\begin{notation}
We let $\pi:\cA\to \cS$ be an  abelian scheme  over a noetherian base $\cS$ and $e:\cS\to \cA$ denotes the unit section. We denote by  $\pi^{\vee}:\cA^{\vee}\to \cS$ the dual abelian scheme with unit section $e^{\vee}$ and let $\sP$ be the Poincar\'e bundle on $\cA\times_\cS\cA^{\vee}$, with rigidifications
\[
	(e\times \id)^*\sP\cong \sO_{\cA^\vee}, \quad (\id\times e^{\vee})^*\sP\cong \sO_{\cA}.
\]
\end{notation}
\nomenclature{$\sP$}{\nomrefpage}
We also need the universal vector extension of $\cA^{\vee}$.
\begin{notation}
We denote by $\pi^{\natural}:\cA^{\natural}\to \cS$ \nomenclature{$\pi^{\natural}:\cA^{\natural}\to \cS$}{\nomrefpage}the universal vector extension of $\cA^{\vee}$ with unit section $e^{\natural}$. 
The universal vector extension classifies isomorphism classes $(\sL,\nabla)$
of line bundles with integrable connection relative to $\cS$   which satisfy the theorem of the square (see \cite{Laumon} for further details). The universal line bundle with connection on $\cA\times_\cS\cA^{\natural}$ is denoted by $(\sP^{\natural},\nabla_{\sP^{\natural}})$. Forgetting the connection gives a homomorphism $p:\cA^{\natural}\to \cA^{\vee}$ sitting in 
an exact sequence of commutative group schemes over $\cS$
\begin{equation*}
0\to \omega_{\cA/\cS}\to \cA^{\natural}\xrightarrow{p} \cA^{\vee}\to 0.
\end{equation*}
\end{notation}
\nomenclature{$\cA^\natural$}{\nomrefpage}
\nomenclature{$\sP^{\natural},\nabla_{\sP^{\natural}}$}{\nomrefpage}
The line bundle $\sP^{\natural}$ is just the pull-back of $\sP$ via $\id\times p:\cA\times_\cS \cA^{\natural}\to \cA\times_\cS\cA^{\vee}$ and hence is still rigidified. 
The relative Lie algebra of $\cA^{\natural}$ identifies with the first de Rham cohomology  $\Lie(\cA^{\natural}/\cS)\isom H^{1}_{dR}(\cA/\cS)$ and
the exact sequence of relative Lie algebras 
\begin{equation*}
0\to \omega_{\cA/\cS}\to H^{1}_{dR}(\cA/\cS)\to \Lie(\cA^{\vee}/\cS)\to 0
\end{equation*}
is just the Hodge filtration of $H^{1}_{dR}(\cA/\cS)$. Recall that we denoted the dual of $H^{1}_{dR}(\cA/\cS)$ by $\sH$. The universal vector extension gives another interpretation of $\sH$ as the relative co-Lie algebra of $\cA^\natural$:
\begin{equation}\label{eq:sH-def}
\sH:=\ul{\Hom}_{\sO_\cS}(H^{1}_{dR}(\cA/\cS),\sO_{\cS})\isom\omega_{\cA^{\natural}/\cS}
\end{equation}
The dual of the Gau\ss-Manin connection gives a canonical connection on $\sH$.

We review the formal completions of group schemes.
Consider $\pi:\cG\to \cT$  a smooth commutative group scheme over a base scheme $\cT$ with unit section $e:\cT\to \cG$. Let $\sJ\subset \sO_\cG$ be the ideal defining the closed immersion $e$. We denote by 
\begin{equation*}
\cG^{(n)}:=\ul{\Spec}_{\cO_\cG}(\sO_\cG/\sJ^{n+1})\to \cG
\end{equation*} 
the $n$-th infinitesimal neighbourhood of the closed immersion $e$ and $\widehat{\cG}$ the formal completion. Then $\cG^{(n)}$ is a ringed space with underlying topological space $\cT$ and structure sheaf $\sO_{\cG^{(n)}}=\sO_\cG/\sJ^{n+1}$ which is affine over $\cT$. One has $\widehat{\cG}=\indlim_n\cG^{(n)}$ and 
\begin{equation*}
\sO_{\widehat{\cG}}=\prolim_n\sO_{\cG^{(n)}}.
\end{equation*}
The group structure induces maps
\begin{equation*}
\cG^{(n)}\times_\cT\cG^{(m)}\to \cG^{(n+m)}
\end{equation*}
and hence a co-commutative coproduct
\begin{equation*}
\sO_{\cG^{(n+m)}}\to \sO_{\cG^{(n)}}\otimes\sO_{\cG^{(m)}}.
\end{equation*}
In the limit one gets the formal group structure on $\widehat{\cG}$ and a coproduct $\widehat{\sO}_\cG\to \widehat{\sO}_\cG\widehat{\otimes}\widehat{\sO}_\cG$ (completed tensor product). As $\cG$ is smooth, the unit section is a regular immersion and has conormal bundle
\begin{equation*}
\sJ/\sJ^{2}\isom \omega_{\cG/\cT}=e^{*}\Omega_{\cG/\cT}
\end{equation*}
and $\sJ^{n}/\sJ^{n+1}\isom \Sym^{n}_{\sO_\cT}(\omega_{\cG/\cT})$ so that one has exact sequences
\begin{equation}\label{eq:Gn-exact-sequence}
0\to \Sym^{n}_{\sO_\cT}(\omega_{\cG/\cT})\to \sO_{\cG^{(n)}}\to \sO_{\cG^{(n-1)}}\to 0.
\end{equation}
In particular, one has a filtration on $\sO_{\cG^{(n)}}$ whose associated graded is $\bigoplus_{b=0}^{n}\Sym^{b}_{\sO_\cT}(\omega_{\cG/\cT})$.
As the formal completion is along a section, the $\sO_{\cG^{(n)}}$ are $\sO_\cT$-algebras and hence the sequence
\begin{equation*}
0\to \omega_{\cG/\cT}\to \sO_{\cG^{(1)}}\to \sO_{\cT}\to 0
\end{equation*}
splits canonically: $\sO_{\cG^{(1)}}\isom \sO_{\cS}\oplus \omega_{\cG/\cT}$. Using the comultiplication iteratively, the co-commutativity and this splitting one gets the \emph{moment map}
\begin{equation*}
\mom_{\cG^{(n)}}:\sO_{\cG^{(n)}}\to \TSym^{n}_{\sO_\cT}(\sO_{\cG^{(1)}})\isom\bigoplus_{b=0}^{n}\TSym^{b}_{\sO_\cT}(\omega_{\cG/\cT}).
\end{equation*}
If $\widehat{\TSym}(\omega_{\cG/\cT})$ denotes the completion of $\TSym^{\cdot}(\omega_{\cG/\cT})$ along its augmention ideal we get in the limit
\begin{equation}\label{eq:moment-G}
\mom_{\widehat{\cG}}:\sO_{\widehat{\cG}}\to \widehat{\TSym}(\omega_{\cG/\cT}).
\end{equation}
\begin{remark}
This map is the analogue of the moment map for completed group rings defined in \cite[3.13]{BKL-topological-polylog}. 
Notice that in the case where $\cT$ is a scheme of characteristic zero $\mom_{\widehat{\cG}}$ is an isomorphism. This can be seen on the associated graded which is the canonical homomorphism $\Sym^{\cdot}(\omega_{\cG/\cT})\to {\TSym}^{\cdot}(\omega_{\cG/\cT})$.
\end{remark}
It is obvious that all these notions are functorial: if $\varphi:\cG\to \cH$ is a homomorphism of group schemes, then one gets induced maps
\begin{equation*}
\varphi^{(n)}:\cG^{(n)}\to \cH^{(n)}\mbox{ and }\widehat{\varphi}:\widehat{\cG}\to \widehat{\cH}
\end{equation*}
and an obvious commutative diagram for the moment maps, which we do not write down.
\begin{definition}\label{def:infinitesimal-sheaf}
For a coherent sheaf of $\sO_\cG$-modules $\sF$ on $\cG$, we let $\widehat{\sF}$ be the completion of $\sF$ with respect to the unit section
$e:\cT\to \cG$, i.e.
\begin{equation*}
\widehat{\sF}=e^{-1}(\prolim \sF\otimes_{\sO_\cG} \sO_\cG/\sJ^{n+1}),
\end{equation*}
which is a $\sO_{\widehat{\cG}}$-module on $\cT$ and also isomorphic to 
$\iota^{*}_{\widehat{\cG}}\sF$, if $\iota_{\widehat{\cG}}:\widehat{\cG}\to \cG$ is the canonical map (the isomorphism is given by $\widehat{\pi}_*$, where  $\widehat{\pi}:\widehat{\cG}\to \cT$ is the structure map). We let
\begin{equation*}
\sF^{(n)}:=e^{-1}(\sF\otimes_{\sO_\cG} \sO_\cG/\sJ^{n+1}), 
\end{equation*}
which is a
$\sO_{\cG^{(n)}}$-module on $\cT$.
\end{definition}
Note that  $\widehat{\sF}\isom \prolim \sF^{(n)}$. Further, as $\pi^{(n)}:\cG^{(n)}\to \cT$ is affine
\begin{equation*}
\sF^{(n)}\isom\pi^{(n)}_*(\sF\mid_{\cG^{(n)}}),
\end{equation*}
where $\sF\mid_{\cG^{(n)}}$ is the pull-back in the sense of ringed spaces.
\begin{definition}
Applying  the above discussion to $\cA$, $\cA^{\vee}$ and $\cA^{\natural}$, we let 
\nomenclature{$\pi^{(n)}:\cA^{(n)}\to\cS$}{\nomrefpage}
\nomenclature{$\pi^{\vee(n)}:\cA^{\vee(n)}\to \cS$}{\nomrefpage}
\nomenclature{$\pi^{\natural(n)}:\cA^{\natural(n)}\to \cS$}{\nomrefpage}
\begin{align*}
\pi^{(n)}:\cA^{(n)}\to\cS&&\pi^{\vee(n)}:\cA^{\vee(n)}\to \cS&&\pi^{\natural(n)}:\cA^{\natural(n)}\to \cS
\end{align*}
be the $n$-th infinitesimal neighbourhood of the respective unit section $e$, $e^{\vee}$, $e^{\natural}$ and $\widehat{\pi}:\widehat{\cA}\to \cS$, $\widehat{\pi}^{\vee}:\widehat{\cA}^{\vee}\to \cS$ and $\widehat{\pi}^{\natural}:\widehat{\cA}^{\natural}\to \cS$ 
the formal completions.
\end{definition}
Consider the relative group scheme $\pr:\cA\times_{\cS}\cA^{\vee}\to \cA$ and its unit section $\id\times e^{\vee}$. Then the $n$-th infinitesimal neighbourhood is
\begin{equation*}
\cA\times_{\cS}\cA^{\vee(n)}\to \cA\times_{\cS}\cA^{\vee}
\end{equation*}
and $\cA\times_\cS\widehat{\cA}^{\vee}$ is the formal completion. The conormal bundle is $\pi^{*}\omega_{\cA^{\vee}/\cS}$. Similarly, for the relative group scheme $\pr:\cA\times_{\cS}\cA^{\natural}\to \cA$ with unit section $\id\times e^{\natural}$ we get $\cA\times_{\cS}\cA^{\natural(n)}$ and $\cA\times_{\cS}\widehat{\cA}^{\natural}$. Here the conormal bundle is $\pi^{*}\sH$, see \eqref{eq:sH-def}. From the moment map \eqref{eq:moment-G} we get two moment maps in this situation:

\begin{definition}\label{def:moment-map} We denote the pull-back via $e$ of  the moment maps associated to the relative group scheme $\pr:\cA\times_{\cS}\cA^{\vee}\to \cA$ by
\begin{equation*}
\mom_e:=\mom_{\widehat{\cA}^{\vee}}:\sO_{\widehat{\cA}^{\vee}}\to \widehat{\TSym}({\omega}_{\cA^{\vee}/\cS}).
\end{equation*}
Similarly, for the relative group scheme
$\pr:\cA\times_{\cS}\cA^{\natural}\to \cA$ 
\begin{equation*}
\mom_e:=\mom_{\widehat{\cA}^{\natural}}:\sO_{\widehat{\cA}^{\natural}}\to \widehat{\TSym}(\sH).
\end{equation*}
\end{definition}

\begin{definition}
We  apply the Definition \ref{def:infinitesimal-sheaf} to the Poincar\'e bundle $\sP$ on $\cA\times_{\cS}\cA^{\vee}$ and the map $\id\times e^{\vee}:\cA\to \cA\times_{\cS}\cA^{\vee}$ and get the \emph{completion of the Poincar\'e bundle} \nomenclature{$\widehat{\sP}$}{\nomrefpage}
\begin{equation*}
\widehat{\sP}:=(\id\times e^{\vee})^{-1}(\prolim\sP\otimes_{\sO_{\cA\times_{\cS}\cA}}\sO_{\cA\times_{\cS}\cA^{\vee(n)}})\isom (\id\times\widehat{\pi})_*\iota_{\cA\times_\cS\widehat{\cA}}^{*}\sP
\end{equation*}
where $\iota_{\cA\times_\cS\widehat{\cA}}:\cA\times_\cS\widehat{\cA}\to \cA\times_\cS\cA^{\vee}$ is the natural map. Then $\wP$ is an
$\sO_{\cA\times_{\cS}\widehat{\cA}^{\vee}}$-module
on $\cA$. We also consider the $\sO_{\cA\times_{\cS}\cA^{\vee(n)}}$-module 
\begin{equation*}
\sP^{(n)}:=(\id\times\pi^{\vee(n)})_*(\sP\mid_{\cA\times_\cS\cA^{\vee(n)}})
\end{equation*}
\nomenclature{$\sP^{(n)}$}{\nomrefpage}
on $\cA$. One has $\widehat{\sP}\isom \prolim_n\sP^{(n)}$.
Similarly, for $\iota_{\cA\times_{\cS}\widehat{\cA}^{\natural}}:\cA\times_{\cS}\widehat{\cA}^{\natural}\to \cA\times_{\cS}\cA^{\natural}$ one has
\nomenclature{$\widehat{\sP}^{\natural}$}{\nomrefpage}
\begin{equation*}
\widehat{\sP}^{\natural}\isom(\id\times \widehat{\pi}^{\natural})_* \iota_{\cA\times_{\cS}\widehat{\cA}^{\natural}}^{*}\sP
\end{equation*}
which is an $\sO_{\cA\times_{\cS}\widehat{\cA}^{\natural}}$-module on $\cA$. Further, we let
\begin{equation*}
\sP^{\natural(n)}:=\pi^{\natural(n)}_*(\sP^{\natural}\mid_{\sO_{\cA\times_\cS\cA^{\natural(n)}}})
\end{equation*}
so that $\widehat{\sP}^{\natural}=\prolim_n\sP^{\natural(n)}$. The  sheaves
$\widehat{\sP}^{\natural}$ and $\sP^{\natural(n)}$ also inherit a relative connection    $\nabla_{\wP}$ resp. $\nabla_{\sP^{\natural(n)}}$.
\end{definition}

Observe that the rigidifications of $\sP$ and $\sP^{\natural}$ induce isomorphisms $\sP^{(0)}\isom \sO_\cA$ and $\sP^{\natural(0)}\isom \sO_\cA$.
With this one deduces from \eqref{eq:Gn-exact-sequence} for $\cG=\cA^{\vee}$ and $\cG=\cA^{\natural}$ the exact sequences
\begin{equation*}
0\to \pi^{*}\Sym_{\sO_\cS}^{n}(\omega_{\cA^{\vee}/\cS})\to \sP^{(n)}\to \sP^{(n-1)}\to 0
\end{equation*}
and 
\begin{equation*}
0\to \pi^{*}\Sym_{\sO_\cS}^{n}(\sH)\to \sP^{\natural(n)}\to \sP^{\natural(n-1)}\to 0
\end{equation*}
where $\sH$ is defined in Definition \ref{def:sH}, see also \eqref{eq:sH-def}.
Let 
\begin{align*}
\sP\to \sP^{(n)}&&\mbox{and}&&\sP^{\natural}\to \sP^{\natural(n)}
\end{align*}
be the natural maps. Then the rigidifications of $\sP$ and $\sP^{\natural}$ induce sections 
\begin{equation*}
\bfone^{(n)}:\sO_\cS\to e^{*}\sP^{(n)}\mbox{ and }\bfone^{(n)}:\sO_\cS\to e^{*}\sP^{\natural(n)}
\end{equation*} 
which are compatible for different $n$ and give in the limit
\begin{align}\label{eq:splittings}
\bfone:\sO_\cS\to e^{*}\wP\isom \sO_{\widehat{\cA}^{\vee}}&&\mbox{and}&&\bfone:\sO_\cS\to e^{*}\widehat{\sP^{\natural}}\isom \sO_{\widehat{\cA}^{\natural}}.
\end{align}
Further, as $\sP^{\natural}$ is the pull-back of $\sP$ to $\cA\times_{\cS}\cA^{\natural}$ one has an inclusion 
\begin{equation}\label{eq:sP-inclusion}
\sP^{(n)}\to\sP^{\natural(n)}\isom \sP^{(n)}\otimes_{\sO_{\cA\times_{\cS}\cA^{\vee(n)}}}\sO_{\cA\times_{\cS}\cA^{\natural(n)}}.
\end{equation}
which is 
compatible with the sections $\bfone^{(n)}$.

\subsection{Properties of the completed Poincar\'e bundle}
In this section we discuss the properties of  the completed Poincar\'e bundles $\sP^{(n)}$, $\wP$, $\sP^{\natural(n)}$ and $\widehat{\sP}^{\natural}$ that are important for the construction of our equivariant coherent Eisenstein-Kronecker classes. The properties are similar to the one of the logarithm sheaf (see for example \cite{HuKiPol}). 
In fact, one can show that $\widehat{\sP}^{\natural}$ is isomorphic to the logarithm sheaf if the base scheme $\cS$ is of characteristic zero (see Theorem \ref{thm:scheider-theorem} below).
We advise the reader to skip the (technical) proofs in this section in a first reading.

Let $\varphi:\cA\to \cB$ be an isogeny and $\varphi^{\vee}:\cB^{\vee}\to \cA^{\vee}$ resp. $\varphi^{\natural}:\cB^{\natural}\to \cA^{\natural}$ its dual. The universal property of the Poincar\'e bundles gives rise to  isomorphisms
\begin{align}\begin{split}
\label{eq:isogeny-prop}
( \varphi\times \id)^{*}\sP_\cB&\isom (\id\times\varphi^{\vee})^{*}\sP_\cA\\
( \varphi\times \id)^{*}\sP^{\natural}_\cB&\isom (\id\times\varphi^{\natural})^{*}\sP^{\natural}_\cA
\end{split}
\end{align}
where $\sP_\cA$ resp. $\sP_\cB$ is the Poincar\'e bundle of $\cA$ resp. $\cB$ and similarly for $\sP^{\natural}$.
\begin{theorem}[Functoriality]\label{thm:functoriality}
Let $\varphi:\cA\to \cB$ be an isogeny. Then one has a canonical maps
\nomenclature{$\varphi^{(n)}_{\#}$}{\nomrefpage}
\begin{align*}
\varphi^{(n)}_{\#}:\sP^{(n)}_\cA\to \varphi^{*}\sP^{(n)}_\cB&&\varphi^{(n)}_{\#}:\sP^{\natural(n)}_\cA\to \varphi^{*}\sP^{\natural(n)}_\cB.
\end{align*}
If $\varphi^{\vee}$ resp. $\varphi^{\natural}$ is \'etale (for example if $\deg\varphi$ is invertible on $\cS$) $\varphi^{(n)}_{\#}$ is an isomorphism. 
The maps $\varphi^{(n)}_\#$ are compatible for different $n$ and one obtains \nomenclature{$\varphi_{\#}$}{\nomrefpage}
\begin{align*}
\varphi_{\#}:\widehat{\sP}_\cA\to \varphi^{*}\widehat{\sP}_\cB.&&\varphi_{\#}:\widehat{\sP}_\cA^{\natural}\to \varphi^{*}\widehat{\sP}_\cB^{\natural}.
\end{align*}
\end{theorem}
\begin{proof} We prove the statement for $\sP^{(n)}$ using \eqref{eq:isogeny-prop}. For $\sP^{\natural(n)}$ it is completely analogous.

Let $\varphi^{\vee(n)}:\cB^{\vee(n)}\to \cA^{\vee(n)}$ be the map induced by $\varphi^{\vee}$ and denote by $\pi_\cA^{\vee(n)}:\cA^{\vee(n)}\to \cS$ and $\pi_\cB^{\vee(n)}:\cB^{\vee(n)}\to \cS$ the structure maps. Then by definition
$\sP^{(n)}_\cA\isom (\id\times\pi_\cA^{\vee(n)})_*\sP_\cA{\mid_{\cA\times \cA^{\vee(n)}}}$ and similarly for $\sP^{(n)}_\cB$. From \eqref{eq:isogeny-prop} one has 
\begin{equation*}\label{eq:inf-isogeny-prop}
(\id_\cA\times\varphi^{\vee(n)})^{*}\sP_\cA{\mid_{\cA\times \cA^{\vee(n)}}}\isom (\varphi\times \id_{\cB^{\vee(n)}})^{*}\sP_\cB{\mid_{\cB\times \cB^{\vee(n)}}}.
\end{equation*}
The desired map
$\varphi^{(n)}_\#:\sP^{(n)}_\cA\to \varphi^{*}\sP^{(n)}_\cB$ is now the composition
\begin{align}\begin{split}
\label{eq:adjunction}\raisetag{-40pt}
\sP^{(n)}_\cA=(\id_\cA\times \pi_\cA^{\vee(n)})_*\sP_\cA{\mid_{\cA\times \cA^{\vee(n)}}}&\to (\id_\cA\times \pi_\cA^{\vee(n)})_*(\id_\cA\times\varphi^{\vee(n)})_*(\varphi\times\id_{\cB^{\vee(n)}})^{*}\sP_\cB{\mid_{\cB\times \cB^{\vee(n)}}}\\
&\isom (\id_\cA\times \pi_\cB^{\vee(n)})_*(\varphi\times \id_{\cB^{\vee(n)}})^{*}\sP_\cB{\mid_{\cB\times \cB^{\vee(n)}}}\\
&\isom \varphi^{*}(\id_\cB\times\pi_\cB^{\vee(n)})_*\sP_\cB{\mid_{\cB\times \cB^{\vee(n)}}}=\varphi^{*}\sP^{(n)}_\cB
\end{split}
\end{align}
where the isomorphism between the last two lines comes from the base change
\begin{equation*}
\xymatrixcolsep{5pc}\xymatrix{\cA\times_\cS \cB^{\vee(n)}\ar[r]^{\varphi\times\id_\cB^{\vee(n)}}\ar[d]_{\id_\cA\times\pi_\cB^{\vee(n)}}&\cB\times_\cS \cB^{\vee(n)}\ar[d]^{\id_\cB\times\pi_\cB^{\vee(n)}}\\
\cA\ar[r]^{\varphi}&\cB.
}
\end{equation*}
If  $\varphi^{\vee}$ is \'etale the map $\varphi^{\vee(n)}:\cB^{\vee(n)}\to \cA^{\vee(n)}$ is an isomorphism and hence the adjunction in \eqref{eq:adjunction} is an isomorphism. 
\end{proof}
The functoriality for isogenies leads to the important splitting principle.
\begin{corollary}
[Splitting principle]\label{cor:coh-log-splitting} Let $\varphi:\cA\to \cB$ be an isogeny and $x:\cS\to \ker\varphi$ a $\varphi$-torsion section.
Then $\varphi_{\#}$ induces a map
\begin{equation*}
\varphi_{\# x}:x^{*}\widehat{\sP}_\cA\to x^{*}\varphi^{*}\widehat{\sP}_\cB\isom e^{*}\widehat{\sP}_\cB\isom \sO_{\widehat{\cB}^{\vee}}.
\end{equation*}
If  $\varphi^{\vee}$ is \'etale one has an isomorphism
\begin{equation*}
\widehat{\sP}_\cA{\mid_{\ker{\varphi}}}\isom\pi_{\ker(\varphi)}^{*}\sO_{\widehat{\cB}^{\vee}},
\end{equation*}
where  $\pi_{\ker(\varphi)}:\ker\varphi\to \cS$ is the structure map and a canonical isomorphism 
\begin{equation*}
\varrho_x:x^{*}\widehat{\sP}_\cA\isom \sO_{\widehat{\cA}^{\vee}}.
\end{equation*}
\nomenclature{$\varrho_x$}{\nomrefpage}
The same statement holds for $\widehat{\sP}^{\natural}_\cA$ and $\widehat{\sP}^{\natural}_\cB$ if $\varphi^{\natural}$ is \'etale.
\end{corollary}
\begin{proof}
The first two statements are clear. 
The isomorphism $\varrho_x$ is the composition
\begin{equation*}
\varrho_x:x^{*}\widehat{\sP}_\cA\xrightarrow[\isom]{\varphi_\#}x^{*}\varphi^{*}\widehat{\sP}_\cB=e^{*}\widehat{\sP}_\cB \xleftarrow[\isom]{\varphi_\#}e^{*}\widehat{\sP}_\cA\isom \sO_{\widehat{\cA}^{\vee}}.\qedhere
\end{equation*}
\end{proof}
\begin{definition}\label{def:moment-x} Let $\varphi:\cA\to \cB$ be an isogeny with \'etale dual $\varphi^{\vee}$ and $x$ a $\varphi$-torsion section. Let 
\begin{equation*}
_\varrho\mom_x:=\mom_e\circ \varrho_{x} :x^{*}\widehat{\sP}_\cA\isom  \sO_{\widehat{\cA}^{\vee}}\to \widehat{\TSym}(\omega_{\cA^{\vee}/\cS})
\end{equation*}
\nomenclature{$_\varrho\mom_x$}{\nomrefpage}
be the composition of $\varrho_x$ from Corollary \ref{cor:coh-log-splitting} with the moment map defined in Definition \ref{def:moment-map}. Similarly
\begin{equation*}
_\varrho\mom_x:=\mom_e\circ \varrho_{x} :x^{*}\widehat{\sP}^{\natural}_\cA\isom  \sO_{\widehat{\cA}^{\natural}}\to \widehat{\TSym}(\sH),
\end{equation*}
if $\varphi^\natural$ is \'etale. In both cases, we write 
\begin{align*}
_\varrho\mom_x^{b}:x^{*}\widehat{\sP}_\cA\to \TSym^{b}(\omega_{\cA^{\vee}/\cS})&&
_\varrho\mom_x^{b}:x^{*}\widehat{\sP}_\cA^{\natural}\to \TSym^{b}(\sH)
\end{align*}
for the projections onto the $b$-component.
\end{definition}
The functoriality implies immediately that $\wP$ and $\widehat{\sP^{\natural}}$ are $\Gamma$-equivariant sheaves:
\begin{corollary}[$\Gamma$-equivariance]
Let $\Gamma$ be a discrete group acting (from the left) via (relative) automorphism on $\cA/\cS$. For each $\gamma\in \Gamma$ one has an isomorphisms
\begin{align*}
(\gamma_\#)^{-1}:\gamma^{*}\widehat{\sP}\isom \widehat{\sP}&&(\gamma_\#)^{-1}:\gamma^{*}\widehat{\sP}^{\natural}\isom \widehat{\sP}^{\natural}.
\end{align*}
\end{corollary}

\begin{proposition}[Comultiplication]
For all $m,n$ there are canonical homomorphisms 
\begin{align*}
\sP^{(n+m)}\to \sP^{(n)}\otimes_{\sO_\cA}\sP^{(m)}
\end{align*}
and hence  homomorphisms
\begin{align*}
\widehat{\sP}\to \widehat{\sP}\widehat{\otimes}_{\sO_\cA}\widehat{\sP}
\end{align*} 
(completed tensor product) whose associated graded
\begin{equation*}
\pi^{*}\Sym_{\sO_\cS}^{\cdot}(\omega_{\cA^{\vee}/\cS})\to \pi^{*}\Sym_{\sO_\cS}^{\cdot}(\omega_{\cA^{\vee}/\cS})\otimes_{\sO_\cA} \pi^{*}\Sym_{\sO_\cS}^{\cdot}(\omega_{\cA^{\vee}/\cS})
\end{equation*}
is the homomorphism induced by the diagonal $\omega_{\cA^{\vee}/\cS}\to \omega_{\cA^{\vee}/\cS}\oplus \omega_{\cA^{\vee}/\cS}$. In particular, the comultiplication is co-commutative. The analogous statement holds for $\sP^{\natural(n)}$ and $\widehat{\sP}^{\natural}$.
\end{proposition}
\begin{proof} We give the proof for  $\sP^{(n)}$ for $\sP^{\natural(n)}$ it is exactly the same.

Let $\mu^{{\vee}(n,m)}:\cA^{{\vee}(n)}\times_\cS\cA^{{\vee}(m)}\to \cA^{{\vee}(n+m)}$ be induced by the group law and $\pi^{\vee(n)}:\cA^{\vee(n)}\to \cS$ the structure map. As $\sP$ satisfies the theorem of the square one gets
\begin{equation*}
(\id_\cA\times \mu^{{\vee}(n,m)})^{*}\sP{\mid_{\cA\times\cA^{{\vee}(n+m)}}}\isom
(\id_{\cA}\times \pr_1)^{*}\sP{\mid_{\cA\times\cA^{{\vee}(n)}}}\otimes (\id_{\cA}\times\pr_2)^{*}\sP{\mid_{\cA\times\cA^{{\vee}(m)}}}.
\end{equation*}
Using adjunction for $(\id_\cA\times \mu^{{\vee}(n,m)})^{*}$ and then push-forward with $(\id_\cA\times \pi^{{\vee}(n+m)})_*$ one gets a homomorphism
\begin{align*}
\sP^{(n+m)}&\isom(\id_\cA\times \pi^{{\vee}(n+m)})_*\sP{\mid_{\cA\times\cA^{{\vee}(n+m)}}}\\
&\to (\id_\cA\times \pi^{{\vee}(n)}\times \pi^{{\vee}(m)})_*\left((\id_{\cA}\times \pr_1)^{*}\sP{\mid_{\cA\times\cA^{{\vee}(n)}}}\otimes (\id_{\cA}\times\pr_2)^{*}\sP{\mid_{\cA\times\cA^{{\vee}(m)}}}\right)\\
&\isom (\id_\cA\times \pi^{{\vee}(n)})_*\sP{\mid_{\cA\times\cA^{{\vee}(n)}}}\otimes (\id_\cA\times \pi^{{\vee}(m)})_*\sP{\mid_{\cA\times\cA^{{\vee}(m)}}}\\
&=\sP^{(n)}\otimes_{\sO_{\cA}} \sP^{(m)}
\end{align*}
(using  $\pi^{{\vee}(n+m)}\circ \mu^{{\vee}(n,m)}=\pi^{{\vee}(n)}\times \pi^{{\vee}(m)}$).
The associated graded of $\wP$ is $\pi^{*}\Sym_{\sO_\cS}(\omega_{\cA^{\vee}/\cS})$
which is also the associated graded of $e^{*}\wP\isom \sO_{\widehat{\cA}^{\vee}}$. To see what the comultiplication looks like one might look at the pull-back $e^{*}$. Hence the associated graded of the map
$\wP\to \wP\widehat{\otimes}\wP$ coincides with the associated graded of 
the comultiplication $\sO_{\widehat{\cA}^{\vee}}\to \sO_{\widehat{\cA}^{\vee}}\widehat{\otimes}_{\sO_\cA}\sO_{\widehat{\cA}^{\vee}}$
of the formal group $\indlim_n{\cA^{\vee(n)}}$,
which is 
\begin{equation*}
\pi^{*}\Sym_{\sO_\cS}^{\cdot}(\omega_{\cA^{\vee}/\cS})\to \pi^{*}\Sym_{\sO_\cS}^{\cdot}(\omega_{\cA^{\vee}/\cS})\otimes_{\sO_\cA} \pi^{*}\Sym_{\sO_\cS}^{\cdot}(\omega_{\cA^{\vee}/\cS})
\end{equation*}
This map is induced from the dual of the addition $\Lie(\cA^{\vee}/S)\oplus \Lie(\cA^{\vee}/S)\to \Lie(\cA^{\vee}/S)$ which is the diagonal $\omega_{\cA^{\vee}/\cS}\to \omega_{\cA^{\vee}/\cS}\oplus\omega_{\cA^{\vee}/\cS}$.
\end{proof}
\begin{remark}
The comultiplication structure for  $\widehat{\sP}$ and $\widehat{\sP}^{\natural}$ in the proposition just reflects the fact that the torsors associated  to $\sP$ and $\sP^{\natural}$ have a partial group law. 
\end{remark}
\begin{corollary}\label{cor:comultiplication}
One has  homomorphisms
\begin{align*}
\sP^{(n)}\to \TSym^{n}_{\sO_\cA}(\sP^{(1)})&&
\sP^{\natural(n)}\to \TSym^{n}_{\sO_\cA}(\sP^{\natural(1)})
\end{align*}
which are  isomorphisms, if $n!$ is invertible on $\cS$.
\end{corollary}
\begin{proof} Again we prove only the case $\sP^{(n)}$ as the proof for $\sP^{\natural(n)}$ is entirely similar. The map $\sP^{(n)}\to \sP^{(1)}\otimes\cdots \otimes \sP^{(1)}$ ($n$-factors) factors through $ \TSym^{n}_{\sO_\cA}(\sP^{(1)})$ as the comultiplication is co-commutative and the map on the associated graded  in degree $k\le n$ is the canonical map
\begin{equation*}
\Sym_{\sO_\cS}^{k}(\omega_{\cA^{\vee}/\cS})\to \TSym^{k}_{\sO_\cS}(\omega_{\cA^{\vee}/\cS}),
\end{equation*}
which is an isomorphism if $k!$ is invertible. 
\end{proof}

The known higher direct images of the Poincar\'e bundle  allow to compute the cohomology of $\wP$. 
\begin{theorem}[Vanishing of cohomology]\label{thm:cohomology-vanishing}There is a canonical isomorphism 
\begin{equation*}
R^{i}\pi_*(\wP\otimes {\Omega^{d}_{\cA/\cS}})\isom \begin{cases}
\sO_\cS&\mbox{ if }i=d\\
0&\mbox{ if }i\neq d.
\end{cases}
\end{equation*}
\end{theorem}
\begin{remark}One has a similar result for $\sP^{\natural}$ if  $\cS$ is a scheme over a field of characteristic zero.  A proof can be found in \cite[Theorem 1.2.1]{Scheider}. 
\end{remark}
\begin{proof} Recall from \eqref{eq:omega-i-def} that $\omega_{\cA^{\vee}/\cS}^{d}:=\Lambda^{d}\omega_{\cA^{\vee}/\cS}$. By the projection formula, one has
\begin{equation*}
R\pi_*(\wP\otimes {\Omega^{d}_{\cA/\cS}})\isom (R\pi_*\wP)\otimes {\omega^{d}_{\cA/\cS}}.
\end{equation*}
Consider the closed immersion $\id\times e^{\vee}:\cA\to \cA\times_\cS\cA^{\vee}$ the natural map $\iota_{\cA\times_\cS\widehat{\cA}^{\vee}}:\cA\times_\cS\widehat{\cA}^{\vee}\to \cA\times_\cS\cA^{\vee}$
and the projection $\pr^{\vee}:\cA\times_\cS\cA^{\vee}\to \cA^{\vee}$. Then \begin{equation*}
(\id\times e^{\vee})_*\wP\isom (\id\times e^{\vee})_*(\id\times\widehat{\pi})_*\iota_{\cA\times_\cS\widehat{\cA}}^{*}\sP\isom \iota_{\cA\times_\cS\widehat{\cA}^{\vee}*}\iota_{\cA\times_\cS\widehat{\cA}^{\vee}}^{*}\sP
\end{equation*} 
and  $\pi=\pi^{\vee}\circ \pr^{\vee}\circ (\id\times e^{\vee})$.
This gives (note that $\iota_{\cA\times_\cS\widehat{\cA}^{\vee}}$ is flat)
\begin{equation*}
R\pi_*\wP\isom R\pi^{\vee}_*R\pr^{\vee}_*R\iota_{\cA\times_\cS\widehat{\cA}^{\vee}*}\iota_{\cA\times_\cS\widehat{\cA}^{\vee}}^{*}\sP
\end{equation*}
Let $\iota_{\widehat{\cA}^{\vee}}:\widehat{\cA}^{\vee}\to {\cA}^{\vee}$ be the natural map and $\widehat{\pr}^{\vee}:\cA\times_\cS\widehat{\cA}^{\vee}\to \widehat{\cA}^{\vee}$ the projection. Then $\pr^{\vee}\circ \iota_{\cA\times_\cS\widehat{\cA}^{\vee}}=\iota_{\widehat{\cA}^{\vee}}\circ \widehat{\pr}^{\vee}$ so that
\begin{equation*}
R\pr^{\vee}_*R\iota_{\cA\times_\cS\widehat{\cA}^{\vee}*}\iota_{\cA\times_\cS\widehat{\cA}^{\vee}}^{*}\sP\isom
 R\iota_{\widehat{\cA}^{\vee}*}R\widehat{\pr}^{\vee}_*\iota_{\cA\times_\cS\widehat{\cA}^{\vee}}^{*}\sP.
\end{equation*}
By the theorem on formal functions \cite[Thm. 4.1.5]{EGAIII} one has 
\begin{equation*}
R\widehat{\pr}^{\vee}_*\iota_{\cA\times_\cS\widehat{\cA}^{\vee}}^{*}\sP
\isom \iota_{\widehat{\cA}^{\vee}}^{*} R\pr^{\vee}_*\sP
\end{equation*}
and using the well-known cohomology of the Poincar\'e bundle (see \cite[Lemme 1.2.5]{Laumon}, observe that his $\omega_{\cA^{\vee}/\cS}$ is our $\omega_{\cA^{\vee}/\cS}^{d}$) one has
\begin{equation*}
R\pr^{\vee}_*\sP\isom e^{\vee}_*\omega_{\cA^{\vee}/\cS}^{-d}[-d].
\end{equation*}
Putting these results together gives 
\begin{equation*}
R\pi_*(\wP\otimes {\Omega^{d}_{\cA/\cS}})\isom R\pi^{\vee}_*e^{\vee}_*\omega^{d}_{\cA/\cS}\otimes\omega_{\cA^{\vee}/\cS}^{-d}[-d]\isom \sO_\cS[-d]. \qedhere
\end{equation*}
\end{proof}

\begin{corollary}\label{cor:vanishing-equiv-coh} Let $\Gamma$ act by relative automorphisms on $\cA/\cS$. 
For the equivariant  cohomology (see Appendix \ref{app1}) one has 
\begin{equation*}
H^{i}(\cA,\Gamma;\wP\otimes\Omega^{d}_{\cA/\cS})\isom 
\begin{cases}
H^{0}(\cS,\sO_{\cS})&i=d\\
0&i<d.
\end{cases}
\end{equation*}
\end{corollary}
\begin{proof} 
The Leray spectral sequence
\begin{equation*}
H^{p}(\cS,\Gamma;R^{q}\pi_*(\wP\otimes\Omega^{d}_{\cA/\cS}))\Rightarrow H^{p+q}(\cA,\Gamma;\wP\otimes \Omega^{d}_{\cA/\cS})
\end{equation*}
degenerates by  Theorem \ref{thm:cohomology-vanishing} and gives 
\begin{equation*}
H^{p}(\cA,\Gamma;\wP\otimes \Omega^{d}_{\cA/\cS})\isom H^{p-d}(\cS,R^{d}\pi_*(\wP\otimes \Omega^{d}_{\cA/\cS}))\isom \begin{cases}
H^{0}(\cS,\sO_{\cS})&p=d\\
0&p<d.
\end{cases}\qedhere
\end{equation*}
\end{proof}

\subsection{The equivariant coherent Eisenstein-Kronecker class}
We consider the abelian scheme $\cA$ over  a noetherian base $\cS$ with a relative action of $\Gamma$ by automorphisms. 

\begin{definition} For $\cD\subset \cA$ a closed subscheme which is finite \'etale over $\cS$ and stable under $\Gamma$, we let 
\begin{equation*}
\sO_{\cS}[\cD]:=H^{0}(\cD,\sO_\cD).
\end{equation*}
\nomenclature{$\sO_{\cS}[\cD]$}{\nomrefpage}
Further, we write $\sO_\cS[\cD]^{0}:=\ker(H^{0}(\cD,\sO_\cD)\to H^{0}(\cS,\sO_\cS))$ where the map is induced by the trace map and let
\begin{equation*}
\sO_\cS[\cD]^{0,\Gamma}:=\ker(H^{0}(\cD,\sO_\cD)\to H^{0}(\cS,\sO_\cS))^{\Gamma}
\end{equation*}
be the $\Gamma$-invariants.
\end{definition}
Note that for an \'etale isogeny $\delta$, the scheme $\ker\delta/\cS$ is unramified. Hence the section $e$ is an open immersion, so that $\ker\delta\smallsetminus\{e(\cS)\}\subset \cA$ is still a closed subscheme which is finite \'etale over $\cS$.

We want to consider closed subschemes $\cD\subset \cA$ \'etale over $\cS$ with special properties. 
\begin{notation}\label{not:D-def} Let $\delta:\cA\to \cA'$ be an \'etale isogeny with \'etale dual $\delta^{\vee}$.   We denote by $\cD$ one of the two closed subschemes 
\begin{align*}
\cD:= \ker\delta&&\mbox{or}&& \cD:=\ker\delta\smallsetminus \{e(\cS)\}.
\end{align*}
Let $\iota:\cD\to \cA$ be the closed immersion and $\pi_{\cD}:\cD\to \cS$ the finite \'etale structure map. We let  $\cU_\cD:=\cA\smallsetminus \cD$ be the complement of $\cD$ and $j:\cU_\cD\to \cA$ the open immersion. 
\end{notation}
Note that $\cD$ and $\cU_{\cD}$ are  $\Gamma$-stable subschemes and that   $\iota:\cD\to \cA$ is a locally complete intersection (see \cite[III. Prop. 1.2]{Hartshorne}). 
Summarizing, one obtains the following $\Gamma$-equivariant diagram
\begin{equation}\label{eq:set-up-over-R}
\xymatrix{
\cD\ar[r]^{\iota}\ar[rd]_{\pi_\cD}& \cA\ar[d]^{\pi}&\cU_\cD:=\cA\smallsetminus \cD\ar[l]_/.7em/{j}\ar[ld]^{\pi_\cU}\\&\cS.
}
\end{equation}
Moreover, for such $\cD$ one has (as $\delta$ and $\delta^{\vee}$ are \'etale) by Corollary \ref{cor:coh-log-splitting} an isomorphism
\begin{equation*}
\delta_\#:\wP{\mid_{\cD}}\isom \pi_\cD^{*}\widehat{\sO}_{\cA'}. 
\end{equation*}
With Corollary \ref{cor:vanishing-equiv-coh} the equivariant localization sequence (see Appendix \ref{sec:appendix_equivariant}) for  the closed subscheme $\cD\subset \cA$ gives 
\begin{equation}\label{eq:localization-sequence}
 0\to H^{d-1}(\cU_\cD,\Gamma;\wP\otimes \Omega^{d}_{\cA/\cS} )\to H^{d}_\cD(\cA,\Gamma;\wP\otimes \Omega^{d}_{\cA/\cS} )\to H^{0}(\cS,\sO_\cS).
\end{equation}
The next theorem associates to  functions in $\sO_{\cS}[\cD]^{0,\Gamma}$ cohomology classes.
\begin{theorem}\label{thm:coh-with-support-descr}Let 
$\cD$ be as in Notation \ref{not:D-def}. Then there is a canonical inclusion 
\begin{equation*}
\sO_\cS[\cD]^{\Gamma}\subset  H^{d}_\cD(\cA,\Gamma;\wP\otimes \Omega^{d}_{\cA/\cS} )
\end{equation*}
whose composition with 
\begin{equation*}
H^{d}_\cD(\cA,\Gamma;\wP\otimes \Omega^{d}_{\cA/\cS} )\to H^{0}(\cS,\sO_\cS)
\end{equation*}
is the trace map $\sO_\cS[\cD]^{\Gamma}\to H^{0}(\cS,\sO_\cS)$. In particular, one has a map
\begin{equation*}
\EK_{\Gamma,\cA}:\sO_\cS[\cD]^{0,\Gamma}\hookrightarrow 
H^{d-1}(\cU_\cD,\Gamma;\wP\otimes \Omega^{d}_{\cA/\cS} ).
\end{equation*}
Furthermore one has 
\begin{equation*}
H^{d-1}(\cU_\cD,\Gamma;\wP\otimes \Omega^{d}_{\cA/\cS} )\isom \prolim
H^{d-1}(\cU_\cD,\Gamma;\sP^{(n)}\otimes \Omega^{d}_{\cA/\cS} ).
\end{equation*}
\end{theorem}
The proof will be given in Section \ref{section-proof-of-thm}.

\begin{definition}[Equivariant coherent Eisenstein-Kronecker class]\label{def:equiv-coh-EK}Let $\EK_{\Gamma,\cA}$ be the map from Theorem \ref{thm:coh-with-support-descr}. 
For $f\in \sO_\cS[\cD]^{0,\Gamma}$  we call 
\begin{equation*}
\EK_{\Gamma,\cA}(f)\in H^{d-1}(\cU_\cD,\Gamma;\wP\otimes \Omega^{d}_{\cA/\cS} )
\end{equation*}
\nomenclature{$\EK_{\Gamma,\cA}(f)$}{\nomrefpage}
the \emph{equivariant coherent Eisenstein-Kronecker class} associated to $f$. 
Using the map $\sP^{(n)}\to \sP^{\natural(n)}$ from \eqref{eq:sP-inclusion} we denote by
\begin{equation*}
\EK_{\Gamma,\cA}^{\natural}(f)\in H^{d-1}(\cU_\cD,\Gamma;\wP^{\natural}\otimes \Omega^{d}_{\cA/\cS} )
\end{equation*}
the image of $\EK_{\Gamma,\cA}(f)$ under the homomorphisms
\begin{equation*}
H^{d-1}(\cU_\cD,\Gamma;\wP\otimes \Omega^{d}_{\cA/\cS} )\to
H^{d-1}(\cU_\cD,\Gamma;{\wP}^{\natural}\otimes \Omega^{d}_{\cA/\cS} ).
\end{equation*}
If $\cA$ is clear from the context, we write $\EK_{\Gamma}(f)$ and
$\EK_{\Gamma}^{\natural}(f)$.
\end{definition}

\begin{remark}
	In the case $d=1$, i.e. $\cA/\cS$ is an elliptic curve, the Eisenstein--Kronecker classes for $\Gamma=\{1\}$ coincide with the so-called \emph{infinitesimal Kronecker sections} 
\begin{equation*}
(l_n^D)_n\in \varprojlim_nH^0(\cU_\cD,\sP^{\natural(n)}\otimes\Omega^1_{\cA/\cS})\cong H^0(\cU_\cD,\widehat{\sP}^{\natural}\otimes\Omega^1_{\cA/\cS})
\end{equation*} 
which have been defined and studied in \cite{Sprang-deRham}.
\end{remark}

The equivariant coherent Eisenstein-Kronecker class  $\EK_{\Gamma}^{\natural}(f)$ is not yet sufficient to construct all cohomology classes we need for the critical $L$-values. 
%
Using the connection on ${\wP^{\natural}}$
\begin{equation*}
\nabla:=\nabla_{{\wP^{\natural}}}:{\wP^{\natural}}\to \Omega_{\cA/\cS}^{1}\otimes{\wP^{\natural}}
\end{equation*}
we construct further cohomology classes out of $\EK_{\Gamma}^{\natural}(f)$ as follows: 
Iterating this connection $a$-times gives a map 
\begin{equation*}
\nabla^{a}:{\widehat{\sP}^{\natural}}\to \TSym^{a}(\Omega_{\cA/\cS}^{1})\otimes{\widehat{\sP}^{\natural}}.
\end{equation*}
\begin{definition}\label{def:pol-derivatives} Let $f\in \sO_{\cS}[\cD]^{0,\Gamma}$. 
The image of $\EK_{\Gamma}^{\natural}(f)$ under $\nabla^{a}$  \begin{equation*}
\nabla^{a}\EK_{\Gamma}^{\natural}(f)\in 
H^{d-1}(\cU_\cD,\Gamma;\TSym^{a}(\Omega_{\cA/\cS}^{1})\otimes{\widehat{\sP}^{\natural}}\otimes \Omega^{d}_{\cA/\cS} ).
\end{equation*}
is called 
\emph{the $a$-th derivative of the equivariant coherent Eisenstein-Kronecker class}.
\end{definition}
We want to specialize these classes at torsion sections and decompose them into components via the moment map from Definition \ref{def:moment-x}
\begin{equation*}
_\varrho\mom_x^{b}:x^{*}\widehat{\sP^{\natural}}\to{\TSym}^{b}(\sH)
\end{equation*}
where $x:\cS\to \cU_{\cD}$ is a $\varphi$-torsion section where $\varphi$ is some isogeny $\varphi:\cA\to \cB$ with \'etale dual $\varphi^{\vee}$. Assume also that $x$ is  fixed by $\Gamma$.
Then the class $\nabla^{a}\EK_{\Gamma}^{\natural}(f)$ can be pulled-back by $x$ to give
\begin{equation*}
x^{*}\nabla^{a}\EK_{\Gamma}^{\natural}(f)\in 
 H^{d-1}(\cS,\Gamma;\TSym^{a}(\omega_{\cA/\cS})\otimes x^{*}{\widehat{\sP}^{\natural}}\otimes \omega^{d}_{\cA/\cS} ).
\end{equation*}
Composing with  the moment map $_\varrho\mom_x^{b}$
gives a class
\begin{equation}\label{eq:Eis-class}
_\varrho\mom^{b}_{x}(\nabla^{a}\EK_{\Gamma}^{\natural}(f))\in 
H^{d-1}(\cS,\Gamma;\TSym^{a}(\omega_{\cA/\cS})\otimes{\TSym}^{b}(\sH)\otimes \omega^{d}_{\cA/\cS} ).
\end{equation}
If $\cS=\cR=\Spec R$ is affine, the spectral sequence for equivariant cohomology \eqref{eq:coh-equiv-spectral-seq} collapses so that
\begin{multline}\label{eq:collapsing}
H^{d-1}(\cR,\Gamma;\TSym^{a}(\omega_{\cA/\cR})\otimes\TSym^{b}(\sH_\cA)\otimes \omega^{d}_{\cA/\cR} )\\
\isom 
H^{d-1}(\Gamma,\TSym^{a}(\omega_{\cA/\cR})\otimes\TSym^{b}(\sH)\otimes \omega^{d}_{\cA/\cR} ),
\end{multline} 
where we consider $\TSym^{a}(\omega_{\cA/\cR})\otimes\TSym^{b}(\sH)\otimes \omega^{d}_{\cA/\cR}$ as $R$-module.

\begin{definition}[Eisenstein-Kronecker class at $x$]\label{def:Eis-classes} Let $\cS=\cR$ affine and 
$x:\cR\to \cU_{\cD}$ a $\varphi$-torsion section fixed by 
$\Gamma\subset \Aut_{\cR}(\cA)$  and assume that $\varphi^{\natural}$ is \'etale. For $f\in R[\cD]^{0,\Gamma}$ we let 
\begin{equation*}
\Eis_{\Gamma}^{b,a}(f,x)\in 
H^{d-1}(\Gamma,\TSym^{a}(\omega_{\cA/\cR})\otimes\TSym^{b}(\sH)\otimes \omega^{d}_{\cA/\cR} )
\end{equation*}
be the image of  the class $_\varrho\mom^{b}_ x(\nabla^{a}\EK_{\Gamma}^{\natural}(f))$ in \eqref{eq:Eis-class} under the map in \eqref{eq:collapsing}. 
\end{definition}
\begin{remark} Let $\cA_0$ be an abelian scheme with CM by $\sO_K$ the ring of integers in a CM field $K$ and let $L$ be a finite extension of $K$. Then $\GL_{\sO_K}(\sO_L)$ is the automorphism group of $\sO_L\otimes_{\sO_K}\cA_0$ and one gets a group cohomology class
\begin{equation*}
\Eis_{\Gamma}^{b,a}(f,x)\in H^{d-1}(\Gamma,\TSym^{a}(\omega_{\cA/\cR})\otimes\TSym^{b}(\sH)\otimes \omega^{d}_{\cA/\cR} ).
\end{equation*}
for some $\Gamma\subset \GL_{\sO_K}(\sO_L)$ of finite index. 
For our purposes it is enough to consider only the restriction of this class to a subgroup of $ \sO_L^{\times}\subset \GL_{\sO_K}(\sO_L)$. 
\end{remark}

\subsection{The Eisenstein-Kronecker class for CM abelian varieties}
\label{subsection:Eisenstein-Kronecker-class}
In this section we decompose  the Eisenstein-Kronecker classes further in the case of CM abelian varieties.

Consider $(\cA/\cR,\omega(\cA),\omega(\cA^{\vee}),x)$ as in Notation \ref{notation-for-CM-ab}. In particular, $R\subset \C$ is a $\sO_{L^{Gal}}[1/d_L]$-algebra, $\cR:=\Spec R$ and $\cA/\cR$  an abelian scheme with CM by $\sO_L$. The point $x$ is a $\frf$-torsion section for some ideal $\frf\subset \sO_L$. We assume in this section that $\norm\frf\in R^{\times}$, so that $[\frf]$ and $[\frf]^{\vee}$ are \'etale. This restriction will be removed in Section \ref{section:integral-refinement}. 

Let us further specify $\cD\subset \cA$ in case where $\cA$ has CM.
\begin{definition}\label{def:D-specific}
Fix an integral ideal $\frc\subset \sO_L$ coprime to $\frf$, such that $N\frc\in R^{\times}$, i.e. $[\frc]$ and $[\frc]^{\vee}$ are \'etale. Let
$\cD=\cA[\frc]\setminus \{x(\cR)\}$ be as in Notation \ref{not:D-def} so that either 
\begin{equation*}
\cD=\cA[\frc]\mbox{ or }\cD=\cA[\frc]\setminus e(\cS).
\end{equation*}
Note that $\cD$ is $\Gamma$-stable for $\Gamma\subset \sO_L^{\times}$. We will consider functions $f\in R[\cD]^{0,\Gamma}$.
\end{definition}

Any subgroup $\Gamma\subset \sO_L^{\times}$ of finite index acts on $\cA$ and we use this action to further decompose the class $\Eis_{\Gamma}^{b,a}(f,x)$ of Definition \ref{def:Eis-classes}.
Recall for this the decomposition results from Proposition \ref{prop:CM-decompostion}. 

The splitting of the Hodge filtration $\sH\isom\sH(\ol{\Sigma})\oplus\sH(\Sigma)$ and the isomorphism $\sH(\ol{\Sigma})\isom \omega_{\cA^{\vee}/\cR}$ induces a projection
\begin{equation*}
\TSym^{b}(\sH)\to  \TSym^{b}(\omega_{\cA^{\vee}/\cR})
\end{equation*}
and by Corollary \ref{cor:Gamma-invariants} the $\Gamma$-invariants of 
$\TSym^{a}(\omega_{\cA/\cR})\otimes\TSym^{b}(\omega_{\cA^{\vee}/\cR})\otimes \omega^{d}_{\cA/\cR}$
are given by
\begin{equation*}
\bigoplus_{\stackrel{\beta-\alpha-\ul{1}\in \CI_L(\Gamma)}{|\alpha|=a,\ |\beta|=b}}\TSym^{\alpha}(\omega_{\cA/\cR})\otimes \TSym^{\beta}(\omega_{\cA^{\vee}/\cR})\otimes\omega^{d}_{\cA/\cR}
\end{equation*}
which is a direct summand and a trivial $\Gamma$-module. Denoting by $\pr_\Gamma$ the projection onto the $\Gamma$-invariants, this allows to project $\Eis_{\Gamma}^{b,a}(f,x)$ to
\begin{equation}\label{eq:pr-Eis}
\pr_\Gamma (\Eis_{\Gamma}^{b,a}(f,x))\in \bigoplus_{\stackrel{\beta-\alpha-\ul{1}\in \CI_L(\Gamma)}{|\alpha|=a,\ |\beta|=b}} H^{d-1}(\Gamma,\TSym^{\alpha}(\omega_{\cA/\cR})\otimes\TSym^{\beta}(\omega_{\cA^{\vee}/\cR})\otimes \omega^{d}_{\cA/\cR} ).
\end{equation}
We would like to get rid of the group cohomology, which means to evaluate on some fundamental cycle. In our case this can be done canonically.
\begin{proposition}\label{prop:orientation} Let $\Gamma\subset \sO_L^{\times}$ be of finite index, and let $\ul{1}\in I_\Sigma$ be the element defined in Notation \ref{not:I-Sigma}. For $\alpha\in I_{{\Sigma}}^+$ and $\beta\in I_{\ol{\Sigma}}^+$  with $\beta-\alpha-\ul{1}\in \CI_L(\Gamma)$, there is a canonical homomorphism
\begin{multline*}
H^{d-1}(\Gamma, \TSym^{\alpha}(\omega_{\cA/\cR})\otimes_R \TSym^{\beta}(\omega_{\cA^{\vee}/\cR})\otimes_R\omega^{d}_{\cA/\cR})\to\\ 
\TSym^{\alpha+\ul{1}}(\omega_{\cA/\cR})\otimes_R \TSym^{\beta}(\omega_{\cA^{\vee}/\cR}).
\end{multline*}
\end{proposition}
\begin{proof}Because $ \TSym^{\alpha}(\omega_{\cA/\cR})\otimes_R \TSym^{\beta}(\omega_{\cA^{\vee}/\cR})\otimes_R\omega^{d}_{\cA/\cR}$ is a trivial $\Gamma$-module, it is sufficient to define a canonical homomorphism
$H^{d-1}(\Gamma,\Z)\otimes\omega^{d}_{\cA/\cR}\to \TSym^{\ul{1}}(\omega_{\cA/\cR})$. 
Let $\Gamma'\subset\Gamma$ be free of finite index. Then $H_{d-1}(\Gamma',\Z)$ is non-canonically isomorphic to $\Z$ and for any generator $\xi'\in H_{d-1}(\Gamma',\Z)$ we get an isomorphism
\begin{equation*}
H^{d-1}(\Gamma',\Z)\xrightarrow{\cap \xi'} \Z.
\end{equation*}
Choosing an ordering of $\Sigma$ allows to choose an orientation 
on $\R\otimes L$ and hence one on $L_\R^{1}:=\ker((\R\otimes L)^{\times}\xrightarrow{\norm_{L/\Q}}\R^{\times})$. This orientation induces an isomorphism $\Z\isom H_{d-1}(\Gamma'\backslash L_\R^{1},\Z)\isom H_{d-1}(\Gamma',\Z)$. 
Using the same ordering we get at the same time an isomorphism $\omega^{d}_{\cA/\cR}\isom \TSym^{\ul{1}}\omega_{\cA/\cR}$ and hence an isomorphism 
\begin{equation*}
H^{d-1}(\Gamma',\Z)\otimes\omega^{d}_{\cA/\cR}\isom \TSym^{\ul{1}}\omega_{\cA/\cR}
\end{equation*}
independent of the choice of the generator $\xi'$ and independent of the choice of the ordering. To treat the general case, we first show that 
\begin{equation*}
\res: H_{d-1}(\Gamma,\Z)\to H_{d-1}(\Gamma',\Z)
\end{equation*} 
is surjective. This follows by first decomposing $\Gamma=\Gamma^{\free}\times \Gamma^{\tors}$ into its free and torsion part. Then $\res: H_{d-1}(\Gamma,\Z)\to H_{d-1}(\Gamma^{\free},\Z)$ is surjective. Moreover, $\res: H_{d-1}(\Gamma^{\free},\Z)\to H_{d-1}(\Gamma',\Z)$ is an isomorphism, which can be seen as follows: One has $\Lambda^{i}\Gamma^{\free}\isom H_i(\Gamma^{\free},\Z)$ and similarly for $H_i(\Gamma',\Z)$. The map $\cor:  H_{i}(\Gamma',\Z)\to H_{i}(\Gamma^{\free},\Z)$ is induced by the inclusion $\Lambda^{i}\Gamma'\subset \Lambda^{i}\Gamma^{\free}$. Hence,  $\cor:H_{d-1}(\Gamma',\Z)\to H_{d-1}(\Gamma^{\free},\Z)$ is an inclusion onto a subgroup of index $[\Gamma^{\free}:\Gamma']$. As  $\cor\circ \res=[\Gamma^{\free}:\Gamma']$ 
one sees  that $\res:H_{d-1}(\Gamma',\Z)\to H_{d-1}(\Gamma^{\free},\Z)$ is an isomorphism. This shows the claim.

Let $\xi\in H_{d-1}(\Gamma,\Z)$ be an element with $\res(\xi)=\xi'\in H_{d-1}(\Gamma',\Z)$  and define using the same ordering of $\Sigma$
\begin{equation*}
H^{d-1}(\Gamma,\Z)\otimes \omega^{d}_{\cA/\cR}\xrightarrow{\cap\xi} \TSym^{\ul{1}}\omega_{\cA/\cR}.
\end{equation*} 
To see that this is independent of the choice of $\xi$ and the ordering, we note that the diagram
\begin{equation}\label{eq:res}
\xymatrix{
H^{d-1}(\Gamma,\Z)\otimes\omega^{d}_{\cA/\cR}\ar[r]^{\cap\xi}\ar[d]_\res& \TSym^{\ul{1}}\omega_{\cA/\cR}\ar[d]^{[\Gamma:\Gamma']}\\
H^{d-1}(\Gamma',\Z)\otimes \omega^{d}_{\cA/\cR}\ar[r]^{\cap\res \xi}&\TSym^{\ul{1}}\omega_{\cA/\cR}
}
\end{equation}
commutes, because for any $\eta\in H^{d-1}(\Gamma,\omega^{d}_{\cA/\cR})$ one has by the projection formula
\begin{equation*}
\cor(\res\eta\cap\res\xi)=\eta\cap \cor\circ\res\xi=[\Gamma:\Gamma']\eta\cap\xi
\end{equation*}
and $\cor:H_0(\Gamma',\Z)\to H_0(\Gamma,\Z)$ is an isomorphism.
As $R\subset \C$ is torsion free, the multiplication with $[\Gamma:\Gamma']$ is injective. This shows the proposition.
\end{proof}
Recall that we let $(\cA/\cR,\omega(\cA),\omega(\cA^{\vee}),x)$ be as in Notation \ref{notation-for-CM-ab}, where $x$ is $[\frf]$-torsion section and $N\frf\in R^{\times}$.
The bases $\omega(\cA),\omega(\cA^{\vee})$ define a trivialization 
\begin{equation}\label{eq:trivialization}
\TSym^{\alpha+\ul{1}}(\omega_{\cA/\cR})\otimes_R\TSym^{\beta}(\omega_{\cA^{\vee}/\cR})\isom R.
\end{equation}
\begin{definition}\label{def:Eis-beta-alpha} Let $x$ be a  $[\frf]$-torsion section with $x\in \cU_{\cD}$. Let  $\Gamma\subset \sO_L^{\times}$ be of finite index fixing $x$,
$f\in  R[\cD]^{0,\Gamma}$, and $\beta-\alpha-\ul{1}\in \CI_L(\Gamma)$ be critical. We define the class \nomenclature{$\Eis_{\Gamma}^{\beta,\alpha}(f,x)$}{\nomrefpage}
\begin{align*}
 \Eis_{\Gamma}^{\beta,\alpha}(f,x)\in 
\TSym^{\alpha+\ul{1}}(\omega_{\cA/\cR})\otimes_R\TSym^{\beta}(\omega_{\cA^{\vee}/\cR})
\end{align*}
to be the $\alpha,\beta $-component of $\pr_{\Gamma}\Eis^{b,a}_{\Gamma}(f,x)$ in \eqref{eq:pr-Eis} using the map in Proposition \ref{prop:orientation}. Using the trivialization \eqref{eq:trivialization} by the bases $\omega(\cA)^{[\alpha+\ul{1}]},\omega(\cA^{\vee})^{[\beta]}$ gives an element of $R$ which is denoted by
\begin{align*}
\Eis_{\Gamma}^{\beta,\alpha}(f,x)(\omega(\cA)^{[\alpha+\ul{1}]},\omega(\cA^{\vee})^{[\beta]})\in R.
\end{align*}
\end{definition}
For later use we record the behaviour of $\Eis_{\Gamma}^{\beta,\alpha}(f,x)$ under change of $\Gamma$.
\begin{corollary}\label{cor:change-of-group}Let $\Gamma'\subset \Gamma\subset \sO_L^{\times}$ be two groups of finite index and $\beta-\alpha-\ul{1}\in \CI_L(\Gamma)$, then
\begin{equation*}
\Eis_{\Gamma'}^{\beta,\alpha}(f,x)=[\Gamma:\Gamma']\Eis_{\Gamma}^{\beta,\alpha}(f,x).
\end{equation*}
\end{corollary}
\begin{proof}
This follows immediately from the commutative diagram \eqref{eq:res}.
\end{proof}

\subsection{A refinement of the Eisenstein-Kronecker class}\label{section:integral-refinement} In this section we define a refinement of the Eisenstein-Kronecker class from Definition \ref{def:Eis-beta-alpha} which has better integrality properties. This definition  is inspired by a construction in \cite{BKL-topological-polylog}. The price to pay is that this class depends on the isogeny $[\frf]$ for an integral ideal $\frf\subset \sO_L$ such that the torsion section $x\in\cA[\frf]$.

Fix an integral ideal $\frf\subset \sO_L$. We keep the notation from section \ref{subsection:Eisenstein-Kronecker-class}, so that  
\begin{equation*}
(\cA/\cR,\omega(\cA),\omega(\cA^{\vee}),x)
\end{equation*}is as in Notation \ref{notation-for-CM-ab} and $\cD= \cA[\frc]$ for an ideal $\frc\neq \sO_L$ coprime to $\frf$ and $N\frc$ invertible in $R$. Let $x:\cS\to \cU_\cD$ be a $\frf$-torsion section fixed by $\Gamma$. Here $[\frf]$ is the isogeny 
\begin{equation*}
[\frf]:\cA\to \cB:=\cA(\frf^{-1}).
\end{equation*} 
We do \emph{not} assume in this section
that $\norm\frf$ is invertible in $R$, so that $[\frf]$ and $[\frf]^{\vee}$ are not necessarily \'etale. 

\begin{theorem}[Refined Eisenstein-Kronecker class]\label{thm:refined-EK}
Let $\frf$ and $\frc\neq \sO_L$ and $x$ be as above. For $\cB:=\cA(\frf^{-1})$ fix bases $\omega(\cB)$ of $\omega_{\cB/\cR}$  and $\omega(\cB)$ of $\omega_{\cB^{\vee}/\cR}$.  Then for each 
$f\in R[\cD]^{0,\Gamma}$ there  is a  section 
\begin{equation*}
_{[\frf]}\Eis_{\Gamma,\cB}^{\beta,\alpha}(f,x)\in \TSym^{\alpha+\ul{1}}(\omega_{\cB/\cR})\otimes_R \TSym^{\beta}(\omega_{\cB^{\vee}/\cR})
\end{equation*}
depending on $\frf$ called the \emph{refined Eisenstein-Kronecker class}. Using the bases $\omega(\cB),\omega(\cB^{\vee})$ one gets as in Definition \ref{def:Eis-beta-alpha} 
\begin{align*}
_{[\frf]}\Eis_{\Gamma,\cB}^{\beta,\alpha}(f,x)(\omega(\cB)^{[\alpha+\ul{1}]},\omega(\cB^{\vee})^{[\beta]})\in R.
\end{align*}
\end{theorem}
\begin{proof}
Consider the class 
\begin{equation*}
\Eis_{\Gamma,\cA}^{\natural}(f)\in \prolim_n H^{d-1}(\cU_\cD,\Gamma;{\sP^{\natural(n)}}\otimes \Omega^{d}_{\cA/\cR})
\end{equation*}
from Definition \ref{def:equiv-coh-EK}. Recall from \eqref{eq:omega-equality} that 
$\Omega^{d}_{\cA/\cS}\isom  \pi^{*}\omega^{d}_{\cA^{\vee}/\cS}$.
Applying the maps $[\frf]_\#:{\sP_\cA^{\natural(n)}}\to [\frf]^{*}{\sP_{\cB}^{\natural(n)}}$ and $[\frf]^{\vee *}:\omega_{\cA^{\vee}/\cR}\to \omega_{\cB^{\vee}/\cR}$ gives a class
\begin{equation*}
[\frf]^{\vee *}[\frf]_\#\Eis_{\Gamma,\cA}^{\natural}(f)\in\prolim_n
H^{d-1}(\cU_\cD,\Gamma;[\frf]^{*}{\sP_{\cB}^{\natural(n)}}\otimes \Omega^{d}_{\cB/\cR}),
\end{equation*}
where we have used the identification ${\pi}^{*}\omega^{d}_{\cB^{\vee}/\cR}\isom\Omega^{d}_{\cB/\cR}$. This class  can be derived with the operator
\begin{equation*}
[\frf]^{*}(\nabla^{a}):[\frf]^{*}{\sP_{\cB}^{\natural(n)}}\to[\frf]^{*}\TSym^{a}(\Omega^{1}_{\cB/\cR})\otimes [\frf]^{*}{\sP_{\cB}^{\natural(n)}}.
\end{equation*}
This gives
\begin{equation*}
[\frf]^{*}(\nabla^{a})([\frf]^{\vee *}[\frf]_\#\Eis_{\Gamma,\cA}^{\natural}(f))\in \prolim_n
H^{d-1}(\cU_\cD,\Gamma;\pi^{*}\TSym^{a}(\omega_{\cB/\cR})\otimes[\frf]^{*}{\sP_{\cB}^{\natural(n)}}\otimes \Omega^{d}_{\cB/\cR}).
\end{equation*}
Pulling-back by $x$, using the moment map $\mom_{e_\cB}^{b}:x^{*}
[\frf]^{*}\widehat{\sP_{\cB}^{\natural}}\isom e_\cB^{*}\widehat{\sP_{\cB}^{\natural}}\to \TSym^{b}(\sH_\cB)$, 
and observing that $\cR$ is affine gives a class
\begin{equation*}
\mom_{e_\cB}^{b}x^{*}[\frf]^{*}(\nabla^{a})([\frf]^{\vee *}[\frf]_\#\Eis_{\Gamma,\cA}^{\natural}(f))\in 
H^{d-1}(\Gamma,\TSym^{a}(\omega_{\cB/\cR})\otimes\TSym^{b}(\sH_\cB)\otimes \omega^{d}_{\cB/\cR}).
\end{equation*}
Applying the projection 
$\TSym^{b}(\sH_\cB)\to \TSym^{b}(\omega_{\cB^{\vee}/\cR})$, the evaluation from Proposition \ref{prop:orientation}, and the decomposition into $\alpha,\beta$, 
yields the desired section
\begin{equation*}
_{[\frf]}\Eis_{\Gamma,\cB}^{\beta,\alpha}(f,x)\in
\TSym^{\alpha+\ul{1}}(\omega_{\cB/\cR})\otimes \TSym^{\beta}(\omega_{\cB^{\vee}/\cR}).\qedhere
\end{equation*}
\end{proof}
We want to compare this refined class with the one defined in Definition \ref{def:Eis-beta-alpha}. For this consider the map $[\frf]^{*}\Omega_{\cB/\cR}^{1}\to \Omega_{\cA/\cR}^{1}$, which induces
\begin{equation*}
\TSym^{\alpha+\ul{1}}([\frf]^{*}):\TSym^{\alpha+\ul{1}}(\omega_{\cB/\cR})\to \TSym^{\alpha+\ul{1}}(\omega_{\cA/\cR})
\end{equation*}
and the image of the refined class $_{[\frf]}\Eis_{\Gamma,\cB}^{\beta,\alpha}(f,x)$ under this map
\begin{equation*}
\TSym^{\alpha+\ul{1}}([\frf]^{*}) (_{[\frf]}\Eis_{\Gamma,\cB}^{\beta,\alpha}(f,x))\in \TSym^{\alpha+\ul{1}}(\omega_{\cA/\cR})\otimes \TSym^{\beta}(\omega_{\cB^{\vee}/\cR}).
\end{equation*}
On the other hand one has by the covariant functoriality $[\frf]_\#=[\frf]^{\vee *}:\omega_{\cA^{\vee}/\cR}\to \omega_{\cB^{\vee}/\cR}$ a map 
\begin{equation*}
\TSym^{\beta}([\frf]_\#):\TSym^{\beta}(\omega_{\cA^{\vee}/\cR})\to \TSym^{\beta}(\omega_{\cB^{\vee}/\cR}),
\end{equation*}
which applied to non-refined class $\Eis_{\Gamma,\cA}^{\beta,\alpha}(f,x)$ gives, over  the base is $R[\frac{1}{\norm \frf}]$,
\begin{equation*}
\TSym^{\beta}([\frf]_\#)(\Eis_{\Gamma,\cA}^{\beta,\alpha}(f,x))\in
\TSym^{\alpha+\ul{1}}(\omega_{\cA/\cR[\frac{1}{\norm \frf}]})\otimes \TSym^{\beta}(\omega_{\cB^{\vee}/\cR[\frac{1}{\norm \frf}]}).
\end{equation*}
The natural comparison map induced by the flat map $R\to R[\frac{1}{\norm\frf}]$ 
\begin{equation}\label{eq:comparison-map}
\TSym^{\alpha+\ul{1}}(\omega_{\cA/\cR})\otimes\TSym^{\beta}(\omega_{\cB^{\vee}/\cR})\to 
\TSym^{\alpha+\ul{1}}(\omega_{\cA/\cR[\frac{1}{\norm\frf}]})\otimes\TSym^{\beta}(\omega_{\cB^{\vee}/\cR[\frac{1}{\norm\frf}]})
\end{equation}
(which is injective as $R\subset \C$ is torsion free) allows to compare the two classes and one gets:
\begin{proposition}\label{prop:integral-Eis-class}
Let $x\in \cA[\frf](\cR)$ be a $[\frf]$-torsion section and
$f\in R[\cD]^{0,\Gamma}$. Then under the map in \eqref{eq:comparison-map} one has 
\begin{equation*}
\TSym^{\alpha+\ul{1}}([\frf]^{*}) (_{[\frf]}\Eis_{\Gamma,\cB}^{\beta,\alpha}(f,x))=
\TSym^{\beta}([\frf]_\#)(\Eis_{\Gamma,\cA}^{\beta,\alpha}(f,x))
\end{equation*}
in 
\begin{equation*}
\TSym^{\alpha+\ul{1}}(\omega_{\cA/\cR})\otimes_R \TSym^{\beta}(\omega_{\cB^{\vee}/\cR})\otimes_RR[\frac{1}{\norm\frf}].
\end{equation*}
In particular,
\begin{equation*}
\TSym^{\beta}([\frf]_\#)(\Eis_{\Gamma,\cA}^{\beta,\alpha}(f,x))(\omega(\cA)^{[\alpha+\ul{1}]},\omega(\cB^{\vee})^{[\beta]})\in R\subset R[\frac{1}{\norm\frf}].
\end{equation*}
\end{proposition}
\begin{proof}We may assume $R=R[\frac{1}{\norm\frf}]$. As $[\frf]_\#:\widehat{\sP}_{\cA}^{\natural}\to [\frf]^{*}\widehat{\sP}_{\cB}^{\natural}$ is a morphism of sheaves with connection, one has a commutative diagram (using $x^{*}[\frf]^{*}=e_\cB^{*}$)
\begin{equation*}
\xymatrixcolsep{7pc}\xymatrix{
x^{*}\widehat{\sP}_{\cA}^{\natural}\ar[r]^{x^{*}\nabla_\cA^{\alpha}}\ar[d]_{[\frf]_{\#,x}}&\TSym^{\alpha}(\omega_{\cA/\cR})\otimes x^{*}\widehat{\sP}_{\cA}^{\natural}\ar[dd]^{\id\otimes [\frf]_\#}\\
e_\cB^{*}\widehat{\sP}_\cB^{\natural}\ar[d]_{e_\cB^{*}\nabla^{\alpha}}\\
\TSym^{\alpha}(\omega_{\cB/\cR})\otimes e_\cB^{*}\widehat{\sP}_\cB^{\natural}\ar[r]^{\TSym^{\alpha}([\frf]^{*})}&\TSym^{\alpha}(\omega_{\cA/\cR})\otimes e_\cB^{*}\widehat{\sP}_\cB^{\natural}.
}
\end{equation*}
Further we claim that the diagram
\begin{equation*}
\xymatrixcolsep{5pc}\xymatrix{
x^{*}\widehat{\sP}_{\cA}^{\natural}\ar[r]^{_\varrho\mom_x^{\beta}}\ar[d]_{[\frf]_{\#, x}}&\TSym^{\beta}(\sH_{\cA})\ar[d]^{\TSym^{\beta}([\frf]_\#)}\\
e_\cB^{*}\widehat{\sP}_\cB^{\natural}\ar[r]^{\mom_{e_\cB}^{\beta}}&\TSym^{\beta}(\sH_\cB)
}
\end{equation*}
commutes. Recall that $_\varrho\mom_x^{\beta}=\mom_{e_\cA}^{\beta}\circ\varrho_x$ and that
$\varrho_x$ is the composition
\begin{equation*}
x^{*}\widehat{\sP}_{\cA}^{\natural}\xrightarrow[\isom]{[\frf]_{\#,x}}e_\cB^{*}\widehat{\sP}_{\cB}^{\natural}\xrightarrow[\isom ]{[\frf]_{\#,e_\cA}^{-1}}e_\cA^{*}\widehat{\sP}_{\cA}^{\natural}.
\end{equation*}  
(see Corollary \ref{cor:coh-log-splitting}). With this the commutativity of the diagram follows from the functoriality of the moment map, which gives
\begin{equation*}
\TSym^{\beta}\circ \mom_{e_\cA}^{\beta}=\mom_{e_\cB}^{\beta}\circ[\frf]_{\#,e_\cA}\qedhere
\end{equation*}
\end{proof}
\subsection{Proof of Theorem \ref{thm:coh-with-support-descr}}\label{section-proof-of-thm}
We start with the computation of an $\Ext$-group.
\begin{proposition}
Let $\sI\subset \sO_\cA$ be the ideal sheaf defining $\cD$ and 
write $\sO_{\cA_k}:=\sO_\cA/\sI^{k+1}$. Then for all $k\ge 0$
\begin{equation*}
\Ext^{q}_{\sO_\cA}(\sI^{k}/\sI^{k+1},\sP^{(n)}\otimes\Omega^{d}_{\cA/\cS})\isom \begin{cases}
0& q<d\\
H^{q-d}(\cD,\ul{\Hom}_{\sO_{\cD}}(\sI^{k}/\sI^{k+1},\iota^{*}\sP^{(n)}))&q\ge d.
\end{cases}
\end{equation*}
\end{proposition}
\begin{proof}
The local to global spectral sequence for $\Ext$ gives
\begin{equation*}
H^{p}(\cA,\sExt_{\sO_\cA}^{q}(\sI^{k}/\sI^{k+1},\sP^{(n)}\otimes\Omega^{d}_{\cA/\cS}))\Rightarrow \Ext^{p+q}_{\sO_\cA}(\sI^{k}/\sI^{k+1},\sP^{(n)}\otimes\Omega^{d}_{\cA/\cS}).
\end{equation*}
Recall the notations in Notation \ref{not:D-def}. As $\iota:\cD\to \cA$ is locally a complete intersection, the local $\Ext$ can be computed by \cite[III. prop. 7.2]{Hartshorne}. For this we need to determine the dualizing sheaf $\omega_{\cD/\cA}$. Let $\delta:\cA\to \cB$ be the \'etale isogeny such that $\cD\subset\ker\delta$ is a closed subscheme. Then by \cite{Hartshorne} p. 141 Remark 2 one has 
\begin{equation*}
\omega_{\cD/\cA}\isom \iota^{*}\delta^{*}\omega_{S/\cB}\isom \iota^{*}\delta^{*}\pi_\cB^{*}\omega^{-d}_{\cB/\cS}\isom \iota^{*}\pi^{*}\omega^{-d}_{\cA/\cS}.
\end{equation*}
Further
$\sI^{k}/\sI^{k+1}$ is a locally free $\sO_\cD$-module and 
one gets
\begin{align*}
\sExt_{\sO_\cA}^{q}(\sI^{k}/\sI^{k+1},\sP^{(n)}\otimes\Omega^{d}_{\cA/\cS})&\isom\sExt^{q}_{\sO_\cA}(\sO_{\cD},\sP^{(n)}\otimes\Omega^{d}_{\cA/\cS})\otimes\ul{\Hom}_{\sO_{\cD}}(\sI^{k}/\sI^{k+1},\sO_{\cD}) \\
&\isom \begin{cases}
0&q\neq d\\
\iota_{*}\ul{\Hom}_{\sO_{\cD}}(\sI^{k}/\sI^{k+1},\iota^{*}\sP^{(n)}) &q=d.
\end{cases}
\end{align*}
Inserting this result into the local to global spectral sequence gives the desired result.
\end{proof}
By \cite{SGA2} Exp. II, Th\'eorème 6 b)
one has then the  following description of the cohomology with support:
\begin{equation}\label{eq:support-descr}
H^{i}_\cD({\cA},\sP^{(n)}\otimes\Omega_{\cA/\cS}^{d})\isom
\indlim_k\Ext^{i}_{\sO_\cA}(\sO_{\cA_k},\sP^{(n)}\otimes\Omega_{\cA/\cS}^{d}).
\end{equation}
\begin{corollary}\label{cor:ext-vanishing}
For all $k\ge 1$ one has for $i<d$
\begin{equation*}
\Ext_{\sO_\cA}^{i}(\sO_{\cA_{k}},\sP^{(n)}\otimes\Omega^{d}_{\cA/\cS})=0
\end{equation*}
and for $i=d$ injections
\begin{equation*}
\Ext_{\sO_\cA}^{d}(\sO_{\cA_{k-1}},\sP^{(n)}\otimes\Omega^{d}_{\cA/\cS})\subset \Ext_{\sO_\cA}^{d}(\sO_{\cA_{k}},\sP^{(n)}\otimes\Omega^{d}_{\cA/\cS}).
\end{equation*}
In particular, one has $H^{i}_\cD({\cA},\sP^{(n)}\otimes\Omega_{\cA/\cS}^{d})=0$ for $i<d$ and
\begin{equation*}
H^{0}(\cD, \iota^{*}\sP^{(n)})\isom\Ext_{\sO_\cA}^{d}(\sO_{\cA_0},\sP^{(n)}\otimes\Omega^{d}_{\cA/\cS})\subset H^{d}_\cD({\cA},\sP^{(n)}\otimes\Omega_{\cA/\cS}^{d}).
\end{equation*}
\end{corollary}
\begin{proof}
This follows from the proposition and the long exact $\Ext$-sequence associated to 
\begin{equation*}
0\to \sI^{k}/\sI^{k+1}\to \sO_{\cA_{k}}\to \sO_{\cA_{k-1}}\to 0.\qedhere
\end{equation*}
\end{proof}
\begin{corollary}
One has
\begin{equation*}
H^{i}_\cD(\cA,\Gamma;\sP^{(n)}\otimes\Omega_{\cA/\cS}^{d})\isom 
\begin{cases}
H^{d}_\cD(\cA,\sP^{(n)}\otimes\Omega_{\cA/\cS}^{d})^{\Gamma}&i=d\\
0&i<d.
\end{cases}
\end{equation*}
and 
\begin{equation*}
H^{i}_\cD(\cA,\Gamma;\wP\otimes\Omega_{\cA/\cS}^{d})\isom \begin{cases}
\prolim
H^{d}_\cD(\cA,\Gamma;\sP^{(n)}\otimes\Omega_{\cA/\cS}^{d})&i=d\\
0&i<d.
\end{cases}
\end{equation*}
\end{corollary}
\begin{proof}
Corollary \ref{cor:ext-vanishing} shows that $H^{i}_\cD(\cA,\sP^{(n)}\otimes\Omega_{\cA/\cS}^{d})=0$ for $i\neq d$ and hence the spectral sequence for equivariant cohomology gives the first statement. By Lemma \ref{lemma:derived-limits} one has $\wP\isom R\prolim \sP^{(n)}$ and by Corollary \ref{cor:limit-sequences} an exact sequence
\begin{equation*}
0\to R^{1}\prolim H^{i-1}_\cD(\cA,\Gamma;\sP^{(n)}\otimes\Omega_{\cA/\cS}^{d})\to 
H^{i}_\cD(\cA,\Gamma;\wP\otimes\Omega_{\cA/\cS}^{d})\to \prolim
H^{i}_\cD(\cA,\Gamma;\sP^{(n)}\otimes\Omega_{\cA/\cS}^{d})\to 0,
\end{equation*}
which shows the second statement. 
\end{proof}
\begin{corollary}
One has 
\begin{equation*}
\prolim
H^{d-1}(\cA,\Gamma;\sP^{(n)}\otimes\Omega_{\cA/\cS}^{d})=0.
\end{equation*}
\end{corollary}
\begin{proof}
This follows from the surjection 
\begin{equation*}
H^{d-1}(\cA,\Gamma;\wP\otimes\Omega_{\cA/\cS}^{d})\to \prolim
H^{d-1}(\cA,\Gamma;\sP^{(n)}\otimes\Omega_{\cA/\cS}^{d})\to 0,
\end{equation*}
from Corollary \ref{cor:limit-sequences} together with the vanishing of 
$H^{d-1}(\cA,\Gamma;\wP\otimes\Omega_{\cA/\cS}^{d})$ from Corollary \ref{cor:vanishing-equiv-coh}.
\end{proof}
\begin{proof}[Proof of Theorem \ref{thm:coh-with-support-descr}]

From the vanishing of $H^{d-1}_\cD(\cA,\Gamma;\sP^{(n)}\otimes\Omega_{\cA/\cS}^{d})$ and the localization sequence one gets
\begin{equation*}
0\to H^{d-1}(\cA,\Gamma;\sP^{(n)}\otimes\Omega_{\cA/\cS}^{d})\to H^{d-1}(\cU_\cD,\Gamma;\sP^{(n)}\otimes\Omega_{\cA/\cS}^{d})\to H^{d}_\cD(\cA,\Gamma;\sP^{(n)}\otimes\Omega_{\cA/\cS}^{d}).
\end{equation*}
Taking the inverse limit one gets a commutative diagram
\begin{equation*}
\xymatrix{0\ar[r]&H^{d-1}(\cU_\cD,\Gamma;\wP\otimes\Omega_{\cA/\cS}^{d})\ar[r]\ar[d]&H^{d}_\cD(\cA,\Gamma;\wP\otimes\Omega_{\cA/\cS}^{d})\ar[d]\\
0\ar[r]&\prolim H^{d-1}(\cU_\cD,\Gamma;\sP^{(n)}\otimes\Omega_{\cA/\cS}^{d})\ar[r]&\prolim H^{d}_\cD(\cA,\Gamma;\sP^{(n)}\otimes\Omega_{\cA/\cS}^{d})
}
\end{equation*}
where the vertical arrows are surjections by Corollary \ref{cor:limit-sequences}. As the right vertical arrow is an isomorphism, it follows that 
\begin{equation*}
H^{d-1}(\cU_\cD,\Gamma;\wP\otimes \Omega^{d}_{\cA/\cS} )\isom \prolim H^{d-1}(\cU_\cD,\Gamma;\sP^{(n)}\otimes\Omega_{\cA/\cS}^{d}).
\end{equation*}

Using the splitting principle  $\iota^{*}{\sP}\isom\pi_{\cD}^{*}\sO_{\widehat{\cB}}$ of Corollary \ref{cor:coh-log-splitting} for the isogeny $\delta:\cA\to \cA'$ defining $\cD$
and that $\sO_\cD\subset \pi_{\cD}^{*}\sO_{\widehat{\cB}}$ one gets from  Corollary \ref{cor:ext-vanishing}
\begin{equation*}
\cO[\cD]=H^{0}(\cD,\sO_{\cD})\subset H^{0}(\cD, \iota^{*}\sP^{(n)})\subset 
H^{d}_\cD({\cA},\sP^{(n)}\otimes\Omega_{\cA/\cS}^{d}).
\end{equation*}
Further the composition with 
\begin{equation*}
H^{d}_\cD({\cA},\sP^{(n)}\otimes\Omega_{\cA/\cS}^{d})\to H^{d}({\cA},\sP^{(n)}\otimes\Omega_{\cA/\cS}^{d})\isom H^{0}(\cS,\sO_\cS)
\end{equation*}
is the trace map $\cO[\cD]=H^{0}(\cD,\sO_{\cD})\to H^{0}(\cS,\sO_\cS)$. Taking invariants under $\Gamma$ and the inverse limit over $n$ gives
\begin{equation*}
\cO[\cD]^{\Gamma}\hookrightarrow  \prolim H^{d}_\cD({\cA},\sP^{(n)}\otimes\Omega_{\cA/\cS}^{d})^{\Gamma}\isom
H^{d}_\cD({\cA},\Gamma;\wP\otimes\Omega_{\cA/\cS}^{d}).
\end{equation*}
Putting everything together one obtains an injection 
\begin{equation*}
\EK_{\Gamma,\cA}:\sO_\cS[\cD]^{0,\Gamma}\hookrightarrow 
H^{d-1}(\cU_\cD,\Gamma;\wP\otimes \Omega^{d}_{\cA/\cS} )
\end{equation*}
which proves finally all statements of Theorem \ref{thm:coh-with-support-descr}.
\end{proof}

\subsection{Relation to the de Rham polylogarithm}
In characteristic zero the bundle $\widehat{\sP}^{\natural}$ is nothing but the de Rham logarithm sheaf, a fact which was first shown in \cite{Scheider}. As we will use this fact later, we review the proof. 

We recall the definition of the de Rham logarithm sheaf. For any scheme $\cX$ over a field $k$ of characteristic zero we denote by $\sD_\cX$ the ring of differential operators relative $k$. Recall that 
\begin{equation*}
\sH:=\sHom_{\sO_\cS}(H^{1}_{dR}(\cA/\cS),\sO_{\cS})
\end{equation*}
is equipped with the dual of the Gau\ss-Manin connection. There is an exact sequence \cite[(1.1.1)]{Scheider} comming from the local to global spectral sequence for $\Ext$
\begin{equation}\label{eq:LocalGlobalExt}
0\to \Ext^{1}_{\sD_\cS}(\sO_{\cS},\sH)\xrightarrow{\pi^{*}} \Ext^{1}_{\sD_\cA}(\sO_{\cA},\pi^{*}\sH)\xrightarrow{\widetilde{\delta}} \Hom_{\sD_\cS}(\sH,\sH)\to 0
\end{equation}
which is split by $e^{*}$. Then the first logarithm sheaf is an extension of $\sD_\cA$-modules
\begin{equation*}
0\to \pi^{*}\sH\to \Log{1}\to \sO_{\cA}\to 0
\end{equation*}
which maps to $\id\in \Hom_{\sD_\cS}(\sH,\sH)$ and with a fixed splitting $\one{1}:e^{*}\Log{1}\isom \sO_{\cS}\oplus \sH$. The pair $(\Log{1},\one{1})$ is unique up to unique isomorphism. 
One then defines 
\begin{equation*}
\Log{n}:=\Sym^{n}\Log{1}
\end{equation*}
\nomenclature{$\Log{n}$}{\nomrefpage}
and $\sLog:=\prolim_n \Log{n}$. Recall from Corollary \ref{cor:comultiplication} that in characteristic zero there is also an isomorphism $\sP^{\natural(n)}\isom \TSym^{n}\sP^{\natural(1)}$, and that by \eqref{eq:splittings} we have a splitting $\one{1}:e^*\sP^{\natural(1)} \isom \sO_{\cS}\oplus \sH$ induced by the rigidification $(e \times \id)^*\sP^{\natural}\cong \sO_{\cA^\natural}$ of $\sP^{\natural}$. 
\begin{theorem}[Scheider, Thm. 2.3.1 \cite{Scheider}]\label{thm:scheider-theorem}
Let $\cS=\Spec(k)$  be a field of characterstic zero.
There is a canonical isomorphism $(\Log{1},\nabla_{\Log{1}},\bfone^{(1)})\isom (\sP^{\natural(1)},\nabla_{\sP^{\natural(1)}}, \one{1})$. In particular, one has an isomorphism
$\sLog\isom \widehat{\sP^{\natural}}$ respecting the sections $\bfone$ along $e:\cS\to \cA$.
\end{theorem}
\begin{proof} This is a special case of a result by Scheider see \cite[Thm. 2.3.1]{Scheider}. For the convenience of the reader, let us give the proof. Recall that the extension class $[(\Log{1},\nabla_{\Log{1}})]\in \Ext^{1}_{\sD_\cA}(\sO_{\cA},\pi^{*}\sH)$ of $\Log{1}$ is uniquely characterized by
\[
	e^*[(\Log{1},\nabla_{\Log{1}})]=0,\quad \widetilde{\delta}([(\Log{1},\nabla_{\Log{1}})])=\id_\sH\in \Hom_{\sD_\cS}(\sH,\sH).
\]
The choice of a splitting $\one{1}:e^{*}\Log{1}\isom \sO_{\cS}\oplus \sH$ pins down the actual extension within this extension class. Since we assumed $\cS=\Spec(k)$, the first term in the short exact sequence \eqref{eq:LocalGlobalExt} vanishes and $\Hom_{\sD_\cS}(\sH,\sH)=\Hom_{\sO_\cS}(\sH,\sH)$. Thus, it remains to prove $\widetilde{\delta}([(\sP^{\natural(1)},\nabla_{\sP^{\natural(1)}}))=\id_\sH$. The map $\widetilde{\delta}$ in \eqref{eq:LocalGlobalExt} has the following explicit description. Given
\[
	[\sL]=[0\to \pi^*\sH \to \sL\to \sO_\cA\to 0] \in \Ext^{1}_{\sD_\cA}(\sO_{\cA},\pi^{*}\sH),
\]
we get a long exact sequence in de Rham cohomology
\[
	0\to H^{0}_{dR}(\cA/\cS,  \pi^*\sH) \to H^{0}_{dR}(\cA/\cS,   \sL) \to  \underbrace{H^{0}_{dR}(\cA/\cS)}_{=k}\xrightarrow{\delta}  \underbrace{H^{1}_{dR}(\cA/\cS,  \pi^*\sH)}_{=\Hom_{k}(\sH,\sH)} \to \dots,
\]
and $\widetilde{\delta}([\sL])=\delta(1)\in \Hom_{k}(\sH,\sH)$.  Let us now describe an explicit inverse of the map $\widetilde{\delta}$. For a finite dimensional $k$-vector space $M$, let us write $\cS[M]:=\Spec(k\oplus M)$ for the corresponding square-zero extension. Given $\varphi\in \ker\left( \cA^\natural(\cS[M])\to \cA^\natural(\cS) \right)$, we get a vector bundle $(\sL_\varphi,\nabla_{\sL_\varphi}):=\pr_{A,*}(\id_\cA\times \varphi)^*(\sP^\natural,\nabla_{\sP^\natural})$ with integrable $S$-connection on $\cA$ fitting into a short exact sequence of vector bundles with integrable $S$-connection
\[
	0\to \pi^*M \to \sL_\varphi \to \sO_\cA\to 0.
\]
This construction provides a map $\Phi$
\[
	\Phi\colon\ker\left( \cA^\natural(\cS[M])\to \cA^\natural(\cS) \right) \to \Ext^{1}_{\sD_\cA}(\sO_{\cA},\pi^{*}M),\quad \varphi\mapsto \mathrm{KS}(\varphi):=[(\sL_\varphi,\nabla_{\sL_\varphi})].
\]
By the universal property of $\sP^\natural$, any vector bundle with integrable $\cS[M]$-connection on $\cA\times \cS[M]$ is obtained by pullback of $\sP^\natural$ along a unique map $\cS[M]\to \cA^\natural$. Thus, the map $\Phi$ is an isomorphism
\[
	\Phi\colon\ker\left( \cA^\natural(\cS[M])\to \cA^\natural(\cS) \right) \xrightarrow{\sim}  \Ext^{1}_{\sD_\cA}(\sO_{\cA},\pi^{*}M) .
\]
By \cite[(4.1.4)]{Mazur-Messing}, the Lie algebra of the universal vectorial extension $\cA^\natural$ is canonically isomorphic to $\sH^\vee$, and we obtain
\begin{equation}\label{eq:MazurMessing-Lie-algebra}
	\ker\left( \cA^\natural(\cS[M])\to \cA^\natural(\cS) \right) = \Lie(\cA^{\natural}/\cS)\otimes M \cong \Hom_{k}(\sH,M).
\end{equation}
For $M=\sH$, we get $S[\sH]=\Spec(k\oplus \sH)=\Inf^1_e \cA^\natural$ and the canonical inclusion $\iota \colon \Inf^1_e \cA^\natural \subseteq \cA^{\natural}$ gives an element $\iota \in \ker\left( \cA^\natural(\cS[\sH])\to \cA^\natural(\cS) \right) $ which corresponds to the identity $\id_\sH\in \Hom_{k}(\sH,\sH)$ under the isomorphism \eqref{eq:MazurMessing-Lie-algebra}. Note that $\Phi(\iota)=[(\sP^{\natural(1)},\nabla_{\sP^{\natural(1)}})]$, and the statement of the Theorem follows from the commutative diagram of isomorphisms
\[
	\xymatrix{
		\Ext^{1}_{\sD_\cA}(\sO_{\cA},\pi^{*}\sH)  \ar[r]^-{\widetilde{\delta}}_-{\cong} & \Hom_k(\sH,\sH)\\
		\Ext^{1}_{\sD_\cA}(\sO_{\cA},\pi^{*}\sH) \ar[u]^{\id}  & \ker\left( \cA^\natural(\cS[\sH])\to \cA^\natural(\cS) \right). \ar[l]_-{\Phi}^-{\cong} \ar[u]^{\eqref{eq:MazurMessing-Lie-algebra}}_-{\cong}
	}
\]
The commutativity of the diagram is checked by a direct computation using the explicit description of $\widetilde{\delta}$ given above and the explicit description of \eqref{eq:MazurMessing-Lie-algebra} given in \cite[(4.1.4)]{Mazur-Messing}.
\end{proof}
%

\begin{remark}
Let us briefly explain the relationship between the Eisenstein--Kronecker classes and the de Rham polylogarithm. Let $\cS=\Spec(k)$  be a field of characterstic zero. The de Rham complex for the connection $\nabla$ on ${\sP^{\natural(n)}}$
\begin{equation*}
{\sP^{\natural(n)}}\to \sP^{\natural(n)}\otimes \Omega_{\cA}^{1}\to \cdots\to {\sP^{\natural(n)}}\otimes \Omega_{\cA}^{d}
\end{equation*}
gives rise via the Hodge spectral sequence  to a map
\begin{equation*}
H^{d-1}(\cU_\cD,\Gamma;{\sP^{\natural(n)}}\otimes \Omega^{d}_{\cA} )\to  H^{2d-1}_\dR(\cU_\cD,\Gamma;{\sP^{\natural(n)}}).
\end{equation*}
It can be shown that the image of $\Eis_{\Gamma}^{\natural}(f)$ under this map is the de Rham realization of the polylogarithm defined in \cite{HuKiPol}. 

\end{remark}

\section{Explicit computation of the equivariant coherent Eisenstein--Kronecker classes}\label{section:explicit-computation}
In this section we write down an explicit representative over $\C$ of the equivariant Eisenstein-Kronecker class in terms of (generalized) Eisenstein-Kronecker series, see Definition \ref{def:EK-series}. The set up of this chapter is as in Notation \ref{notation-for-CM-ab}, i.e. $(\cA/\cR,\omega(\cA),\omega(\cA^\vee),x)$ is an abelian scheme with fixed bases of $\omega_{\cA/\cR}$ and $\omega_{\cA^{\vee}/\cR}$ and a $\frf$-torsion section $x$. For the computation, we will assume $\cR=\Spec(\C)$. This section relies on the ideas of Levin \cite{levin}, Nori \cite{Nori} and Graf \cite{Graf}.

\subsection{The completed Poincar\'e bundle as $\cC^{\infty}$-bundle} We will make the connection on the completed Poincar\'e bundle over $\C$ explicit.

Recall from \eqref{eq:C-Sigma} and \eqref{eq:uniformization} that the basis of differential forms $\omega(\cA)$ and the resulting identification $\omega(\cA):\Lie(\cA/\CC)\cong \CC^\Sigma$ gives a \emph{period lattice} $\Lambda\subseteq \CC^\Sigma$ and an induced \emph{complex uniformization}
\begin{equation}\label{eq:lattice}
	\theta\colon \cA(\CC)\cong \CC^\Sigma/\Lambda.
\end{equation}
\begin{notation}\label{not:coordinates}
We will choose an ordering of our CM type $\Sigma=\{\sigma_1,\dots,\sigma_d\}$. For a family of objects $a=(a_\sigma)_{\sigma\in \Sigma}$ parametrized by $\Sigma$ we will write
\[
	a=(a_i)_{i=1}^d,\quad \text{ with } a_i=a_{\sigma_i}.
\]
In particular, this gives the coordinates on $\C^{\Sigma}=\prod_{\sigma\in \Sigma}\C$
\[
	z=(z_1,\dots,z_d):=(z(\sigma_i))_{i=1}^d,
\]
and $I_L^+=I_{\overline{\Sigma}}^+\oplus I_\Sigma^+$ becomes identified with $\NN^d\times \NN^d$. It will be convenient to introduce the notation
\[
	z^\mu:=\prod_{i=1}^d z_i^{\mu_i},
\]
for $\mu=(\mu_1,\dots, \mu_d)\in \NN^d$ and $z\in \CC^{\Sigma}$.
\end{notation}
Since the coordinates on $\Lie(\cA/\CC)$ are induced by the basis $\omega(\cA)$, the action of a subgroup $\Gamma\subseteq\sO_L^\times$ of finite index is given by the formula
\[
	\gamma.z:=(\sigma_1(\gamma^{-1})z_1,\dots,\sigma_d(\gamma^{-1})z_d),\quad \gamma\in \sO_L^\times, z\in \CC^\Sigma.
\]

We start describing the smooth connection on the completion of the Poincar\'e bundle. By Proposition \ref{prop:CM-decompostion} the splitting induced by the $\Gamma$-action on $\cA$ coincides with the Hodge decomposition
\[
	\sH\cong \sH(\overline{\Sigma})\oplus\sH(\Sigma)\cong \overline{\Lie(\cA/\CC)}\oplus \Lie(\cA/\CC).
\]
\begin{definition}\label{eq:nu-def}
Let us define $\nu=\nu^{1,0}+\nu^{0,1}\in \sH\otimes_\CC (\omega_{\cA/\CC}\oplus \ol{\omega_{\cA/\CC}})$ with
\[
	\nu^{1,0}\in \sH(\Sigma) \otimes_\CC \omega_{\cA/\CC},\quad \nu^{0,1}\in \sH(\overline{\Sigma}) \otimes_\CC \overline{\omega_{\cA/\CC}}
\]
corresponding to the identity morphisms in 
	$$\sH(\Sigma) \otimes_\CC \omega_{\cA/\CC}\cong \Lie(\cA/\CC) \otimes_\CC \omega_{\cA/\CC}=\Hom_{\CC}(\omega_{\cA/\CC},\omega_{\cA/\CC})$$ 
respectively 
	$$\sH(\overline{\Sigma}) \otimes_\CC \overline{\omega_{\cA/\CC}}\cong \overline{\Lie(\cA/\CC)}\otimes_\CC \overline{\omega_{\cA/\CC}}=\Hom_{\CC}(\overline{\omega_{\cA/\CC}},\overline{\omega_{\cA/\CC}}).$$
\end{definition}
\begin{definition}\label{def:ubar}
We write $(\overline{\mathrm{u}}_1,\dots,\overline{\mathrm{u}}_d,\mathrm{u}_1,\dots,\mathrm{u}_d)$ for the basis of $\sH$ corresponding to the basis
\[
	\frac{\partial}{\partial \bar{z}_1},\dots,\frac{\partial}{\partial \bar{z}_d},\frac{\partial}{\partial z_1},\dots,\frac{\partial}{\partial z_d}\in \overline{\Lie(\cA/\CC)}\oplus\Lie(\cA/\CC)
\]
under the isomorphism $\sH\cong \overline{\Lie(\cA/\CC)}\oplus \Lie(\cA/\CC)$.
\end{definition}
In this basis, we may write
\begin{equation}
	\nu^{0,1}=\sum_{i=1}^d \overline{\mathrm{u}}_i d\overline{z}_i,\quad \nu^{1,0}=\sum_{i=1}^d \mathrm{u}_i dz_i.
\end{equation}
For a smooth manifold $X$, we will write $\cE^\cdot_{X}$ for the differential graded algebra of sheaves of smooth differential forms and $\cC^\infty_X$ for the sheaf of smooth functions on $X$. 
\nomenclature{$\cE^\cdot_{X}$}{\nomrefpage}
\begin{notation}
Let us write
\begin{align*}
\cP^{(n)}:={\sP}^{(n),an} \otimes_{\cO^{an}_{\cA(\CC)}}\cC^\infty_{\cA(\CC)},&\quad{\cP}^{\natural(n)}:={\sP}^{\natural(n), an} \otimes_{\cO^{an}_{\cA(\CC)}} \cC^\infty_{\cA(\CC)}\\
	\widehat{\cP}:=\prolim_n{\cP}^{(n)},& \quad	\widehat{\cP}^{\natural}:=\prolim_n{\cP}^{\natural(n)}
\end{align*}
\nomenclature{$\cP^{(n)}$}{\nomrefpage}
\nomenclature{${\cP}^{\natural(n)}$}{\nomrefpage}
for the sheaves of smooth sections of the completed Poincar\'e bundles $\widehat{\sP}$ and $\widehat{\sP}^\natural$ considered as pro-bundles. We also let 
\[
	\cH=\sH\otimes_\CC \cC^\infty_{\cA(\CC)}.
\]
We will write $\cH=\cH(\Sigma)\oplus\cH(\overline{\Sigma})$ for the decomposition induced by $\sH=\sH(\Sigma)\oplus\sH(\overline{\Sigma})$. 
\end{notation}

The connection $\nabla$ on $\sP^{\natural(n)}$ defines a smooth connection $\nabla_{\cC^\infty}$ on $\cP^{\natural(n)}$ as follows. We have $\cP^{\natural(n)} = (\sP^{\natural(n),an})^\nabla\otimes_\CC \cC^\infty_{\cA(\CC)}$, where $(\sP^{\natural(n),an})^\nabla$ is the local system of $\CC$ vector spaces given by horizontal sections of ${\sP}^{\natural(n),an}$. Then $\nabla_{\cC^\infty}$ is defined as:
\[
	\nabla_{\cC^\infty} := \id\otimes d\colon ({\sP}^{\natural(n),an})^\nabla\otimes_{\CC}\cC^\infty_{\cA(\CC)} \rightarrow  ({\sP}^{\natural(n),an})^\nabla\otimes_{\CC} \cE^1_{\cA(\CC)}.
\]
The connection $\nabla_{\cC^\infty}$ decomposes into a holomorphic and an anti-holomorphic connection $\nabla_{\cC^\infty}=\nabla'+\nabla''$ where
\[
	\nabla'\colon {\cP}^{\natural(n)}\rightarrow {\cP}^{\natural(n)}\otimes \cE^{1,0}_{\cA(\CC)},\quad \nabla''\colon {\cP}^{\natural(n)}\rightarrow {\cP}^{\natural(n)}\otimes \cE^{0,1}_{\cA(\CC)}.
\]
Note that the transition maps ${\cP}^{\natural(n)}\to {\cP}^{\natural(n-1)}$ are horizontal with respect to $\nabla'$ and $\nabla''$. We will sometimes simplify the notation and write $\otimes$ if the base of the tensor product is clear from the context. In the above case, the tensor product is over the sheaf of smooth functions on $\cA(\CC)$. Let us observe that the connection $\nabla''$ induces a Dolbeault resolution of the holomorphic sheaf ${\sP}^{\natural(n),an}$:
\[
	({\sP}^{\natural(n),an})[0]\xrightarrow{\sim} \left({\cP}^{\natural(n)}\otimes  \cE^{0,\bullet}_{\cA(\CC)},\nabla'' \right).
\]
Our next aim is to describe ${\cP}^{\natural(n)}$ explicitly. 
Recall the forms $\nu$, $\nu^{1,0}$ and $\nu^{0,1}$ defined in \eqref{eq:nu-def}. 
By abuse of notation, we will also write  for the  corresponding sections
\begin{equation*}
\nu=\nu^{1,0}+\nu^{0,1}\in \Gamma(\cA(\CC), \cH\otimes \cE^1_{\cA(\CC)})
\end{equation*}
with
\begin{align*}
	\nu^{1,0}\in \Gamma(\cA(\CC), \cH\otimes \cE^{1,0}_{\cA(\CC)}) &&\text{ and }&& \nu^{0,1}\in \Gamma(\cA(\CC), \cH\otimes \cE^{0,1}_{\cA(\CC)}).
\end{align*}
With this notation, we have the following result:
\begin{theorem}\label{thm_analyticPoincare}
	For $n\geq 0$, there  is a compatible system of horizontal isomorphism
	\[
		({\cP}^{\natural(n)},\nabla_{\cC^\infty})\cong \left( \bigoplus_{k=0}^n{\TSym^k} (\cH), d+\nu \right).
	\]
	fitting into a commutative diagram
	\[
		\xymatrix{  {{\cP}^{\natural(n)}}\ar[r]^-{\cong} & {\bigoplus_{k=0}^n{\TSym^k} (\cH)} \\  {\cP} \ar[u]\ar[r]^-{\cong} & \bigoplus_{k=0}^n{\TSym^k}(\cH(\overline{\Sigma}))\ar[u]  }
	\]
	where the vertical maps are the canonical inclusions. The pullback of the above isomorphism along $e$ is furthermore compatible with the moment maps on ${\cA}^{\natural(n)}$:
	\[
		\mom_{{\cA}^{\natural(n)}}\colon e^* {\cP}^{\natural(n)}= \cO_{{\cA}^{\natural(n)}} \xrightarrow{\sim} {\bigoplus_{k=0}^n{\TSym^k}(\sH)}
	\]
where one uses $e^{*}\cH=\sH$.
\end{theorem}
\begin{proof} 
Let us first show that it is enough to prove that there is a horizontal isomorphism
	\begin{equation}\label{eq_Poincare_con_1}
		({\cP}^{\natural(1)},\nabla_{\cC^\infty})\cong \left( (\CC\oplus \sH)\otimes_\CC \cC^\infty_{\cA(\CC)}, d+\nu \right).
	\end{equation}
	such that the pullback of the above isomorphism along $e$ is the identity:
	\[
		e^* {\cP}^{\natural(1)}= \CC\oplus\sH \xrightarrow{\sim} \CC\oplus\sH.
	\]
	Indeed, the co-multiplication maps
	\[
		\cP^{\natural(n)}\rightarrow \TSym^n(\cP^{\natural(1)})
	\]
	are isomorphisms by Corollary \ref{cor:comultiplication} and we can define
	\[
		{\cP^{(1)}}\cong (\CC\oplus \sH(\overline{\Sigma}))\otimes_\CC \cC^\infty_{\cA(\CC)},
	\]
	as the restriction of \eqref{eq_Poincare_con_1} to ${\cP^{(1)}}\subseteq  {\cP^{\natural(1)}}$. Once \eqref{eq_Poincare_con_1} is constructed, the desired horizontal morphism
	\[
		({\cP^{\natural(n)}},\nabla_{\cC^\infty})\cong \left( \bigoplus_{k=0}^n{\TSym^k} (\cH), d+\nu \right)
	\]
	follows from
	\[
		 {\cP^{\natural(n)}} \cong \TSym^n({\cP^{\natural(1)}}) \cong \TSym^n (\CC\oplus \sH)\otimes_\CC \cC^\infty_{\cA(\CC)}\cong \prod_{b=0}^n \TSym (\cH).
	\]
	So it remains to construct the horizontal isomorphism \eqref{eq_Poincare_con_1}. According to Scheider's theorem \ref{thm:scheider-theorem} there is a unique horizontal isomorphism 
\[
	\sP^{\natural(1)}\xrightarrow{\sim} \Log{1}
\]
which is compatible with the trivialization
\[
	\CC\oplus \sH\cong  e^*\sP^{\natural(1)}\xrightarrow{\sim} e^*\Log{1}=\CC\oplus \sH.
\]
It has been shown by Levin in \cite[Proposition 2.4.5]{levin} that there is a horizontal isomorphism
\[
	(\Log{1}_{\cC^\infty},\nabla)\xrightarrow{\sim} ((\CC\oplus\sH)\otimes\cC^\infty_{\cA(\CC)},d+\nu)
\]
compatible with the splitting along $e^*$ and the claim follows.
\end{proof}

The anti-holomorphic part $\nu^{0,1}$ of $\nu$ is an anti-holomorphic differential form with values in $\cH(\overline{\Sigma}) \subseteq \cH$:
\[
\nu^{0,1}\in H^{0}(\cA(\CC), \cH(\overline{\Sigma})\otimes \cE^{0,1}_{\cA(\CC)})\subseteq H^{0}(\cA(\CC), \cH\otimes \cE^{0,1}_{\cA(\CC)}).
\]
Thus, it follows from the above theorem that the anti-holomorphic part $\nabla''$ of $\nabla_{\cC^\infty}$ restricts to a connection on ${\cP}^{(n)}$ and we obtain a Dolbeault resolution for ${\sP}^{(n),an}$:
\begin{corollary}
For $n\geq 0$ there is a compatible system of quasi-isomorphisms
\begin{equation}\label{eq_Dolbeault_resolution}
	({\sP}^{(n),an})[0]\xrightarrow{\sim} \left({\cP}^{(n)}\otimes \cE^{0,\bullet}_{\cA(\CC)},\nabla'' \right)
\end{equation}
and tensoring the left hand side with the sheaf of holomorphic $p$-forms gives the quasi-isomorphism
\begin{equation}\label{eq_PB_resolution}
	({\sP}^{(n),an}\otimes \Omega^p_{\cA(\CC)})[0]\xrightarrow{\sim} \left( {\cP}^{(n)}\otimes\cE^{p,\bullet},\nabla'' \right).
\end{equation}
\end{corollary}
\subsection{A model to compute  equivariant cohomology}\label{section:model}
The goal of this section is to construct a  model which allows to compute the equivariant coherent polylogarithm class and the associated Eisenstein-Kronecker classes explicitly. We will use the  fact (see Appendix  \eqref{eq:coh-equiv}) that the equivariant sheaf cohomology
\[
	H^{i}(X,\Gamma,\cF)
\]
for a torsion free group $\Gamma$ can be computed using the Borel construction. For the rest of this subsection let us assume that $\Gamma$ is torsion free. 

Let 
\[
	L^1_{\RR}:=\left\{ (r_1,\dots,r_d)\in \RR^d_{>0}\mid \prod_{i=1}^d r_i=1 \right\}.
\]
with an action of $\Gamma$ given by
\[
	\Gamma\times L^1_\RR\rightarrow L^1_\RR,\quad (\gamma,r)\mapsto (|\sigma_1(\gamma)|^2r_1,\dots,|\sigma_d(\gamma)|^2r_d).
\]
This serves as an explicit model for the universal bundle $E\Gamma$ over $B\Gamma$:
\[
	E\Gamma:=L^1_\RR\rightarrow B\Gamma:=\Gamma\backslash E \Gamma.
\]
The inclusion of $L^1_\RR$ into $\RR^d$ gives a function
\[
	r=(r_1,\dots,r_d)\colon E\Gamma=L^1_\RR \subseteq \RR^d
\]
and we will consider $(r_1,\dots,r_{d-1})$ as coordinates on $E\Gamma$. In particular, we get the following explicit model for $\cA(\CC)\times_\Gamma E\Gamma$:
\[
	\cA(\CC)\times_\Gamma E\Gamma= \Gamma\backslash (\cA(\CC)\times L^1_\RR)
\]
where $\gamma\in \Gamma$ acts on $(z,r)\in \cA(\CC)\times L^1_\RR $ by
\[
	\gamma.(z,r)=((\sigma_1(\gamma^{-1})z_1,\cdots, \sigma_d(\gamma^{-1})z_d),(|\sigma_1(\gamma)|^2r_1,\dots,|\sigma_d(\gamma)|^2r_d)).
\]
Here, the topological space $E\Gamma$ carries the structure of a real manifold and the pullback of abelian sheaves $\pr^{-1}\cC^\infty_{\cA(\CC)}$ along $$\pr\colon \cA(\CC)\times E\Gamma\rightarrow \cA(\CC)$$ can be resolved by the complex of smooth relative differentials:
\begin{equation}\label{eq_diff_resolution}
	(\pr^{-1}\cC^\infty_{\cA(\CC)})[0]\xrightarrow{\sim} \cE^{\bullet}_{E\Gamma\times \cA(\CC)/\cA(\CC)}.
\end{equation}
\begin{definition}
We define
\[
	\cE^{1,0}_{\cA(\CC)\times E\Gamma}:=\pr^*\cE^{1,0}_{\cA(\CC)},\quad \cE^{0,1}_{\cA(\CC)\times E\Gamma}:=\pr^*\cE^{0,1}_{\cA(\CC)}\oplus \pr_{E\Gamma}^*\cE^1_{E\Gamma}.
\]
Using this we define the sheaf $\cE^{p,q}_{\cA(\CC)\times E\Gamma}$ of $(p,q)$-forms on $\cA(\CC)\times E\Gamma$ for non-negative integers $p$ and $q$ by taking exterior powers.
\end{definition}
 By summarizing the above discussion, we obtain a kind of Dolbeault resolution:
\begin{lemma}\label{lem_cP_resolution}
	We have a quasi-isomorphism
	\[
		\pr^{-1}({\sP}^{(n),an}\otimes \Omega^{d}_{\cA(\CC)})\xrightarrow{\sim} \left(\pr^*{\cP}^{(n)}\otimes \cE^{d,\bullet}_{\cA(\CC)\times E\Gamma}, \nabla'' \right).
	\]
	In particular, the cohomology
	\[
		H^i(\cU_\cD(\CC),\Gamma,\widehat{\sP}^{an}\otimes\Omega^{d}_{\cA(\CC)})\cong \varprojlim_n H^i(\cU_\cD(\CC),\Gamma,{\sP}^{(n),an}\otimes\Omega^{d}_{\cA(\CC)})
	\]
	can be described in terms of a compatible system of smooth $\Gamma$-invariant $\nabla''$-closed $(d,i)$-forms with values in $\pr^*{\cP}^{(n)}$.
\end{lemma}
\begin{proof}
In \eqref{eq_PB_resolution} we have shown that there  is the resolution
\[
({\sP}^{(n),an}\otimes \Omega^p_{\cA(\CC)})[0]\xrightarrow{\sim} \left( {\cP}^{(n)}\otimes\cE^{p,\bullet}_{\cA(\CC)},\nabla'' \right).
\]
The pullback of this along $\pr\colon \cA(\CC)\times E\Gamma\rightarrow \cA(\CC)$ tensored with \eqref{eq_diff_resolution} over $\CC$ gives after passing to the associated double complex the desired resolution:
\[
	\pr^{-1}({\sP}^{(n),an}\otimes \Omega^{d}_{\cA(\CC)})\xrightarrow{\sim} \left(\pr^*{\cP}^{(n)}\otimes \cE^{d,\bullet}_{\cA(\CC)\times E\Gamma}, \nabla'' \right).
\]
As explained in Appendix \eqref{eq:coh-equiv}, the equivariant cohomology $H^i(\cU_\cD(\CC),\Gamma,{\sP}^{(n),an}\otimes\Omega^{d}_{\cA(\CC)})$ can be computed by the Borel construction
\[
	H^{i}(E\Gamma\times_\Gamma \cU_\cD(\CC), ({\sP}^{(n),an}\otimes\Omega^{d}_{\cA(\CC)})^\sim),
\]
where $({\sP}^{(n)an}\otimes\Omega^{d}_{\cA(\CC)})^\sim$ is the sheaf on $E\Gamma\times_\Gamma \cU_\cD(\CC)=(E\Gamma\times \cU_\cD(\CC))/\Gamma$ induced by the $\Gamma$-equivariant sheaf $\pr^{-1}({\sP}^{(n),an}\otimes \Omega^{d}_{\cA(\CC)})$. By the above resolution, this cohomology group coincides with the $i$-th cohomology of the complex
\[
	\left(H^{0}\left(E\Gamma\times \cU_{\cD}(\CC),\pr^*{\cP}^{(n)}\otimes \cE^{d,\bullet}_{\cA(\CC)\times E\Gamma}\right)^\Gamma, \nabla'' \right).\qedhere
\]
\end{proof}
\begin{definition}
We define the sheaf of $(p,q)$-currents $\cD_{\cA(\CC)\times E\Gamma}^{p,q}:=\left(\cE_{\cA(\CC)\times E\Gamma,c}^{d-p,2d-1-q}\right)^*$ as the dual of smooth $(d-p,2d-1-q)$-forms with compact support.
\end{definition}
The map $\omega\mapsto (\eta \mapsto \int \eta\wedge \omega)$ gives a morphism $\cE^{p,q}_{\cA(\CC)\times E\Gamma}\to \cD_{\cA(\CC)\times E\Gamma}^{p,q}$, which induces a quasi-isomorphism of complexes
\[
	\cE^{d,\bullet}_{\cA(\CC)\times E\Gamma} \rightarrow \cD_{\cA(\CC)\times E\Gamma}^{d,\bullet}.
\]
In particular, we can describe the equivariant cohomology $H^i(\cU_\cD(\CC),\Gamma,\widehat{\sP}^{an}\otimes\Omega^{d}_{\cA(\CC)})$ in terms of $(d,i)$-currents:
\begin{corollary}
	For $n\geq 0$, we have quasi-isomorphisms
	\[
		\pr^{-1}(\widehat{\sP}^{(n),an}\otimes \Omega^{d}_{\cA(\CC)})\xrightarrow{\sim} \left(\pr^*{\cP}^{(n)}\otimes \cD_{\cA(\CC)\times E\Gamma}^{d,\bullet}, \nabla'' \right).
	\]
	In particular, the cohomology
	\[
		H^i(\cU_\cD(\CC),\Gamma,\widehat{\sP}^{an}\otimes\Omega^{d}_{\cA(\CC)})\cong \varprojlim_n H^i(\cU_\cD(\CC),\Gamma,{\sP}^{(n),an}\otimes\Omega^{d}_{\cA(\CC)})
	\]
	can be described in terms of a pro-system of smooth $\Gamma$-invariant $\nabla''$-closed $(d,i)$-currents with values in $\pr^*{\cP}^{(n)}$.
\end{corollary}

\subsection{Generalized Eisenstein-Kronecker series}\label{section:EK_Series}
In this section we introduce the generalized Eisenstein--Kronecker series which appears in the explicit description of the equivariant coherent Eisenstein--Kronecker class. 
\begin{definition}\label{def:EK-series}Using the coordinates from Notation \ref{not:coordinates}, let $\mu\in \NN^d$, $\Lambda\subseteq \CC^\Sigma$, $z,w\in \CC^\Sigma$ and $H(z,w):=\sum_{i=1}^d h_i z_i\overline{w}_i$ be a Hermitian form which is in diagonal form in the standard basis of $\CC^\Sigma$ with $h=(h_1,\dots,h_d)\in \RR_+^\Sigma$ and $\langle\cdot ,\cdot \rangle:=\Im H(\cdot,\cdot)$ the associated alternating form. Then we define the Eisenstein-Kronecker series
\begin{equation*}
K^{\mu}(H,z,w, s,\Lambda):= 
\dashsum_{\lambda\in \Lambda}\frac{(\zbar+\lambdabar )^{\mu}}{\lVert z+\lambda \rVert_H^{2s}}e^{2\pi i\langle \lambda,w\rangle}.
\end{equation*}
\nomenclature{$K^{\mu}(H,z,w, s,\Lambda)$}{\nomrefpage}
Here $\Hnorm{z}=H(z,z):=\sum_{i=1}^d h_i|z_i|^2$ is the associated absolute value and $\dashsum$ means that we only sum over $\lambda\in \Lambda\smallsetminus \{-z\}$. 
\end{definition}
\begin{remark}
These series are also a special case of the (generalization of) Epstein zeta functions considered by Siegel in \cite{Siegel}.
\end{remark}
The series $K^{\mu}(H,z,w, s,\Lambda)$ converges absolutely and uniformly for $\Re s>d+\frac{|\mu|}{2}$ and for $z,w$ in a compact subset of $\CC^\Sigma$. The following result shows that it admits an analytic continuation to $\C$ with possible poles in $s=0$ and $s=d$:
\begin{proposition}\label{cor:eisenstein-kronecker-fct-eq}
The Eisenstein-Kronecker series $K^{\mu}(H,z,w,s,\Lambda)$ has an analytic continuation to $\C$ with possible poles in $s=0$ and $s=d$ of order one, with residue $-\delta_{\mu,z}e^{2\pi i\langle -z,w\rangle}$ in $s=0$, and residue $\frac{\delta_{\mu,w}}{\vol_H(\Lambda)}$ in $s=d$. Moreover, it satisfies the functional equation
\begin{equation*}
\frac{\Gamma(s)K^{\mu}(H,z,w,s,\Lambda)}{\pi^s}=
\frac{e^{2\pi i\langle z,w\rangle}}{\vol_H(\Lambda)}
\frac{\Gamma(d+|\mu|-s)K^{\mu}(H,w,z,d+|\mu|-s,\Lambda^*)}{\pi^{d+|\mu|-s}}.
\end{equation*}
\end{proposition}
\begin{proof}
The proof follows in a standard way from the fact that the Eisenstein--Kronecker series can be written as the Mellin transform of a suitable theta function, see for example \cite[Theorem 3]{Siegel}. For the convenience of the reader, let us sketch the proof:
For $\mu\in \NN^d$, $t\in\RR_{>0}$ and $z,w\in \CC^\Sigma$ define the theta function
\begin{equation*}
\vartheta^{\mu}_t(H,z,w,\Lambda):=
\sum_{\lambda\in \Lambda}(\zbar+\lambdabar)^{\mu}e^{-\pi t\lVert z+\lambda\rVert_H^{2}}e^{2\pi i \langle \lambda,w\rangle}.
\end{equation*}
The function $\vartheta^{\mu}_t(H,z,w,\Lambda)$ is uniformly and absolutely convergent on each compact set in $\RR_{>0}\times\CC^\Sigma\times \CC^\Sigma$, and satisfies the functional equation (see \cite[Prop. 8]{Siegel} or \cite[Chapter VII,(3.6)]{Neukirch})
\begin{equation*}
\vartheta^{\mu}_t(H,z,w,\Lambda)=t^{-d-|\mu|}\frac{1}{\vol_H(\Lambda)}e^{2\pi i \langle w, z\rangle}\vartheta^{\mu}_{t^{-1}}(H,w,z,\Lambda^{*})
\end{equation*}
where 
\[
	\Lambda^*=\left\{ l\in\CC^\Sigma : \langle l,\Lambda \rangle\subseteq \ZZ\right\}
\]
 is the dual of $\Lambda$ with respect to $\langle-,-\rangle$ and $\vol_H(\Lambda)$ is the volume of $\CC^\Sigma/\Lambda$ with respect to the metric induced by the Hermitian form $H$. The statement of the proposition follows now from the fact that that the higher Eisenstein--Kronecker series can be written as a Mellin transform of the theta function
\begin{equation*}
\pi^{-s}\Gamma(s)K^{\mu}(H, z,w,s,\Lambda)= \int_{0}^{\infty}(\vartheta^{\mu}_t(H,z,w,\Lambda)-\delta_{\mu,z}e^{2\pi i\langle -z,w\rangle})t^{s}\frac{dt}{t}.
\end{equation*}
where $\delta_{\mu,z}$ is $1$ if both $z\in \Lambda$ and $|\mu|=0$, and zero in all other cases. 
\end{proof}

\subsection{The Eisenstein--Kronecker current} In this section, we will define the Eisenstein--Kronecker current, discuss its basic properties and relate it to the Eisenstein--Kronecker series introduced in section \ref{section:EK_Series}.  We assume in this subsection that $\Gamma$ is torsion-free. 

\begin{remark}Let us briefly  explain why it is crucial to work with currents instead of differential forms. In the following, we will define a pro-system of $\Gamma$-invariant smooth $(d,d-1)$-currents $\phi=(\phi_n)_n$ on $\cA(\CC)$ which satisfies the differential equation
\[
	\nabla''(\phi_n)=\delta_f\vol,
\]
where $\vol$ is a suitably normalized volume form on $\cA$ and $\delta_f$ is a linear combination of $\delta$-distributions with support on the divisor $\cD(\CC)$. The characterizing differential equation of $\phi$ shows that the restriction of $\phi$ to $\cU_{\cD}(\CC)$ is $\nabla''$-closed, hence it defines a equivariant cohomology class
\[
	[\phi|_{\cU_\cD(\CC)}]\in H^{d-1}(\cU_\cD(\CC),\Gamma,\widehat{\sP}^{an}\otimes\Omega^d_{\cA(\CC)}).
\]
We will also see that the current $\phi|_{\cU_\cD(\CC)}$ is represented by a smooth differential form. The main advantage of working with currents instead of differential forms is that $\phi|_{\cU_\cD(\CC)}$ extends as a smooth current to $\cA(\CC)$ but it does not extend as a differential form.  This extension together with its characterizing differential equation allows us to compute its residue along the divisor $\cD(\CC)$. This computation is crucial for showing that $\phi|_{\cU_\cD(\CC)}$ is an explicit representative for the equivariant Eisenstein--Kronecker class.
\end{remark}

Recall the notations from Notation \ref{not:coordinates} and let $H(z,w):=\sum_{i=1}^d z_i\overline{w}_i$ be the standard Hermitian form on $\CC^\Sigma$. For $r\in\RR_{>0}^d$ we write $Hr$ for the scaled Hermitian form $Hr(z,w)=\sum_{i=1}^d r_i z_i\overline{w}_i$ on $\CC^\Sigma$ with the associated alternating form
 \[
 	\langle \cdot , \cdot \rangle_{Hr} \colon \CC^\Sigma\times \CC^\Sigma \rightarrow \RR
 \]
 and the dual lattice
 \[
	\Lambda^{Hr,*}:=\{ l\in\CC^\Sigma \mid \langle l,\Lambda \rangle_{Hr}\subseteq \ZZ \}.
\]
 For $r=\underline{1}$ this gives back the standard Hermitian form $H(z,w)=\sum_{i=1}^d z_i\overline{w}_i$ with the alternating form $\langle\cdot,\cdot\rangle:=\langle\cdot,\cdot\rangle_{H}$ and dual lattice $\Lambda^*:=\Lambda^{H,*}$.

Let us define the volume form $\vol:=\frac{(2\pi i)^d}{\vol(\cA)}\bigwedge_{i=1}^d dz_i\wedge d\bar{z}_i$ with $\vol(\cA)=\int_\cA \bigwedge_{i=1}^d dz_i\wedge d\bar{z}$. Every element $l\in\Lambda^*$ of the dual lattice gives us a character
\[
	\chi_l\colon \CC^\Sigma/\Lambda\rightarrow \CC^\times,\quad w\mapsto \exp(2\pi i\langle l,w \rangle).
\]
Recall, that we write $r\colon L^1_\RR\rightarrow \RR^d$ for the inclusion and $\overline{\mathrm{u}}=(\overline{\mathrm{u}}_1,\dots,\overline{\mathrm{u}}_d)$ for the basis of $\sH(\overline{\Sigma})$. For $l\in \Lambda^*$ we let 
\[
	\tilde{l}:=\sum_{i=1}^d \frac{\bar{l}_i\bar{\mathrm{u}}_i}{r_i}\in H^{0}(\cA(\CC)\times E\Gamma,\pr^*(\cH(\overline{\Sigma})).
\]
By identifying $\omega_{\cA^\vee}$ with $\overline{\Lie(\cA/\CC)}$ we may view $\tilde{l}$ as an anti-holomorphic vector field
\[
	\tilde{l}=\sum_{i=1}^d \frac{\bar{l}_i}{r_i} \frac{\partial}{\partial \bar{z}_i}
\]
on $\cA(\CC)\times E\Gamma$. We will write $\iota_{\tilde{l}}$ for the contraction along $\tilde{l}$. 
\begin{definition}
For non-negative integers $j,b$ and $s\in \CC$, let us define the following $(d,d-1)$-current with values in $\TSym^{b}(\pr^*\cH(\overline{\Sigma}))$ on $\cA(\CC)\times E\Gamma$:
\begin{equation}\label{eq_phi_kjs}
	\phi_s^{(b,j)}:=\frac{(-1)^j \Gamma(s)}{j!}\sum_{t\in \frc^{-1}\Lambda/\Lambda}f(t)\sum_{0\neq l\in \Lambda^*}\frac{\tilde{l}^{[b]}}{\left( \pi \lVert \frac{l}{\sqrt{r}} \rVert_H^2 \right)^s} \chi_l\left(z-t \right) \iota_{\tilde{l}}(d\iota_{\tilde{l}})^j\vol.
\end{equation}
\end{definition}
Note, that locally in $(r,s)\in \RR_{>0}^d \times \CC$ the coefficients of $\chi_l$ have at most polynomial growth in $\lVert l \rVert$. Hence, the convergence of the current follows from \cite[VII,\S 1]{Schwartz}. In the following, we will show that the current can be represented by a smooth differential form on the open subset $\cU_\cD$. Let us start with a more explicit expression for $\iota_{\tilde{l}}(d\iota_{\tilde{l}})^j \vol$:
\begin{lemma}\label{lem_difform}Let $j$ be a non-negative integer:
\begin{enumerate}
\item \label{lem_difform_a} We have $\iota_{\tilde{l}}(d\iota_{\tilde{l}})^j \vol=0$ for $j\geq d$.
\item \label{lem_difform_b} For $1\leq j\leq d-1$ we have
	\[
		\iota_{\tilde{l}}(d\iota_{\tilde{l}})^j \vol=\frac{(-1)^jj!(2\pi i)^d}{\vol(\cA)}\sum_{\substack{\underline{\epsilon}\in\{0,1\}^d \\ |\underline{\epsilon}|=j+1}} \left( \frac{\bar{l}}{r} \right)^{\underline{\epsilon}} \omega_{\underline{\epsilon}}
	\]
	with
	\[
	\omega_{\underline{\epsilon}}=\iota_{-\sum_{i=1}^d r_i\frac{\partial}{\partial r_i} } \left( \bigwedge_{i=1}^d \omega_{\underline{\epsilon}}^i \right),\quad \omega_{\underline{\epsilon}}^i=\begin{cases}
	dz_i\wedge \frac{dr_i}{r_i} & \epsilon_i=1\\
	dz_i\wedge d\bar{z}_i & \epsilon_i=0,
\end{cases}
\]
and
\[
	\left( \frac{\bar{l}}{r} \right)^{\underline{\epsilon}}:=\prod_{i=1}^d \left( \frac{\bar{l}_i}{r_i} \right)^{\epsilon_i}.
\]
\item \label{lem_difform_c} For $j=d-1$ we get:
 	\[
\iota_{\tilde{l}}(d\iota_{\tilde{l}})^{d-1}\vol=- \frac{(-1)^{\frac{d(d-1)}{2}}d!(2\pi i)^d}{\vol(\cA)}\left(\frac{\bar{l}}{r}\right)^{\underline{1}}\cdot \bigwedge_{i=1}^{d-1}\frac{dr_i}{r_i}\wedge \bigwedge_{i=1}^ddz_i.
\]
Here, we use the notation $\underline{1}=(1,\dots,1)\in \NN^d$.
\end{enumerate}
\end{lemma}
\begin{proof}
\eqref{lem_difform_a}: This is clear since the degree of the anti-holomorphic differential in $\vol$  is $d$, so contracting more than $d$ times along $\tilde{l}$ gives $0$.

\noindent\eqref{lem_difform_b}: Recall that we have
\[
\vol=\frac{(2\pi i)^d}{\vol(\cA)}\bigwedge_{i=1}^d dz_i\wedge d\overline{z}_i.
\]
Now the formula follows from
\[
	d\iota_{\tilde{l}}(dz_i\wedge d\overline{z}_i)=dz_i\wedge d\left( \frac{\overline{l}_i}{r_i} \right)=-\frac{\overline{l}_i}{r_i} dz_i\wedge \frac{dr_i}{r_i}
\]
and
\[
	\iota_{\tilde{l}}(dz_i\wedge d\bar{z}_i)=-\frac{\overline{l}_i}{r_i} dz_i=\iota_{-\sum_{i=1}^d r_i\frac{\partial}{\partial r_i} }\left( d\iota_{\tilde{l}}(dz_i\wedge d\bar{z}_i) \right).
\]
\eqref{lem_difform_c}: In the case $j=d-1$ the only $\underline{\epsilon}$ in the above sum is $\underline{\epsilon}=\underline{1}$. So the formula in \eqref{lem_difform_b} reduces to
\begin{equation*}
		\iota_{\tilde{l}}(d\iota_{\tilde{l}})^{d-1}\vol=-\frac{(d-1)!(2\pi i)^d}{\vol(\cA)} \left(\frac{\bar{l}}{r}\right)^{\underline{1}} \cdot \omega_{\underline{1}}.
\end{equation*}
Now, a straightforward computations using $\prod_{i=1}^dr_i=1$ shows
\begin{align*}
	\omega_{\underline{1}}=\iota_{-\sum_{i=1}^d r_i\frac{\partial}{\partial r_i} } \left( \bigwedge_{i=1}^d dz_i\wedge \frac{dr_i}{r_i} \right)=d\cdot \left(\bigwedge_{i=1}^{d-1} dz_i \wedge \frac{dr_i}{r_i}\right) \wedge dz_d.
\end{align*}
By reordering the differential forms, we get
\[
	\omega_{\underline{1}}=(-1)^{\frac{d(d-1)}{2}} \bigwedge_{i=1}^{d-1}\frac{dr_i}{r_i}\wedge \bigwedge_{i=1}^ddz_i,
\]
and the result follows.
\end{proof}

In the next step, let us prove that $\phi_{s}^{(b,j)}$ is represented by smooth differential forms:
\begin{proposition}\label{prop_polform}
For $n\geq b$, the current $\phi_{s}^{(b,j)}|_{\cU_\cD(\CC)\times E\Gamma}$ is represented by a smooth $(d,d-1)$-form, i.e.,
\[
	\phi_s^{(b,j)}|_{\cU_\cD(\CC)\times E\Gamma}\in \Gamma(\cU_\cD\times E\Gamma, \pr^*{\cP}^{(n)}\otimes \cE^{d,d-1}_{\cA(\CC)\times E\Gamma}).
\]
More precisely, we have the following explicit formula in terms of Eisenstein--Kronecker series
\begin{equation}
	\label{eq_phijk_form} \phi_s^{(b,j)}|_{\cU_\cD(\CC)\times E\Gamma}
	=\frac{(2\pi i)^d}{\vol(\cA)}\sum_{\substack{\beta\in\NN^d \\ |\beta|=b}} \sum_{t\in \frc^{-1}\Lambda/\Lambda}f(t)\sum_{\substack{\underline{\epsilon}\in\{0,1\}^d \\ |\underline{\epsilon}|=j+1}} \frac{\Gamma(s)K^{\beta+\underline{\epsilon}}(Hr,0,w-t,s,\Lambda^{Hr,*})}{\pi^s}\omega_{\underline{\epsilon}}\otimes \bar{\mathrm{u}}^{[\beta]},
\end{equation}
with
\[
	\omega_{\underline{\epsilon}}=\iota_{-\sum_{i=1}^d r_i\frac{\partial}{\partial r_i} } \left( \bigwedge_{i=1}^d \omega_{\underline{\epsilon}}^i \right),\quad \omega_{\underline{\epsilon}}^i=\begin{cases}
	dz_i\wedge \frac{dr_i}{r_i} & \epsilon_i=1\\
	dz_i\wedge d\bar{z}_i & \epsilon_i=0.
\end{cases}
\]
In particular, $s\mapsto \phi_s^{(b,j)}$ admits an analytic continuation to $\CC$.
\end{proposition}
\begin{proof}
 Both sides of \eqref{eq_phijk_form} vary holomorphically in $s$, so it is enough to prove the formula in the region where the defining series of the Eisenstein--Kronecker series converge, i.e. for $\Re(s)>2d+b$. Recall that we defined for $l\in \Lambda^*$ the vector field $\tilde{l}:=\sum_{i=1}^d \frac{\bar{l}_i\bar{\mathrm{u}}_i}{r_i}$. By Lemma \ref{lem_difform} \eqref{lem_difform_b} and the equation
\[
	\tilde{l}^{[b]} =\sum_{\substack{ \beta\in\NN^d \\ |\beta|=b}} \left( \frac{\bar{l}}{r}\right)^{\beta} \ul{\mathrm{u}}^{[\beta]}
\] 
  it is enough to show
 \begin{align}\label{eq_phijkEK}
 	\sum_{l\in \Lambda^{*}\smallsetminus \{0\}} \left( \frac{\bar{l}}{r} \right)^{\beta+\underline{\epsilon}} \frac{1}{\left( \pi \left\lVert \frac{l}{\sqrt{r}} \right\rVert^2_H \right)^s}\chi_l(w-t)=\frac{K^{\beta+\underline{\epsilon}}(Hr,0,w-t,s,\Lambda^{Hr,*})}{\pi^s}.
 \end{align}
Using $\Hnorm{\frac{rl}{\sqrt{r}}}=\lVert l\rVert^2_{Hr}$, $r\Lambda^{Hr,*}=\Lambda^{H,*}$ this follows from the definition of the Eisenstein--Kronecker series. Finally, let us note that the possible simple poles of the Eisenstein--Kronecker series at $s=0$ and $s=d$ cancel by the assumption $\sum_{t\in \frc^{-1}\Lambda/\Lambda}f(t)=0$.
\end{proof}

Using the explicit description of $\widehat{\cP}$ from Theorem \ref{thm_analyticPoincare} we may define the Eisenstein--Kronecker current:
\begin{definition}\label{def:EK-current}
The \emph{Eisenstein--Kronecker current} is the pro-system $\phi=(\phi^{(n)})_n$ of $(d,d-1)$-currents
\[
	\phi^{(n)}:=\sum_{b=0}^n\sum_{j=0}^{2d} \phi_{b+j+1}^{(b,j)} \in H^{0}\left( \cA(\CC)\times E\Gamma, \cD_{\cA(\CC)\times E\Gamma}^{d,d-1}(\pr^*{\cP}^{(n)})\right)
\]
with values in $\pr^*{\cP}^{(n)}$, where the $\phi_{b+j+1}^{(b,j)}$ were defined in \eqref{eq_phi_kjs}.
\end{definition}
\nomenclature{$\phi=(\phi^{(n)})_n$}{\nomrefpage}

\subsection{The equivariant coherent Eisenstein--Kronecker class}
In this section  we will compute explicit representatives for the Eisenstein-Kronecker classes 
$\Eis^{\beta,\alpha}_\Gamma(f,x)$
defined in Definition \ref{def:Eis-classes} on the model $E\Gamma$ for the equivariant cohomology constructed in Section \ref{section:model}.

Let $x$ be a $\Gamma$-invariant $\frf$-torsion section of $\cU_{\cD}(\CC)$ and $\cD:=\cA[\frc]\setminus\{x(\cR)\}$ be as in Definition \ref{def:D-specific}. Note that $[\frc]$ is automatically \'etale since we are working in this section over $\CC$.  Recall from Definition \ref{def:Eis-classes}, that the definition of the Eisenstein--Kronecker class $\EK_\Gamma^{\beta,\alpha}(f,x)$ depends on a $\Gamma$-invariant function $f\in R[\cD]^{0,\Gamma}$.
Using the above complex uniformization $\cA(\CC)=\CC^\Sigma/\Lambda$ allows us to view $f$ as a $\Gamma$-invariant function
\[
	f\colon \frc^{-1}\Lambda/\Lambda \rightarrow R
\]
satisfying $\sum_{t\in \frc^{-1}\Lambda/\Lambda}f(t)=0$. The chosen $\frf$-torsion section $x$ corresponds to a $\Gamma$-invariant element $x\in\frf^{-1}\Lambda / \Lambda$. 
\begin{definition}
For $t\in \cA(\CC)=\CC^\Sigma/\Lambda$ let us write $\delta_t$ for the $\delta$-distribution concentrated in $t$ and 
\begin{equation*}
\delta_f:=\sum_{t\in\frc^{-1}\Lambda/\Lambda} f(t)\delta_t.
\end{equation*}
\end{definition}

\begin{theorem}\label{thm_explicit_pol} For $n\geq 0$, the $(d,d-1)$-current $\phi^{(n)}$ from Definition \ref{def:EK-current} is $\Gamma$-equivariant and  solves the differential equation
\begin{align}
	\label{eq_eqPol_2}\nabla''(\phi^{(n)})&=\delta_f\vol.
\end{align}
\end{theorem}
\begin{proof}
The $\Gamma$-invariance  of $\phi^{(n)}$ will follow from the more general formula $\gamma^*\phi_s^{(b,j)}=\phi_s^{(b,j)}$. For $\gamma\in \Gamma$ and $l\in \Lambda^*$, the function $\gamma^*\langle  l, \cdot \rangle$ is $\ZZ$-valued on $\Lambda\subseteq \CC^\Sigma$, we deduce that 
$$\overline{\gamma}. l:=(\overline{\sigma_1(\gamma^{-1})}l_1,\cdots,  \overline{\sigma_d(\gamma^{-1})}l_d)\in \Lambda^*\subseteq \CC^\Sigma.$$
gives a well-defined action on $\Lambda^*$. We have
\begin{align*}
	\gamma^*\Hnorm{\frac{l}{\sqrt{r}}} & =\gamma^*\left( \sum_{i=1}^d \frac{|l_i|^2}{r_i} \right)=\sum_{i=1}^d \frac{|l_i|^2}{\sigma_i(\gamma)\overline{\sigma_i(\gamma)}r_i}=\Hnorm{\frac{\overline{\gamma}.l}{\sqrt{r}}} \\
	\gamma^*\langle  l, \cdot \rangle &=\langle \overline{\gamma}.l,\cdot  \rangle \quad \text{and} \quad \gamma^*\circ \iota_{\tilde{l}}=\iota_{\widetilde{\overline{\gamma}.l}}\circ \gamma^*.
\end{align*}
Since $f$ is $\Gamma$-invariant, we deduce $\gamma^*\phi_s^{(b,j)}=\phi_s^{(b,j)}$ as desired. It remains to show that $\phi^{(n)}$ satisfies the differential equation
\[
	\nabla''(\phi^{(n)})=\delta_f \vol.
\]
Let us assume that there is a $(d,d-1)$-current $\phi^{(n)}$ solving the above differential equation and consider its Fourier expansion 
$$\phi^{(n)}=\sum_{l\in \Lambda^*}\phi_l \chi_l.$$
The above differential equation \eqref{eq_eqPol_2} can be restated, because $\nabla'(\phi^{(n)})=0$, as the following differential equation for the Fourier coefficients
\[
	d\phi_l +(2\pi i d\langle l, \cdot \rangle+\nu)\wedge \phi_l=\sum_{t\in \frc^{-1}\Lambda/\Lambda}f(t)\chi_l(-t)\vol .
\]
Our aim is to solve this equation coefficient-wise. This strategy goes back to Nori \cite{Nori}. For $l=0$ it is solved by $\phi_0=0$. For $l\in \Lambda^*$ define $A_l:=2\pi i d \langle l, \cdot \rangle+\nu \in \cE^{1}_{\cA(\CC)\times E\Gamma}\otimes {\TSym}^{\leq n}(\sH)$. We get an operator
$$C_l:=d+A_l \colon \cE^{\bullet}_{\cA(\CC)\times E\Gamma}\otimes {\TSym}^{\leq n}(\sH)\rightarrow \cE^{\bullet+1}_{\cA(\CC)\times E\Gamma}\otimes {\TSym}^{\leq n} (\sH).$$
\emph{Claim: For $l\neq 0$ the operator $C_l\circ \iota_{\tilde{l}}+\iota_{\tilde{l}}\circ C_l$ is invertible on $(d,d)$-currents and 
\[
	\phi_{l,t}:=\iota_{\tilde{l}}\circ( C_l\circ \iota_{\tilde{l}}+\iota_{\tilde{l}}\circ C_l )^{-1}(\exp(2\pi i \langle l,-t\rangle)\vol)
\]
solves the equation}
\[
	d\phi_{l,t}+A_l\wedge \phi_{l,t}=\chi_l(-t)\vol.
\]
\emph{Proof of the Claim:} Let us first observe	$C_l \circ C_l=0$  and  $\iota_{\tilde{l}}\circ \iota_{\tilde{l}}=0$. Our next aim is to prove that the operator 
\[
	C_l\circ \iota_{\tilde{l}}+\iota_{\tilde{l}}\circ C_l=d\iota_{\tilde{l}} +\iota_{\tilde{l}}d+A_l\iota_{\tilde{l}}+\iota_{\tilde{l}}A_l
\]
is invertible on $(d,d)$-currents. The first term $\cL_{\tilde{l}}:=d\iota_{\tilde{l}} +\iota_{\tilde{l}}d$ is the Lie-derivative along the vector field $\tilde{l}$. It is nilpotent, more precisely $\cL_{\tilde{l}}^{2d+1}(\vol)=0$. Let us compute the second term evaluated on a form $\omega$:
\[
	(A_l\iota_{\tilde{l}}+\iota_{\tilde{l}}A_l)(\omega)=A_l\wedge \iota_{\tilde{l}}(\omega)+\iota_{\tilde{l}}(A_l\wedge \omega)=\iota_{\tilde{l}}(A_l)\wedge\omega.
\]
Thus the operator $A_l\iota_{\tilde{l}}+\iota_{\tilde{l}}A_l$ is just multiplication by
\[
	\iota_{\tilde{l}}(A_l)=2\pi i \iota_{\tilde{l}} d(\langle l, \cdot \rangle)+\tilde{l}=\pi \Hnorm{\frac{l}{\sqrt{r}}}+\tilde{l}.
\] 
Since $\pi \Hnorm{\frac{l}{\sqrt{r}}}\in \CC^\times$ and $\tilde{l}\in \TSym^1(\pr^*\cH)$ we conclude that $A_l\iota_{\tilde{l}}+\iota_{\tilde{l}}A_l$ is invertible in $\widehat{\TSym}^\cdot (\pr^*\cH)$. We have already seen that $\cL_{\tilde{l}}$ is nilpotent and deduce that $C_l\circ \iota_{\tilde{l}}+\iota_{\tilde{l}}\circ C_l$ is invertible on $(d,d)$-currents.\par 
Next, let us prove that $\phi_{l,t}$ solves the equation
\[
	d\phi_{l,t}+A_l\wedge \phi_{l,t}=\exp (2\pi i\langle l,-t \rangle)\vol.
\]
Let us define $\omega_{l,t}:=\exp(2\pi i\langle l,-t \rangle)\vol$. We have $C_l(\omega_{l,t})=d\omega_{l,t}+A_l\wedge\omega_{l,t}=0$ and since $C_l$ commutes with $(C_l\iota_{\tilde{l}}+\iota_{\tilde{l}}C_l)^{-1}$ we deduce
\begin{align*}
	\omega_{l,t}&=(C_l\iota_{\tilde{l}}+\iota_{\tilde{l}}C_l)(C_l\iota_{\tilde{l}}+\iota_{\tilde{l}}C_l)^{-1}(\omega_{l,t})=C_l\iota_{\tilde{l}}(C_l\iota_{\tilde{l}}+\iota_{\tilde{l}}C_l)^{-1}\omega_{l,t}.
\end{align*}
Thus $\phi_{l,t}=\iota_{\tilde{l}}(C_l\iota_{\tilde{l}}+\iota_{\tilde{l}}C_l)^{-1}\omega_{l,t}$ solves the equation
\[
	d\phi_{l,t}+A_l\wedge \phi_{l,t}=\chi_l(-t)\vol
\]
as desired. This finishes the proof of the Claim.\par 

In a next step, we compute $\phi_{l,t}$ explicitly.
\begin{align*}
	\phi_{l,t}&=\iota_{\tilde{l}}(C_l \iota_{\tilde{l}}+ \iota_{\tilde{l}}C_l)^{-1}\omega_{l,t}=\iota_{\tilde{l}}(\cL_{\tilde{l}}+\iota_{\tilde{l}}(A_l))^{-1}\omega_{l,t}= \sum_{j=0}^{2d} (-1)^j \iota_{\tilde{l}}(A_l)^{-(j+1)}\iota_{\tilde{l}}(\cL_{\tilde{l}})^j\omega_{l,t}\\
	&=\sum_{j=0}^{2d}(-1)^j \frac{1}{\left(\pi \Hnorm{\frac{l}{\sqrt{r}}} +\tilde{l}\right)^{j+1}}\iota_{\tilde{l}}(d\iota_{\tilde{l}})^j\omega_{l,t}
\end{align*}
After expanding the last sum using the binomial series
\[
	 \frac{1}{\left(\pi \Hnorm{\frac{l}{\sqrt{r}}} +\tilde{l}\right)^{j+1}}=\sum_{b= 0}^n \frac{(j+b)!}{j!}\left(\pi \Hnorm{\frac{l}{\sqrt{r}}}\right)^{-(b+j+1)}\tilde{l}^{[b]}
\]
we see that $\sum_{t\in \frc^{-1}\Lambda/\Lambda}f(t)\phi_{l,t}$ are the Fourier coefficients of $\phi$ and the result follows.
\end{proof}
\begin{remark}
	Using the inclusion ${\cP}^{(n)}\subseteq {\cP}^{\natural(n)}$ we may view $\phi=(\phi^{(n)})_n$ as a pro-system of $(2d-1)$ current in ${\cP}^{\natural(n)}$ satisfying
	\[
		\nabla''(\phi)=\delta_f \vol.
	\]
	Furthermore, we have $\nabla'(\phi)=0$ for trivial reasons: $\phi$ is already of top degree in the holomorphic differential forms. Thus, $\phi$ is an equivariant current satisfiying
	\[
		\nabla(\phi)=\delta_f \vol.
	\]
	By Theorem \ref{thm:scheider-theorem} the sheaf ${\sP}^{\natural(n)}$ is canonically isomorphic to the de Rham logarithm sheaf
	so that  $[\phi]$ is a representative of the equivariant de Rham polylogarithm. 
\end{remark}

In Lemma \ref{lem_cP_resolution} we have seen that 
\[
		\pr^{-1}({\sP}^{(n),an}\otimes \Omega^{d}_{\cA(\CC)})\xrightarrow{\sim} \left(\pr^*{\cP}^{(n)}\otimes \cE^{d,\bullet}_{\cA(\CC)\times E\Gamma}, \nabla'' \right).
	\]
	In particular, the cohomology
	\[
		H^i(\cU_\cD(\CC),\Gamma,\widehat{\sP}^{an}\otimes\Omega^{d}_{\cA(\CC)})\cong \varprojlim_n H^i(\cU_\cD(\CC),\Gamma,{\sP}^{(n),an}\otimes\Omega^{d}_{\cA(\CC)})
	\]
	can be described in terms of smooth $\Gamma$-invariant $\nabla''$-closed $(d,d-1)$-forms with values in $\pr^*{\cP}^{(n)}$. As a corollary of the above Proposition and Theorem \ref{thm_explicit_pol} we obtain:
\begin{corollary}\label{cor_cohpol}
The pro-system of $(d,d-1)$-forms $\phi|_{\cU_\cD(\CC)\times E\Gamma}$ represents the equivariant coherent Eisenstein--Kronecker class
\[
	[\phi]_{coh}=\EK_{\Gamma}(f)\in H^{d-1}(\cU_\cD(\CC),\Gamma,\widehat{\sP}^{an}\otimes\Omega^d_{\cA(\CC)}).
\]	
\end{corollary}
\begin{proof}
The pro-system of smooth $(d,d-1)$-forms $(\phi^{(n)}|_{\cU_\cD(\CC)\times E\Gamma})_n$ satisfies the differential equation
\[
	\nabla''(\phi^{(n)}|_{\cU_\cD(\CC)})=0,
\]
so it is a $\Gamma$-equivariant $\nabla''$-closed $(d,d-1)$-form and gives a cohomology class in 
\[
	H^{d-1}(\cU_\cD(\CC),\Gamma,\widehat{\sP}^{an}\otimes\Omega^d_{\cA(\CC)})\cong \varprojlim_n H^{d-1}(\cU_\cD(\CC),\Gamma,{\sP}^{(n),an}\otimes\Omega^d_{\cA(\CC)}).
\]
It remains to compute the residue map
\[
	H^{d-1}(\cU_\cD(\CC),\Gamma,\widehat{\sP}^{an}\otimes\Omega^d_{\cA(\CC)})\rightarrow H^{d}_{\cD}(\cA(\CC),\Gamma,\widehat{\sP}^{an}\otimes\Omega^d_{\cA(\CC)}).
\]
The residue map in Dolbeault cohomology is computed by extending $\phi^(n)|_{\cU_\cD(\CC)\times E\Gamma}$ to a current and applying $\nabla''$. By definition, the current $\phi^{(n)}$ extends the differential form $\phi^{(n)}|_{\cU_\cD(\CC)\times E\Gamma}$. Thus $[\nabla''(\phi^{(n)})]$ represents the class $\res([\phi^{(n)}|_{\cU_\cD(\CC)\times E\Gamma}])$. According to Theorem \ref{thm_explicit_pol} we have $\nabla''(\phi^{(n)})=\delta_f \vol$. The class $\delta_f \vol$ represents the cohomology class in $H^d_{\cD(\CC)}(\cA(\CC),\Gamma; {\sP}^{(n),an}\otimes \Omega^d_{\cA(\CC)/\cS})$ corresponding to $f \in \CC[\cD]^{0,\Gamma}$ under the inclusion $\CC[\cD]^{0,\Gamma}\hookrightarrow H^d_{\cD(\CC)}(\cA(\CC),\Gamma; {\sP}^{(n),an}\otimes \Omega^d_{\cA(\CC)})$ and the result follows.
\end{proof}

 
Our next aim is to compute the specialization of the Eisenstein-Kronecker class at the chosen $\Gamma$-invariant $\frf$-torsion section $x\in \cU_\cD(\CC)$. In the following, we consider tuples $(\beta,\alpha)\in \NN^d\times\NN^d$ such that the associated infinity type is critical, i.e.
\[
	\sum_{i=1}^d \beta_i \ol{\sigma}_i -\sum_{i=1}^d (\alpha_i+1)\sigma_i \in \CI_L(\Gamma).
\]
Using the isomorphism $\NN^d\times\NN^d \cong I^+_{\ol{\Sigma}}\oplus I^+_{\Sigma} $, we will write this condition in the following as $\beta-\alpha-\ul{1}\in \CI_L(\Gamma)$. For $(\beta,\alpha)\in \NN^d\times\NN^d$ such that $\beta-\alpha-\ul{1}\in \CI_L(\Gamma)$, we have an isomorphism
 \begin{multline*}
 	H^{d-1}(\Gamma, \TSym^{\alpha}(\omega_\cA) \otimes \TSym^{\beta}(\sH) \otimes \omega^d_{\cA/\CC} )\\ \cong H^{d-1}(B\Gamma, \CC)\otimes_\CC \TSym^{\alpha}(\omega_\cA) \otimes \TSym^{\beta}(\sH) \otimes \omega^d_{\cA/\CC},
 \end{multline*}
 which allows us to compute $\EK_\Gamma^{\beta,\alpha}(f,x) \in H^{d-1}(\Gamma, \TSym^{\alpha}(\omega_\cA) \otimes \TSym^{\beta}(\sH(\overline{\Sigma}))\otimes \omega^d_{\cA/\CC} )$ in terms of differential forms on $E\Gamma$:
 
\begin{proposition}\label{prop_psi} 
Let $(\beta,\alpha)\in\NN^d\times \NN^d$ such that $\beta-\alpha-\underline{1}\in\CI_L(\Gamma)$ is a critical infinity type and write $a:=|\alpha|$ and $b:=|\beta|$. The $\Gamma$-invariant differential form $$\psi^{(\beta,\alpha)}(f,x)\in  H^0(E\Gamma, \TSym^{\alpha}(\omega_{\cA/\CC}) \otimes_\CC \TSym^{\beta}(\sH(\overline{\Sigma}))\otimes_\CC \omega^d_{\cA/\CC}\otimes_\CC\cE^{d-1}_{E\Gamma})$$ with
\begin{multline*}
\psi^{(\beta,\alpha)}(f,x):=d\cdot (-1)^{\frac{d(d-1)}{2}}\sum_{t\in \frc^{-1}\Lambda/\Lambda}f(t) \frac{\Gamma(a+d)K^{\beta+\alpha+\underline{1}}(Hr,x-t,0,a+d,\Lambda)}{\pi^{a+d}}\\
\times (\pi r)^{\alpha+\ul{1}} \bigwedge_{i=1}^{d-1} \frac{dr_i}{r_i} \otimes dz^{[\alpha]} \otimes \bar{\mathrm{u}}^{[\beta]} \otimes \bigwedge_{i=1}^d dz_i
\end{multline*}
represents the cohomology class
\[
	\EK_\Gamma^{\beta,\alpha}(f,x)=[\psi^{(\beta,\alpha)}(f,x)] \in H^{d-1}(\Gamma, \TSym^{\alpha}(\omega_\cA) \otimes \TSym^{\beta}(\sH(\overline{\Sigma}))\otimes \omega^d_{\cA/\CC} ).
\]
\end{proposition}
\begin{proof}
Let us use the embedding ${\cP}^{(n)}\subseteq {\cP}^{\natural(n)}$ to view $\phi^{(n)}$ as a smooth $(d,d-1)$-form in $\pr^*\cP^{\natural(n)}$. Then $\Eis_{\Gamma}^{\beta,\alpha}$ is represented for $n\geq |\beta|$ by the  projection of
\[
	\varrho_x(x^*(\nabla')^a \phi)\in H^{0}\left(E\Gamma, \TSym^{a}(\omega_{\cA/\CC}) \otimes_\CC {\TSym^{\leq n}} (\sH(\overline{\Sigma} ))\otimes_\CC \omega^d_{\cA/\CC}\otimes_\CC \cE^{d-1}_{E\Gamma} \right)
\]
to the direct summand $\TSym^{\alpha}(\omega_{\cA/\CC})\otimes_\CC \TSym^{\beta}(\sH(\overline{\Sigma} ))\otimes \omega^d_{\cA/\CC}$. The $\beta$-component $\TSym^{\beta}(\sH(\overline{\Sigma}))$ is contained in the subspace of $\TSym^{b}(\sH(\overline{\Sigma}))$ spanned by powers of $\bar{\mathrm{u}}_1,\dots,\bar{\mathrm{u}}_d$. Thus the holomorphic connection 
$$\nabla'=d^{1,0}+\nu^{1,0}=d^{1,0}+\sum_{i=1}^d  dz_i\otimes u_i$$
acts just by the holomorphic exterior derivation $d^{1,0}$ after projection to $\TSym^{\beta}(\sH(\overline{\Sigma}))$. Furthermore, we have $e^*\omega_{\underline{\epsilon}}=0$ if $\underline{\epsilon}\neq (1,\dots,1)$. We deduce from Corollary \ref{cor_cohpol}, Proposition \ref{prop_polform} and Lemma \ref{lem_difform} \eqref{lem_difform_c} that the coherent equivariant Eisenstein--Kronecker class on $B\Gamma$ is represented by the differential form:
\begin{multline*}
	(-1)^{\frac{d(d-1)}{2}}\frac{d(2\pi i)^d}{\vol(\cA)}\sum_{t\in \frc^{-1}\Lambda/\Lambda}f(t) \partial_z^{\alpha}\left. \frac{\Gamma(b+d)K^{\beta+\underline{1}}(Hr,0,z-t,b+d,\Lambda^{Hr,*})}{\pi^{b+d}}\right|_{z=x}\\
	\times \bigwedge_{i=1}^{d-1} \frac{dr_i}{r_i} \otimes dz^{[\alpha]} \otimes \bar{\mathrm{u}}^{[\beta]} \otimes \bigwedge_{i=1}^d dz_i
\end{multline*}
Combining this with
\[
	\partial_z^{\alpha} K^{\beta+\underline{1}}(Hr,0,z-t,b+d,\Lambda^{Hr,*})|_{z=x}=(\pi r)^{\alpha}K^{\beta+\alpha+\underline{1}}(Hr,0,z-t,b+d,\Lambda^{Hr,*})|_{z=x}
\]
gives the representative
\begin{multline*}
(-1)^{\frac{d(d-1)}{2}}\frac{d(2\pi i)^d}{\vol(\cA)}\sum_{t\in \frc^{-1}\Lambda/\Lambda}f(t) \frac{\Gamma(b+d)K^{\beta+\alpha+\underline{1}}(Hr,0,x-t,b+d,\Lambda^{Hr,*})}{\pi^{b+d}}\\
\times(\pi r)^{\alpha} \bigwedge_{i=1}^{d-1} \frac{dr_i}{r_i} \otimes dz^{[\alpha]} \otimes \bar{\mathrm{u}}^{[\beta]} \otimes \bigwedge_{i=1}^d dz_i
\end{multline*}
for $\Eis_{\Gamma}^{\beta,\alpha}(f,x)$. Recall the formula $\vol(\cA)=\int_{\cA/\CC}\bigwedge_{i=1}^d dz_i\wedge d\bar{z}_i$. The measure $\bigwedge_{i=1}^d dz_i\wedge d\bar{z}_i$ differs from the measure $\mu_{st}$ induced from the metric given by the standard scalar product by a factor $(2i)^d$, so
\[
	\frac{(2\pi i)^d}{\vol(\cA)}=\frac{\pi^d }{\vol_H(\Lambda)}=\frac{(\pi r)^{\ul{1}} }{\vol_{Hr}(\Lambda)}.
\]
Now, the desired result follows from the functional equation
\[
	\frac{\Gamma(b+d)K^{\beta+\alpha+\underline{1}}(Hr,0,x-t,b+d,\Lambda^{Hr,*})}{\pi^{b+d}\vol_{Hr}(\Lambda) }=\frac{\Gamma(a+d)K^{\beta+\alpha+\underline{1}}(Hr,x-t,0,a+d,\Lambda)}{\pi^{a+d}}.\qedhere
\]
\end{proof}

\subsection{Eisenstein series and fiber integration} The map $E\Gamma\to B\Gamma$ is on the level of differential forms given by a fibre integration. This fibre integration turns the  Eisenstein--Kronecker series into  Eisenstein series in the following sense:

\begin{definition}\label{def:E-beta-alpha}
	Let $\Gamma\subseteq \sO_L^\times$ be an arbitrary subgroup of finite index, $(\beta,\alpha)\in \NN^d\times\NN^d$ with $\beta-\alpha\in \CI_L(\Gamma)$, $s\in \CC$ with $\Re(s)>\frac{b-a}{2}+d$. For $t'\in \QQ\otimes \Lambda$, we define the Eisenstein series
	\[
		E^{\beta,\alpha}(t',s;\Lambda,\Gamma):=\dashsum_{\lambda\in \Gamma \backslash (\Lambda+\Gamma t')}\frac{{\overline{\lambda}}^{\beta}}{{\lambda}^{\alpha}N(\lambda)^s},
	\]
	where $\lambda$ runs over all non-zero $\Gamma$-cosets of $\Lambda+\Gamma t'$. 
\end{definition}
\nomenclature{$E^{\beta,\alpha}(t',s;\Lambda,\Gamma)$}{\nomrefpage}
Note that the summand is independent of the chosen representative in the $\Gamma$-coset, since $\beta-\alpha\in \HCI_L(\Gamma)$ (compare Definition \ref{def:critical}), and since the norm is $\Gamma$-invariant. In particular, the Eisenstein series $E^{\beta,\alpha}(t',s;\Lambda,\Gamma)$ does only depend on the double coset of $t'$ in $\Gamma \backslash (\QQ\otimes \Lambda/\Lambda)$.
For later reference, let us observe the following lemma, which allows us to pass to a torsion free subgroup for studying certain properties of the Eisenstein series.
\begin{lemma}\label{lem_change_of_group}
For a subgroups of finite index $\Gamma'\subset \Gamma\subseteq \sO_L^\times$, $\beta-\alpha\in \CI_L(\Gamma)$, $s$ and $t'$ as above, we have:
\[
	E^{\beta,\alpha}(t',s;\Lambda,\Gamma')=[\Gamma:\Gamma']E^{\beta,\alpha}(t',s;\Lambda,\Gamma).
\]
\end{lemma}
\begin{proof}
This is a straight-forward computation.
\end{proof}
Next, we describe the relation of the Eisenstein series $E^{\beta,\alpha}(t',s;\Lambda,\Gamma)$ to the Eisenstein-Kronecker series.

\begin{lemma}\label{lem_Eis}
	Let $(\beta,\alpha)\in \NN^d\times \NN^d$ with $\beta-\alpha\in\CI_L(\Gamma)$ and $t'\in  \QQ\otimes\Lambda/\Lambda$ a torsion section, then $\Gamma(\alpha+\underline{s})E^{\beta,\alpha}(t',s;\Lambda,\Gamma)$ admits an analytic continuation to $\CC$. More precisely, if $\Gamma\subseteq \sO_L^\times$ is torsion free, we have
	\begin{align*}
		&\Gamma(\alpha+\underline{s})E^{\beta,\alpha}(t',s;\Lambda,\Gamma)\\
		=&d\sum_{\substack{t\in \QQ\otimes\Lambda/\Lambda \\  \Gamma t=\Gamma t'}}\int_{r\in B\Gamma} \frac{\Gamma(a+ds)K^{\beta+\alpha}(Hr,t,0,a+ds;\Lambda)}{\pi^{a+ds}}(\pi  r)^{\alpha+\ul{s}}\bigwedge_{i=1}^{d-1}\frac{dr_i}{r_i},
	\end{align*}
	where the first sum runs over all torsion sections in $\QQ\otimes\Lambda/\Lambda$ which are in the same $\Gamma$-coset as $t'$. Furthermore, we write $\underline{s}=(s,\dots,s)$ and $\Gamma(\alpha+\underline{s}):=\prod_{i=1}^d \Gamma(\alpha_i+s) $ for the product of $\Gamma$-functions.
\end{lemma}
\begin{proof}By the previous Lemma, we may without loss of generality assume that $\Gamma$ is torsion free.
	The right hand side of the above equation is defined for $s\in \CC$, so it remains to compare it to the left hand side for $\Re(s)>\frac{b-a}{2}+d$:
	\begin{align*}
		& d\sum_{\substack{t\in \QQ\otimes\Lambda/\Lambda \\ \Gamma t=\Gamma t'}}\int_{r\in B\Gamma} \frac{\Gamma(a+ds)K^{\beta+\alpha}(Hr,t,0,a+ds;\Lambda)}{\pi^{a+ds}}(\pi  r)^{\alpha+\ul{s}}\bigwedge_{i=1}^{d-1}\frac{dr_i}{r_i}=\\
		=& d\int_{r\in B\Gamma}  \sum_{\lambda \in \Lambda+\Gamma t'} \int_{u\in\RR_{>0}} \overline{\lambda}^{\beta+\alpha}\exp\left( -\pi \lVert \lambda \rVert_{Hr}^2 u \right) u^{a+ds} \frac{du}{u} (\pi r)^{\alpha+\ul{s}}\bigwedge_{i=1}^{d-1}\frac{dr_i}{r_i}=\\
	\end{align*}
		Let us now make the substitution $\tilde{r}_i:=u r_i$ corresponding to the bijection
	\[
		B\Gamma\times \RR_{>0}=(\Gamma\backslash L^1_{\RR})\times\RR_{>0}\xrightarrow{\sim} \Gamma\backslash (\RR_{>0})^d,\quad ((r_1,...,r_d),u)\mapsto (ur_1,\dots,ur_d).
	\]
	Applying the Mellin transform and using the above substitution gives
		\begin{equation*}
		=  \sum_{\lambda\in \Gamma\backslash (\Lambda+\Gamma t')}\Gamma(\alpha+\underline{s})\frac{\overline{\lambda}^{\beta+\alpha}}{ \lambda^{\alpha+\underline{s}} (\pi \overline{\lambda})^{\alpha+\underline{s}}}\pi^{a+ds}=\Gamma(\alpha+\underline{s})E^{\beta,\alpha}(t',s;\Lambda,\Gamma) \qedhere
	\end{equation*}
\end{proof}

\subsection{The main formula for the Eisenstein-Kronecker class}
In this section we will finally  compute the Eisenstein-Kronecker class $\Eis_{\Gamma}^{\beta,\alpha}(f,x)$ explicitly. 

Recall that these classes are obtained by capping the classes $\pr_\Gamma(\Eis_{\Gamma}^{\beta,\alpha}(f,x))$ from \eqref{eq:pr-Eis} with $\xi\in H_{d-1}(\Gamma,\ZZ)$ (cf. Proposition \ref{prop:orientation}):
\[
	H^{d-1}(\Gamma,\TSym^{\alpha}(\omega_\cA) \otimes \TSym^{\beta}(\sH(\overline{\Sigma}))\otimes \omega^d_{\cA/\CC})\rightarrow \TSym^{\alpha+\underline{1}}(\omega_\cA) \otimes \TSym^{\beta}(\sH(\overline{\Sigma})).
\]
It will be convenient to introduce the notation $\alpha!:=\alpha_1!\cdots \alpha_d!$ for $\alpha\in \NN^d$. Furthermore, we use the canonical isomorphism $\omega_{\cA^\vee/\CC}\cong \sH(\ol{\Sigma})$ to view $\ol{\mathrm{u}}$ as a basis of $\omega_{\cA^\vee/\CC}$. With this convention, we have the following explicit formula for the Eisenstein-Kronecker class.
\begin{theorem}[Explicit form of the Eisenstein-Kronecker class]\label{thm:computation-of-Eis} Let $\Gamma\subseteq \sO_L^\times$ be a subgroup of finite index, $\frf$ and $\frc$ co-prime ideals in $\sO_L$, $x\in \frf^{-1}\Lambda/\Lambda$ a $\Gamma$-invariant torsion section of $\cU_{\cD}(\CC)$, and $f\colon \frc^{-1}\Lambda/\Lambda\to \CC$ a $\Gamma$-invariant function satisfying $\sum_{t\in \frc^{-1}\Lambda/\Lambda}f(t)=0$, where $\Lambda$ is the period lattice from \eqref{eq:lattice}. For $(\beta,\alpha)\in \mathbb{N}^d\times \mathbb{N}^d$ such that $\beta-\alpha-\underline{1}\in\CI_L(\Gamma)$ is a critical infinity type, we have the following equality in $\TSym^{\alpha+\underline{1}}(\omega_{\cA/\CC}) \otimes \TSym^{\beta}(\omega_{\cA^\vee/\CC})$:
	\[
		\Eis_{\Gamma}^{\beta,\alpha}(f,x)=(-1)^{\frac{d(d-1)}{2}}\cdot \alpha!\sum_{ t \in \Gamma\backslash (\frc^{-1}\Lambda/\Lambda)} f(-t) E^{\beta,\alpha+\underline{1}}(t+x,0;\Lambda,\Gamma) \cdot  dz^{[\alpha+\underline{1}]} \otimes  \bar{\mathrm{u}}^{[\beta]},
	\]
	where the sum runs over all $\Gamma$-cosets of $\frc$-torsion sections. Note that the summand is well-defined on $\Gamma$-cosets, since $f$ and $x$ are $\Gamma$-invariant and the Eisenstein series depends only on the double coset of $t+x$ in $\Gamma\backslash (\QQ\otimes \Lambda/\Lambda)$, see Definition \ref{def:E-beta-alpha}.
\end{theorem}
\begin{proof} By Lemma \ref{lem_change_of_group} and Corollary \ref{cor:change-of-group} it suffices to prove the statement
for a torsion free $\Gamma$. The cap product
\begin{multline*}
	H_{d-1}(B\Gamma,\ZZ)\otimes_\ZZ H^{d-1}(B\Gamma,\TSym^{\alpha}(\omega_{\cA/\CC})\otimes \TSym^{\beta}(\omega_{\cA^\vee/\CC})\otimes\omega^d_{\cA/\CC}) \xrightarrow{\cap}\\ H_0(B\Gamma,\TSym^{\alpha+\underline{1}}(\omega_{\cA/\CC}) \otimes \TSym^{\beta}(\omega_{\cA^\vee/\CC})\otimes\omega^d_{\cA/\CC})
\end{multline*}
is given by $([Z],[\omega])\mapsto \int_Z\omega$  if $Z$ is a cycle in singular homology and $\omega$ a $(d-1)$-form representing a cohomology class in degree $d-1$. Using Proposition \ref{prop_psi}, Lemma \ref{lem_Eis} and Proposition \ref{prop:orientation} we see that the image of $\Eis_{\Gamma}^{\beta,\alpha}(f,x)$ under
	\[
		H^{d-1}(B\Gamma,\TSym^{\alpha}(\omega_{\cA/\CC}) \otimes \TSym^{\beta}(\omega_{\cA^\vee/\CC})\otimes\omega^d_{\cA/\CC}) \cong  \TSym^{\alpha+\underline{1}}(\omega_{\cA/\CC}) \otimes \TSym^{\beta}(\omega_{\cA^\vee/\CC})
	\]
	is given by the following fiber integral:
	\begin{align*}
	&\int_{r\in B\Gamma}\psi^{(\beta,\alpha)}(f,x)=(-1)^{\frac{d(d-1)}{2}} \cdot \Gamma(\alpha+\underline{1})\sum_{ t \in\Gamma\backslash( \frc^{-1}\Lambda/\Lambda)} f(-t)E^{\beta,\alpha+\underline{1}}(x+t,0;\Lambda,\Gamma)dz^{\otimes (\alpha+\underline{1})}\bar{\mathrm{u}}^{\otimes (\beta)}.
	\end{align*}	
\end{proof}
We have chosen our basis $(dz_1,\ldots,dz_d)$ in such a way that it is the given basis $\omega(\cA)=(\omega(\cA)(\sigma))_{\sigma\in \Sigma}$. 
On the other hand, recall from Definition \ref{def:ubar} that $(\ol{\mathrm{u}}_1,\ldots,\ol{\mathrm{u}}_d)$ corresponds to 
$(\frac{\partial}{\partial\bar{z}_1},\ldots,\frac{\partial}{\partial\bar{z}_d})$ under the isomorphism
\begin{equation*}
 \overline{\Lie(\cA/\CC)}\isom \omega_{\cA^\vee/\C},
\end{equation*}
and hence is not an algebraic basis of $\omega_{\cA^\vee/\C}$. Nevertheless, each 
$\mathrm{\bar{u}}_i$ is a generator of the $1$-dimensional $\CC$-vector space $\omega_{\cA^\vee/\C}(\overline{\sigma}_i)$, and hence a complex multiple of our fixed algebraic basis $\omega(\cA^{\vee})(\overline{\sigma}_i)$. Recall the pairing 
\begin{equation*}
\langle \cdot ,\cdot \rangle_{\dR,\cA}:\ol{\omega_{\cA/\C}}\times \omega_{\cA^{\vee}/\C}\to \C^{\ol{\Sigma}}
\end{equation*}
from Corollary \ref{cor:pairing}.
This  allows us to state the explicit computation of the Eisenstein--Kronecker class in terms of our algebraic bases $\omega(\cA)$ and $\omega(\cA^\vee)$:
\begin{corollary}\label{cor:EK-in-terms-of-period-lattice}
For $(\beta,\alpha)\in \mathbb{N}^d\times \mathbb{N}^d$ such that $\beta-\alpha-\underline{1}\in\CI_L(\Gamma)$ is a critical infinity type and $\Lambda$ the period lattice from \eqref{eq:lattice}, we have the following equality in $\TSym^{\alpha+\underline{1}}(\omega_{\cA/\CC}) \otimes \TSym^{\beta}(\omega_{\cA^{\vee}/\C})$:
\begin{multline*}
		\Eis_{\Gamma}^{\beta,\alpha}(f,x)=\\\frac{(-1)^{\frac{d(d-1)}{2}}\cdot \alpha!}{\langle \ol{\omega(\cA)},\omega(\cA^{\vee})\rangle_{\dR,\cA}^\beta}\sum_{ t\in \Gamma\backslash (\frc^{-1}\Lambda/\Lambda)} f(-t) E^{\beta,\alpha+\underline{1}}(t+x,0;\Lambda,\Gamma) \cdot  \omega(\cA)^{[\alpha+\underline{1}]} \otimes  \omega(\cA^\vee)^{[\beta]}.
\end{multline*}
\end{corollary}
\begin{remark}
Note that $\langle\ol{\omega(\cA)},\omega(\cA^{\vee})\rangle_{\dR,\cA}$ is an invariant of the tuple  $(\cA,\omega(\cA),\omega(\cA^\vee))$, which does not depend on auxiliary choices like a polarization. In Section \ref{section:periods}, we will explain the relation of this invariant to classical periods of $\cA$ and $\cA^\vee$.
\end{remark}
\section{Application to  special values of $L$-functions for Hecke characters}
In this section we discuss the consequences of the computation in Theorem \ref{thm:computation-of-Eis} and Corollary \ref{cor:EK-in-terms-of-period-lattice} for the special values of $L$-functions of algebraic Hecke characters. 
\subsection{Hecke $L$-functions} 
In this section we define the relevant Hecke $L$-function and relate them to the Eisenstein series from Definition \ref{def:E-beta-alpha}.

Let $L$ be a  number field and $\frf\subset \sO_L$ an integral ideal, $\cI(\frf)$ the group of fractional ideals prime to $\frf$. We let $\cP_\frf\subset \cI(\frf)$ be the subgroup of 
the principal ideals generated by $\lambda\in L^{\times }$ with $\lambda\equiv 1{\mod^{\times}}\frf$ and totally positive at the real places.
%

\begin{definition}
An algebraic Hecke character $\chi$ with values in a number field $E\subset \C$, of conductor dividing $\frf$ and infinity type $\mu\in I_L$ is a homomorphism 
\begin{equation*}
\chi:\cI(\frf)\to E^{\times}\subset \C^{\times}
\end{equation*}
such that for all $\lambda\in \cP_\frf$ one has
\begin{equation*}
\chi((\lambda))=\prod_{\sigma\in J_L}\sigma(\lambda)^{\mu(\sigma)}.
\end{equation*}
A finite Hecke character is called a Dirichlet character. 
\end{definition}
If $\frf\mid \frf'$ one has $\cI(\frf')\subset \cI(\frf)$ and one identifies the Hecke character of conductor dividing $\frf$ with the ones of conductor dividing $\frf'$ obtained by restriction. The smallest $\frf$ where $\chi$ can be defined, is called the conductor of $\chi$. 

In the case where $L$ contains a maximal CM field $K\subset L$, it follows from the discussion after Theorem \ref{thm:decomposition} that an algebraic Hecke character is of the form 
\begin{equation}\label{eq:Hecke-character}
\chi=\varrho\cdot (\chi_0\circ\norm_{L/K})
\end{equation}
where $\chi_0$ is an algebraic Hecke character of $K$ and $\varrho$ a  Dirichlet character of $L$.

For a Hecke character $\chi$ of $L$ of conductor $\frf$, 
one defines  the $L$-function for $\Re s\gg 0$ by
\begin{equation*}
L_\frf(\chi,s):=\sum_{\stackrel{\fra\in \cI(\frf)}{\fra\subset \sO_L}}\frac{\chi(\fra)}{\norm\fra^{s}},
\end{equation*}
\nomenclature{$L_\frf(\chi,s)$}{\nomrefpage}
where the sum is extended over all integral ideals $\fra$ coprime to $\frf$.
For an integral ideal $\frb$ of $\sO_L$ with ray class $[\frb]\in \cI(\frf)/P_\frf$ we introduce the partial $L$-function
\begin{align*}
L_\frf(\chi,s,[\frb]):=\sum_{\stackrel{\fra\in [\frb]}{\fra\subset \sO_L}}\frac{\chi(\fra)}{\norm\fra^{s}}&&\mbox{so that}&&
L_\frf(\chi,s)=\sum_{[\frb]\in \cI(\frf)/P_\frf}L_\frf(\chi,s,[\frb]).
\end{align*}
The $L$-series has an analytic continuation to $\C$ and satisfies a functional equation.
The value at $s=0$ of $L_\frf(\chi,s)$ is called critical, if none of the $\Gamma$-factors occurring on each side of the functional equation have a pole at $s=0$. 
It is known that critical values can occur only if either $L$ is totally real or if $L$ contains a CM field (see \cite{Deligne-conj}). 
In the latter case, one has  that $s=0$ is critical for $L_\frf(\chi,s)$ if and only if the infinity type $\mu$ of $\chi$ is critical in the sense of Definition \ref{def:critical}. In this case, the set
\[
	\Sigma:=\{\sigma \in J_L\mid \mu(\sigma)<0\}
\]
is a CM type of $L$ which will be called the \emph{CM type attached to $\chi$}, and we can write $\mu$ uniquely as
\begin{equation*}
\mu=\beta-\alpha
\end{equation*}
with $\beta\in I_{\ol{\Sigma}}^{+}$ and $\alpha-\ul{1}\in I_\Sigma^{+}$ (see Notation \ref{not:I-Sigma}). 

From now on, let $\chi$ be a critical Hecke character of conductor $\frf$ with attached CM type $\Sigma$ and infinity type $\mu=\beta-\alpha$. Let $\sO_{\frf}^{\times}$ be the subgroup of units of $\sO_L$, which are congruent to $1$ mod $\frf$. 
The partial $L$-functions $L_\frf(\chi,s,[\frb])$ can be expressed through the series $E^{\beta,\alpha}(1,s;\frf\frb^{-1},\sO_\frf^{\times})$ defined in Definition
\ref{def:E-beta-alpha} as follows: If the conductor $\frf$ is not trivial there is a bijection
\begin{align*}
\sO_{\frf}^{\times}\backslash(1+\frf\frb^{-1})\isom \{\fra \in [\frb]\mid \fra\mbox{ integral}\}&& \lambda\mapsto \lambda\frb.
\end{align*}
In this case, one has for each $\lambda\in 1+\frf\frb^{-1}$ the formula $\chi((\lambda))=\prod_{\sigma\in \Sigma}\frac{\ol{\sigma}(\lambda)^{\beta}}{\sigma(\lambda)^{\alpha}}$. For the trivial conductor $\frf=\sO_L$ the integral ideals in the class $[\frb]$ are parametrized by  $\sO_L^{\times}\backslash(\frb^{-1}\smallsetminus\{0\})$, and then $1\in \frb^{-1}$.
Thus, one gets for arbitrary conductor $\frf$:
\begin{lemma}\label{lemma:L-and-Eis}
One has
\begin{equation}
L_\frf(\chi,s,[\frb])=\chi(\frb)\norm\frb^{-s}\dashsum_{\lambda\in \sO_{\frf}^{\times}\backslash(1+\frf\frb^{-1})}
\frac{\chi((\lambda))}{\norm\lambda^{s}}=\chi(\frb)\norm\frb^{-s}E^{\beta,\alpha}(1,s,\frf\frb^{-1},{\sO_{\frf}^{\times}}),
\end{equation}
where $\dashsum$ means that we do not sum over $\lambda=0$ and $E^{\beta,\alpha}(1,s,\frf\frb^{-1},{\sO_{\frf}^{\times}})$ is the Eisenstein series from Definition \ref{def:E-beta-alpha}. In particular, one can evaluate both sides at $s=0$ and gets
\begin{equation}\label{eq:L-in-terms-of-E}
L_\frf(\chi,0,[\frb])=\chi(\frb)E^{\beta,\alpha}(1,0;\frf\frb^{-1},{\sO_{\frf}^{\times}}).
\end{equation}
\end{lemma}

\subsection{Periods}\label{section:periods}
Consider $(\cA/\cR,\omega(\cA),\omega(\cA^{\vee}))$ as in Notation \ref{notation-for-CM-ab}. For the integrality results on special values of $L$-functions it is necessary to choose the periods  compatible for $\cA$ and $\cA(\fra)$ for certain ideals $\fra$ of $\sO_L$.

One would like to fix an $\sO_L$-basis of $H_1(\cA(\C),\Z)$, but this is not always possible, so that we define:
\begin{definition}\label{def:a-structure}
Let $\cA/\cR$ be an abelian scheme with CM by $\sO_L$ and $R\subset \C$. Let $\fra$ be a fractional ideal of $L$. An $\fra$- resp. $L$-structure of $\cA$ is an isomorphism of $\sO_L$-modules
\begin{align*}
\xi_\fra:\fra\isom H_1(\cA(\C),\Z)&&\mbox{resp.}&& \xi_{L}:L\isom  H_1(\cA(\C),\Q).
\end{align*}
\end{definition}
\nomenclature{$\xi_\fra,\xi_L$}{\nomrefpage}
Notice that if $\cA$ has an $\fra$-structure $\xi_\fra$, then 
$\frb\otimes\cA$ has an $\fra\frb$-structure because 
\begin{equation*}
H_1(\frb\otimes\cA(\C),\Z)\isom \frb\otimes_{\sO_L} H_1(\cA(\C),\Z).
\end{equation*}
An $\fra$-structure induces an $L$-structure by tensoring with $\Q$
\begin{equation*}
\xi_\fra: L\isom \fra\otimes_{\Z}\Q \xrightarrow{\isom}
H_1(\cA(\C),\Q),
\end{equation*}
which we again denote by $\xi_\fra$ by abuse of notation. To compare the period of $\cA$ with the one of $\cA^{\vee}$ consider the following commutative diagram, which expresses the compatibility of the Betti-pairing  with the de Rham pairing (see \cite[(10.2.8)]{Deligne-HodgeIII})
\begin{equation*}\xymatrix{
\langle\cdot,\cdot\rangle_{top}\colon H_{1}(\cA(\C),\Z)\times H_{1}(\cA^{\vee}(\C),\Z) \ar[r]\ar[d]&(2\pi i)\Z\ar[d]\\
\langle\cdot,\cdot\rangle_\dR\colon \sH_{\cA/\CC}\times \sH_{\cA^\vee/\CC}\ar[r]& \CC,
}
\end{equation*}
here we use the notation $\sH_{\cA/\CC}:=H^1_{\dR}(\cA/\CC)^\vee$ (see Definition \ref{def:sH}).
In particular, 
an $L$-structure $\xi_L$ on $\cA$ induces an $L$-structure $\xi_L^{\vee}$ on $\cA^{\vee}$, as $H_1(\cA^{\vee}(\C),\Q)$ is the $L$-dual of $H_1(\cA(\C),\Q)$. More precisely, it induces a natural $\fra^{-1}\frd^{-1}_L$-structure (see \eqref{eq:R-O-pairing}).
\begin{remark}
CM abelian varieties with an $\fra$-structure always exist, as a CM type $\Sigma$ induces an isomorphism $\Phi_\Sigma:L\otimes\R\isom \C^{\Sigma}$,
which allows to view fractional ideals $\fra\subset L$ as lattices in $\C^{\Sigma}$. It is well-known (see \cite{Lang}) that $\C^{\Sigma}/\fra$ is an abelian variety with CM by $\sO_L$, which can be defined over an algebraic number field $k$ and hence there is an affine open $\cR=\Spec R\subset \Spec \sO_k$ over which it can be defined. On the other hand, for each abelian variety $\cA/\cR$ with CM by $\sO_L$ the $\sO_L$-module $H_1(\cA(\C),\Z)$ is isomorphic to some fractional ideal.
\end{remark}
Recall the period pairing from Corollary \ref{cor:pairing}
\begin{equation*}
\langle\cdot,\cdot \rangle_\cA:\omega_{\cA/\C}\times H_1(\cA(\C),\Z)\to \C^{\Sigma}.
\end{equation*}
Similarly, one has the pairing 
\begin{equation*}
\langle\cdot ,\cdot\rangle_{\cA^{\vee}}:\omega_{\cA^{\vee}/\C}\times H_1(\cA^{\vee}(\C),\Z)\to \C^{\ol{\Sigma}}.
\end{equation*} 
for the dual abelian scheme $\cA^{\vee}$.
\begin{definition}\label{def:periods}
Let $(\cA/\cR,\omega(\cA),\omega(\cA^{\vee}))$ be as in Notation \ref{notation-for-CM-ab} and suppose it has an $\sO_L$-structure $\xi_{\sO_L}:\sO_L\isom H_1(\cA(\C),\Z)$. Denote by $\xi_L^{\vee}$ the induced $L$-structure on $\cA^{\vee}$. Then we define periods
\begin{align*}
\Omega:=\langle\omega(\cA), \xi_{\sO_L}(1)\rangle_\cA\in \C^{\Sigma}&&\mbox{and}&&
\Omega^{\vee}:=\langle\omega(\cA^{\vee}), \xi_{L}^{\vee}(1)\rangle_{\cA^{\vee}}\in \C^{\ol{\Sigma}}.
\end{align*}
\end{definition}
\nomenclature{$\Omega,\Omega^\vee$}{\nomrefpage}
For later use and for the comparison with more classical formulations, we express $\Omega^{\vee}$ differently.
\begin{proposition}\label{prop:Omega-dual}
Suppose that $(\cA/\cR,\omega(\cA),\omega(\cA^{\vee}))$ has an $\sO_L$-structure $\xi_{\sO_L}:\sO_L\isom H_1(\cA(\C),\Z)$. Then
\begin{equation*}
\Omega^{\vee}=\frac{2\pi i\langle\ol{\omega(\cA)},\omega(\cA^{\vee})\rangle_{\dR,\cA}}{\ol{\Omega}}.
\end{equation*}
\end{proposition}
\begin{proof}
The pairing $\langle,\rangle_\cA$ comes from the isomorphism \begin{equation*}
H^1(\cA(\C),\C)\isom H^{1}_\dR(\cA/\C)\isom \Lie(\cA/\C)\oplus \ol{\Lie(\cA/\C)},
\end{equation*} 
which maps $\xi_{\sO_L}(1)\mapsto  (\Omega\omega(\cA)^{\vee},\ol{\Omega\omega(\cA)^{\vee}})$. The dual of this isomorphism induces the pairing $\langle,\rangle_{\cA^{\vee}}$ and maps the dual basis to the dual basis, hence
\begin{equation*}
(\frac{1}{\Omega}\omega(\cA),\frac{1}{\ol{\Omega}}\ol{\omega(\cA)})\mapsto \frac{1}{2\pi i}\xi_{L}^{\vee}(1).
\end{equation*}
Taking the pairing $\langle\cdot,\cdot\rangle_{\dR,\cA}$ with $\omega(\cA^{\vee})$ gives the desired result. 
\end{proof}
From now on we fix $(\cA/\cR,\omega(\cA),\omega(\cA^{\vee}))$ with an $\sO_L$-structure $\xi_{\sO_L}:\sO_L\isom H_1(\cA(\C),\Z)$ as our reference variety. For a fractional ideal
$\fra$ of $L$ we write 
\begin{equation*}
\cA(\fra):=\fra\otimes\cA
\end{equation*}
with its induced $\fra$-structure $\xi_\fra$. Note that $\Lie(\cA(\fra)/\cR)\isom \fra\otimes \Lie(\cA/\cR)$. The basis 
$\omega(\cA)$ can be viewed as an isomorphism (cf. \eqref{eq:omega-as-map})
\begin{equation*}
\omega(\cA):\Lie(\cA/\cR)\isom (\sO_L\otimes R)(\Sigma)
\end{equation*}
which induces an isomorphism
\begin{equation}\label{eq:fra-omega}
\fra\otimes \omega(\cA):\Lie(\cA(\fra)/\cR)\isom (\fra\otimes R)(\Sigma)
\end{equation}
and hence an isomorphism $\Lie(\cA(\fra)/\C)\isom \C^{\Sigma}$.
\begin{proposition}\label{prop:period-properties}
With the above notation one has 
\begin{equation*}
\langle\fra\otimes \omega(\cA), \xi_\fra(1)\rangle_{\cA(\fra)}=\langle \omega(\cA), \xi_{\sO_{L}}(1)\rangle_{\cA}=\Omega
\end{equation*}
and the image of
\begin{equation*}
H_1(\cA(\fra)(\C),\Z)\to \Lie(\cA(\fra)/\C)\xrightarrow[\isom]{\fra\otimes \omega(\cA)}\C^{\Sigma}
\end{equation*}
is equal to $\fra\Omega$, so that $\cA(\fra)(\C)\isom\C^{\Sigma}/\fra\Omega$. Suppose further that $\fra=\frf\frb^{-1}$ with integral ideals $\frf,\frb$ and that $[\frf]:\cA(\fra)\to \cA(\frb^{-1})$ and
$[\frb]:\cA\to \cA(\frb^{-1})$  are \'etale.  Then 
\begin{equation*}
\fra\otimes\omega(\cA)=[\frf]^{*}([\frb]^{*})^{-1}\omega(\cA)
\end{equation*} 
is a basis of $\omega_{\cA(\fra)/\cR}$. 
\end{proposition}
\begin{proof} The map
\begin{equation*}
H_1(\cA(\fra)(\C),\Q)\to \Lie(\cA(\fra)/\C)\xrightarrow[\isom]{\fra\otimes \omega(\cA)}\C^{\Sigma}
\end{equation*}
is induced from 
\begin{equation*}
H_1(\cA(\C),\Q)\to \Lie(\cA/\C)\xrightarrow[\isom]{ \omega(\cA)}\C^{\Sigma}
\end{equation*}
by tensoring over $\Z$ with $\fra$. Thus, by definition of the induced $\fra$-structure, the image of $\xi_\fra(1)$ in $\C^{\Sigma}$ is the same as the image of $\xi_{\sO_L}(1)$. This shows 
\begin{equation*}
\langle\fra\otimes \omega(\cA), \xi_\fra(1)\rangle_{\cA(\fra)}=\Omega.
\end{equation*}
The image of $H_1(\cA(\C),\Z)$ in $\C^{\Sigma}$ is the $\sO_L$-module generated by 
\begin{equation*}
\omega(\cA)(\xi_{\sO_L}(1))=\langle\omega(\cA), \xi_{\sO_L}(1)\rangle_\cA=\Omega.
\end{equation*}
and again from the definitions one gets that the image of 
$H_1(\cA(\fra)(\C),\Z)$ in $\C^{\Sigma}$ is $\fra\Omega$. 
In the case $\fra=\frf\frb^{-1}$ one gets as period lattice $\frf\frb^{-1}\Omega$ and hence $\Omega$ is a $[\frf]$-torsion section.
If 
$[\frf]$ and $[\frb]$ are \'etale, one has isomorphisms
\begin{equation*}
\omega_{\cA(\fra)/\cR}\xrightarrow[\isom]{[\frf]^{*}}\omega_{\cA(\frb^{-1})/\cR}\xleftarrow[\isom]{[\frb]^{*}}\omega_{\cA/\cR}. 
\end{equation*}
This implies that $\fra\otimes \omega(\cA)=[\frf]^{*}([\frb]^{*})^{-1}\omega(\cA)$ is a basis of $\omega_{\cA(\fra)/\cR}$. 
\end{proof}

\subsection{Integrality of special values of Eisenstein series} We can now formulate our fundamental integrality result.\label{section:integrality-of-Eisenstein-series}
We use the notations from Section \ref{section:periods}. 
Let $\frf$, $\frb$ and $\frc$ be pairwise coprime integral ideals in $\sO_L$ and write $\fra:=\frf\frb^{-1}$ as before. 

Let $\chi$ be a critical algebraic Hecke character of conductor $\frf$ and attached CM type $\Sigma$. By Lemma \ref{lemma:L-and-Eis} the partial $L$-values $L_\frf(\chi,0,[\frb])$ are expressed by Eisenstein series on an abelian variety with CM type $\Sigma$ and with an $\fra= \frf\frb^{-1}$-structure evaluated at an $\frf$-torsion section. To have compatible periods on all these abelian varieties we fix an abelian variety with CM type $\Sigma$, an $\sO_L$-structure $\xi_{\sO_L}:\sO_L\isom H_1(\cA(\C),\Z)$ and bases  $(\cA/\cR,\omega(\cA),\omega(\cA^{\vee}))$, where $R$ is a $\sO_{L^{Gal}}[1/d_L]$-algebra, as in Notation \ref{notation-for-CM-ab}.

Using Proposition \ref{prop:period-properties} we have a uniformization 
\begin{equation*}
\cA(\fra)(\C)\isom \C^{\Sigma}/\frf\frb^{-1}\Omega,
\end{equation*}
where $\Omega$ is the period. In particular, for all integral ideals $\frb^{-1}$, the point $\Omega$ is an $\frf$-torsion section on $\cA(\fra)$.
Notice that $\cA(\fra)(\cR)\isom \fra\otimes\cA(\cR)$ so that if an $\frf$-torsion section is defined over $\cR$ for one of the Serre constructions $\cA(\fra)$, then all other  Serre constructions have a $\cR$-rational $\frf$-torsion section. We now choose a particular such point.

\begin{definition}\label{def:x-Omega}
Let $(\cA/\cR,\omega(\cA),\omega(\cA^{\vee}))$ be an abelian variety with CM type $\Sigma$ and $\sO_L$-structure $\xi_{\sO_L}:\sO_L\isom H_1(\cA(\C),\Z)$. Let 
\begin{equation*}
x_\Omega\in \cA(\frf)(\C)
\end{equation*} 
be the $\frf$-torsion section on $\cA(\frf)$ which corresponds under the uniformization $\cA(\frf)(\C)\isom \C^{\Sigma}/\frf\Omega$ to $\Omega$ and enlarge $\cR$ such that $x_\Omega$ is defined over $\cR$. Via the isogeny
\begin{equation*}
[\frb]:\cA(\frf)\to \cA(\fra)
\end{equation*}
whose degree is coprime to $\frf$, this induces $\frf$-torsion sections $[\frb](x_\Omega)$ on $\cA(\fra)$ which are all defined over $\cR$. 
\end{definition}
Notice that on $\C$-valued points the isogeny $[\frb]$
\begin{equation*}
[\frb]:\C^{\Sigma}/\frf\Omega\to \C^{\Sigma}/\frf\frb^{-1}\Omega
\end{equation*}
is just the quotient map for the lattice $\frf\Omega\subset \frf\frb^{-1}\Omega$. In particular, one has $[\frb](\Omega)=\Omega$. 

Let $\cD=\cA[\frc]\smallsetminus\{x(\cR)\}$ be as in Definition \ref{def:D-specific}. Note that over $\cR[1/N\frc]$ the isogenies $[\frc]$ and $[\frc]^{\vee}$ are \'etale. Finally, we let
$\Gamma=\sO_{\frf}^{\times}$
be the units which are congruent to $1$ modulo $\frf$. This group fixes all $\frf$-torsion sections.

\begin{theorem}\label{thm:integral-Eisenstein-series}  With the above notations  let $f\in R[\frac{1}{\norm\frc}][\cD]^{0,\sO_\frf^{\times}}$. Then for critical $\beta-\alpha$
\begin{equation*}
\frac{(\alpha-\ul{1})!(2\pi i)^{|\beta|}}{\Omega^{\alpha}\Omega^{\vee \beta}}
\sum_{t \in \sO_{\frf}^{\times}\backslash (\frc^{-1}\fra/\fra)} f(- t) E^{\beta,\alpha}(t+1,0;\fra,\sO_{\frf}^{\times}) \in R[\frac{1}{\frf\norm(\frb\frc)}].
\end{equation*}
\end{theorem}

\begin{proof}  Because the maps
$\Lie(\cA(\fra)/\cR)\to \Lie(\cA(\frb^{-1})/\cR)$ and $\Lie(\cA/\cR)\to \Lie(\cA(\frb^{-1})/\cR)$ are isomorphisms over $R[\frac{1}{\frf\norm(\frb\frc)}]$, the isogenies $[\frf]$ and $[\frb]$ are \'etale. Hence the form  $\fra\otimes\omega(\cA)=[\frf]^{*}([\frb]^{*})^{-1}\omega(\cA)$ is a basis of $\omega_{\cA(\fra)/\cR}$. 
By Proposition \ref{prop:period-properties} the period lattice of $\cA(\fra)$ with respect to $\fra\otimes\omega(\cA)$ is $\fra\Omega$.
Further, $[\frb]^{\vee}:\cA(\frb^{-1})^{\vee}\to \cA^{\vee}$ is \'etale, so that $[\frb]_\#\omega(\cA^{\vee})=[\frb]^{\vee *}\omega(\cA^{\vee})$ is a basis of $\omega_{\cA(\frb^{-1})^{\vee}/\cR}$.
Choose any 
basis $\omega(\cA(\fra)^{\vee})$ of $\omega_{\cA(\fra)^{\vee}/\cR}$ and consider the $[\frf]$-torsion sections $[\frb](x_\Omega)\in \cA(\fra)(\cR)$, which is $\Omega\in \cA(\fra)(\C)\isom  \C^{\Sigma}/\fra\Omega$ as explained after Definition \ref{def:x-Omega}. Then we have
\begin{equation}\label{eq:test-var}
(\cA(\fra)/\cR,[\frf]^{*}([\frb]^{*})^{-1}\omega(\cA),\omega(\cA(\fra)^{\vee}),[\frb](x_\Omega)).
\end{equation}
From Proposition \ref{prop:integral-Eis-class}, applied to $[\frf]:\cA(\fra)\to \cA(\frb^{-1})$, we know that 
\begin{equation*}
\TSym^{\beta}([\frf]_\#)
\left( \Eis_{\Gamma,\cA(\fra)}^{\beta,\alpha-\ul{1}}(f,[\frb](x_\Omega))\right)
(([\frf]^{*}([\frb]^{*})^{-1}\omega(\cA))^{[\alpha]},
([\frb]_\#\omega(\cA^{\vee}))^{[\beta]})
\in R[\frac{1}{\frf\norm(\frb\frc)}].
\end{equation*}
By  Corollary
\ref{cor:EK-in-terms-of-period-lattice}, $\TSym^{\beta}([\frf]_\#)
\left( \Eis_{\Gamma,\cA(\fra)}^{\beta,\alpha-\ul{1}}(f,[\frb](x_\Omega))\right)$ is equal to
\begin{multline*}
\frac{(-1)^{\frac{d(d-1)}{2}} (\alpha-\ul{1})!}{\langle\ol{[\frf]^{*}([\frb]^{*})^{-1}\omega(\cA)},\omega(\cA(\fra)^{\vee})\rangle_{\dR,\cA(\fra)}^\beta}\cdot\\
\sum_{t\in \sO_{\frf}^{\times}\backslash (\frc^{-1}\fra/\fra)} f(- t) E^{\beta,\alpha}((t+1)\Omega,0;\fra\Omega,\sO_{\frf}^{\times}) \cdot ([\frf]^{*}([\frb]^{*})^{-1}\omega(\cA))^{[\alpha]} \otimes  ([\frf]_\#\omega(\cA(\fra)^\vee))^{[\beta]}.
\end{multline*}
Observe that the de Rham pairing from Corollary \ref{cor:pairing} behaves as follows under isogenies
\begin{equation*}
\langle\ol{[\frf]^{*}([\frb]^{*})^{-1}\omega(\cA)},\omega(\cA(\fra)^{\vee})\rangle_{\dR,\cA(\fra)}=
\langle\ol{([\frb]^{*})^{-1}\omega(\cA)},[\frf]_\#\omega(\cA(\fra)^{\vee})\rangle_{\dR,\cA(\frb^{-1})}
\end{equation*}
and that for the basis $([\frb]_\#\omega(\cA^{\vee}))^{\vee}$ of $\Lie(\cA(\frb^{-1})^{\vee}/\cR)$ one has
\begin{equation*}
\langle([\frb]_\#\omega(\cA^{\vee}))^{\vee}, [\frf]_\#\omega(\cA(\fra)^\vee)\rangle_{\dR,\cA(\frb^{-1})}([\frb]_\#\omega(\cA^{\vee}))=[\frf]_\#\omega(\cA(\fra)^\vee).
\end{equation*}
This implies
\begin{multline*}
\TSym^{\beta}([\frf]_\#)
\left( \Eis_{\Gamma,\cA(\fra)}^{\beta,\alpha-\ul{1}}(f,[\frb](x_\Omega))\right)
(([\frf]^{*}([\frb]^{*})^{-1}\omega(\cA))^{[\alpha]},
([\frb]_\#\omega(\cA^{\vee}))^{[\beta]})=\\(-1)^{\frac{d(d-1)}{2}}\cdot(\alpha-\ul{1})!
\frac{\langle([\frb]_\#\omega(\cA^{\vee}))^{\vee},[\frf]_\#\omega(\cA(\fra)^{\vee})\rangle^{\beta}_{\dR,\cA(\frb^{-1})}}{\langle\ol{([\frb]^{*})^{-1}\omega(\cA)},[\frf]_\#\omega(\cA(\fra)^{\vee})\rangle^{\beta}_{\dR,\cA(\frb^{-1})}}\sum_{t\in \sO_{\frf}^{\times}\backslash (\frc^{-1}\fra/\fra)}f(- t) E^{\beta,\alpha}((t+1)\Omega,0;\fra\Omega,\sO_{\frf}^{\times}).
\end{multline*}
In the quotient in front of the sum, one can substitute 
$[\frf]_\#\omega(\cA(\fra)^{\vee})$ by $[\frb]_\#\omega(\cA^{\vee})$ without changing the value.
This gives 
\begin{equation*}
\frac{\langle([\frb]_\#\omega(\cA^{\vee}))^{\vee},[\frb]_\#\omega(\cA^{\vee})\rangle_{\dR,\cA(\frb^{-1})}}{\langle\ol{([\frb]^{*})^{-1}\omega(\cA)},[\frb]_\#\omega(\cA^{\vee})\rangle_{\dR,\cA(\frb^{-1})}}=
\frac{1}{\langle\ol{\omega(\cA)},\omega(\cA^{\vee})\rangle_{\dR,\cA}}.
\end{equation*}
Using the equality
\begin{equation*}
E^{\beta,\alpha}((t+1)\Omega,0;\fra\Omega,\sO_{\frf}^{\times})=\left(\frac{\ol{\Omega}^{\beta}}{\Omega^{\alpha}}\right)E^{\beta,\alpha}(t+1,0;\fra,\sO_{\frf}^{\times})
\end{equation*}
one concludes with Proposition \ref{prop:Omega-dual}.
\end{proof}

\subsection{Special values of $L$-series for algebraic Hecke characters}\label{section-special-values}
The next theorem is one of the main results in this paper and generalizes (and strengthens) the work of Katz \cite{Katz-CM} in the case of CM fields.
We use the notations from Sections \ref{section:periods} and \ref{section:integrality-of-Eisenstein-series}. In particular, we fix 
\begin{equation*}
(\cA/\cR,\omega(\cA),\omega(\cA^{\vee}))
\end{equation*}
with an $\sO_L$-structure $\xi_{\sO_L}:\sO_L\isom H_1(\cA(\C),\Z)$. 
\begin{theorem}[Special values of Hecke $L$-functions]\label{thm:special-values} Let $\frf$  be an integral ideal in $\sO_L$ and assume that the $\frf$-torsion section $x_\Omega$  on $\cA(\frf)$ from Definition \ref{def:x-Omega} is defined over $\cR$.
Let $\chi$ be a critical algebraic Hecke character with values in $E\subset\C$, of conductor $\frf$ and infinity type $\mu=\beta-\alpha\in\CI_L(\sO_{\frf}^{\times})$ (see Notation\ref{not:I-Sigma}). Then if $\frf\neq \sO_L$ one has for any integral ideal $\frc$ coprime to $\frf$
\begin{equation*}
\frac{(\alpha-\ul{1})!(2\pi i)^{|\beta|}}{\Omega^{\alpha}\Omega^{\vee\beta}}
\left(\chi(\frc)\norm\frc -1\right)L_\frf(\chi,0)\in \sO_ER[\frac{1}{\frf\norm\frc}],
\end{equation*}
where $\sO_ER\subset \C$ is the ring generated by $R$ and the integers $\sO_E$ of $E$.
If $\frf=\sO_L$ one has for any $\frc,\frc'$ coprime to each other
\begin{equation*}
\frac{(\alpha-\ul{1})!(2\pi i)^{|\beta|}}{\Omega^{\alpha}\Omega^{\vee\beta}}
\left(1-\chi(\frc')\right)\left(\chi(\frc)\norm\frc -1\right)L_{\sO_L}(\chi,0)\in\sO_ER[\frac{1}{\norm(\frc\frc')}].
\end{equation*}
Moreover, one can choose $\Spec R= \Spec\sO_k[\frac{1}{d_L\frf\norm\frc}]$ where $\sO_k$ are the integers  of an algebraic number field $k$. In particular, if $E\subset k$, the value
\begin{equation*}
\frac{(\alpha-\ul{1})!(2\pi i)^{|\beta|}}{\Omega^{\alpha}\Omega^{\vee\beta}}
L_\frf(\chi,0)\in k
\end{equation*}
is an algebraic number in $k$. 
\end{theorem}
The proof is given below.

\begin{remark}Let $K$ be an imaginary quadratic field and $L/K$ an extension. For Hecke characters of the form $\phi=\chi\cdot(\psi_K\circ N_{L/K}$, where $\phi_K$ is a  Hecke character  of $K$ and $\chi$  a character of finite order, Bergeron-Charollois-Garcia (see \cite{BCG} and \cite{BCG2}) proved that the critical values of $L$-functions of $\phi$ can be expressed as polynomials in Eisenstein--Kronecker series evaluated at torsion sections on elliptic curves. Thus they obtain similar algebraicity  results in this special case.
\end{remark}

Recall that  $\cD:=\ker[\frc]\smallsetminus \{x(\cR)\}$ is as in Definition
\ref{def:D-specific}.  In the case where $x\neq e$, $\cD=\ker[\frc]$ and we would like to  use  the special function 
\begin{equation}\label{eq:special-function}
f_{[\frc]}=\norm\frc 1_{e(\cR)}-1_{\ker[\frc]}\in R[\cD]^{0,\sO_{\frf}^{\times}}.
\end{equation}
For the evaluation of the sum over the $\frc$-torsion sections occurring in the theorem we use a distribution relation for the Eisenstein series defined in Definition \ref{def:E-beta-alpha}:
\begin{proposition}\label{prop:distribution-relation} Let $\frf,\frb,\frc$, with $\frc\neq \sO_L$ be pairwise coprime integral ideals in $\sO_L$ as before and let $\fra=\frf\frb^{-1}$. Then $1\in \frb^{-1}/\fra$ is fixed by $\sO_\frf^{\times}$ and the Eisenstein series from Definition \ref{def:E-beta-alpha} satisfy the distribution relation
\begin{equation*}
\sum_{ t\in \sO_{\frf}^{\times}\backslash (\frc^{-1}\fra/\fra)}  {E^{\beta,\alpha}( t+1,0;\fra,\sO_\frf^{\times})}=
{E^{\beta,\alpha}(1,0;\frc^{-1}\fra,\sO_\frf^{\times})}.
\end{equation*}
In particular, for the function $f_{[\frc]}$ defined in \eqref{eq:special-function} one gets
\begin{multline*}
\sum_{ t\in \sO_{\frf}^{\times}\backslash (\frc^{-1}\fra/\fra)} f_{[\frc]}(- t) {E^{\beta,\alpha}( t+1,0;\fra,\sO_\frf^{\times})}=\norm\frc E^{\beta,\alpha}(1,0;\fra,\sO_\frf^{\times})-{E^{\beta,\alpha}(1,0;\frc^{-1}\fra,\sO_\frf^{\times})}.
\end{multline*}
\end{proposition}
\begin{proof}
One has a disjoint decomposition
\begin{equation*}
\frc^{-1}\fra+1=\bigcup_{\sO_\frf^{\times} t\in \sO_{\frf}^{\times}\backslash (\frc^{-1}\fra/\fra)}(\fra+\sO_{\frf}^{\times}t+1), 
\end{equation*}
which implies
\begin{equation*}
\sO_{\frf}^{\times}\backslash(\frc^{-1}\fra+1)=
\bigcup_{\sO_\frf^{\times} t\in \sO_{\frf}^{\times}\backslash (\frc^{-1}\fra/\fra)}\sO_{\frf}^{\times}\backslash(\fra+\sO_{\frf}^{\times}t+1).
\end{equation*}
Using the definition of ${E^{\beta,\alpha}( t+1,0;\fra,\sO_\frf^{\times})}$ in Definition \ref{def:E-beta-alpha} and of the function $f_{[\frc]}$ the result follows.
\end{proof}
With this distribution relation we get from Theorem \ref{thm:integral-Eisenstein-series} the following corollary:
\begin{corollary}\label{cor:partial-L-values}
With the notations and assumptions of Theorem \ref{thm:special-values} and for $\frf\neq \sO_L$ one has
\begin{equation*}
\frac{(\alpha-\ul{1})!(2\pi i)^{|\beta|}}{\Omega^{\alpha}\Omega^{\vee\beta}}
\left(\chi(\frc)\norm\frc L_\frf(\chi,0,[\frb])- L_\frf(\chi,0,[\frb\frc])\right)\in \sO_ER[\frac{1}{\frf \norm(\frb\frc)}].
\end{equation*}
\end{corollary}
\begin{proof}
From Theorem \ref{thm:integral-Eisenstein-series} and using the distribution relation from Proposition \ref{prop:distribution-relation} for the function $f_{[\frc]}$ from \eqref{eq:special-function} one gets
\begin{equation*}
\frac{(\alpha-\ul{1})!(2\pi i)^{|\beta|}}{\Omega^{\alpha}\Omega^{\vee\beta}}\left(\norm\frc E^{\beta,\alpha}(1,0;\frf\frb^{-1},\sO_\frf^{\times})-{E^{\beta,\alpha}(1,0;\frf\frb^{-1}\frc^{-1},\sO_\frf^{\times})}
\right)\in R[\frac{1}{\frf\norm(\frb\frc)}].
\end{equation*}
To conclude, observe that 
for integral $\frb$ the value $\chi(\frb)\in \sO_E[\frac{1}{\norm\frb}]$, because there is a power $m$ such that $\frb^{m}=(\lambda)$ is a principal ideal with $\lambda\in \sO_L$, so that 
\begin{equation*}
\chi(\frb)^{m}=\prod_{\sigma\in \Sigma}\frac{\ol{\sigma}(\lambda)^{\beta(\ol{\sigma})}}{\sigma(\lambda)^{\alpha(\sigma)}}\in \sO_E[\frac{1}{\norm\frb}].
\end{equation*}
Using \eqref{eq:L-in-terms-of-E} gives the desired result.
\end{proof}
\begin{proof}[Proof of Theorem \ref{thm:special-values}] Let us first assume that $\frf\neq \sO_L$. 
Choose representatives $\frb_1,\ldots,\frb_m$ for the ideal classes in $\cI(\frf)/\cP_\frf$. Then we get from Corollary \ref{cor:partial-L-values} by summing over all classes in $\cI(\frf)/\cP_\frf$ 
\begin{equation*}
\frac{(\alpha-\ul{1})!(2\pi i)^{|\beta|}}{\Omega^{\alpha}\Omega^{\vee\beta}}
\left(\chi(\frc)\norm\frc -1\right)L_\frf( \chi,0)\in\sO_ER[\frac{1}{\frf\norm(\frb_1\cdots\frb_m\frc)}].
\end{equation*}
Choosing other representatives $\widetilde{\frb}_1,\ldots,\widetilde{\frb}_m$ which are coprime to $ \frf\frb_1\cdots\frb_m\frc$ one gets
\begin{equation*}
\frac{(\alpha-\ul{1})!(2\pi i)^{|\beta|}}{\Omega^{\alpha}\Omega^{\vee\beta}}
\left(\chi(\frc)\norm\frc -1\right)L_\frf( \chi,0)\in\sO_ER[\frac{1}{\frf\norm(\widetilde{\frb}_1\cdots\widetilde{\frb}_m\frc)}].
\end{equation*}
The result follows, because the intersection of both rings is $R[\frac{1}{\frf\norm(\frc)}]$.
In the case where $\frf=\sO_L$, choose on $(\cA[\frc']\smallsetminus\{e(\cR)\})\times\cA[\frc]\subset \cD=\ker[\frc\frc']\smallsetminus\{e(\cR)\}$ the function $\widetilde{f_{[\frc]}}$ which is the pull-back of the function $f_{[\frc]}$ under the projection to the second factor. Then the distribution relation gives
\begin{multline*}
\sum_{t \in \sO_L^{\times}\backslash ((\frc\frc'\frb)^{-1}/\frb^{-1})} \widetilde{f}_{[\frc]}(- t) {E^{\beta,\alpha}(t+1,0;\frb^{-1},\sO_L^{\times})}=\\=
\norm\frc {E^{\beta,\alpha}(1,0;(\frb\frc')^{-1},\sO_L^{\times})}-\norm\frc{E^{\beta,\alpha}(1,0;\frb^{-1},\sO_L^{\times})}\\+{E^{\beta,\alpha}(1,0;(\frb\frc)^{-1},\sO_L^{\times})}
-{E^{\beta,\alpha}(1,0;(\frb\frc\frc')^{-1},\sO_L^{\times})}.
\end{multline*}
The rest of the argument is as in the case $\frf\neq \sO_L$.

By the theory of complex multiplication one can assume that $R$ is contained in an algebraic number field. In fact one choose $k$ so big that it contains $\sO_E$, that the locally free modules $\omega_{\cA/\sO_k}$ and $\omega_{\cA^{\vee}/\sO_k}$  become free, and so that $\cA$ has everywhere good reduction. 
\end{proof}

\subsection{A weak form of the Deligne conjecture and a generalization of a formula of Shimura}
In the introduction of \cite{Katz-CM} Katz formulated a conjectural generalization of a formula of Shimura, which is a version of the algebraicity statement in Theorem \ref{thm:special-values} with different periods. In this section we show how one gets this formula from our result. Moreover, one can deduce from Theorem \ref{thm:special-values} a weak form of the Deligne conjecture on the critical values  $L(\chi,0)$ (we thank Blasius for pointing this out), but note that meanwhile Kufner \cite{Kufner} has deduced the full Deligne conjecture from our main results. 

Let us start with a discussion on periods. Let $L$ be an algebraic number field which contains a CM field $K\subset L$ and a fixed CM type $\Sigma$ lifted from the CM type  $\Sigma_0$ of $K$. Let $(\cA_0,\omega(\cA_0),\omega(\cA_0^\vee))$ be an abelian variety with CM by $K$ with a $\sO_K$-structure $\xi_0\colon \sO_K \cong H_1(\cA_0(\CC),\ZZ)$. We may without loss of generality assume that our base ring is $\ol{\QQ}$, i.e. $R=\ol{\QQ}$. 
The datum of $\xi_0,\omega(\cA_0)$ and $\omega(\cA^\vee_0)$ specifies periods
\[
	\Omega_0:=\langle\omega(\cA_0),\xi_0(1)\rangle_{\cA_0},\quad \Omega_0^\vee:=\langle\omega(\cA^\vee_0),\xi^\vee_0(1)\rangle_{\cA^\vee_0}
\]
Let us recall that the datum of $(\omega(\cA_0),\omega(\cA_0^\vee))$ is equivalent to the datum of a basis of the $\sO_K\otimes_\ZZ R$-module $\sH_0:=H^1_{dR}(\cA_0/\cR)^\vee$ and hence gives a trivialization:
\begin{equation}\label{eq_basisH}
	\sH_0 \cong \sO_K\otimes_\ZZ \ol{\QQ},\quad (\text{recall }R=\ol{\QQ}).
\end{equation}
The abelian variety $\cA:=\sO_L\otimes_{\sO_K}\cA_0$ has CM by $\sO_L$. The isomorphism \eqref{eq_basisH} induces an isomorphism
\[
	\sH=\sO_L\otimes_{\sO_K}\sH_0\cong \sO_L\otimes \ol{\QQ}
\]
and thus gives a pair of bases $(\omega(\cA),\omega(\cA^\vee))$ for $\omega_\cA$ respectively $\omega_{\cA^\vee}$. The $\sO_K$-structure on $\cA_0$ induces a $\sO_L$-structure $\xi=\sO_L\otimes_{\sO_K}\xi_0$ on $\cA$. With this choice, we get
\begin{align*}
	\Omega:&=\langle \omega(\cA), \xi(1) \rangle_{\cA}=\langle \omega(\cA_0), \xi_0(1) \rangle_{\cA_0}\otimes_{K} 1 \in L\otimes_K\CC^{\Sigma_0}\cong \CC^{\Sigma},\\
	\Omega^\vee:&=\langle \omega(\cA^\vee), \xi^\vee(1) \rangle_{\cA^{\vee}}=\langle \omega(\cA^\vee_0), \xi^\vee_0(1) \rangle_{\cA_0^{\vee}}\otimes_{K} 1 \in L\otimes_K\CC^{\ol{\Sigma}_0} \cong \CC^{\ol{\Sigma}}
\end{align*}
and thus
\[
	\Omega(\sigma)=\Omega_0(\sigma|_K),\quad \Omega^\vee(\ol{\sigma})=\Omega_0^\vee(\ol{\sigma}|_K),\quad \sigma\in \Sigma,\ol{\sigma}\in\ol{\Sigma}.
\]
Finally, let us observe that the choice of a $\sO_K$-linear polarization $\cA_0\rightarrow \cA_0^\vee$ shows
\[
	\Omega_0^\vee(\ol{\sigma})\ol{\QQ}^\times=\Omega_0(\sigma)\ol{\QQ}^\times.
\]
With this identification our main result implies the following formula which has been formulated by Katz in the introduction of \cite[p. 203]{Katz-CM}:
\begin{corollary}\label{cor:Katz-formula}
	Let $\chi$ be a critical Hecke character of conductor $\frf\subseteq \sO_L$ of the totally imaginary field $L$ and of infinity type $\beta-\alpha=N_{L/K}^*\beta_0-N_{L/K}^*\alpha_0$. Then
	\[
		\left(\frac{(2\pi i)^{|\beta_0|}}{\Omega_0^{\alpha_0+\ol{\beta}_0}}\right)^{[L:K]}L_\frf(\chi,0)
	\]
	is algebraic. Here, we write $\ol{\beta}_0$ for the image of $\beta_0$ under the canonical map $\CC^{\ol{\Sigma}_0}\rightarrow \CC^{\Sigma_0}$ induced by the complex conjugation $\ol{(\cdot)}\colon{\Sigma}_0\cong \ol{\Sigma}_0, \sigma\mapsto \ol{\sigma}$.
\end{corollary}
We deduce a weak version of the Deligne conjecture for critical Hecke $L$-values from Theorem \ref{thm:special-values}. Recall that one can attach to a critical Hecke character $\chi$ of our totally imaginary field $L$ a motive $M(\chi)$ (say of absolute Hodge cycles, to fix ideas) over $L$ with coefficients in $E$. Let $R_{L/\Q}M(\chi)$ be the restriction of scalars from $L$ to $\Q$ and suppose that $R_{L/\Q}M(\chi)$ is critical. Deligne  \cite{Deligne-conj} has attached periods to the motive $R_{L/\Q}M(\chi)$
\begin{equation*}
c^{+}R_{L/\Q}M(\chi)\in (E\otimes\C)^{\times}
\end{equation*}
which are well-defined up to elements in $E^{\times}$. Using our fixed embedding $\tau:E\to \C$, we get 
$c^{+}_\tau R_{L/\Q}M(\chi)\in \C^{\times}$. 
\begin{corollary}[Deligne conjecture up to $\Qbar^{\times}$]\label{cor:deligne-conjecture}
Let $\chi:\cI(\frf)\to \C^{\times}$ be as before a critical Hecke character of type $\beta-\alpha\in \CI_L(\sO_{\frf}^{\times})$. Then in $\C^{\times}$ one has
\begin{equation*}
c^{+}_\tau R_{L/\Q}M(\chi)\sim \left((2\pi i)^{-|\beta_0|}\Omega_0^{\alpha_0+\ol{\beta}_0}\right)^{[L:K]},
\end{equation*}
where $a\sim b$ if there is a $q\in \Qbar^{\times}$ with $a=qb$.
In particular, one has
\begin{equation*}
\frac{L_\frf(\chi,0)}{c^{+}_\tau R_{L/\Q}M(\chi)}\in\Qbar^{\times}.
\end{equation*}
\end{corollary}
\begin{remark}
Blasius \cite{Blasius} in the case of a CM field $K$ and Kufner \cite{Kufner} in general have shown Deligne's conjecture, which claims in the case of algebraic Hecke characters of $K$ that 
\begin{equation*}
\frac{L_\frf(\chi,0)}{c^{+}_\tau R_{L/\Q}M(\chi)}\in E^{\times}.
\end{equation*}
if $\chi$ has values in $E$. 
\end{remark}
\begin{proof} We use the notations and results from \cite{Deligne-conj}.
Let $\mu$ be the infinity type of $\chi$ considered as a homomorphism $\mu:L^{\times}\to E^{\times}$ and denote by $\mu^{*}:E^{\times}\to L^{\times}$ the homomorphism defined by Deligne in \cite[8.19]{Deligne-conj} (the reflex of $\chi$ \cite{Blasius}). If $\mu^{*}$ is the infinity type of a Hecke character $\chi^{*}$, then the associated motive $M(\chi^{*})$ is the $H^{1}$ of an abelian variety defined over $E$ with complex multiplication by $L$.
Deligne 
\cite[8.21]{Deligne-conj} shows that (note that Deligne's conventions (loc.cit. section 8.3) for the infinity type are the inverse of ours)
\begin{equation*}
c^{+}_\tau R_{L/\Q}M(\chi)\sim \prod_{\sigma\in \Sigma}\frac{p(\mu^{*};\tau,\ol{\sigma})^{\beta(\ol{\sigma})}}{p(\mu^{*};\tau,\sigma)^{\alpha(\sigma)}},
\end{equation*}
where $p(\mu^{*};\tau,\sigma)$ is the period defined by Deligne in \cite[8.18]{Deligne-conj}. This is a period modulo $\Qbar^{\times}$ and is invariant under extension of the ground field and coefficients. Consider $\cA=\sO_L\otimes_{\sO_K}\cA_0$ as above.  We can consider $\mu^{*}$ also as a homomorphism $\mu^{*}:(Ek)^{\times}\to L^{\times}$ without changing the periods $p(\mu^{*};\tau,\sigma)$ and then $\mu^{*}$ is the infinity type of the Hecke character defined by the motive $H^{1}(\cA)$.  As the dual of  $H^{1}(\cA)$ is $H_1(\cA)$ and the periods $\Omega$ defined above are the periods of $H_1(\cA)$, one gets
\begin{equation*}
p(\mu^{*};\tau,\sigma)\sim\Omega(\sigma)^{-1}\sim\Omega_0(\sigma\mid_K)^{-1}.
\end{equation*}
Further, using a polarization of $\cA_0$ gives a semi-linear isomorphism $H^{1}(\cA_0)\isom H_1(\cA_0^{\vee})(-1)$, which implies that 
\begin{equation*}
p(\mu^{*};\tau,\ol{\sigma})\sim \frac{\Omega_0({\sigma}\mid_K)}{2\pi i}.
\end{equation*}
From this one gets the desired result.
\end{proof}
\begin{remark}
We thank one of the referees for pointing out that the formalism of motives clarifies the above proof. Recall from \eqref{eq:Hecke-character} that $\chi=\varrho\chi_0\circ N_{L/K}$ so that by \cite{Blasius-periods} M.17 Corollary 1 and 3 one has
\begin{equation*}
c^{+}_\tau(R_{L/\Q}M(\chi))\sim c^{+}_\tau(R_{L/\Q}M(\chi_0\circ N_{L/K}))\sim c^{+}_\tau(R_{L/\Q}M(\chi_0))^{[L:K]}\bmod \ol{\Q}^{\times}
\end{equation*}
On the other hand Deligne \cite[8.21]{Deligne-conj} shows that 
\begin{equation*}
c^{+}_\tau(R_{L/\Q}M(\chi_0))\sim (2\pi i)^{-|\beta_0|}\Omega_0^{\alpha_0+\ol{\beta}_0}\bmod \ol{\Q}^{\times}.
\end{equation*}
The formula in Corollary \ref{cor:Katz-formula} implies then that
$\frac{L_\frf(\chi,0)}{c^{+}_\tau R_{L/\Q}M(\chi)}\in\Qbar^{\times}$.
\end{remark}
\section{$p$-adic interpolation}
Let $L$ be a totally imaginary number field containing a CM field $K$ and $p$ an 'ordinary' prime which is co-prime to the discriminant $d_L$. In this section we will construct a $p$-adic measure interpolating $p$-adically all critical Hecke $L$-values of conductor dividing $p^\infty\frf$, where $\frf\subseteq \sO_L$ is a prime-to-$p$ ideal. This section profited from valuable comments by Blasius, which helped to eliminate some unnecessary choices in  our construction in an earlier version.

\subsection{The geometric setup}\label{sec:p-adic-geometric-setup}
In the following $\CC_p$ denotes the completion of a fixed algebraic closure of $\QQ_p$ and $\iota_p$ denotes a fixed embedding $\iota_p\colon \overline{\QQ}\rightarrow \CC_p$. Let us denote by $\Gamma$ a torsion-free subgroup of the units $\sO_L^\times$. For a CM type $\Sigma$ of $L$, let us define the sets
\begin{align*}
	\Sigma_p:=\{\frp \text{ induced by the $p$-adic embeddings $\iota_p\circ \sigma$ with }\sigma\in\Sigma \}\\
	\overline{\Sigma}_p:=\{\overline{\frp} \text{ induced by the $p$-adic embeddings $\iota_p\circ \overline{\sigma}$ with }\overline{\sigma}\in\overline{\Sigma} \}.
\end{align*}
\nomenclature{$\Sigma_p,\overline{\Sigma}_p$}{\nomrefpage}
Following Katz, we will say that the CM type $\Sigma$ of $L$ is \emph{ordinary at $\iota_p$} or \emph{$p$-ordinary} if the following condition holds:
\begin{equation*}
\tag{ORD-$p$}  \text{The sets $\Sigma_p$ and $\overline{\Sigma}_p$ are disjoint.}
\end{equation*}
If $\Sigma$ is lifted from a CM field $K$ with maximal totally real subfield $F$, then the condition (ORD-$p$) means that all primes in $F$ over $p$ split completely in $K$. A $p$-ordinary CM type induces a disjoint union
\[
	\{\frp\subseteq \sO_L: \frp\mid p\}=\Sigma_p \cup \overline{\Sigma}_p
\]
of the primes dividing $p$. In particular, a $p$-ordinary CM type $\Sigma$ induces the decomposition
\[
	\sO_L\otimes_\ZZ \Zp\cong \sO_L(\Sigma_p)\oplus \sO_L(\overline{\Sigma}_p),\quad  \sO_L(\Sigma_p)=\prod_{\frp\in\Sigma_p}\sO_{L_\frp},\quad \sO_L(\overline{\Sigma}_p)=\prod_{\overline{\frp}\in\overline{\Sigma}_p}\sO_{L_{\overline{\frp}}}.
\]
It will be convenient to write $\frp_\Sigma:=\prod_{\frp\in\Sigma_p} \frp$ and $\ol{\frp}_\Sigma:=\prod_{\ol{\frp}\in\ol{\Sigma}_p} \ol{\frp}$ for the product of primes in $\Sigma_p$ respectively $\ol{\Sigma}_p$. With this notation, we have:
\[
	\sO_L(\Sigma_p)=\varprojlim_n \sO_L/\frp_\Sigma^n,\quad \sO_L(\ol{\Sigma}_p)=\varprojlim_n \sO_L/\ol{\frp}_\Sigma^n.
\]
\nomenclature{$\frp_\Sigma,\ol{\frp}_\Sigma$}{\nomrefpage}
\nomenclature{$\sO_L(\Sigma_p)$}{\nomrefpage}
\begin{notation}\label{notation_padicAbelianVariety}
For a prime $p$ let $(\cA/\cR,\Sigma,\omega(\cA),\omega(\cA^\vee),x)$ be a datum as in Notation \ref{notation-for-CM-ab} with the following extra hypotheses:
\begin{itemize}
\item $\cR=\Spec R$ with $\mathrm{Quot}(R)=k$ a number field, $p\notin R^\times$, $d_L\in R^\times$ and $L^{Gal}\subseteq k$,
\item The CM type $\Sigma$ of $\cA$ satisfies $(\mathrm{ORD}\text{-}p)$ with respect to the embedding $\iota_p$.
\item $x\in \cA[\frf]$ for some ideal $\frf$ with $p\nmid \frf$. In particular, $[\frf]^{\vee}$ is \'etale over $\OCp$.
\end{itemize}
By abuse of notation, we will denote the base change of $\cA$ to $\OCp$ again by $\cA$. If we want to emphasize that we consider $\cA$ over $R$ respectively $\OCp$, we will write $\cA/R$ respectively $\cA/\OCp$.
\end{notation}

From now on let $(\cA/\cR,\Sigma,\omega(\cA),\omega(\cA^\vee),x)$ be a tuple as in Notation \ref{notation_padicAbelianVariety}.

\begin{definition} Let us define the following level structures on $\cA/\OCp$:
\begin{enumerate}
\item A \emph{$\Gamma_{00}(p^\infty)$-structure} on $\cA$ is an $\sO_L$-equivariant isomorphism
\begin{align*}
	\theta_p\colon \mu_{p^\infty}\otimes_\Zp \sO_{L}(\Sigma_p)\xrightarrow{\sim} \widehat{\cA}[p^\infty].
\end{align*}
A \emph{$\Gamma_{00}(p^\infty)$-structure} on $(\cA,\cA^\vee)$ is a pair $(\theta_p,\theta_p^\vee)$ of \emph{$\Gamma_{00}(p^\infty)$-structures} on $\cA$ and $\cA^\vee$. Note, that $\cA^\vee$ is of CM type $\overline{\Sigma}$, so $\theta_p^\vee$ is an $\sO_L$-equivariant isomorphism
\begin{align*}
	\theta_p^\vee\colon \mu_{p^\infty}\otimes_\Zp \sO_{L}(\ol{\Sigma}_p)\xrightarrow{\sim} \widehat{\cA^\vee}[p^\infty].
\end{align*}
\nomenclature{$\theta_p,\theta_p^\vee$}{\nomrefpage}
\item A $\Gamma_{arith}(p^\infty)$-structure on $\cA$ is a $\sO_L$-equivariant isomorphism of $p$-divisible groups
\[
	\mu_{p^\infty}\otimes_\Zp \sO_L(\Sigma_p)\times \underline{\Qp/\Zp}\otimes_\Zp \sO_L(\ol{\Sigma}_p)\xrightarrow{\sim} \cA[p^\infty].
\]
\end{enumerate}
\end{definition}

Note that  the $p$-divisible group $ \widehat{\cA}[p^\infty]$ is the infinitesimal part of  $\cA[p^\infty]$. Let us recall that we denoted for a fractional ideal $\fra \subseteq \sO_L$ by $\cA(\fra):=\fra\otimes \cA$ the Serre construction. Since the $p^n$-multiplication on $\cA$ factors in the following two ways
\begin{equation}\label{eq:p-isogenies}
	\xymatrix{
		[p^n]\colon \cA \ar[r]^-{[\frp^n_{\Sigma}]}_-{\mathrm{inf.}} & \cA(\ol{\frp}^n_{\Sigma}) \ar[r]^-{[\ol{\frp}^n_{\Sigma}]}_-{\et}  & \cA, & 
		[p^n]\colon \cA \ar[r]^-{[\ol{\frp}^n_{\Sigma}]}_-{\et} & \cA(\frp^n_{\Sigma}) \ar[r]^-{[\frp^n_{\Sigma}]}_-{\mathrm{inf.}} & \cA
	}
\end{equation}
with $\frp^n_{\Sigma}$ infinitesimal and $\ol{\frp}^n_{\Sigma}$ \'etale, we get the following decomposition of $\cA[p^n]$ over $\OCp$:
\[
	\cA[p^n]\cong \cA[\frp^n_{\Sigma}]\times \cA[\ol{\frp}^n_{\Sigma}]=\cA[p^n]^0\times \cA[p^n]^{\et}.
\]
By duality, the dual abelian variety is equipped with an $\cO_L$-endomorphisms structure as well and the dual of the diagram \eqref{eq:p-isogenies}  reads:
\begin{equation*}
	\xymatrix{
		 \cA^\vee & \ar[l]_-{[\frp^n_{\Sigma}]}^-{\et}  \cA(\ol{\frp}^n_{\Sigma})^\vee & \ar[l]_-{[\ol{\frp}^n_{\Sigma}]}^-{\mathrm{inf.}}  \cA^\vee \colon [p^n], & 
		\cA^\vee & \ar[l]_-{[\ol{\frp}^n_{\Sigma}]}^-{\mathrm{inf.}} \cA(\frp^n_{\Sigma})^\vee & \ar[l]_-{[\frp^n_{\Sigma}]}^-{\et}  \cA^\vee\colon [p^n].
	}
\end{equation*}
In particular, we get the decomposition
\[
	\cA^\vee[p^n]\cong \cA^\vee[\ol{\frp}^n_{\Sigma}]\times \cA^\vee[{\frp}^n_{\Sigma}]=\cA^\vee[p^n]^0\times \cA^\vee[p^n]^{\et}.
\]
\begin{definition}\label{def:A-n}
In the following, we use the shorter notation 
\begin{equation*}
\cA_{(n)}:=\cA(\ol{\frp}^n_{\Sigma})
\end{equation*} 
\nomenclature{$\cA_{(n)}$}{\nomrefpage}
so that one has an isogeny $[\ol{\frp}_\Sigma^{n}]:\cA_{(n)}\to \cA$.
\end{definition}
Using duality we can state the datum of a $\Gamma_{00}(p^\infty)$-structure on $(\cA,\cA^\vee)$ in several equivalent ways.

\begin{proposition}\label{prop:GammaStructures}The following are equivalent:
\begin{enumerate}
\item A $\Gamma_{00}(p^\infty)$-structure on $(\cA,\cA^\vee)$.
\item A $\Gamma_{arith}(p^\infty)$-structure on $\cA$.
\item A compatible system of $\sO_L$-equivariant isomorphisms of \'etale group schemes
\begin{align*}
	\cA_{(n)}[\ol{\frp}_\Sigma^n]\cong \sO_L(\ol{\Sigma}_p)/p^n\sO_L(\ol{\Sigma}_p),\quad \cA^\vee_{(n)}[\frp_\Sigma^n]\cong \sO_L(\Sigma_p)/p^n\sO_L(\Sigma_p),
\end{align*}
with $\cA_{(n)}:=\cA(\ol{\frp}^n_{\Sigma})$.
\end{enumerate}
\end{proposition}
\begin{proof}
Since we assumed $p\nmid d_L$, the trace induces isomorphisms
\[
	\sO_L(\Sigma_p)\cong \Hom_{\Zp}(\sO_L(\Sigma_p),\Zp),\quad \sO_L(\ol{\Sigma}_p)\cong \Hom_{\Zp}(\sO_L(\ol{\Sigma}_p),\Zp).
\]
The equivalence of $(1)$ and $(2)$  follows from the perfect pairing
\[
	\cA[\ol{\frp}_\Sigma^n]\times\cA^\vee[\ol{\frp}_\Sigma^n]  \rightarrow \mu_{p^n},
\]
obtained by restricting the Weil pairing 
\[
\langle \cdot, \cdot\rangle_{[p^n]} \colon\cA[p^n]\times \cA^\vee[p^n]\to \mu_{p^n}
\]
 to $\cA[\ol{\frp}_\Sigma^n]\times\cA^\vee[\ol{\frp}_\Sigma^n]$. The equivalence of $(1)$ and $(3)$ follows from the perfectness of the Weil pairings
\[
	\langle \cdot,\cdot\rangle_{[\frp_\Sigma^n]}\colon \cA[\frp_\Sigma^n]\times\cA_{(n)}^\vee [\frp_\Sigma^n]  \rightarrow \mu_{p^n}, \quad \langle \cdot,\cdot\rangle_{[\ol{\frp}_\Sigma^n]}\colon\cA_{(n)}[\ol{\frp}_\Sigma^n]\times \cA^\vee[\ol{\frp}_\Sigma^n]\to \mu_{p^n}
\]
associated to the isogenies $[\frp_\Sigma^n]\colon \cA\to \cA_{(n)}$ respectively $[\ol{\frp}_\Sigma^n]\colon \cA_{(n)}\to \cA$.
\end{proof}

\begin{lemma}\label{lem:GammaArith} Suppose that $\cA$ has an $\fra$-structure $\xi\colon\fra\cong H_1(\cA(\CC),\ZZ)$ for a 
prime-to-$p$ fractional ideal $\fra$ then using the embedding $\iota_p:\ol{\Q}\to \C_p$ one has natural isomorphisms:
	\begin{align*}
		\cA[\ol{\frp}_\Sigma^n](\OCp)\isom \cA[\ol{\frp}_\Sigma^n](\CC)&\cong \fra/\ol{\frp}_\Sigma^n\fra\cong \sO_L(\ol{\Sigma}_p)/p^n\sO_L(\ol{\Sigma}_p)\\
		\cA[\frp_\Sigma^n](\OCp)\isom \cA[\frp_\Sigma^n](\CC)&\cong \fra/\frp_\Sigma^n\fra\cong \sO_L(\Sigma_p)/p^n\sO_L(\Sigma_p).
	\end{align*}
	In particular, using $\ZZ/p^n\ZZ\cong \mu_{p^n}(\CC), 1\mapsto \exp(\frac{2\pi i}{p^n})$ there is a unique $\Gamma_{arith}(p^\infty)$-structure on $\cA$ given by the isomorphisms
	\begin{align*}
		\cA[\ol{\frp}_\Sigma^n]\cong \ZZ/p^n\ZZ \otimes_\ZZ \sO_L(\ol{\Sigma}_p),\quad  \cA[\frp_\Sigma^n]\cong \mu_{p^n}\otimes_\ZZ \sO_L(\Sigma_p).
	\end{align*}
	which induce the above isomorphisms on $\CC$-valued points.
\end{lemma}
\begin{proof}
The first isomorphism $\cA[\ol{\frp}_\Sigma^n](\OCp)\isom \cA[\ol{\frp}_\Sigma^n](\CC)$ is induced by $\iota_p$ and the second is clear. As $\fra$ is prime to $p$, the natural map 
\begin{equation*}
\fra/\ol{\frp}_\Sigma^n\fra\subset \sO_L/\ol{\frp}_\Sigma^n\fra\isom \sO_L/\ol{\frp}_\Sigma^n\times \sO_L/\fra\xrightarrow{\pr} \sO_L/\ol{\frp}_\Sigma^n=\sO_L(\ol{\Sigma}_p)/p^n\sO_L(\ol{\Sigma}_p)
\end{equation*}
is an isomorphism. The proof for $\cA[\frp_\Sigma^n](\OCp)$ is the same.
\end{proof}
Later, we will choose a reference abelian variety $\cA$ with $\sO_L$-structure and use the above Lemma to get an induced $\Gamma_{00}(p^\infty)$-structure.\par

Our next goal is to show that the choice of a $\Gamma_{00}(p^\infty)$-structure on $(\cA,\cA^\vee)$ induces canonical trivializations of the Tate module of the Cartier duals $\widehat{\cA}^t$ and $(\widehat{\cA}^\vee)^t$ of the $p$-divisible formal groups $\widehat{\cA}$ and $\widehat{\cA}^\vee$. Let us recall the notation $\cA_{(n)}:=\cA(\ol{\frp}^n_\Sigma)$ and note that the Weil pairings 
\[
	\langle \cdot,\cdot\rangle_{[\frp_\Sigma^n]}\colon \cA[\frp_\Sigma^n]\times\cA_{(n)}^\vee [\frp_\Sigma^n]  \rightarrow \mu_{p^n}, \quad \langle \cdot,\cdot\rangle_{[\ol{\frp}_\Sigma^n]}\colon\cA_{(n)}[\ol{\frp}_\Sigma^n]\times \cA^\vee[\ol{\frp}_\Sigma^n]\to \mu_{p^n}
\]
give the following explicit description of the Cartier duals of $\widehat{\cA}[p^n]=\cA[p^n]^0= \cA[\frp^n_\Sigma]$ and $\widehat{\cA}^\vee[p^n]=\cA^\vee[p^n]^0=\cA^\vee[\ol{\frp}^n_\Sigma]$:
\begin{equation}\label{eq:cartier-duals}
	(\cA[\frp^n_\Sigma])^t\cong \cA_{(n)}^\vee [\frp_\Sigma^n] ,\quad (\cA^\vee[\ol{\frp}^n_\Sigma])^t \cong \cA_{(n)}[\ol{\frp}_\Sigma^n].
\end{equation}
Thus we have
\[
	T_p\widehat{\cA}^t \cong \varprojlim_n \cA_{(n)}^\vee [\frp_\Sigma^n](\OCp),\quad T_p(\widehat{\cA}^\vee)^t \cong \varprojlim_n \cA_{(n)}[\ol{\frp}_\Sigma^n](\OCp).
\]

\begin{proposition}\label{prop:padicFG}
	Let $\cA$ be an abelian variety with a fixed $\Gamma_{00}(p^\infty)$-structure $(\theta_p,\theta_p^\vee)$ on $(\cA,\cA^\vee)$.
	Over $\OCp$ there are canonical $\sO_L$-equivariant isomorphisms
	\begin{align*}
		T_p\widehat{\cA}^t \cong \sO_L(\Sigma_p),\quad T_p(\widehat{\cA}^\vee)^t \cong \sO_{L}(\overline{\Sigma}_p).	
	\end{align*}
	Here, $T_p\widehat{\cA}^t$ denotes the $p$-adic Tate module of the Cartier dual $\widehat{\cA}^t$ of the $p$-divisible group $\widehat{\cA}$. Furthermore, we have isomorphisms
	\begin{align*}
		\widehat{\cA}\xrightarrow{\sim} \Gmf\otimes_{\Zp}\sO_L(\Sigma_p),\quad  \widehat{\cA}^\vee\xrightarrow{\sim} \Gmf\otimes_{\Zp}\sO_{L}(\overline{\Sigma}_p).
	\end{align*}
\end{proposition}

\begin{proof}
By Proposition \ref{prop:GammaStructures}, the $\Gamma_{00}(p^\infty)$-structure gives compatible $\sO_L$-equivariant isomorphisms
\begin{align*}
	\cA_{(n)}^\vee [\frp_\Sigma^n] \cong \sO_L(\Sigma_p)/p^n\sO_L(\Sigma_p),\quad \cA_{(n)}[\ol{\frp}_\Sigma^n]\cong \sO_L(\ol{\Sigma}_p)/p^n\sO_L(\ol{\Sigma}_p)
\end{align*}
and the isomorphisms $T_p\widehat{\cA}^t \cong \sO_L(\Sigma_p)$ and $T_p(\widehat{\cA}^\vee)^t \cong \sO_{L}(\overline{\Sigma}_p)$ follow by passing to the limit. The second claim follows immediately from the definition of a $\Gamma_{00}(p^\infty)$-structure and the equivalence of $p$-divisible formal groups and infinitesimal $p$-divisible groups.
\end{proof}

In the following, let us consider a tuple  $(\cA/\cR,\Sigma,\omega(\cA))$ as in Notation \ref{notation_padicAbelianVariety} with a fixed $\Gamma_{00}(p^\infty)$-structure $\theta_p$ over $\OCp$. The $\sO_L^\times$-action on $\widehat{\cA}$ induces a decomposition
\[
	\omega_{\widehat{\cA}}= \bigoplus_{\sigma\in\Sigma }\omega_{\widehat{\cA}}(-\sigma).
\]
The isomorphism $\widehat{\cA}\cong \Gmf\otimes \sO_L(\Sigma_p)$ over $\OCp$ allows us to be more concrete. We have
\[
	\Lie(\widehat{\cA})\cong \Lie(\Gmf)\otimes_{\Zp}\sO_L(\Sigma_p)\cong \OCp\otimes_{\Zp} \sO_L(\Sigma_p) \cong \prod_{\sigma\in \Sigma}\OCp.
\]
Here, we have used the canonical isomorphism $\Lie(\Gmf)\cong \OCp$ coming from the canonical invariant derivation $(1+T)\frac{\partial}{\partial T}$ on $\Gmf$. Passing to the dual gives us a canonical basis of $\omega_\cA$:
\[
	\omega_{can}(\cA)=(\omega_{can}(\cA)(-\sigma))_{\sigma \in \Sigma}\in \bigoplus_{\sigma\in \Sigma}\omega_{\widehat{\cA}}(-\sigma)=\omega_{\widehat{\cA}}\cong \omega_{\cA}.
\]
The comparison of this basis $\omega_{can}(\cA)$ to the algebraic basis $\omega(\cA)$ of $\omega_\cA$ gives us the $p$-adic periods:

\begin{definition}
	\emph{The $p$-adic periods} $\Omega_p=(\Omega_p(\sigma))_{\sigma\in \Sigma}\in\OCp^\Sigma$ of $\cA$ are defined by the equation:
	\begin{align*}
		\omega(\cA)=\Omega_p\omega_{can}(\cA).
	\end{align*}
\end{definition}
\nomenclature{$\Omega_p,\Omega_p^\vee$}{\nomrefpage}
Of course, the $p$-adic periods depend on the triple $(\cA,\theta_p,\omega(\cA))$. If we want to stress this dependence we will write $\Omega_p(\cA,\theta_p,\omega(\cA))$. Usually, the dependence will be clear from the context. Applying this construction to $(\cA^\vee,\theta_p^\vee,\omega(\cA^\vee))$ gives $p$-adic periods
\[
	\Omega_p^\vee=(\Omega_p^\vee(\ol{\sigma}))_{\ol{\sigma}\in\ol{\Sigma}}\in \OCp^{\ol{\Sigma}}.
\]
The $p$-adic periods are independent under prime-to-$p$ isogenies in the following sense:
\begin{lemma}\label{lem_Gamma00_isogeny}
Let $\cA\rightarrow\cA'$ be an isogeny over $\OCp$ of prime-to-$p$ degree, hence \'etale, and let $\omega(\cA')$ be the basis of $\omega_{\cA'}$ induced from the isomorphism:
	\[
		\omega_{\cA}\cong \omega_{\cA'}.
	\]
	The $\Gamma_{00}(p^\infty)$-structure on $\cA$ induces a $\Gamma_{00}(p^\infty)$-structure $\theta_p'$ on $\cA'$ and we have the comparison
	\[
		\Omega_p(\cA,\theta_p,\omega(\cA))=\Omega_p(\cA',\theta_p',\omega(\cA')).
	\]
\end{lemma}
\begin{proof}
Indeed, we have the equation $\omega(\cA')=\Omega_p(\cA,\theta_p,\omega(\cA))\omega_{can}(\cA')$	since $\omega(\cA')$ and $\omega_{can}(\cA')$ are pullbacks of the corresponding differentials on $\cA$.
\end{proof}

\subsection{Infinitesimal trivialization of the Poincar\'e bundle}
In this section, we will construct an isomorphism $\widehat{\sP}|_{\widehat{\cA}}\xrightarrow{\sim} \sO_{\widehat{\cA\times\cA^\vee}}$. This isomorphism will play a key role in the construction of the $p$-adic measure interpolating the Eisenstein-Kronecker series $p$-adically. Afterwards, we will discuss the basic properties of this isomorphism.

We keep the notation of the previous section and $(\cA/\cR,\Sigma,\omega(\cA),\omega(\cA^\vee),x)$ is as in Notation \ref{notation_padicAbelianVariety}.  In this section, we will always consider $\cA$ as an abelian variety over $\OCp$. Motivated by Norman's construction of $p$-adic theta functions, we obtain:
\begin{proposition}\label{prop_trivialization_Poincare}
	There is a canonical isomorphism
	\[
		\widehat{\sP}|_{\widehat{\cA}}\xrightarrow{\sim} \sO_{\widehat{\cA\times\cA^\vee}}.
	\]
\end{proposition}
\begin{proof}The construction is the same as in the case of elliptic curves, see \cite[II, \S 6]{Sprang_EisensteinKronecker}.  For the convenience of the reader, let us sketch the argument. Recall that $\frp_\Sigma:=\prod_{\frp\in\Sigma_p} \frp$ and $\overline{\frp}_{\Sigma}:=\prod_{\ol{\frp}\in\overline{\Sigma}_p} \ol{\frp}$. The infinitesimal part of the $p^n$-torsion is given by the subgroup scheme
	\[
		C_n:=\cA\left[\frp_\Sigma^n\right]=\ker [\frp_\Sigma^n].
	\]
	After the base change  $\cA\times_{\OCp} C_n\rightarrow C_n$ of $\cA$ to $C_n$ we have a canonical $C_n$-valued section
	\[
		\Delta\colon C_n\rightarrow \cA\times_{\OCp} C_n
	\]
	given by the diagonal. The completed Poincar\'e bundle on $\cA\times_{\OCp} C_n$ is given by $\pr_{\cA}^*\widehat{\sP}$, where $\pr_\cA\colon \cA\times_{\OCp} C_n\rightarrow \cA$ is the projection to the first component. Since the dual isogeny $[\frp_\Sigma^n]^\vee$ is \'etale we can apply Corollary \ref{cor:coh-log-splitting} and obtain an isomorphism
	\[
		\Delta^*\pr_{\cA}^*\widehat{\sP}\cong e^*\pr_{\cA}^*\widehat{\sP}.
	\]
	This gives an $\sO_{C_n}$-linear isomorphism
	\[
		\widehat{\sP}|_{C_n}\cong \Delta^*\pr_{\cA}^*\widehat{\sP} \cong e^*\pr_{\cA}^*\widehat{\sP} \cong \sO_{C_n}\otimes_{\OCp} \sO_{\widehat{\cA^\vee}}
	\]
	Passing to the limit over $n$ and observing $\varinjlim {C_n}=\widehat{\cA}$ proves the claim. 
\end{proof}

In particular, there is an isomorphism of $\sO_{\widehat{\cA}}$-modules 
\begin{equation}\label{eq_1splitting_1}
	\sP^{(1)}|_{\widehat{\cA}}\cong \sO_{\widehat{\cA}}\otimes_{\OCp}(\OCp\oplus \omega_{\cA^\vee}).
\end{equation}
Since $\sP^{\natural(1)}$ is the pushout of
\[
	0\rightarrow \pi^*\omega_{\cA^\vee}\rightarrow \sP^{(1)}\rightarrow \sO_\cA\rightarrow 0
\]
along $\pi^*\omega_{\cA^\vee}\rightarrow \pi^*\sH$ we also get a splitting
\begin{equation}\label{eq_1splitting_2}
	\sP^{\natural(1)}|_{\widehat{\cA}}\cong \sO_{\widehat{\cA}}\otimes_{\OCp}(\OCp\oplus \sH).
\end{equation}

Applying $\TSym^n$ to \eqref{eq_1splitting_1} and \eqref{eq_1splitting_2}, using the co-multiplication maps and passing to the limit over $n$ gives
\begin{align}
	\label{eq_Poincare_TSym1}\widehat{\sP^{\natural}}|_{\widehat{\cA}}&\hookrightarrow \sO_{\widehat{\cA}}\hat{\otimes}_{\OCp} \widehat{\TSym}^\cdot (\sH)\\
	\label{eq_Poincare_TSym2}\widehat{\sP}|_{\widehat{\cA}}&\hookrightarrow \sO_{\widehat{\cA}}\hat{\otimes}_{\OCp} \widehat{\TSym}^\cdot(\omega_{\cA^\vee}).
\end{align}
These maps are injective since $n!$ is a non-zero divisor in $\sO_{\CC_p}$ for all $n\in\NN$. For some arguments, it will be convenient to pass to the generic fiber of $\cA$. Let us introduce the following notation for the generic fiber of our abelian variety:
\begin{notation}\label{not_genfib}
	Let $(\cA/\cR,\omega(\cA),\omega(\cA^\vee),x)$ be an abelian scheme as in Notation \ref{notation_padicAbelianVariety}.
	We will write
	\[
		A:=\cA\times_{\cR}\Spec\Cp,\quad A^\vee:=\cA^\vee\times_{\cR}\Spec\Cp
	\]
	for the generic fiber of $\cA$. The pullback of the differential forms $\omega(\cA)$ and $\omega(\cA^\vee)$ along the maps $A\to \cA$ and $A^\vee \to \cA^\vee$ will be denoted by $\omega(A)$ and $\omega(A^\vee)$. The Poincar\'e bundle on $A\times A^\vee$ will be denoted by $\sP_{\Cp}$ and its completion by $\widehat{\sP}_{\Cp}$ and similarly we use $\widehat{\sP}^{\natural}_\Cp$.
\end{notation}
The following Lemma allows us to compare the connection $\nabla$ on $\widehat{\sP}_{\Cp}^\natural$ to the invariant differential of the formal group $\widehat{A}\times \widehat{A}^\vee$.
\begin{lemma}\label{lem_sP_triv}
\begin{enumerate}
\item The maps \eqref{eq_Poincare_TSym1} and \eqref{eq_Poincare_TSym2} induce isomorphisms
\begin{align*}
	\widehat{\sP}_{\Cp}^{\natural}|_{\widehat{A}}&\xrightarrow{\sim} \sO_{\widehat{A}}\hat{\otimes}_{\Cp} \widehat{\TSym}^\cdot (\sH_{\Cp})\\
	\widehat{\sP}_{\Cp}|_{\widehat{A}}&\xrightarrow{\sim} \sO_{\widehat{A}}\hat{\otimes}_{\Cp} \widehat{\TSym}^\cdot(\omega_{A^\vee}).
\end{align*}
In particular, the splitting of the Hodge filtration induces a canonical $\sO_{\widehat{A}}$-linear retraction $p\colon \widehat{\sP}_{\Cp}^{\natural}|_{\widehat{A}}\twoheadrightarrow \widehat{\sP}_{\Cp}|_{\widehat{A}} $ of the canonical injection $i\colon \widehat{\sP}_{\Cp}|_{\widehat{A}} \hookrightarrow \widehat{\sP}^{\natural}_{\Cp}|_{\widehat{A}}$.
\item The diagram
\begin{equation}\label{eq:PbigDiagram}
	\xymatrix{
		\sO_{\widehat{A\times A^\vee}} \ar[d]^{d\otimes\id} \ar[r]^-{\cong} &  \widehat{\sP}_{\Cp}|_{\widehat{A}} \ar[r]^{i} & \widehat{\sP}^{\natural}_{\Cp}|_{\widehat{A}} \ar[d]_{\nabla}\\
		\sO_{\widehat{A\times A^\vee}}\otimes\Omega^1_{\widehat{A}} & \widehat{\sP}_{\Cp}|_{\widehat{A}} \otimes\Omega^1_{\widehat{A}} \ar[l]_-{\cong} & \widehat{\sP}^{\natural}_{\Cp}|_{\widehat{A}} \otimes\Omega^1_{\widehat{A}}\ar[l]_{p}
	}
\end{equation}
commutes.
\end{enumerate}
\end{lemma}
\begin{proof}
$(1)$ By \eqref{eq_1splitting_1} and \eqref{eq_1splitting_1}, we have isomorphisms
\begin{align*}
		\sP_{\Cp}^{\natural(1)}|_{\widehat{A}}&\cong \sO_{\widehat{A}}\otimes_{\Cp}(\Cp\oplus \sH_{\Cp})\\
		\sP_{\Cp}^{(1)}|_{\widehat{A}}&\cong \sO_{\widehat{A}}\otimes_{\Cp}(\Cp\oplus \omega_{A^\vee}).
\end{align*}
Over a field of characteristic zero, the co-multiplication maps are isomorphisms
\begin{align*}
	\sP_{\Cp}^{\natural(n)} &\to \TSym^n \sP_{\Cp}^{\natural(1)}\\
	\sP_{\Cp}^{(n)} &\to \TSym^n \sP_{\Cp}^{(1)},
\end{align*} 
and the claim follows by passing to the limit.\\
$(2)$ By applying $\widehat{\TSym}$ it suffices to prove the commutativity of diagram \eqref{eq:PbigDiagram} for $\sP_{\Cp}^{\natural(1)}|_{\widehat{A}}$ instead of $\widehat{\sP}_{\Cp}^{\natural}$. Let us recall that $\sP_{\Cp}^{\natural(1)}|_{\widehat{A}}$ sits in a short exact sequence
\[
	0\rightarrow \pi^*\sH_\Cp\rightarrow \sP^{\natural(1)}_\Cp \rightarrow \sO_{A}\rightarrow 0
\]
which is horizontal if we equip $\pi^*\sH$ and $\sO_{A}$ with the canonical $\Cp$-linear derivation. Since this sequence splits over $\widehat{A}$, the connection $\nabla_{\sP^{\natural(1)}_\Cp}$ is uniquely determined by $\nabla_{\sP^{\natural(1)}_\Cp}(e^{(1)})\in \pi^*\sH_\Cp\otimes\Omega^1_{\widehat{A}}$, where $e^{(1)}$ is the image of $1\otimes(1,0)$ under
\[
	\sP_\Cp^{\natural(1)}|_{\widehat{A}}\cong \sO_{\widehat{A}}\otimes_{\Cp}(\Cp\oplus \sH_\Cp).
\]
The commutativity of the above diagram is now equivalent to the formula
\begin{equation}
	\nabla_{\sP_\Cp^{\natural(1)}}(e^{(1)})\in \pi^*\sH_\Cp(\overline{\Sigma})\otimes\Omega^1_{\widehat{A}}=\ker( \pi^*\sH_\Cp\otimes\Omega^1_{\widehat{A}} \twoheadrightarrow  \pi^*\omega_{A^\vee}\otimes\Omega^1_{\widehat{A}}).
\end{equation}
Before we verify this formula, let us recall that the canonical map
\[
	[p]_\sharp\colon \sH_\Cp\rightarrow \sH_\Cp
\]
acts by multiplication by $p$ on $\sH_\Cp(\Sigma)$ and is invertible on the 'unit root' space $\sH_\Cp(\overline{\Sigma})$. Let us write $\eta_\Sigma$ for the component of
\[
	\nabla_{\sP_\Cp^{\natural(1)}}(e^{(1)})
\]
in $\pi^*\sH_\Cp(\Sigma)\otimes\Omega^1_{\widehat{A}}$. Our aim is to show $\eta_\Sigma=0$. The horizontality of the invariance under isogenies map
\[
	[p]_\sharp \colon\sP_\Cp^{\natural(1)} \rightarrow [p]^*\sP_\Cp^{\natural(1)}
\]
and the fact that $[p]^*\colon \Gamma(A,\Omega^1_{A})\rightarrow \Gamma(A,\Omega^1_{A})$ and $[p]_\sharp \colon \sH_\Cp(\Sigma)\rightarrow \sH_\Cp(\Sigma)$ are both multiplication by $p$ implies the formula
\[
	p^2\eta_\Sigma=p\eta_\Sigma.
\]
Since $\Cp$ is $p$-torsion free this implies $\eta_\Sigma=0$.
\end{proof}

\subsection{Construction of $p$-adic theta functions}
In this section, we will use the isomorphism $\widehat{\sP}|_{\widehat{\cA}}\xrightarrow{\sim} \sO_{\widehat{\cA\times\cA^\vee}}$ from Proposition \ref{prop_trivialization_Poincare} to construct a certain $p$-adic theta function from our Eisenstein--Kronecker classes. 

Let $(\cA/\cR,\Sigma,\omega(\cA),\omega(\cA^\vee),x)$ be as in Notation \ref{notation_padicAbelianVariety}. We will work in this section over $\OCp$. 
\begin{lemma}
Let $y\in \cA(\OCp)$ be in the kernel of an isogeny $\varphi:\cA\to \cB$ whose dual $\varphi^{\vee}$ is \'etale. Let  $T_y:\cA\to \cA$ be the translation with $y$. Then there is a canonical isomorphism 
\begin{equation}\label{eq_translation_invariance}
	T_y^*\widehat{\sP}\isom \widehat{\sP}.
\end{equation}
In particular, on the generic fibre $A$ of $\cA$ over $\C_p$ there is always an isomorphism 
\begin{equation*}
T_y^{*}\wP_{\C_p}\isom \wP_{\C_p}.
\end{equation*}
\end{lemma}
\begin{proof}
By Theorem \ref{thm:functoriality} one has $T_y^*\widehat{\sP}\cong T_y\varphi^*\widehat{\sP}_{\cB} =\varphi^*\widehat{\sP}_{\cB}\cong  \widehat{\sP}$.
\end{proof}
\begin{definition}Let $y\in \cA(\OCp)$ be a $\varphi$-torsion section, i.e. $y\in \ker\varphi$ for an isogeny $\varphi:\cA\to \cB$ and assume the dual $\varphi^{\vee}$ to be \'etale.
Define 
\begin{equation}\label{eq_rho_hat_integral}
	\hat{\varrho}_y\colon T_y^* \widehat{\sP}|_{\widehat{\cA}}\cong \widehat{\sP}|_{\widehat{\cA}}\cong \sO_{\widehat{\cA \times \cA^\vee}}.
\end{equation}
\nomenclature{$\hat{\varrho}_y$}{\nomrefpage}
to be the composition of the isomorphism $T_y^*\widehat{\sP}\isom \widehat{\sP}$ from \eqref{eq_translation_invariance} restricted to $\widehat{\cA}$ with the isomorphism 
$\widehat{\sP}|_{\widehat{\cA}} \cong \sO_{\widehat{\cA\times \cA^\vee}}$ from  Proposition \ref{prop_trivialization_Poincare}. 
\end{definition}
In particular, for any torsion section $y\in \cA(\OCp)$ one has always an isomorphism on the generic fibre:
\begin{equation}\label{eq_rho_hat_generic}
	\hat{\varrho}_y\colon T_y^* \widehat{\sP}_{\Cp}|_{\widehat{A}}\cong \widehat{\sP}_{\Cp}|_{\widehat{A}}\cong \sO_{\widehat{A \times A^\vee}}.
\end{equation}
We will use this definition for the $\frf$-torsion section $x$ and for $\sws\in \cA[\frp_\Sigma^n](\OCp)$. Recall from \eqref{eq:p-isogenies} that the dual $[\frp_\Sigma^n]^{\vee}=[\ol{\frp_\Sigma^n}]$ is \'etale. Further, the  group $\cA[\frp_\Sigma^n](\OCp)$ identifies with the $p^{n}$-torsion $\widehat{\cA}[p^{n}]$ of the formal group $\widehat{\cA}$. Hence $\sws$ induces a translation 
\[
	T_\sws\colon \widehat{\cA}\to \widehat{\cA}
\]
on the formal group $\widehat{\cA}$. 

For later reference, let we need some  basic properties and compatibilities of $\hat{\varrho}_y$. 
Recall from Corollary \ref{cor:coh-log-splitting} the isomorphism 
\begin{equation*}
\varrho_y:y^{*}\wP\isom \sO_{\widehat{\cA}^{\vee}}
\end{equation*}
which holds if $y\in \ker\varphi$ with $\varphi^{\vee}$ \'etale. Further recall from \eqref{eq:moment-G} the moment map (we work over $\OCp$)
\begin{equation*}
\mom_{\sO_{\widehat{\cA}}}:\Gamma(\widehat{\cA},\sO_{\widehat{\cA}})\to \widehat{\TSym}(\omega_{\cA/\OCp}).
\end{equation*}

\begin{lemma}\label{lem_rhox}Let $y\in \ker \varphi(\OCp)\subset \cA(\OCp)$ be a $\varphi$-torsion section. Note that the first two properties are statements about the generic fiber, while the last statement holds integrally:
	\begin{enumerate}
		\item $\mom_{\widehat{A^\vee}}\circ( e^*\hat{\varrho}_y)=\varrho_y$ as maps $y^* \widehat{\sP}_{\Cp}\rightarrow \widehat{\TSym^\bullet(\omega_{A^\vee}})$.
		\item $\hat{\varrho}_y\circ p\circ \nabla\circ i =d_{\hat{A}}\circ \hat{\varrho}_y$ as maps $T_y^*\widehat{\sP}_{\Cp}|_{\widehat{A}}\rightarrow \sO_{\widehat{A\times A^\vee}}\otimes_{\Cp} \omega_{\widehat{A}}$. Here, $i\colon \widehat{\sP}_{\Cp}\subseteq \widehat{\sP_{\Cp}^\natural}$ is the canonical inclusion and $p\colon  \widehat{\sP}_{\Cp}^\natural\twoheadrightarrow \widehat{\sP_{\Cp}}$ the retraction induced by the splitting of the Hodge filtration.
		\item \label{lem_rhox_3} Assume that  $\varphi^{\vee}$ is \'etale. Then for $\sws\in \widehat{\cA}[p^n](\OCp)=\cA[\frp_\Sigma^n](\OCp)$ we have a commutative diagram
		\[
			\xymatrix{
				H^0(\cA,\widehat{\sP}) \ar[r]^{(\cdot)|_{\widehat{\cA}}\circ T_y^*}\ar[d]^{T_\sws^*} & H^0(\widehat{\cA},T_y^*\widehat{\sP}) \ar[r]^-{\hat{\varrho}_{y}} &  H^0(\widehat{\cA\times \cA^\vee},\sO_{\widehat{\cA\times \cA^\vee}}) \ar[d]_{(T_\sws\times \id)^*}\\
				H^0(\cA,T_\sws^*\widehat{\sP}) \ar[r]^{(\cdot)|_{\widehat{\cA}}\circ T_y^*} & H^0(\widehat{\cA},T_{y+\sws}^*\widehat{\sP}) \ar[r]^-{\hat{\varrho}_{y+\sws}} &  H^0(\widehat{\cA\times \cA^\vee},\sO_{\widehat{\cA\times \cA^\vee}}).
			}
		\]
	\end{enumerate}
\end{lemma}
\begin{proof}
The first claim follows immediately from the definitions of $\varrho_y$ and $\hat{\varrho}_y$. The second claim follows from Lemma \ref{lem_sP_triv} using that the canonical isomorphism
\[
	\varphi^* \widehat{\sP}_{\Cp}^\natural\cong \widehat{\sP}^\natural_{\Cp}
\]
is horizontal with respect to $\nabla$. The last claim reduces to the commutativity of the translation isomorphisms
\[
	\xymatrix{
		T_{y+\sws}^*\sP\ar[r]\ar[d] & T_y^*\sP\ar[d]\\
		T_{\sws}^*\sP\ar[r] & \sP.
	}
\]
\end{proof}
Fix an ideal $\frc\subseteq \sO_L$  coprime to $p\frf$ such that $N\frc\in R^{\times}$ and  let $\cD=\ker[\frc]\setminus \{x(\cR)\}$ be as in Definition \ref{def:D-specific} and 
\[
	f\colon \cD(\OCp)\rightarrow \OCp
\]
a $\Gamma$-invariant function on $\cD$. 
Let us write $D\subseteq A$ for the generic fiber of $\cD\subseteq \cA$ and $U_D:=A\smallsetminus D$ for its complement. 

Let again $y\in \cA(\OCp)$ be a $\varphi$-torsion section, with $\varphi^{\vee}$ \'etale. In the following, we will associate a $p$-adic theta function
\[
	\vartheta_\Gamma(f,y)\in H^0(\widehat{\cA\times \cA^\vee},\sO_{\widehat{\cA\times \cA^\vee}})^\Gamma
\]
to the coherent Eisenstein--Kronecker class $\EK_{\Gamma}(f)\in H^{d-1}(\cU_\cD,\Gamma,\widehat{\sP}\otimes \Omega_{\cA}^d)$. 

Let us write $\widehat{\cA}_{y}$ for the completion of $\cA$ at $y$.  The restriction of the equivariant coherent Eisenstein--Kronecker class $\EK_{\Gamma}(f)$ to $\widehat{\cA}_{y}$ gives:
\[
	\EK_{\Gamma}(f)|_{\widehat{\cA}_{y}}\in H^{d-1}(\widehat{\cA}_{y},\Gamma, \widehat{\sP}\otimes \Omega_{\cA}^d|_{\widehat{\cA}_{y}}).
\]
We will use the basis $\omega_{can}(\cA)$ to trivialize $\omega_{\cA}^d\cong\OCp$. Since $\widehat{\cA}_{y}$ is an affine formal scheme, we have
\[
	H^{d-1}(\widehat{\cA}_{y},\Gamma, \widehat{\sP}\otimes \Omega_{\cA}^d|_{\widehat{\cA}_{y}})=H^{d-1}(\Gamma,H^0(\widehat{\cA}_{y}, \widehat{\sP}\otimes \Omega_{\cA}^d|_{\widehat{A}_{y}})).
\]
Pullback along the translation $T_y$ by $y$ gives an isomorphism
\[
	T_y^*\colon H^0(\widehat{\cA}_{y}, \widehat{\sP}\otimes \Omega_{\cA}^d|_{\widehat{\cA}_{y}})\cong H^0(\widehat{\cA}, (T_y^*\widehat{\sP})\otimes \Omega_{\cA}^d|_{\widehat{\cA}}).
\]

\begin{lemma}\label{lem:capproduct}
	Let $\Gamma'\subseteq \Gamma$ be a torsion-free subgroup of finite index and $\xi\in H_{d-1}(\Gamma,\ZZ)$ an element such that $\res(\xi)\in H_{d-1}(\Gamma',\ZZ)$ is a generator. Then the cap-product with $\xi$ induces a canonical homomorphism
	\begin{multline}\label{eq:inclusion}
\cap\xi\colon H^{d-1}(\Gamma,H^0(\widehat{\cA\times \cA^\vee}, \sO_{\widehat{\cA\times \cA^\vee}}\otimes\Omega^d_{\cA}))\rightarrow H^0(\widehat{\cA\times \cA^\vee}, \sO_{\widehat{\cA\times \cA^\vee}}\otimes \TSym^{\ul{1}}\omega_\cA)^\Gamma\\\subseteq  H^0(\widehat{\cA\times \cA^\vee},\sO_{\widehat{\cA\times \cA^\vee}}),
\end{multline}
where the last inclusion uses the isomorphism $\TSym^{\ul{1}} \omega_\cA\cong \OCp$ induced by the canonical basis $\omega_{can}(\cA)$ of $\omega_\cA$.
\end{lemma}
\begin{proof} The only non-trivial statement is the assertion that the map is canonical, i.e. it does not depend on the choice of $\xi\in H_{d-1}(\Gamma,\ZZ)$, but this follows from the corresponding statement in Proposition \ref{prop:orientation} and the canonical inclusion given by the moment map
\[
	H^0(\widehat{\cA\times \cA^\vee},\sO_{\widehat{\cA\times \cA^\vee}}) \hookrightarrow \widehat{\TSym}^\cdot(\omega_{\cA}\oplus \omega_{\cA^\vee})=\prod_{\alpha,\beta}\TSym^\alpha(\omega_{\cA})\otimes \TSym^\beta(\omega_{\cA^\vee}).\qedhere
\]
\end{proof}

\begin{definition}\label{def:theta}Let $\cA/\OCp$, $f\in \OCp[\cD]^{0,\Gamma}$ and $y$ a $\varphi$-torsion section such that $\varphi^{\vee}$ is \'etale. We define the \emph{$p$-adic theta function} 
\begin{equation*}
\vartheta_\Gamma(f,y)\in H^0(\widehat{\cA\times \cA^\vee},\sO_{\widehat{\cA\times \cA^\vee}})
\end{equation*}
\nomenclature{$\vartheta_\Gamma(f,y)$}{\nomrefpage}
associated with $\EK_{\Gamma}(f)$ at $y$ as the image of
	\[
		\hat{\varrho}_yT_y^*(\EK_{\Gamma}(f)|_{\widehat{\cA}_y}) \in H^{d-1}(\Gamma,H^0(\widehat{\cA\times \cA^\vee}, \sO_{\widehat{\cA\times \cA^\vee}}\otimes \Omega^d_{\cA}))
	\]
	under the canonical map \eqref{eq:inclusion} of Lemma \ref{lem:capproduct}. Since $\varphi^{\vee}$ is always  \'etale on the generic fiber, we get in any case a $p$-adic theta function on the generic fiber
\begin{equation*}
\vartheta_\Gamma(f,y)\in H^0(\widehat{A\times A^\vee},\sO_{\widehat{A\times A^\vee}}).
\end{equation*} 
\end{definition}

Finally, let us observe the behaviour of $\vartheta_\Gamma(f,y)$ under varying the subgroup $\Gamma\subseteq\sO_L^\times$:
\begin{lemma}
	If $\Gamma'\subseteq \Gamma$ is a subgroup of finite index then
	\[
		\vartheta_{\Gamma'}(f,y)=[\Gamma:\Gamma']\vartheta_{\Gamma}(f,y).
	\]
\end{lemma}
\begin{proof}
	This follows exactly as in the proof of Proposition \ref{prop:orientation} from the fact that the restriction map
	\[
		\OCp\cong H_{d-1}(\Gamma,\OCp)\otimes H^{d-1}(\Gamma,\OCp)\rightarrow H_{d-1}(\Gamma',\OCp)\otimes H^{d-1}(\Gamma',\OCp)\cong \OCp
	\]
	is multiplication with $[\Gamma:\Gamma']$.
\end{proof}

\subsection{Construction of the $p$-adic Eisenstein measure}
In the previous section, we have constructed under certain assumptions a $p$-adic theta function 
\[
	\vartheta_\Gamma(f,x)\in H^0(\widehat{\cA\times \cA^\vee},\sO_{\widehat{\cA\times \cA^\vee}}).
\]
In this section, we will finally construct the $p$-adic measure interpolating the Eisenstein-Kronecker series $p$-adically from the $p$-adic theta function.

Let $(\cA/\cR,\Sigma,\omega(\cA),\omega(\cA^\vee),x)$ be as in Notation \ref{notation_padicAbelianVariety}. Furthermore, we assume that $(\cA,\cA^\vee)$ is equipped with a $\Gamma_{00}(p^\infty)$-structure $(\theta_p,\theta_p^\vee)$.

\begin{definition}
For $\alpha\in I^+_\Sigma$, $\beta\in I^+_{\overline{\Sigma}}$ let us define differential operators
\[
	\partial(\cA)^\alpha\partial(\cA^\vee)^\beta\colon \sO_{\widehat{\cA\times\cA^\vee}}\rightarrow \OCp
\]
by the formula
\[
	\mom_{\widehat{\cA\times\cA^\vee}}(f)=\left( (\partial(\cA)^\alpha\partial(\cA^\vee)^\beta f)\cdot\omega_{can}(\cA)^{[\alpha]}\omega_{can}(\cA^\vee)^{[\beta]} \right)_{\alpha,\beta}.
\]
Here $\mom_{\widehat{\cA\times\cA^\vee}}$ denotes the moment map of the formal group $\widehat{\cA\times\cA^\vee}$.
\end{definition}

The inclusion $\Zp\subseteq\OCp$ gives a map
\[
	\sO_L\otimes_\ZZ \Zp\subseteq \sO_L\otimes_\ZZ \OCp
\]
Since $\Sigma$ is assumed to be $p$-ordinary, we have $\sO_L\otimes_\ZZ \Zp= \sO_L(\Sigma_p)\oplus \sO_L(\overline{\Sigma}_p)$ and the above map decomposes into
\begin{align}
	\label{eq:coordinates}\rmt &\colon \sO_L(\Sigma_p)\rightarrow \bigoplus_{\sigma\in \Sigma}\OCp,\quad  \rms \colon \sO_L(\overline{\Sigma}_p)\rightarrow \bigoplus_{\ol{\sigma}\in \ol{\Sigma}}\OCp.
\end{align}
\begin{definition}
For $\alpha\in I_\Sigma$, $\beta\in I_{\overline{\Sigma}}$ we will write $\rmt^\alpha\rms^\beta$ for the map
\[
	\rmt^\alpha\rms^\beta\colon\sO_L\otimes_\ZZ \Zp\rightarrow \OCp,\quad (z\otimes x)\mapsto \prod_{\sigma\in \Sigma}(x\sigma(z))^{\alpha_\sigma}\cdot\prod_{\overline{\sigma}\in \overline{\Sigma}}(x\overline{\sigma}(z))^{\beta_{\overline{\sigma}}}. 
\]

\end{definition}

According to Katz \cite{Katz-another-look}, we have the following isomorphism between functions on formal toroidal groups and $p$-adic measures:
\begin{proposition}\label{prop:Mahler}
	There is a $\Gamma$-equivariant isomorphism of rings
	\[
		H^0(\widehat{\cA\times\cA^\vee},\sO_{\widehat{\cA\times\cA^\vee}})\xrightarrow{\sim} \Meas(\sO_L\otimes_\ZZ \Zp,\OCp),\quad g\mapsto \mu_g
	\]
	which is uniquely determined by the integration formula
	\[
		\int_{\sO_L\otimes_\ZZ \Zp} \rmt^\alpha\rms^\beta d\mu_g= \partial(\cA)^\alpha \partial(\cA^\vee)^\beta g.
	\]
\end{proposition}
\begin{proof}
	According to Proposition \ref{prop:padicFG} we have
	\[
		\widehat{\cA\times\cA^\vee}=\Hom_{\Zp}(\sO_L\otimes\Zp, \Gmf).
	\]
	Now we can apply \cite[Theorem 1]{Katz-another-look}. The $\Gamma$-equivariance follows from the functoriality of Katz' construction of $p$-adic measures.
\end{proof}
Fix an ideal $\frc\subseteq \sO_L$  coprime to $p\frf$ such that $N\frc\in R^{\times}$ so that $[\frc]$ and $[\frc]^{\vee}$ are \'etale.  As in Definition \ref{def:D-specific} we let $\cD=\ker[\frc]\setminus \{x(\cR)\}$ and 
\[
	f\colon \cD(\OCp)\rightarrow \OCp
\]
a $\Gamma$-invariant function on $\cD$. 
\begin{definition}  The unique $p$-adic measure 
\begin{equation*}
\mu_{Eis}(f,x)\in \Meas(\sO_L\otimes_\ZZ \Zp,\OCp)
\end{equation*}
\nomenclature{$\mu_{Eis}(f,x)$}{\nomrefpage}
corresponding to $\vartheta_{\Gamma}(f,x)$ under $	H^0(\widehat{\cA\times\cA^\vee}, \sO_{\widehat{\cA\times\cA^\vee}})\cong \Meas(\sO_L\otimes_\ZZ \Zp,\OCp)$
	is called \emph{the $p$-adic Eisenstein measure} associated to the tuple $(\cA/\cR,\Sigma,\omega(\cA),\omega(\cA^\vee),x)$. If we want to emphasize the dependence on $\Gamma$, we will also write $\mu_{Eis,\Gamma}(f,x)$.
\end{definition}
\begin{theorem}[$p$-adic Eisenstein measure]\label{thm_Eisenstein_measure}
The $p$-adic Eisenstein measure has the following interpolation property: For $\alpha\in I^+_\Sigma$ and $\beta\in I^+_{\overline{\Sigma}}$ we have
	\[
		\frac{1}{\Omega_p^{\alpha+\underline{1}}\Omega_p^{\vee\beta}}\int_{\sO_L\otimes_\ZZ \Zp} \rmt^\alpha\rms^\beta d\mu_{Eis}(f,x)=\Eis_\Gamma^{\beta,\alpha}(f,x)(\omega(\cA)^{[\alpha+\underline{1}]},\omega(\cA^\vee)^{[\beta]}).
	\]
\end{theorem}
\begin{proof}
Let us write $\mom^{\alpha,\beta}_{\widehat{\cA\times\cA^\vee}}$ respectively $\mom^{\beta}_{\widehat{\cA}^\vee}$ for the $(\alpha,\beta)$- resp. $\beta$-component of the moment map. With this notation, we have by the defining property of $\mu_{Eis}(f,x)$ the integration formula
\[
	\int_{\sO_L\otimes_\ZZ \Zp} \rmt^\alpha\rms^\beta d\mu_{Eis}(f,x)\omega_{can}(\cA)^{[\alpha+\underline{1}]}\omega_{can}(\cA)^{[\beta]} =\mom^{\alpha,\beta}_{\widehat{\cA\times\cA^\vee}}(\vartheta_\Gamma(f,x)  ).
\]
For the computation of $\mom^{\alpha,\beta}_{\widehat{\cA\times\cA^\vee}}( \vartheta_{\Gamma}(f,x))$, we can pass without loss of generality to the generic fiber $A$ of $\cA$. Using the inclusions $\omega_\cA\subseteq \omega_A$ and $\omega_{\cA^\vee}\subseteq \omega_{A^\vee}$, we may view $\omega(\cA)$ and $\omega(\cA^\vee)$ as bases of $\omega_A$ and $\omega_{A^\vee}$. Using Lemma \ref{lem_rhox} and the definition of $\vartheta_{\Gamma}(f,x)$ and $\Eis_\Gamma^{\beta,\alpha}(f,x)$ we get:
\begin{align*}
	\mom^{\alpha,\beta}_{\widehat{\cA\times\cA^\vee}}( \hat{\varrho}_xT_x^*(\EK_{\Gamma}(f)|_{\widehat{\cA}_x}) )&=\mom^{\alpha,\beta}_{\widehat{A\times A^\vee}}( \hat{\varrho}_xT_x^*(\EK_{\Gamma}(f)|_{\widehat{A}_x}) )\\
	&=\mom^{\beta}_{\widehat{A}^\vee}\left( \partial(\cA)^\alpha \hat{\varrho}_xT_x^*(\EK_{\Gamma}(f)|_{\widehat{A}_x}) \right)\cdot \omega_{can}(\cA)^{[\alpha]}\\
	&=\mom^{\beta}_{\widehat{A}^\vee}\left(\hat{\varrho}_xT_x^*(\nabla^{[\alpha]}\EK_{\Gamma}(f)|_{\widehat{A}_x}) \right)=\Eis_\Gamma^{\beta,\alpha}(f,x).
\end{align*}
The formulas
\[
	\omega(\cA)^{[\alpha+\ul{1}]}=\Omega_p^{\alpha+\ul{1}}\omega_{can}(\cA)^{[\alpha+\ul{1}]},\quad \omega(\cA^\vee)^{[\beta]}=\Omega^{\vee,\beta}_p\omega_{can}(\cA^\vee)^{[\beta]}
\]
allow us to re-write $\Eis_\Gamma^{\beta,\alpha}(f,x)$ in terms of the basis $\omega_{can}(\cA)$, $\omega_{can}(\cA^\vee)$
\begin{align*}
	&\Eis_\Gamma^{\beta,\alpha}(f,x)=\Eis_\Gamma^{\beta,\alpha}(f,x)(\omega(\cA)^{[\alpha+\ul{1}]},\omega(\cA^\vee)^{[\beta]}) \cdot \omega(\cA)^{[\alpha+\ul{1}]}\otimes\omega(\cA^\vee)^{[\beta]}=\\
	=&\Eis_\Gamma^{\beta,\alpha}(f,x)(\omega(\cA)^{[\alpha+\ul{1}]},\omega(\cA^\vee)^{[\beta]})\cdot  \Omega_p^{\alpha+\ul{1}}\Omega^{\vee\beta}_p\omega_{can}(\cA)^{[\alpha+\ul{1}]}\otimes\omega_{can}(\cA^\vee)^{[\beta]}.
\end{align*}
Combining the above computations shows the Theorem.
\end{proof}

\subsection{Translation operators}
While we already know the moments of the $p$-adic Eisen\-stein-Kronecker measure constructed in the previous section, we will need furthermore an explicit description of the integration of locally constant functions. This question is immediately related to the question of finding an explicit formula for the value of the corresponding $p$-adic theta functions at $p$-power torsion sections of $\widehat{\cA\times \cA^\vee}$, which is addressed in this section.

We continue with the notation of the previous section. In particular, we consider an abelian variery $(\cA/\cR,\omega(\cA),\omega(\cA^\vee),x)$  as in Notation \ref{notation_padicAbelianVariety}, and we will write $A$ for the generic fiber of $\cA$. Recall that $[\frf]^{\vee}$ is \'etale. The aim of this section is to compute the value of certain translations of $\vartheta_\Gamma({f,x})$ by infinitesimal torsion sections. Any infinitesimal torsion section $\sws\in \cA [\frp_\Sigma^n](\OCp)=\widehat{\cA}[p^n](\OCp)$ induces translation maps
\[
	T_\sws\colon \widehat{\cA}\to \widehat{\cA},
\] 
on the corresponding formal groups. These maps can be described explicitly in terms of the coordinate rings (induced by $\omega_{can}$) as the maps
\[
	\OCp\llbracket \ul{T} \rrbracket\to \OCp\llbracket \ul{T} \rrbracket,\quad   \ul{T}\mapsto \ul{T}+_{\widehat{\cA}}\sws,
\]
Similarly, a torsion section $\swt\in \cA^\vee [\ol{\frp}_\Sigma^n]=\widehat{\cA}^\vee[p^n]$ induces maps
\[
	T_\swt\colon \widehat{\cA}^\vee\to \widehat{\cA}^\vee.
\] 
The following result allows us to compute such translations of the $p$-adic theta function explicitly. Let us recall the notation $\cA_{(n)}=\cA(\overline{\frp}_\Sigma^n)$ from Definition \ref{def:A-n} and denote by $A_{(n)}$ the generic fiber of $\cA_{(n)}$. Let us denote by
\begin{align}\label{eq:D-prime}
\cD':=[\overline{\frp}_\Sigma^n]^{-1}(\cD)&&\mbox{and}&& f'\colon \cD'(\sO_{\Cp})\to \sO_{\Cp} R[\cD']^{(0,\Gamma)}
\end{align}
the pull-back of $\cD$ and $f\in \OCp[\cD]^{0,\Gamma}$ by the isogeny $[\overline{\frp}_\Sigma^n]:\cA_{(n)}\to \cA$. 

\begin{theorem}[Translation]\label{prop_translation_formula} With the above notation let $(\sws,\swt)\in \cA [\frp_\Sigma^n]\times \cA^\vee [\ol{\frp}_\Sigma^n]$ be $\Gamma$-invariant torsion sections on $\cA\times \cA^\vee$ and $x \in \cA[\frf ]$ our fixed $\frf$-torsion section (recall $p\nmid \frf$). Let $x'$ resp. $\sws'$ be the unique $\frf$-torsion resp. $\frp_\Sigma^n$-torsion section of $\cA_{(n)}$ mapping to $x$ resp. $\sws$ under $[\overline{\frp}_\Sigma^n]\colon \cA_{(n)}\to \cA$. Let  $\cD'$ and $f'$ be as in \eqref{eq:D-prime}.
Then
\begin{equation}
	(T_\sws\times T_\swt)^*\vartheta_\Gamma({f,x})=\sum_{s'\in \cA_{(n)}[\overline{\frp}^n_\Sigma](\OCp)}\langle s',\swt \rangle_{[\overline{\frp}_\Sigma^n]}\cdot [\overline{\frp}_\Sigma^n]_\#\vartheta_{\Gamma,A_{(n)}}({f',x'+s'+\sws'}),
\end{equation}
where $[\overline{\frp}_\Sigma^n]_\#$ denotes the map
\[
	\xymatrix{
		\Gamma(\widehat{A_{(n)}}\times \widehat{A_{(n)}^\vee }, \cO_{\widehat{A_{(n)}}\times \widehat{A_{(n)}^\vee }}) & &
		 \Gamma(\widehat{A}\times \widehat{A_{(n)}^\vee }, \cO_{\widehat{A}\times \widehat{A_{(n)}^\vee }}) \ar[ll]_-{([\overline{\frp}_\Sigma^n]\times \id)^*}^{\sim} \ar[rr]^{(\id\times [\ol{\frp}_\Sigma^n])^*}& &
		  \Gamma(\widehat{A}\times \widehat{A^\vee }, \cO_{\widehat{A}\times \widehat{A^\vee }}),
	}
\]
and
\[
	\langle\cdot,\cdot \rangle_{[\overline{\frp}_\Sigma^n]}\colon \cA_{(n)}[\overline{\frp}_\Sigma^n]\times \cA^\vee[\ol{\frp}_\Sigma^n]\rightarrow \mu_{p^n}
\]
is the Weil pairing.
\end{theorem}

%
Before giving the proof, let us make the following remark regarding the integrality of the above construction: A priori the right hand side of the equation in the above Theorem can only be defined rationally. Indeed, the dual of $[\overline{\frp}^n_\Sigma]$ is not \'etale, hence $\vartheta_{\Gamma,A_{(n)}}({f',x'+s'+\sws'})$ is not defined integrally for $s'\in \cA_{(n)}[\overline{\frp}^n_\Sigma]$. Nevertheless, $[\overline{\frp}^n_\Sigma]^\vee$ becomes \'etale over the generic fiber $\Cp$ of $\OCp$. On the other hand, since $x\in \cA[\frf]$ with $[\frf]^\vee$ \'etale, $(T_\sws\times T_\swt)^*\vartheta_\Gamma({f,x})$ is defined integrally, i.e.
	\[
		(T_\sws\times T_\swt)^*\vartheta_\Gamma({f,x})\in H^0(\widehat{\cA\times\cA^\vee},\sO_{\widehat{\cA\times\cA^\vee}}).
	\]
	In particular, the equation
	\[
		(T_\sws\times T_\swt)^*\vartheta_\Gamma({f,x})=\sum_{s'\in \cA_{(n)}[\overline{\frp}^n_\Sigma](\OCp)}\langle s',\swt \rangle_{[\overline{\frp}_\Sigma^n]}\cdot [\overline{\frp}_\Sigma^n]_\#\vartheta_{\Gamma,A_{(n)}}({f',x'+s'+\sws'}),
	\]
	shows a posteriori the integrality of the right hand side.

\begin{proof}[Proof of Theorem \ref{prop_translation_formula}]
By Lemma \ref{lem_rhox} \eqref{lem_rhox_3}, we have
\[
	(T_\sws\times\id)^*\vartheta_\Gamma(f,x)\stackrel{Def.}{=}  (T_{\sws}\times \id)^*\hat{\varrho}_{x}T^*_{x}(\EK_{\Gamma}(f)|_{\widehat{\cA}_{x}})=\hat{\varrho}_{x+\sws}T^*_{x+\sws}(\EK_{\Gamma}(f)|_{\widehat{\cA}_{x+\sws}})=\vartheta_\Gamma(f,x+\sws).
\]
This formula allows us to assume without loss of generality that $\sws$ is the trivial torsion section. 

Let us now briefly outline the strategy of the rest of the proof. We will define global translation operators on $H^{d-1}(\cU_\cD,\Gamma,\widehat{\sP})$ corresponding to $(\id\times T_\swt)^*$ on the formal group. The advantage of the global translation operators is that we are able to compare $\EK_{\Gamma}(f)$ to its image under the translation operator by a residue computation. The restriction of the global translation operator is compatible with the translation on the formal group. Since $\vartheta_\Gamma(f,x)$ is obtained by such a restriction from $\EK_{\Gamma}(f)$, the result will follow.

As $n$ is fixed, we will write $\cA':=\cA_{(n)}$ to simplify the notation in the proof. 
The translation $T_\swt\colon \cA^\vee\rightarrow \cA^\vee$ restricts to the translation of the formal group $T_\swt\colon \widehat{\cA}^\vee\rightarrow \widehat{\cA}^\vee$. 
The adjunction between $(\id\times T_\swt)^*$ and $(\id\times T_\swt)_*$ induces a map
\begin{equation}\label{eq_translation_1}
	\widehat{\sP}\rightarrow \pr_{\cA,*}(\id\times T_\swt)_*(\id\times T_\swt)^*\left( \sP|_{\cA\times\widehat{\cA}^\vee}\right).
\end{equation}
Let us denote the Poincar\'e bundle on $\cA'\times \cA'^\vee$ by $\sP'$. The universal property of the Poincar\'e bundle gives us an isomorphism
\[
	([\overline{\frp}^n_\Sigma]\times \id_{\cA^\vee})^* \sP \cong (\id_{\cA'}\times [\ol{\frp}^n_\Sigma]^\vee)^*\sP'.
\]
Using this isomorphism, we get
\begin{align}
	\notag &\pr_{\cA',*}(([\overline{\frp}^n_\Sigma]\times T_\swt)^* \sP|_{\cA'\times \widehat{\cA}^\vee} )\cong \pr_{\cA',*}((\id_{\cA'}\times [\ol{\frp}^n_\Sigma]^\vee)^*\sP'|_{\cA'\times \widehat{\cA}^\vee} )\cong\\
	\label{eq_translation_2} &\cong  \pr_{\cA',*}(([\overline{\frp}^n_\Sigma]\times \id_{\widehat{\cA}^\vee})^* \sP|_{\cA'\times \widehat{\cA}^\vee} )=[\overline{\frp}^n_\Sigma]^*\widehat{\sP}.
\end{align}
 Restricting the isogeny $[\overline{\frp}^n_\Sigma]\colon \cA'\rightarrow \cA$ to $\widehat{\cA}'_{x'}$ gives an isomorphism $[\overline{\frp}^n_\Sigma]\colon \widehat{\cA}'_{x'}\xrightarrow{\sim} \widehat{\cA}_x$. By composing \eqref{eq_translation_1} and \eqref{eq_translation_2}, we get
\begin{align*}
	\hat{U}_{\swt}^n\colon H^0(\widehat{\cA}_x,\widehat{\sP})\xrightarrow{\eqref{eq_translation_1}} &H^0(\widehat{\cA}_{x},\pr_{\cA,*}(\id\times T_\swt)^*\left( \sP|_{\cA\times\widehat{\cA}^\vee} \right))\\
	&\cong H^0(\widehat{\cA}'_{x'},\pr_{\cA,*}( [\overline{\frp}^n_\Sigma] \times T_\swt)^*\left( \sP|_{\cA'\times\widehat{\cA}^\vee} \right)) \\
	\xrightarrow{\eqref{eq_translation_2}} &H^0(\widehat{\cA}'_{x'},[\overline{\frp}^n_\Sigma] ^*\widehat{\sP}),
\end{align*}
where the unlabeled isomorphism is induced by $[\overline{\frp}^n_\Sigma] \colon \widehat{\cA}'_{x'}\cong \widehat{\cA}_x$. Unwinding of definitions shows that this isomorphism is compatible with translation on $\widehat{\cA}^\vee$, i.e. the following diagram commutes:
\[
	\xymatrix{
		H^0(\widehat{\cA}_x,\widehat{\sP})\ar[r]^{\hat{U}_{\swt}^n}\ar[d]^{\hat{\varrho}_x} & H^0(\widehat{\cA}'_{x'},[\overline{\frp}^n_\Sigma] ^*\widehat{\sP}) & H^0(\widehat{\cA}_x,\widehat{\sP})\ar[d]^{\hat{\varrho}_x}\ar[l]_{\cong}\\
		H^0(\widehat{\cA\times \cA^\vee}, \sO_{\widehat{\cA\times \cA^\vee}})\ar[rr]^{(\id\times T_\swt)^*} & & H^0(\widehat{\cA\times \cA^\vee}, \sO_{\widehat{\cA\times \cA^\vee}}).
	}
\]
Recall that $\cD'\subseteq \cA'$ denotes the pre-image of $\cD=\cA[\frc]\smallsetminus\{x(\OCp)\}\subseteq \cA$ under $[\overline{\frp}^n_\Sigma]\colon \cA'\rightarrow \cA$ and $\cU'_{\cD}:=\cA'\smallsetminus \cD'$ will denote its complement. The generic fibers are denoted by $D'$ and $U'_{D}$ respectively.  Note, that $\cD'$ is contained in the \'etale group scheme $\cA'[\frc\overline{\frp}^n_\Sigma]\isom \cA'[\frc]\oplus\cA'[\overline{\frp}^n_\Sigma]$.
For any (formal) scheme over $\OCp$, we will write $\eta$ for the inclusion of its generic fiber. In particular, we have the maps
\[
	\eta\colon \widehat{A}_x\to\widehat{\cA}_x,\quad \eta \colon D'\to \cD'.
\]
We claim that it suffices to prove the formula
\begin{equation}\label{eq_reduction1}
	\eta^* \hat{U}^n_\swt(\EK_{\Gamma}(f)|_{\widehat{\cA}_x})=[\overline{\frp}^n_\Sigma]_{\#}\EK_{\Gamma,{A'}}(\eta^*\hat{f})|_{\widehat{A'}_{x'}}
\end{equation}
with
\[
	\hat{f}\colon \cD'(\OCp) \rightarrow \OCp,\quad \hat{f}(c,s'):=f'(c,s')\cdot\langle s',\swt  \rangle_{[\overline{\frp}^n_\Sigma]}.
\]
Indeed, if we consider the image of \eqref{eq_reduction1} under the map
\[
	\xymatrix{
		H^0(\widehat{\cA}'_{x'},[\overline{\frp}^n_\Sigma] ^*\widehat{\sP}) & H^0(\widehat{\cA}_x,\widehat{\sP})\ar[l]_-{\cong} \ar[r]^-{\widehat{\rho}_x} & H^0(\widehat{\cA\times \cA^\vee}, \sO_{\widehat{\cA\times \cA^\vee}}),
	}
\]
then the left hand side maps to  $(\id \times T_\swt)^*\vartheta_\Gamma(f,x)$, while the right hand side gives \begin{equation*}
\sum_{s'\in \cA'[\overline{\frp}^n_\Sigma](\OCp)}\langle s',\swt \rangle_{[\overline{\frp}^n_\Sigma]} \cdot [\overline{\frp}_\Sigma^n]_{\#}\vartheta_{\Gamma,A'}({f',x'+s'}).
\end{equation*}
The reason why we have to pass to the generic fiber is that $\hat{f}$ is supported on the divisor $\cD'$ which is contained in the kernel of an \'etale isogeny, but the dual of this isogeny is not \'etale, so $\EK_{\Gamma,{A'}}(\eta^*\hat{f})$ can not be defined integrally. Of course, this isogeny become \'etale over the generic fiber. In order to prove equation \eqref{eq_reduction1}, we will extend the translation operators $\hat{U}_\swt^n$ to globally (on $\cU_\cD$) defined translation operators and prove the equality by a residue computation. 

These global translation operators are defined as follows:
\[
	U_\swt^n\colon H^{d-1}(\cU_\cD,\Gamma,\widehat{\sP})\xrightarrow{\eqref{eq_translation_2}\circ[\overline{\frp}^n_\Sigma]^*\circ \eqref{eq_translation_1}} H^{d-1}(\cU'_{\cD},\Gamma,[\overline{\frp}^n_\Sigma]^*\widehat{\sP}).
\]
By construction, $U_\swt^n$ is compatible with $\hat{U}_\swt^n$,
\[
	\xymatrix{
		H^{d-1}(\cU_\cD,\Gamma,\widehat{\sP})\ar[r]^-{U_\swt^n}\ar[d] & H^{d-1}(\cU'_{\cD},\Gamma,[\overline{\frp}^n_\Sigma]^*\widehat{\sP})\ar[d]\\
		H^{d-1}(\Gamma, H^0(\widehat{\cA}_x,\widehat{\sP}))\ar[r]^-{\hat{U}_\swt^n} & H^{d-1}(\Gamma, H^0(\widehat{\cA}'_{x'},[\overline{\frp}^n_\Sigma]^*\widehat{\sP}))
	}
\]
where the vertical maps come from restriction to $\widehat{A}_x$ respectively $\widehat{A}'_{x'}$. Thus, it remains to show
\[
	\eta^*U^n_\swt(\EK_{\Gamma}(f))=\EK_{\Gamma,{A'}}(\eta^*\hat{f}).
\]
The translation morphisms $U^n_\swt$ fit into a commutative diagram
\[
	\xymatrix{
		H^{d-1}(\cU_\cD,\Gamma,\widehat{\sP})\ar[r]^-{U^n_\swt}\ar[d]^{\res} & H^{d-1}(\cU'_\cD,\Gamma,[\overline{\frp}^n_\Sigma]^*\widehat{\sP}\ar[d]^{\res})\\
		H^{d}_{\cD}(\cA,\Gamma,\widehat{\sP})\ar[r] & H^{d}_{\cD'}(\cA',\Gamma,[\overline{\frp}^n_\Sigma]^*\widehat{\sP})\\
		H^{0}(\cD,\widehat{\sP}|_{\cD})^\Gamma\ar[r]^-{U^{res}_\swt}\ar[u]_{incl} & H^{0}(\cD',[\overline{\frp}^n_\Sigma]^*(\widehat{\sP}|_{\cD}))^\Gamma\ar[u]_{incl}
	}
\]
where all horizontal maps are induced by $\eqref{eq_translation_2}\circ[\overline{\frp}^n_\Sigma]^*\circ \eqref{eq_translation_1}$ and the vertical maps are the residue maps, respectively the inclusions defined in Section \ref{section-proof-of-thm} (the construction of this inclusion only uses that $\cD$ and $\cD'$ are \'etale). Since $\EK_{\Gamma,{A'}}(\eta^*\hat{f})$ is uniquely determined by its image under the residue map, it suffices to show
\[
	\eta^*U^{res}_\swt(i(f))=i'(\eta^*\hat{f}).
\]
where $i$ resp. $i'$ denotes the inclusions
\begin{align*}
	i\colon H^0(\cD,\sO_{\cD})&\rightarrow H^0(\cD,\widehat{\sP}|_{\cD})\\
	i'\colon H^0(D',\sO_{D'})&\rightarrow H^0(D',[\overline{\frp}^n_\Sigma]^*\widehat{\sP}_{\Cp}|_{D'}),
\end{align*}
induced by the splitting principle, see Corollary \ref{cor:coh-log-splitting}. The inclusions $i$ and $i'$ have canonical retractions $r$ and $r'$ induced by the canonical projections
\[
	r\colon H^0(\cD,\widehat{\sP}|_{\cD})\rightarrow H^0(\cD,\sO_{\cD})
\]
respectively
\[
	r'\colon H^0(\cD',[\overline{\frp}^n_\Sigma]^*(\widehat{\sP}|_{\cD}))\rightarrow H^0(\cD',\sO_{\cD'})
\]
and it suffices to show
\[
	r'(U^{res}_\swt(i(f)))=\hat{f}.
\]
The commutative diagram
\[
	\xymatrix{
		H^{0}(\cA[\frc],\widehat{\sP}|_{\cA[\frc]})^\Gamma\ar[rr]^-{(\cdot )|_{\cA'[\frc]}\circ U^{res}_\swt} & & H^{0}(\cA'[\frc],[\overline{\frp}^n_\Sigma]^*\widehat{\sP}|_{\cA'[\frc]})^\Gamma \ar[d]^{r'}\\
		H^{0}(\cA[\frc],\sO_{\cA[\frc]})^\Gamma\ar[rr]^-{[\overline{\frp}_\Sigma^n]^*}\ar[u]_{i} & & H^{0}(\cA'[\frc],\sO_{\cA'[\frc]})^\Gamma
	}
\]
shows the formula
\begin{equation}\label{eq_Ures_c}
	r'(U^{res}_\swt(i(f)))(c,0)=f([\overline{\frp}^n_\Sigma]c)=\hat{f}(c,0),
\end{equation}
and it remains to show the following claim:\\
\textit{Claim: For $s'\in \cA'[\overline{\frp}^n_\Sigma]$ we have $T_s^*r'(U_\swt^{res}(i(f)))=\langle s',\swt \rangle_{[\overline{\frp}^n_\Sigma]}r'(U_\swt^{res}(i(f)))$}.\\
\textit{Proof of the Claim:} Let us first recall the definition of  the Weil pairing: Let $\varphi\colon \cA'\rightarrow \cA$ be an isogeny and $\cL$ a line bundle defining a torsion section $[\cL]\in \ker\varphi^\vee$ in $\cA^\vee$. Since  $[\cL]$ is in the kernel of $\varphi^\vee$, there is an isomorphism $\alpha\colon \sO_{\cA'} \cong \varphi^*\cL$. For given $s'\in \ker \varphi$ let us consider the $\sO_{\cA'}$-linear isomorphism
\[
	\sO_{\cA'}\xrightarrow{\alpha} \varphi^*\cL=T_{s'}^*\varphi^*\cL\xrightarrow{T_{s'}^*\alpha^{-1}} \sO_{\cA'}.
\]
It's $\deg\varphi$-th power is the identity, so it is given by a $\deg\varphi$-th root of unity. This root of unity is $\langle s',[\cL]\rangle_{\varphi}$. If we apply this to $\cL=(\id\times \swt)^*\sP$ we get a commutative diagram
\begin{equation}\label{eq_comm_Oda}
	\xymatrix{
		([\overline{\frp}^n_\Sigma]\times \swt)^*(\sP|_{\cD\times \cA^\vee})\ar[d]\ar[r] & \sO_{\cD'}\ar[d]^{\langle s',\swt\rangle_{[\overline{\frp}^n_\Sigma]}}\\
		T_{s'}^*([\overline{\frp}^n_\Sigma]\times \swt)^*(\sP|_{\cD\times \cA^\vee})\ar[r] & \sO_{\cD'}.
	}
\end{equation}
Unwinding the definitions, we see that $r'(U^{res}_\swt(i(f)))$ is the image of $([\overline{\frp}^n_\Sigma]\times \swt)^*(i(f))$ under
\[
	H^0(\cD',([\overline{\frp}^n_\Sigma]\times \swt)^*(\sP|_{\cD\times \cA^\vee})) \rightarrow H^0(\cD',\sO_{\cD'})
\]
while $T_{s'}^*r'(U^{res}_\swt(i(f))) $ is the image of $T^*_{s'}([\overline{\frp}^n_\Sigma]\times \swt)^*(i(f))$ under
\[
	H^0(\cD',T_{s'}^*([\overline{\frp}^n_\Sigma]\times \swt)^*(\sP|_{\cD\times \cA^\vee})) \rightarrow H^0(\cD',\sO_{\cD'}).
\]
The Claim follows now from the commutativity of the diagram \eqref{eq_comm_Oda}.
\end{proof}

In the presence of a $\Gamma_{00}(p^\infty)$-structure on $(\cA,\cA^\vee)$, we have by Proposition \ref{prop:GammaStructures} and 
\eqref{eq:cartier-duals} for $n\geq 1$ isomorphisms
\begin{align}\label{eq:Gamma00-Cartier-duals}
	\sO_L(\Sigma)/p^n &\cong (A[\frp_\Sigma^n])^t \cong A_{(n)}^\vee[\overline{\frp}_\Sigma^n], \quad
	\sO_L(\overline{\Sigma})/p^n \cong (A^\vee[\frp_\Sigma^n])^t \cong A_{(n)}[\overline{\frp}_\Sigma^n].
\end{align}
This allows us to define the partial Fourier transform for functions defined on $\sO_L/p^n$ as follows:

\begin{definition}\label{def:Fourier-trafo} Let $\cA/\cR$ be an abelian scheme as in Notation \ref{notation_padicAbelianVariety} and $A$ the generic fibre over $\C_p$.
For a function 
\[
\rho\colon A_{(n)}^\vee[\frp^n_\Sigma]\times A_{(n)}[\ol{\frp}^n_\Sigma]  \rightarrow \ol{\Q}
\]
 let us define the \emph{partial Fourier transform} 
 \[
 P_A (\rho)\colon A_{(n)}[p^n]=A_{(n)}[\frp^n_\Sigma]\times A_{(n)}[\ol{\frp}^n_\Sigma] \rightarrow \ol{\Q}
 \]
 \nomenclature{$P_A (\rho)$}{\nomrefpage}
  by
\[
	P_A(\rho)(\tilde{\sws},\tilde{s}):=\frac{1}{p^{dn}}\sum_{t\in A_{(n)}^\vee[\frp_\Sigma^n]}\langle [\overline{\frp}_\Sigma^n](\tilde{\sws}),t \rangle_{[\frp^n_\Sigma]}^{-1} \rho(t,\tilde{s}),
\]
where $(\tilde{\sws},\tilde{s})\in A_{(n)}[\frp_\Sigma^n]\times A_{(n)}[\ol{\frp}_\Sigma^n]=A_{(n)}[p^n]$. If $(\cA,\cA^\vee)$ is equipped with a $\Gamma_{00}(p^\infty)$-structure, then the isomorphisms \eqref{eq:Gamma00-Cartier-duals} allow us to view the partial Fourier transform as a map
  \[
  		P_A\colon \Map(\sO_L/p^n, \ol{\Q}) \xrightarrow{\sim} \Map(A_{(n)}^\vee[\frp^n_\Sigma]\times A_{(n)}[\ol{\frp}^n_\Sigma],\ol{\Q})\xrightarrow{P_A}\Map(A_{(n)}[p^n],\ol{\Q}),
  \]
  where the first map is induced by the isomorphism \eqref{eq:Gamma00-Cartier-duals}.
\end{definition}

As an immediate Corollary of Theorem \ref{prop_translation_formula}, we deduce the following formula for integration of locally constant functions on the Tate module:
\begin{corollary}\label{cor_rho_integration}
Let $(\cA,\omega(\cA),\omega(\cA^\vee),x)$ be as in Notation \ref{notation_padicAbelianVariety} with a fixed $\Gamma_{00}(p^\infty)$-structure $(\theta_p,\theta_p^\vee)$. Let $\rho\colon \sO_L/p^n \sO_L\to \ol{\Q}$ be a function on $\sO_L/p^n \sO_L$ which we consider as a locally constant function $\rho \colon \sO_L\otimes_\ZZ \Zp\to \C_p$ via the embedding $\ol{\Q}\subset \C_p$. Recall the notation $A_{(n)}=A(\overline{\frp}_\Sigma^n)$ and denote by $\omega(A_{(n)})$ and $\omega(A_{(n)}^\vee)$ the unique basis of differential forms such that $\omega(A_{(n)})=[\overline{\frp}_\Sigma^n]^*(\omega(A))$ and $\omega(A^\vee)=([\overline{\frp}_\Sigma^n]^\vee)^*(\omega(A_{(n)}^\vee))$, then
	\begin{multline*}
		\frac{1}{\Omega_p^{\alpha+\underline{1}}\Omega^{\vee\beta}_p}\int_{\sO_L\otimes\Zp} \rmt^\alpha \rms^\beta \rho(\rmt,\rms) d\mu_{Eis}(f,x)\\
		=\sum_{\tilde{s}\in A_{(n)}[p^n]}P_{A}(\rho)(\tilde{s})\Eis_{\Gamma,A_{(n)}}^{\beta,\alpha}(f',x'+\tilde{s})(\omega(A_{(n)})^{[\alpha+\underline{1}]},\omega(A_{(n)}^\vee)^{[\beta]}),
	\end{multline*}
	where $f'$ and $x'$ denote lifts of the corresponding objects on $\cA$ along the isogeny $[\ol{\frp}_\Sigma^n]\colon\cA_{(n)}\to \cA$ as in Theorem \ref{prop_translation_formula}.
\end{corollary}
\begin{proof} Let  $\sws\in \cA [\frp_\Sigma^n]$ and $\swt\in \cA^\vee [\ol{\frp}_\Sigma^n]$.
The isomorphism \eqref{eq:Gamma00-Cartier-duals} allows us to view $\langle \sws, \cdot\rangle_{[\frp^n_\Sigma]}\colon A^\vee_{(n)}[\frp^n_\Sigma]\to \mu_{p^n}$ and $\langle  \cdot,\swt\rangle_{[\ol{\frp}^n_\Sigma]}\colon A_{(n)}[\ol{\frp}^n_\Sigma]\to \mu_{p^n}$ as functions $\sO_L/p^n\sO_L\to \mu_{p^n}\subseteq \Cp$. It suffices to prove the Corollary for $\rho=\langle \sws, \cdot\rangle_{[\frp^n_\Sigma]}\langle  \cdot,\swt\rangle_{[\ol{\frp}^n_\Sigma]}$. Note that the  measure $\langle \sws, \cdot\rangle_{[\frp^n_\Sigma]}\langle  \cdot,\swt\rangle_{[\ol{\frp}^n_\Sigma]} \mu_{Eis}(f,x)$ correspond under the Mahler isomorphism
	\[
		H^0(\widehat{\cA\times\cA^\vee},\sO_{\widehat{\cA\times\cA^\vee}})\xrightarrow{\sim} \Meas(\sO_L\otimes_\ZZ \Zp,\OCp),\quad g\mapsto \mu_g
	\]
	to the function $(T_\sws\times T_\swt)^*\vartheta_\Gamma(f,x)$. Let us denote by $\sws'\in A_{(n)}[\frp^n_\Sigma]$ the unique lift of $\sws \in A[\frp^n_\Sigma]$ along $[\ol{\frp}^n_\Sigma]$.  Using Theorem \ref{prop_translation_formula}, we compute:
	\begin{align*}
		\int_{\sO_L\otimes\Zp} &\rmt^\alpha \rms^\beta \rho(\rmt,\rms) d\mu_{Eis}(f,x) =\partial(A)^\alpha\partial(A^\vee )^\beta (T_\sws\times T_\swt)^*\vartheta_\Gamma(f,x)\\
		&=\sum_{\tilde{s}\in \cA_{(n)}[\overline{\frp}^n_\Sigma]} \langle \tilde{s},\swt \rangle_{[\overline{\frp}_\Sigma^n]}\cdot \partial(A_{(n)})^\alpha\partial(A^\vee_{(n)})^\beta \vartheta_{\Gamma,A_{(n)}}({f',x'+\tilde{s}+\sws'})\\
		&=\sum_{\tilde{s}\in \cA_{(n)}[\overline{\frp}^n_\Sigma]} \langle \tilde{s},\swt \rangle_{[\overline{\frp}_\Sigma^n]} \Eis_{\Gamma,A_{(n)}}^{\beta,\alpha}(f',x'+\tilde{s}+\sws')(\omega(A_{(n)})^{[\alpha+\underline{1}]},\omega(A_{(n)}^\vee)^{[\beta]})
	\end{align*}
For $(\tilde{\sws},\tilde{s})\in A_{(n)}[\frp^n_\Sigma]\times A_{(n)}[\ol{\frp}^n_\Sigma]=A_{(n)}[p^n]$ we have the formula
\begin{align*}
	P_{A}(\langle \sws, \cdot\rangle_{[\frp^n_\Sigma]}\langle  \cdot,\swt\rangle_{[\ol{\frp}^n_\Sigma]})(\tilde{\sws},\tilde{s})&= \frac{1}{p^{dn}}\sum_{t\in A_{(n)}^\vee[\frp_\Sigma^n]}\langle [\overline{\frp}_\Sigma^n](\tilde{\sws}),t \rangle_{[\frp^n_\Sigma]}^{-1} \cdot \langle \sws, t \rangle_{[\frp^n_\Sigma]}\langle  \tilde{s},\swt\rangle_{[\ol{\frp}^n_\Sigma]} \\
	&=\left(\frac{1}{p^{dn}}\sum_{t\in A_{(n)}^\vee[\frp_\Sigma^n]}\langle \sws- [\overline{\frp}_\Sigma^n](\tilde{\sws}),t \rangle_{[\frp^n_\Sigma]}\right) \cdot \langle  \tilde{s},\swt\rangle_{[\ol{\frp}^n_\Sigma]}\\
	&=\begin{cases}
			0 &  \tilde{\sws}\neq \sws'\\
			\langle  \tilde{s},\swt\rangle_{[\ol{\frp}^n_\Sigma]}  & \tilde{\sws}= \sws',
		\end{cases}
\end{align*}
and the Corollary follows.
\end{proof}

In our main application, the $\Gamma_{00}(p^\infty)$-structure on $(\cA,\cA^\vee)$ will be induced by a complex uniformization. In this case, we can make the partial Fourier transform more explicit. More precisely, let $(\cA/\cR,\omega(\cA),\omega(\cA^\vee),x)$ be an abelian variety as in Notation \ref{notation_padicAbelianVariety} and suppose it has an $\fra$-structure $\xi$ for a fractional ideal $\fra$ which is prime to $p$. By Proposition \ref{prop:GammaStructures} and Lemma \ref{lem:GammaArith}, the $\fra$-structure induces a canonical $\Gamma_{00}(p^\infty)$-structure $(\theta_p,\theta_p^\vee)$ on $(\cA,\cA^\vee)$. On the other hand, the $\fra$-structure $\xi$ and the basis $\omega(\cA)$ induce isomorphisms
\[
	\cA(\CC)\cong \CC^\Sigma/\fra \Omega, \quad\cA_{(n)}(\CC)\cong \CC^\Sigma/ \ol{\frp}^n_\Sigma \fra \Omega.
\]
In particular, we get
\begin{equation}\label{eq_Gamma_str}
	\cA_{(n)}[p^n](\ol{\QQ}) \cong {\frp}_\Sigma^{-n}\fra \Omega/\overline{\frp}_\Sigma^{n}\fra \Omega\cong {\frp}_\Sigma^{-n}\fra /\overline{\frp}_\Sigma^{n}\fra .
\end{equation}
This identification allows us to view the partial Fourier transform $P_A$ (see Definition \ref{def:Fourier-trafo}) as a map
\[
	P_{A}\colon \mathrm{Map}(\sO_L/p^n\sO_L,\overline{\QQ})\rightarrow \mathrm{Map}({\frp}_\Sigma^{-n}\fra /\overline{\frp}_\Sigma^{n}\fra ,\overline{\QQ}).
\]
By the construction of the $\Gamma_{00}(p^\infty)$-structure, we have the following commutative diagram
\begin{equation}\label{eq:Comp-P-PA}
	\xymatrix{
		\Map(\sO_L/p^n\sO_L,\ol{\QQ}) \ar[r]^-{P} \ar[rd]_-{{P_A}} & 	\Map({\frp}_\Sigma^{-n} /\overline{\frp}_\Sigma^{n} ,\overline{\QQ})\ar[d]^-{\cong} \\
		&\Map({\frp}_\Sigma^{-n}\fra /\overline{\frp}_\Sigma^{n}\fra ,\overline{\QQ}),
	}
\end{equation}
where the horizontal map is induced by the canonical isomorphism ${\frp}_\Sigma^{-n}\fra /\overline{\frp}_\Sigma^{n}\fra\cong {\frp}_\Sigma^{-n} /\overline{\frp}_\Sigma^{n}$ and $P$ is defined by
		\begin{equation}\label{eq:canonicalPFT}
			(P\rho)(y,z):=\frac{1}{p^{dn}}\sum_{x\in \sO_L/\frp^n_\Sigma} \exp\left( -2\pi i \cdot Tr_{L/\QQ}(x\cdot y) \right) \rho(x,z),
		\end{equation}
for $\rho\in \Map( \sO_L/\frp^n_\Sigma\times \sO_L/\ol{\frp}^n_\Sigma,\ol{\QQ})=\Map( \sO_L/p^n \sO_L,\ol{\QQ})$ and $(y,z)\in \frp^{-n}_\Sigma/\sO_L\times \sO_L/\ol{\frp}^n_\Sigma= \frp^{-n}_\Sigma/\ol{\frp}^n_\Sigma$.
 Note that we have a canonical projection
\begin{equation*}
	\frp^{-n}_\Sigma\sO_L\otimes \Zp\twoheadrightarrow \frp^{-n}_\Sigma/\ol{\frp}^{n}_\Sigma.
\end{equation*}
Sometimes it will be convenient to view $P\rho$ as a function on $\frp^{-n}_\Sigma\sO_L\otimes \Zp$. In particular, this allows us to evaluate $P\rho$ on any element of $\frp^{-n}_\Sigma \fra$ for a fractional ideal $\fra\subseteq \sO_L$ co-prime to $\frp_\Sigma$.

\subsection{$p$-adic Hecke characters and the local factor}\label{subsec_local_factor}
In the previous section, the integral of the $p$-adic Eisenstein measure over a locally constant function has been expressed in terms of its partial Fourier transform. It is this partial Fourier transform which will be responsible for the appearance of certain local factors associated to an algebraic Hecke character in our interpolation formula. In this section, we will introduce these local factors and discuss their basic properties.

Let $\frf$ be an integral ideal which is prime to $p$. Let us recall that we denote by $\cI(\frf)$ the group of all prime to $\frf$ fractional ideals and by $\cP_{\frf}$ all principal fractional ideals congruent to $1$ modulo $\frf$. Let $\chi\colon \cI(\frf)\rightarrow \overline{\QQ}^\times$ be an algebraic Hecke character of conductor dividing $\frf$ and of infinity type $\mu=\beta-\alpha\in I_L$ with $\alpha\in I^+_\Sigma$ and $\beta\in I^+_{\overline{\Sigma}}$, i.e.
\begin{equation}\label{eq_Hecke_infinity}
	\chi((\lambda))=\prod_{\sigma\in J_L}\sigma(\lambda)^{\mu(\lambda)},\quad \lambda\in \cP_\frf.
\end{equation}
The fixed embedding $\overline{\QQ}\rightarrow \Cp$ allows us to view $\chi$ as a $\Cp$-valued character. We deduce from \eqref{eq_Hecke_infinity} that $\chi(\cI(p\frf))\subseteq \OCp$ and $\chi(\cP_{p^n\frf})\subseteq 1+p^n\OCp$. Passing to the inverse limit allows us to view $\chi$ as a character
\[
	\chi \colon \varprojlim_n \cI(p\frf)/\cP_{p^n\frf}\rightarrow \OCp^\times.
\]
By class field theory, we may view $\chi$ as a character of the Galois group
\[
	\chi \colon \Gal(L(p^\infty\frf)/L)\rightarrow \OCp^\times.
\]
The infinity type of $\chi$ is induced from a CM subfield $K\subseteq L$, i.e., $\alpha=N_{L/K}^*\alpha_0$, $\beta=N_{L/K}^*\beta_0$ for certain $\alpha_0\in I_{\Sigma_K}$ and $\beta_0\in I_{\overline{\Sigma}_K}$.  Let us denote by $\chi_{\fin}$ the unique locally constant character 
$\chi_{\fin}\colon (\sO_L\otimes\Zp)^\times \rightarrow \overline{\QQ}^\times$ which satisfies
	\[
		\chi((\lambda))=\frac{\chi_{\fin}(\lambda) \overline{N_{L/K}(\lambda)}^{\beta_0}}{N_{L/K}(\lambda)^{\alpha_0}},
	\]
	for $\lambda$ co prime to $p\frf$. Similar as in \cite[\S 5.2]{Katz-CM} we will now associate a local term $\Local(\chi;\Sigma_p)\in \overline{\QQ}$ to the Hecke character $\chi$ and the $p$-ordinary CM type $\Sigma$.  Let us write 
	\[
		\tilde{F}\colon (\sO_L\otimes\Zp)^\times=(\sO_L(\overline{\Sigma})\oplus \sO_L(\Sigma))^\times\rightarrow \overline{\QQ}^\times
	\]
	 for the locally constant function given by $\tilde{F}(x,y):=\chi_{\fin}(x^{-1},y)$.	 Using $(\sO_L\otimes\Zp)^\times\cong \prod_{\frp\mid p} \sO_{L,\frp}^\times$, we can write $\tilde{F}$ as a product $\tilde{F}=\prod_{\frp\mid p} \tilde{F}_\frp$. We extend $\tilde{F}$ to a function on $\sO_L\otimes \Zp$ as follows. Define $F_\frp \colon \sO_{L,\frp}\to \ol{\QQ}$ by
	 \[
	 	F_\frp:=
	 	\begin{cases}
	 		1 & \text{if } \frp\in \Sigma_p \text{ and } \tilde{F}_\frp\equiv 1,\\
	 		\tilde{F}_\frp \text{ extended by zero} & \text{else}.
	 	\end{cases}
	 \]
	The locally constant function 
	\[
		F\colon \sO_L\otimes \Zp =\prod_{\frp|p} \sO_{L,\frp}\to \ol{\QQ}
	\]
	is now defined as $F:=\prod_{\frp|p}F_\frp$. For $\frp\in\Sigma_p$ let us write $a_\frp$ for the exact power of $\frp$ in $\mathrm{Cond}(\chi)$. Let us choose a decomposition of the $\Sigma_p$-part of the conductor
	\[
		\prod_{\frp\in\Sigma_p}\frp^{a_\frp}=(c)\frb,
	\]
	with $\frb$ a prime-to-$p$ fractional ideal and $c\in \cP_\frf$. Note that $c^{-1}\in \frp^{-n}_\Sigma$ for any $n\geq \max \{a_\frp\mid \frp\in \Sigma_p\}$.
	\begin{definition}
		The local factor of $\chi$ and $\Sigma$ is defined as
		\[
			\Local(\chi,\Sigma):=\frac{(N_{L/K}(c))^{\alpha_0} (PF)(c^{-1}) }{\overline{N_{L/K}(c)}^{\beta_0}\chi(\frb)}.
		\]
		\nomenclature{$\Local(\chi,\Sigma)$}{\nomrefpage}
	\end{definition}
Let us note that the local factor in non-zero and does not depend on the decomposition $\frb(c)$ of the conductor, see \cite[\S 5.2]{Katz-CM}.

\subsection{$p$-adic interpolation of Hecke $L$-values}
Let $\Sigma$ be a CM type of $L$ and $p$ a prime such that $\Sigma$ is $p$-ordinary. In this section we construct a $p$-adic measure on $\Gal(L(p^\infty\frf)/L)$ interpolating all critical Hecke $L$-values for Hecke characters of CM type $\Sigma$ and conductor dividing $p^\infty\frf$ where $\frf$ is a prime-to-$p$ fractional ideal of $L$.\par 
Let us fix an abelian variety $(\cA/\cR,\Sigma,\omega(\cA),\omega(\cA^\vee))$ as in Notation \ref{notation_padicAbelianVariety} together with a $\sO_L$-structure $\xi\colon\sO_L\cong H_1(\cA(\CC),\ZZ)$. By Proposition \ref{prop:GammaStructures} and Lemma \ref{lem:GammaArith} we get an induced $\Gamma_{00}(p^\infty)$-level structure $(\theta_p,\theta^\vee_p)$ and we denote the corresponding $p$-adic periods of $\cA$ by $\Omega_p$ and $\Omega^\vee_p$. The associated complex periods $\Omega$ and $\Omega^\vee$ are as in Definition \ref{def:periods}.\par
 The following Theorem generalizes results of Katz in case of a CM field \cite{Katz-CM}.
\begin{theorem}\label{thm_p-adic-interpolation} Let $\Sigma$ be a $p$-ordinary CM type of a totally imaginary field $L$, see Section \ref{sec:p-adic-geometric-setup}. For every fractional ideal $\frf$ co-prime to $p$ there exists a $p$-adic measure $\mu_{\frf}$ on $\Gal(L(p^\infty\frf)/L)$ with the following interpolation property: For every critical algebraic Hecke character $\chi$ with attached CM type $\Sigma$ and infinity type $\mu=\beta-\alpha$ (see Notation \ref{not:I-Sigma}) and conductor dividing $p^\infty\frf$, we have:
	\begin{multline*}
		\frac{1}{\Omega_p^{\alpha}\Omega^{\vee\beta}_p}\int_{\Gal(L(p^\infty\frf)/L)} \chi(g)d\mu_{\frf}(g)= \\
		\\ \frac{(\alpha-\ul{1})!(2\pi i)^{|\beta|}}{\Omega^{\alpha}\Omega^{\vee\beta}}
\Local(\chi,\Sigma) \prod_{\frp\in\Sigma_p}\left(1-\frac{\chi(\frp^{-1})}{N\frp}\right) \prod_{\ol{\frp}\in\ol{\Sigma}_p}\left(1-\chi(\ol{\frp})\right) L_{\frf}(\chi,0)
	\end{multline*}
\end{theorem}

Note that in the case of $L$ an extension of an imaginary quadratic field, such an interpolation was obtained by Colmez-Schneps \cite{Colmez-Schneps} for the $L$-values of Hecke characters treated in \cite{Colmez-algebraic}.
\begin{proof} 
 Using 
\[
	\prod_{\ol{\frp}\in\ol{\Sigma}_p}\left(1-\chi(\ol{\frp})\right)L_{\frf}(\chi,0)=L_{\ol{\frp}_\Sigma\frf}(\chi,0)
\]
it will often be more convenient to write $L_{\ol{\frp}_\Sigma\frf}(\chi,0)$ instead of $\prod_{\ol{\frp}\in\ol{\Sigma}_p}\left(1-\chi(\ol{\frp})\right)L_{\frf}(\chi,0)$. Let us first construct a regularized measure $\mu_{\frf,\frc}$ for $\frf$ non-trivial.\\
\textit{\textbf{Step 1:} For $\frf\neq\sO_L$ and for every auxiliary fractional ideal $\frc$ co-prime to $p\frf$ there exists a $p$-adic measure $\mu_{\frf,\frc}$ with the interpolation property:
	\begin{multline*}
		\frac{1}{\Omega_p^{\alpha}\Omega^{\vee\beta}_p}\int_{\Gal(L(p^\infty\frf)/L)} \chi(g)d\mu_{\frf,\frc}(g)= \\
		\frac{(\alpha-\ul{1})!(2\pi i)^{|\beta|}}{\Omega^{\alpha}\Omega^{\vee\beta}}\Local(\chi,\Sigma)(N\frc -\chi(\frc^{-1})) \prod_{\frp\in\Sigma_p}\left(1-\frac{\chi(\frp^{-1})}{N\frp}\right)L_{\frf\ol{\frp}_\Sigma}(\chi,0)
	\end{multline*}
	}
	By class field theory, we have a short exact sequence
	\[
		0\rightarrow G \rightarrow \Gal(L(p^\infty\frf)/L)\rightarrow \Gal(L(\frf)/L)\rightarrow 0.
	\]
	with $G=(\sO_L\otimes \Zp)^\times/E(\sO_L^\times)$ where $E(\Gamma)$ denotes the closure of a subgroup $\Gamma\subseteq \sO_L^\times$ in $(\sO_L\otimes\Zp)^\times$. By abuse of notation, let us write $\fra\in \Gal(L(p^\infty\frf)/L)$  for the image of a fractional ideal under the Artin map. Thus, if we choose prime-to-$p$ fractional ideals $\fra_i$ which are representatives of $\Gal(L(\frf)/L)$ we get an explicit co-set decomposition
	\[
		\Gal(L(p^\infty\frf)/L)=\bigcup_{i=1}^{h_\frf} \fra_i G.
	\]
	By means of this decomposition it is equivalent to give $h_\frf$ distinct measures $\mu(\fra_i)$ on $G\cong \fra_iG$:
	\begin{equation*}
		\int_{\Gal(L(p^\infty\frf)/L)} \rho d\mu_{\frf,\frc} =\sum_{i=1}^{h_\frf}\int_G \rho (\fra_i g) d\mu(\fra_i)(g).
	\end{equation*}
	Let us now define the measures $\mu(\fra_i)$: The Serre construction $\cA_i:=\fra_i^{-1}\frf\otimes_{\sO_L}\cA$ gives for every $1\leq i\leq h_\frf$ a new tuple
	\[
		(\cA_i,\omega(\cA_i),\omega(\cA_i))
	\]
	with
	\[
		\omega(\cA_i)\mapsto \omega(\cA),\quad \omega(\cA_i^\vee)\mapsto \omega(\cA^\vee)
	\]
	under the canonical isomorphisms $\omega_{\cA_i}\cong \omega_\cA$ and $\omega_{\cA_i^\vee}\cong \omega_{\cA^\vee}$. Furthermore, we get on $\cA_i$ by Lemma \ref{lem_Gamma00_isogeny} a $\Gamma_{00}(p^\infty)$-level structure $\theta_{i,p},\theta_{i,p}^\vee$. The complex uniformization induced by $\omega(\cA_i)$ is given by $\theta_{i}\colon \cA_i(\CC)\cong \CC^\Sigma/\Omega\fra_i^{-1}\frf$. For each fractional ideal $\fra_i$ let $x_i\in \cA_i(\overline{\QQ})$ be a torsion section corresponding to $\Omega\in  \CC^\Sigma$ under $\theta_{i}$. Let us write $f_{[\frc]}$ for the function introduced in \eqref{eq:special-function}. The restriction of the $p$-adic Eisenstein measure associated with 
	\[
		(\cA_i,\omega(\cA_i),\omega(\cA_i^\vee),\Sigma,x_i)
	\]
	for the group $\Gamma:=\sO_\frf^\times$ gives us a $p$-adic measure
	\[
		\mu_{Eis}(f_{[\frc]},x_i)|_{(\sO_L\otimes\Zp)^\times}\in \Meas((\sO_L\otimes\Zp)^\times,\OCp).
	\]
	Let us write
	\[
		q\colon (\sO_L\otimes\Zp)^\times\rightarrow (\sO_L\otimes\Zp)^\times/E(\sO_L^\times)
	\]
	for the projection. The desired measure $\mu(\fra_i)$ is obtained as the push-forward along $q$, i.e.,
	\begin{equation}\label{eq_measure_ai}
		\int_{(\sO_L\otimes \Zp)^\times/E(\sO_L^\times)} \rho d\mu(\fra_i):=\int_{(\sO_L\otimes \Zp)^\times} (\rho\circ q)(\rmt^{-1},\rms)\rmt^{-1}d\mu_{Eis}(f_{[\frc]},x_i)(\rmt,\rms).
	\end{equation}
	Finally, we define the measure $\mu_{\frf,\frc}$ using the formula
	\begin{equation}\label{eq_measure_decomposition}
		\int_{\Gal(L(p^\infty\frf)/L)} \rho d\mu_{\frf,\frc} =\sum_{i=1}^{h_\frf}\int_G \rho(\fra_i g) d\mu(\fra_i)(g).
	\end{equation}
It remains to check that this measure satisfies the desired interpolation property. From here on, we can follow Katz' argument quite closely. Let $\chi$ be a Hecke character of critical infinity type $\beta-\alpha$. Without loss of generality, we may assume that $\frf$ is the prime-to-$p$ part of the conductor of $\chi$. As in Section \ref{subsec_local_factor} let us denote by $\chi_{\fin}$ the unique locally constant character 
	\[
		\chi_{\fin}\colon (\sO_L\otimes\Zp)^\times \rightarrow \overline{\QQ}^\times
	\]
	associated to $\chi$ and write 
	\[
		\tilde{F}\colon (\sO_L\otimes\Zp)^\times=(\sO_L(\overline{\Sigma})\oplus \sO_L(\Sigma))^\times\rightarrow \overline{\QQ}^\times
	\]
	 for the locally constant function given by $\tilde{F}(x,y):=\chi_{\fin}(x^{-1},y)$.
	The extension of $\tilde{F}$  by zero to $\sO_L\otimes \Zp$ will also be denoted by $\tilde{F}$. Choose $n$ sufficiently large such that the locally constant function $\tilde{F}$ is constant modulo $p^n$. By Corollary \ref{cor_rho_integration} we get
	\begin{align}
		\label{eq_integral}&\frac{1}{\Omega_p^{\alpha}\Omega^{\vee\beta}_p}\int_{\Gal(L(p^\infty\frf)/L)}\chi(g)d\mu_{\frf,\frc}(g)\\
		\notag &=\sum_{i=1}^{h_\frf} \chi(\fra_i)\sum_{\tilde{s}\in A_{i,(n)}[p^n]}(P_i\tilde{F})(\tilde{s}) \Eis_{\Gamma,A_{i,(n)}}^{\beta,\alpha-\underline{1}}(f_{[\frc]}',x_i'+\tilde{s})(\omega(\cA_{i,(n)})^{[\alpha]},\omega(\cA^\vee_{i,(n)})^{[\beta]}),
	\end{align}
	where we write $P_i$ for $P_{A_{i}}$ defined in \ref{def:Fourier-trafo}. Under the complex uniformization $\cA_{i,(n)}(\CC)\cong \CC^\Sigma/\ol{\frp}^n_\Sigma \frf\fra_i^{-1}\Omega$, we get
	\[
		\cA_{i,(n)}[p^n](\ol{\QQ})\cong \frp^{-n}_\Sigma \frf\fra_i^{-1}\Omega/\ol{\frp}^n_\Sigma \frf\fra_i^{-1}\Omega\cong \frp^{-n}_\Sigma \frf\fra_i^{-1}/\ol{\frp}^n_\Sigma \frf\fra_i^{-1} 
	\]
	and the point $x_i'\in \cA_{i,(n)}(\ol{\QQ})$ corresponds to $\Omega+\ol{\frp}^n_\Sigma \frf\fra_i^{-1}\Omega\in \CC^\Sigma/\ol{\frp}^n_\Sigma \frf\fra_i^{-1}\Omega$. Using Theorem \ref{thm:computation-of-Eis} and observing the relation between $P_i$ and $P$ from \eqref{eq:Comp-P-PA}, we obtain
	\begin{align*}
		&\left( \frac{(\alpha-\ul{1})!(2\pi i)^{|\beta|}}{\Omega^{\alpha}\Omega^{\vee\beta}} \right)^{-1}\sum_{\tilde{s}\in A_{i,(n)}[p^n]}(P_i\tilde{F})(\tilde{s}) \Eis_{\Gamma,A_{i,(n)}}^{\beta,\alpha-\underline{1}}(f_{[\frc]}',x_i'+\tilde{s})(\omega(\cA_{i,(n)})^{[\alpha]},\omega(\cA^\vee_{i,(n)})^{[\beta]})=\\
		&=\left(\left.N\frc \sum_{\lambda\in \Gamma\backslash (1+\frp^{-n}_\Sigma\frf\fra^{-1}_i)}\frac{ (P\tilde{F})(\lambda ) \overline{\lambda}^\beta}{\lambda^{\alpha}N(\lambda)^s}-\sum_{\lambda\in \Gamma\backslash (1+\frc^{-1}\frp^{-n}_\Sigma\frf\fra^{-1}_i)}\frac{(P\tilde{F})(\lambda) \overline{\lambda}^\beta}{\lambda^{\alpha}N(\lambda)^s}\right)\right|_{s=0}
	\end{align*}
		As is Section \ref{subsec_local_factor}, we will denote multiplicity of $\frp\in \Sigma_p$ in $\mathrm{Cond}(\chi)$ by $a_\frp$ and we choose a decomposition
		\[
			\prod_{\frp\in \Sigma_p} \frp^{a_\frp}=c\cdot \frb
		\]
	with $c\in 1+\frf$ and $\frb$ co-prime to $p$. Let us first assume that $a_\frp\geq 1$ for all $\frp\in\Sigma_p$.
	The partial Fourier transform $P\tilde{F}$ on $(1+\frp^{-n}_\Sigma\frf\fra^{-1}_i)\cap \frp^{-n}_\Sigma\otimes \Zp$ is supported in 
	\[
		1+\prod_{\frp\in \Sigma_p} \frp^{-a_\frp}\frf\fra^{-1}_i=c^{-1}(c+\frf\fra^{-1}_i\frb^{-1})=c^{-1}(1+\frf\fra^{-1}_i\frb^{-1}).
	\]
	Since $a_\frp\geq 1$, the function $\tilde{F}$ coincides with the extension $F$ defined in Section \ref{subsec_local_factor}. For $\lambda\in 1+\frf\fra^{-1}_i\frb^{-1}$ a direct computation shows
	\[
		(P\tilde{F})(c^{-1}\lambda)=(PF)(c^{-1})\cdot \chi_{\fin}(\lambda)
	\]
	and together with Theorem \ref{thm:computation-of-Eis} we obtain the formula:
	\begin{align*}
		&\left( \frac{(\alpha-\ul{1})!(2\pi i)^{|\beta|}}{\Omega^{\alpha}\Omega^{\vee\beta}} \right)^{-1}\sum_{\tilde{s}\in A_{i,(n)}[p^n]}(P_i\tilde{F})(\tilde{s}) \Eis_{\Gamma,A_{i,(n)}}^{\beta,\alpha-\underline{1}}(f_{[\frc]}',x_i'+\tilde{s})(\omega(\cA_{i,(n)})^{[\alpha]},\omega(\cA^\vee_{i,(n)})^{[\beta]})=\\
		&=\frac{N_{L/K}(c)^{\alpha_0} (PF)(c^{-1}) }{\overline{N_{L/K}(c)}^{\beta_0}}\left(\left.N\frc \sum_{\lambda\in \Gamma\backslash (1+\frf\fra^{-1}_i\frb^{-1})}\frac{\chi_{\fin}(\lambda) \overline{\lambda}^\beta}{\lambda^{\alpha}N(\lambda)^s}-\sum_{\lambda\in \Gamma\backslash (1+\frc^{-1}\frf\fra^{-1}_i\frb^{-1})}\frac{\chi_{\fin}(\lambda) \overline{\lambda}^\beta}{\lambda^{\alpha}N(\lambda)^s}\right)\right|_{s=0}=\\
		&=\frac{N_{L/K}(c)^{\alpha_0} (PF)(c^{-1}) }{\overline{N_{L/K}(c)}^{\beta_0}}\chi(\fra_i^{-1}\frb^{-1})\left(N\frc L_{p\frf}(\chi,0,[\fra_i\frb]) -\chi(\frc^{-1})L_{p\frf}(\chi,0,[\frc\fra_i\frb])\right)=\\
		&=\Local(\chi,\Sigma)\chi(\fra_i^{-1})\left(N\frc L_{p\frf}(\chi,0,[\fra_i]) -\chi(\frc^{-1})L_{p\frf}(\chi,0,[\frc\fra_i])\right).
	\end{align*}
	Summing over $\fra_i$ gives
	\begin{align*}
		&\frac{1}{\Omega_p^{\alpha}\Omega^{\vee\beta}_p}\int_{\Gal(L(p^\infty\frf)/L)}\chi(g)d\mu_{\frf,\frc}(g)=\\
		=&\frac{(\alpha-\ul{1})!(2\pi i)^{|\beta|}}{\Omega^{\alpha}\Omega^{\vee\beta}} \Local(\chi,\Sigma) (N\frc -\chi(\frc^{-1}))L_{p\frf}(\chi,0).
	\end{align*}
	Since $\chi(\frp)=0$ for any $\frp\in \Sigma_p$, we have $L_{p\frf}(\chi,0)=L_{\ol{\frp}_\Sigma\frf}(\chi,0) $ the claim follows in the case $a_\frp\geq 1$.\par 
	Let us now consider the case of general $a_\frp$. We follow essentially the computations in \cite{Katz-CM}. Let us define $I:=\{\frp\in \Sigma_p\mid a_\frp=0\}$ and write $\frp_I:=\prod_{\frp\in I}\frp$. In this case, the partial Fourier transform $P\tilde{F}$ is supported in
	\[
		1+\frp_I^{-1}\left(\prod_{a_\frp\geq 1} \frp^{-a_\frp}\right)\frf\fra^{-1}_i=c^{-1}(c+\frf\frp_I^{-1}\fra^{-1}_i\frb^{-1})=c^{-1}(1+\frf\frp_I^{-1}\fra^{-1}_i\frb^{-1}).
	\]
	Following Katz' notation in \cite[(5.5)]{Katz-CM}, we write $\chi_1$ for the Hecke character $\chi$ seen as a character on fractional ideals prime to $p\cdot \frf$, and $\chi_2$ for the Hecke character $\chi$ seen as a character on fractional ideals prime to $\prod_{a_\frp\geq 1}\frp \cdot \ol{\frp}_\Sigma \cdot \frf$.
	A straightforward computation shows for $\lambda\in 1+\frf\frp_I^{-1}\fra^{-1}_i\frb^{-1}$ the formula
	\begin{equation}\label{eq_PF_case2}
		(P\tilde{F})(c^{-1}\lambda)=(PF)(c^{-1})\chi_{2,\fin}(\lambda)\Chr(\lambda)
	\end{equation}
	where $F$ is the extension appearing in Section \ref{subsec_local_factor} and
	\[
		\Chr(\lambda)=\prod_{a_\frp=0}\Chr_{\frp}(\lambda),\quad \Chr_{\frp}(\lambda):=
		\begin{cases} 
			1-\frac{1}{N\frp} & \ord_\frp(\lambda)=0\\ 
			-\frac{1}{N\frp} & \ord_\frp(\lambda)=-1\\ 
			0 & \text{else}.
		\end{cases}
	\]
Now essentially the same computation as above allows us to compute
	\begin{multline*}
		\frac{1}{\Omega_p^{\alpha}\Omega^{\vee\beta}_p}\int_{\Gal(L(p^\infty\frf)/L)} \chi(g)d\mu_{\frf,\frc}(g)= \\
		\frac{(\alpha-\ul{1})!(2\pi i)^{|\beta|}}{\Omega^{\alpha}\Omega^{\vee\beta}} \Local(\chi,\Sigma) (N\frc -\chi(\frc^{-1})) \prod_{\frp\in\Sigma_p}\left(1-\frac{\chi(\frp^{-1})}{N\frp}\right) L_{\ol{\frp}_\Sigma\frf}(\chi,0).
	\end{multline*}
	In a next step, let us complete the proof in the case that $\frf$ is non-trivial:\\
	\textit{\textbf{Step 2:}
	For $\frf\neq\sO_L$ there exists a $p$-adic measure $\mu_{\frf}$ with the interpolation property:
	\begin{align*}
		&\frac{1}{\Omega_p^{\alpha}\Omega^{\vee\beta}_p}\int_{\Gal(L(p^\infty\frf)/L)} \chi(g)d\mu_{\frf}(g)= \\
		=& \frac{(\alpha-\ul{1})!(2\pi i)^{|\beta|}}{\Omega^{\alpha}\Omega^{\vee\beta}} \Local(\chi,\Sigma) \prod_{\frp\in\Sigma_p}\left(1-\frac{\chi(\frp^{-1})}{N\frp}\right)L_{\frf\ol{\frp}_\Sigma}(\chi,0).
	\end{align*}
	}
	Getting rid of the auxiliary ideal $\frc$ is a standard argument. We follows de Shalit's book \cite{deShalit} closely. Let us first observe, that for a given ideal $\frc$ the map $\chi\mapsto N\frc -\chi(\frc^{-1})$ defines a $p$-adic measure on $\Gal(L(p^\infty\frf)/L)$. More precisely, we may view $\frc\in \Gal(L(p^\infty\frf)/L)$ as a Galois automorphism via the Artin map. The $p$-adic measure $\chi\mapsto \frc -\chi(\frc^{-1})$ corresponds to $\delta_\frc:=N\frc-\frc^{-1}\in \Zp \llbracket \Gal(L(p^\infty\frf)/L) \rrbracket$ under $\Zp \llbracket \Gal(L(p^\infty\frf)/L) \rrbracket\cong \Meas(\Gal(L(p^\infty\frf)/L),\Zp) $. For two ideals $\frc,\frb\subseteq\sO_L$ with $(\frb\frc,p\frf)=1$ the measures
	\[
		\delta_\frc \mu_{\frf,\frb}=\delta_\frb\mu_{\frf,\frc}
	\]
	coincide, since they coincide on all Hecke characters and the Hecke characters are dense in the continuous functions on $\Gal(L(p^\infty\frf)/L)$. The quotient $\mu_{\frf,\frc}/\delta_\frc$ is a priori only a pseudo-measure and the above equation shows that it is independent of $\frc$ and one can proceed as in \cite[Proof of 4.12]{deShalit}. The main point is, roughly speaking, that the greatest common divisor of the $\delta_\frc$ is $1$.\par 
	It remains to construct the $p$-adic measure in the case $\frf=\sO_L$. We proceed in two steps as in the construction for non-trivial conductor. The only difference is that the torsion section $x=\Omega+\Omega\fra_i^{-1}\frf$ corresponds to zero if $\frf=\sO_L$. So the choice $f=f_{[\frc]}$ does not allow to specialize along $x=0$. This problem can be addressed by adapting the function $f_{[\frc]}$ with an additional auxiliary ideal $\frc'$ as in the proof of Theorem \ref{thm:special-values}:\par
	\textit{\textbf{Step 3:} For auxiliary ideals $\frc$ and $\frc'$ co-prime to each other and co-prime to $p$ there exists a $p$-adic measure $\mu_{\sO_L,\frc,\frc'}$ with the interpolation property:
	\begin{multline*}
		\frac{1}{\Omega_p^{\alpha}\Omega^{\vee\beta}_p}\int_{\Gal(L(p^\infty)/L)} \chi(g)d\mu_{\sO_L,\frc,\frc'}(g)= \\
		\frac{(\alpha-\ul{1})!(2\pi i)^{|\beta|}}{\Omega^{\alpha}\Omega^{\vee\beta}} \Local(\chi,\Sigma) (1-\chi(\frc'))(N\frc -\chi(\frc^{-1})) \prod_{\frp\in\Sigma_p}\left(1-\frac{\chi(\frp^{-1})}{N\frp}\right)L_{\ol{\frp}}(\chi,0).
	\end{multline*}}
Let us choose prime-to-$p$ fractional ideals $\fra_i$ which are representatives of $\Gal(L(1)/L)$:
	\[
		\Gal(L(p^\infty\frf)/L)=\bigcup_{i=1}^{h} \fra_i G.
	\]
	By means of this decomposition it is equivalent to give $h$ distinct measures $\mu(\fra_i)$ on $G\cong \fra_iG$:
	\begin{equation*}
		\int_{\Gal(L(p^\infty)/L)} \rho d\mu_{\sO_L,\frc,\frc'} =\sum_{i=1}^{h}\int_G f(\fra_i g) d\mu_{\sO_L,\frc,\frc'} (\fra_i)(g).
	\end{equation*}
	Let us now define the measures $\mu_{\sO_L,\frc,\frc'} (\fra_i)$: As above, the Serre construction $\cA_i:=\fra_i^{-1}\otimes_{\sO_L}\cA$ gives for every $1\leq i\leq h$ a new tuple
	\[
		(\cA_i,\omega(\cA_i),\omega(\cA_i))
	\]
	with
	\[
		\omega(\cA_i)\mapsto \omega(\cA),\quad \omega(\cA_i^\vee)\mapsto \omega(\cA^\vee)
	\]
	under the canonical isomorphisms $\omega_{\cA_i}\cong \omega_\cA$ and $\omega_{\cA_i^\vee}\cong \omega_{\cA^\vee}$. In this case, we choose the function $\widetilde{f}_{[\frc]}$ on $(\cA[\frc']\smallsetminus\{e\})\times \cA[\frc]$ as the pullback of $f_{[\frc]}$ along the projection to the second factor. With this choice, the torsion section $x_i=e$ is a valid choice for the construction of the Eisenstein measure. Define
	\[
		\mu'_{\sO_L,\frc,\frc'}(\fra_i)\in\Meas((\sO_L\otimes\Zp)^\times,\OCp)
	\]
	as the $p$-adic Eisenstein measure $\mu_{Eis}(\widetilde{f}_{[\frc]},e)|_{(\sO_L\otimes\Zp)^\times}$ associated with 
	\[
		(\cA_i,\omega(\cA_i),\omega(\cA_i^\vee),\Sigma,x_i).
	\]
 The desired measure $\mu_{\sO_L,\frc,\frc'}(\fra_i)$ is obtained as the push-forward along the quotient map $q\colon (\sO_L\otimes\Zp)^\times \to (\sO_L\otimes\Zp)^\times/E(\cO_L^\times)$, i.e.,
	\begin{equation*}
		\int_{(\sO_L\otimes \Zp)^\times/E(\sO_L^\times)} \rho d\mu_{\sO_L,\frc,\frc'}(\fra_i):=\int_{(\sO_L\otimes \Zp)^\times} (q\circ\rho)(\rmt^{-1},\rms)\rmt^{-1}d\mu'_{\sO_L,\frc,\frc'}(\fra_i)(\rmt,\rms).
	\end{equation*}
	We define the measure $\mu_{\sO_L,\frc,\frc'}$ using the formula
	\begin{equation*}
		\int_{\Gal(L(p^\infty)/L)} \rho d\mu_{\sO_L,\frc,\frc'} =\sum_{i=1}^{h}\int_G f(\fra_i g) d\mu_{\sO_L,\frc,\frc'}(\fra_i)(g).
	\end{equation*}	
	From here on, the proof of Step 3 follows exactly the same lines as in Step 1.\par 
	Finally, we have to get rid of the regularization in the case $\frf=\sO_L$.\par 
	\textit{\textbf{Step 4:} In the case $\frf=\sO_L$ there exists a $p$-adic measure $\mu_{\sO_L}$ with the interpolation property:
	\begin{multline*}
		\frac{1}{\Omega_p^{\alpha}\Omega^{\vee\beta}_p}\int_{\Gal(L(p^\infty)/L)} \chi(g)d\mu_{\sO_L}(g)= \\
		\frac{(\alpha-\ul{1})!(2\pi i)^{|\beta|}}{\Omega^{\alpha}\Omega^{\vee\beta}} \Local(\chi,\Sigma) \prod_{\frp\in\Sigma_p}\left(1-\frac{\chi(\frp^{-1})}{N\frp}\right) L_{\ol{\frp}_\Sigma}(\chi,0).
	\end{multline*}}
	Let us first get rid of $\frc'$: Let us define $\delta'_{\frc'}:=1-\frc'$ seen as a $p$-adic measure on $\Gal(L(p^\infty)/L)$. We claim that the pseudo-measure $\mu_{\frc,\frc'}/\delta'_{\frc'}$ is an actual measure. The value of the pseudo-measure $\mu_{\frc,\frc'}/\delta'_{\frc'}$ on $\chi$ is given by the formula
	\[
		\frac{(\alpha-\ul{1})!(2\pi i)^{|\beta|}}{\Omega^{\alpha}\Omega^{\vee\beta}}\Local(\chi,\Sigma) (N\frc -\chi(\frc^{-1}))\prod_{\frp\in\Sigma_p}\left(1-\frac{\chi(\frp^{-1})}{N\frp}\right) L_{\ol{\frp}_\Sigma}(\chi,0).
	\]
	Applying Theorem \ref{thm:special-values} with $\frf=\ol{\frp}_\Sigma$ shows that $(\mu_{\frc,\frc'}/\delta'_{\frc'})(\chi)\in\OCp$ is integral at $p$. Since every $\OCp$-valued continuous functions on $\OCp$ can be approximated by a linear combination of Hecke characters, we conclude that $\mu_{\sO_L,\frc}:=\mu_{\frc,\frc'}/\delta'_{\frc'}$ is a $p$-adic measure. Finally, the same argument as in Step 2 shows the existence of the measure $\mu_{\sO_L}$ with the interpolation property:
	\begin{multline*}
		\frac{1}{\Omega_p^{\alpha}\Omega^{\vee\beta}_p}\int_{\Gal(L(p^\infty)/L)} \chi(g)d\mu_{\sO_L}(g)= \\
		\frac{(\alpha-\ul{1})!(2\pi i)^{|\beta|}}{\Omega^{\alpha}\Omega^{\vee\beta}}\Local(\chi,\Sigma) \prod_{\frp\in\Sigma_p}\left(1-\frac{\chi(\frp^{-1})}{N\frp}\right) L_{\ol{\frp}_\Sigma}(\chi,0).\qedhere
	\end{multline*}
\end{proof}
\appendix
\section{Equivariant cohomology}\label{sec:appendix_equivariant}
The theory of equivariant cohomology is treated by Grothendieck 
\cite{Grothendieck} chapter V. For convenience of the reader we recall the definitions we need in this appendix and include also  a device of Borel computing equivariant cohomology. 
\subsection{Equivariant  cohomology}\label{app1}
Consider pairs $(X,\sF)$ consisting of a ringed space $(X,\sO_X)$ and an $\sO_X$-module $\sF$ on it. A morphism $(f,f^{\#}):(X,\sF)\to (Y,\sG)$ consists of a morphism of ringed spaces
$f:X\to Y$ and an $\sO_Y$-module homomorphism $f^{\#}:\sG\to f_*\sF$ (or equivalently an $\sO_X$-module homomorphism $f^{*}\sG\to \sF$). 
\begin{definition}
Let $\Gamma$ be a group acting on $X$ (from the left). Then $\Gamma$ acts also on $(X,\sO_X)$. A $\Gamma$-equivariant $\sO_X$-module $\sF$ (a $\Gamma$-$\sO_X$-module in the terminology of \cite{Grothendieck}) on $X$ is a pair $(X,\sF)$ as above with a group homomorphism 
\begin{equation*}
\Gamma\to \Aut(X,\sF)
\end{equation*} 
such that the induced action on $X$ coincides with the given one. The category of $\Gamma$-equivariant $\sO_X$-modules is abelian and has enough injectives \cite{Grothendieck} Prop. 5.1.1. 
\end{definition}

Let $\pi:X\to S$ be morphism of ringed spaces and $\Gamma$ be a group acting on $X/S$, i.e., each $\gamma\in G$ gives an automorphism $\gamma:X\to X$ over $S$.  Then higher direct images $R^{i}\pi_*\sF$ of a $\Gamma$-equivariant sheaf $\sF$ 
inherit a $\Gamma$-action. Note that 
\begin{equation*}
\Hom_{\Gamma,\sO_X}(\sF,\sG)\isom \Hom_{\sO_X}(\sF,\sG)^{\Gamma}
\end{equation*}
so that in particular $\Hom_{\Gamma,\sO_X}(\sO_X,\sG)\isom \Gamma(X,\sF)^{\Gamma}$.
\begin{definition} Let $\sF$ and $\sG$  be $\Gamma$-equivariant $\sO_X$-modules on $X$. 
The derived functors of $\sF\mapsto H^{0}(X,\sF)^{\Gamma}$ are the \emph{equivariant cohomology groups} which we denote by
\nomenclature{$H^{i}(X,\Gamma;\sF)$}{\nomrefpage}
\begin{equation*}
H^{i}(X,\Gamma;\sF)
\end{equation*}
and the derived functors of $\sG\mapsto \Hom_{\Gamma,\sO_X}(\sF,\sG)$ are denoted by \nomenclature{$\Ext^{i}_{\Gamma,\sO_X}(\sF,\sG)$}{\nomrefpage}
\begin{equation*}
\Ext^{i}_{\Gamma,\sO_X}(\sF,\sG)
\end{equation*}
and called the equivariant Ext-groups.
\end{definition}
Immediate from this definition are the spectral sequences
\begin{align}\begin{split}\label{eq:coh-equiv-spectral-seq}
E_2^{p,q}=H^{p}(\Gamma,H^{q}(X,\sF))&\Rightarrow H^{p+q}(X,\Gamma;\sF)\\
E_2^{p,q}=H^{p}(\Gamma,\Ext^{q}_{\sO_X}(\sF,\sG))&\Rightarrow \Ext^{p+q}_{\Gamma,\sO_X}(\sF,\sG).
\end{split}
\end{align}
Let $i:D\subset X$ be a closed immersion of ringed spaces which is stable under the $\Gamma$-action and $j:U:=X\setminus D\hookrightarrow X$ the complement. Then one has an exact sequence of $\Gamma$-equivariant $\sO_X$-modules
\begin{equation*}
0\to j_!\sO_U\to \sO_X\to i_*i^{-1}\sO_X\to 0.
\end{equation*}
Then taking $\Hom_{\sO_X}(-,\sF)$ gives 
\begin{equation*}
0\to \Hom_{\sO_X}(i_*i^{-1}\sO_X,\sF)\to \Gamma(X,\sF)\to \Gamma(U,\sF)
\end{equation*}
which shows that $\Hom_{\sO_X}(i_*i^{-1}\sO_X,\sF)$ is isomorphic to $\Gamma_D(X,\sF)$, the sections with support in $D$. Note that $\Gamma_D(X,\sF)$ is a $\Gamma$-module.
\begin{definition}\label{def:coh-with-support}
The equivariant cohomology with support in $D$ are the derived functors of $\sF\mapsto \Gamma_D(X,\sF)^{\Gamma}$ or equivalently of $\sF\mapsto \Hom_{\Gamma,\sO_X}(i_*i^{-1}\sO_X,\sF)$ and are denoted by 
\begin{equation*}
H^{i}_D(X,\Gamma;\sF).
\end{equation*}
\end{definition}
By construction one has a long exact cohomology sequence
\begin{equation*}
\cdots\to H^{i}_D(X,\Gamma;\sF)\to H^{i}(X,\Gamma;\sF)\to H^{i}(U,\Gamma;\sF)\to H^{i+1}_D(X,\Gamma;\sF)\to \cdots
\end{equation*}
and a spectral sequence
\begin{equation*}
E_2^{p,q}=H^{p}(\Gamma,H^{q}_D(X,\sF))\Rightarrow H^{p+q}_D(X,\Gamma;\sF).
\end{equation*}

\subsection{Borel-Equivariant sheaf cohomology} \label{subsec_equivariantmodel} We explain how equivariant cohomology can be computed for smooth manifolds using classifying spaces. We will only concentrate on the case of finitely generated free abelian group $\Gamma$. We use
\[
 E\Gamma=\Gamma\otimes_\ZZ \RR\rightarrow B=(\Gamma\otimes_\ZZ\RR)/\Gamma
\]
as a model for the universal $\Gamma$-bundle $E\Gamma$ over the classifying space $B$. Let $X$ be a smooth manifold with a  $\Gamma$-action and let  $\sF$ be a $\Gamma$-equivariant sheaf of abelian groups on $X$. For any $\Gamma$-equivariant map $f\colon X\to Y$ between smooth manifolds, we define the sheaf $f^\Gamma_*\sF$ as the sheaf of $\Gamma$-invariant sections of $f_*\sF$, i.e. for any open subsets $U\subseteq Y$ we have
\[
	\Gamma(U, f^\Gamma_*\sF):=\Gamma(U,f_*\sF)^\Gamma.
\]
We consider the diagram
\begin{equation*}
E\Gamma\times_\Gamma X\xleftarrow{q}E\Gamma\times X\xrightarrow{p}X
\end{equation*}
where $E\Gamma\times_\Gamma X$ is the orbit space of the diagonal action of $\Gamma$ on $E\Gamma\times X$.
The action of $\Gamma$ is free and properly discontinuously on $E\times X$ so that  by \cite[\S 5.1]{Grothendieck} the functors $q_*^{\Gamma} $ and $q^{-1}$ give an equivalence of categories between $\Gamma$-equivariant sheaves on $E\times X$ and sheaves on $E\times_\Gamma X$. In particular, $q_*^{\Gamma}$ is exact and transforms injective sheaves to injectives.
Consider $\widetilde{\sF}:=q_*^{\Gamma}p^{-1}\sF$ and define 
\[
	H^i_\Gamma(X,\sF):=H^i(E\Gamma\times_\Gamma X,\widetilde{\sF}).
\]
Let us quickly recall the proof that 
\begin{equation}\label{eq:coh-equiv}
	H^i_\Gamma(X,\sF)\cong H^i(X,\Gamma;\sF).
\end{equation}
First note that $\Gamma(E\Gamma\times_\Gamma X,\widetilde{\sF})\isom \Gamma(X,\sF)^{\Gamma}$. Further $p:E\Gamma\times X\to X$ is acyclic because $E\Gamma$ is contractible. Hence, 
$\sF\isom p_*p^{-1}\sF$. 

To show that the derived functors coincide, it remains to see that $\widetilde{\sI}$ is flabby on $E\Gamma\times_\Gamma X$ for an injective $\Gamma$-equivariant sheaf $\sI$ on $X$.  For each $\Gamma$-invariant open $j:V\subset X$ one has 
\begin{equation*}
\Gamma(V,\widetilde{\sI})^{\Gamma}\isom \Hom_{\Gamma}(j_!\Z,\sI)
\end{equation*}
so that 
\begin{equation*}
\Gamma(X,\sI)^{\Gamma}\to \Gamma(V,\sI)^{\Gamma}
\end{equation*}
is surjective. 
Let $U\subset E\Gamma\times_\Gamma X$ be open. Then
\begin{equation*}
\Gamma(U,\widetilde{\sI})\isom \Gamma(q^{-1}(U),p^{-1}\sI)^{\Gamma}\isom \Gamma(p(q^{-1}(U)),\sI)^{\Gamma}
\end{equation*}
because $p$ is an open map (view $\sI$ as \'etale space over $X$). As $p(q^{-1}(U))\subset X$ is  $\Gamma$-invariant open this shows that 
\begin{equation*}
\Gamma(E\Gamma\times_\Gamma X,\widetilde{\sI})\to \Gamma(U,\widetilde{\sI})
\end{equation*}
is surjective. 

\section{Derived limits}
For the convenience of the reader we collect some results in the stack project about derived limits. We warn the reader that derived limits are not derived functors of limits. 
\begin{definition}
Let $(K_n,f_n)$, $n\in \N$ be an inverse system of objects in a triangulated category $\sD$, then $K\in \sD$ is a \emph{derived limit} of $(K_n,f_n)$ if the product $\prod_n K_n$ exists and there is a distinguished triangle
\begin{equation*}
K\to \prod_nK_n\xrightarrow{k_n\mapsto k_n-f_{n+1}(k_{n+1})}\prod_nK_n\to K[1].
\end{equation*}
One writes $K=R\prolim K_n$.\nomenclature{$R\prolim$}{\nomrefpage}
\end{definition}
The derived limit, if it exists, is only  unique up to non-unique isomorphism.
\begin{lemma}\label{lemma:derived-limit-exact-sequence}
If $R\prolim K_n$ is a derived limit in $\sD$, then one has for each object $L$ in $\sD$ an exact sequence
\begin{equation*}
0\to R^{1}\prolim_n \Hom_\sD(L,K_n[-1])\to\Hom_\sD(L,R\prolim K_n)\to \prolim_n \Hom_\sD(L,K_n)\to 0. 
\end{equation*}
\end{lemma}
\begin{proof}
This is \cite[\href{https://stacks.math.columbia.edu/tag/0919}{Tag 0919}]{stacks-project}.
\end{proof}
\begin{lemma}\label{lemma:derived-limits}
Let $X$ be a scheme with an action of a (discrete) group $\Gamma$  and  $\sD(\sO_X,\Gamma)$ the derived category of  $\Gamma$-equivariant $\sO_X$-modules as in Appendix \ref{sec:appendix_equivariant}. Then $\sD(\sO_X,\Gamma)$ have derived limits and 
\begin{enumerate}
\item If $f:X\to Y$ is a morphism of schemes, $Rf_*$ commutes with derived limits.
\item If $(\sF_n)$ is an inverse system of quasi-coherent $\sO_X$-modules with surjective transition maps, then $R\prolim\sF_n=\prolim_n\sF_n$.
\item If $(\sF_n)$ is an inverse system of quasi-coherent $\Gamma$-equivariant $\sO_X$-modules one has for each $p\ge 0$ and $\sG\in \sD(\sO_X,\Gamma)$
\begin{equation*}
0\to R^{1}\prolim_n\Ext^{p-1}_{\Gamma,\sO_X}(\sG,\sF_n)\to\Ext^{p}_{\Gamma,\sO_X}(\sG,R\prolim \sF_n)\to \prolim_n\Ext^{p}_{\Gamma,\sO_X}(\sG,\sF_n)\to 0. 
\end{equation*}
\end{enumerate}
\end{lemma}
\begin{proof} By Grothendieck \cite{Grothendieck} Prop. 5.1.1 the category $\sD(\sO_X,\Gamma)$ has products and enough injectives, it follows from Lemma 13.34.3 in \cite[\href{https://stacks.math.columbia.edu/tag/0BK7}{Tag 0BK7}]{stacks-project} that the $R\prolim$ exists in $\sD(\sO_X,\Gamma)$. 
(1) is  \cite[\href{https://stacks.math.columbia.edu/tag/0BKP}{Tag 0BKP}]{stacks-project} and (2) is a special case of   \cite[\href{https://stacks.math.columbia.edu/tag/0BKS}{Tag 0BKS}]{stacks-project}. (3) follows from Lemma \ref{lemma:derived-limit-exact-sequence}.
\end{proof}
\begin{corollary}\label{cor:limit-sequences}
With the notation in Definition \ref{def:coh-with-support} one has exact sequences
\begin{equation*}
0\to R^{1}\prolim_nH^{p-1}_D(X,\Gamma;\sF_n)\to H^{p}_D(X,\Gamma;R\prolim \sF_n)
\to \prolim_nH^{p}_D(X,\Gamma;\sF_n)\to 0
\end{equation*}
and similarly for $H^{p}(X,\Gamma;\sF_n)$ and $H^{p}(U,\Gamma;\sF_n)$.
\end{corollary}
\begin{proof}
This follows from Lemma \ref{lemma:derived-limits} and the description in terms of Ext-groups, see Definition \ref{def:coh-with-support}. 
\end{proof}

 \printnomenclature

\renewcommand{\MR}[1]{\relax\ifhmode\unskip\space\fi MR \MRhref{#1}{#1}}
\renewcommand{\MRhref}[2]{%
  \href{http://www.ams.org/mathscinet-getitem?mr=#1}{#2}.
}
\providecommand{\bysame}{\leavevmode\hbox to3em{\hrulefill}\thinspace}

\bibliographystyle{amsalpha}
\bibliography{poincare-eisenstein}
\end{document}